\documentclass[11pt,reqno]{amsart}  %LaTeX2e
\usepackage{amsmath,amssymb,amsthm}
\theoremstyle{plain}
\newtheorem{theorem}{Theorem}[section]
\newtheorem{lemma}[theorem]{Lemma}
\newtheorem{proposition}[theorem]{Proposition}

\newtheorem{corollary}[theorem]{Corollary}

\numberwithin{equation}{section}
\allowdisplaybreaks[1]

\theoremstyle{definition}

\newtheorem{definition}[theorem]{Definition}
\newtheorem{remark}[theorem]{Remark}
\newtheorem{example}[theorem]{Example}

\theoremstyle{remark}

\newcommand{\ev}{\boldsymbol{\rm ev}}
\newcommand{\free}{{\mathbb F}^{+}_{d}}

\newcommand{\bA}{{\mathbf A}}

\newcommand{\bQ}{{\mathbf Q}}
\newcommand{\bU}{{\mathbf U}}
\newcommand{\bV}{{\mathbf V}}
\newcommand{\bZ}{{\mathbf Z}}

\newcommand{\fa}{{\mathfrak a}}
\newcommand{\fb}{{\mathfrak b}}
\newcommand{\dBR}{{\rm dBR}}

\newcommand{\biota}{{\boldsymbol \iota}}

\newcommand{\bcH}{{\boldsymbol{\mathcal H}}}

\newcommand{\bchi}{\boldsymbol{\chi}}

\newcommand{\bphi}{{\boldsymbol \varphi}}

\newcommand{\ba}{{\mathbf a}}
\newcommand{\bb}{{\mathbf b}}

\newcommand{\boldf}{{\mathbf f}}
\newcommand{\bg}{{\mathbf g}}
\newcommand{\bh}{{\mathbf h}}
\newcommand{\bi}{{\mathbf i}}

\newcommand{\bx}{{\mathbf x}}
\newcommand{\by}{{\mathbf y}}

\newcommand{\cA}{{\mathcal A}}
\newcommand{\cB}{{\mathcal B}}
\newcommand{\cC}{{\mathcal C}}
\newcommand{\cD}{{\mathcal D}}
\newcommand{\cE}{{\mathcal E}}
\newcommand{\cF}{{\mathcal F}}

\newcommand{\cH}{{\mathcal H}}
\newcommand{\cI}{{\mathcal I}}

\newcommand{\cK}{{\mathcal K}}
\newcommand{\cL}{{\mathcal L}}
\newcommand{\cM}{{\mathcal M}}

\newcommand{\cO}{{\mathcal O}}

\newcommand{\cR}{{\mathcal R}}
\newcommand{\cS}{{\mathcal S}}
\newcommand{\cSA}{\mathcal{SA}}
\newcommand{\cT}{{\mathcal T}}
\newcommand{\cU}{{\mathcal U}}
\newcommand{\cV}{{\mathcal V}}
\newcommand{\cW}{{\mathcal W}}
\newcommand{\cX}{{\mathcal X}}
\newcommand{\cY}{{\mathcal Y}}
\newcommand{\cZ}{{\mathcal Z}}

\newcommand{\fA}{{\mathfrak A}}
\newcommand{\bGamma}{\boldsymbol{\Gamma}}

\newcommand{\rowvec}{{\rm \textbf{row-vec}}}

\newcommand{\C}{{\mathbb C}}

\newcommand{\sbm}[1]{\left[\begin{smallmatrix} #1
		\end{smallmatrix}\right]}

\newcommand{\sm}[1]{\begin{smallmatrix} #1
                \end{smallmatrix}}

\newcommand{\bcD}{{\boldsymbol{\mathcal D}}}
\newcommand{\bcR}{{\boldsymbol{\mathcal R}}}

\newcommand{\diag}{{\rm diag}_{1 \le i \le k_{z}}}

\begin{document}

\title[nc interpolation and realization]
{Interpolation and transfer-function realization for the 
noncommutative Schur-Agler class}
\author[J.A.\ Ball]{Joseph A. Ball}
\address{Department of Mathematics,
Virginia Tech,
Blacksburg, VA 24061-0123, USA}
\email{joball@math.vt.edu}
\author[G.\ MarX]{Gregory Marx}
\address{Department of Mathematics,
Virginia Tech,
Blacksburg, VA 24061-0123, USA}
\email{marxg@vt.edu}
\author[V.\ Vinnikov]{Victor Vinnikov}
\address{Department of Mathematics, Ben-Gurion University of the 
Negev, Beer-Sheva, Israel, 84105}
\email{vinnikov@cs.bgu.ac.il}

\begin{abstract} The Schur-Agler class consists of functions over a 
    domain satisfying an appropriate von Neumann inequality.  
    Originally defined over the polydisk,  the idea has been 
    extended to general domains in multivariable complex 
    Euclidean space with matrix polynomial defining function  as well as to 
    certain multivariable noncommutative-operator domains with a noncommutative 
    linear-pencil defining function.  Still more recently there has 
    emerged a free noncommutative function theory (functions of noncommuting matrix 
    variables respecting direct sums and similarity 
    transformations).  The purpose of the present paper is to extend 
    the Schur-Agler-class theory to the free noncommutative function 
    setting.  This includes the positive-kernel-decomposition 
    characterization of the class, transfer-function realization and 
    Pick interpolation theory.  A special class of defining functions 
    is identified for which the associated Schur-Agler class 
    coincides with the contractive-multiplier class on an associated 
    noncommutative reproducing kernel Hilbert space; in this case, 
    solution of the Pick interpolation problem is in terms of the 
    complete positivity of an associated Pick matrix which is 
    explicitly determined from the interpolation data.
	
 \end{abstract}

\subjclass{47B32; 47A60}
\keywords{Noncommutative function, completely positive noncommutative 
kernel, noncommutative Schur-Agler class,
noncommutative contractive-multiplier class, noncommutative Pick interpolation}

\maketitle

\tableofcontents

\section{Introduction}  \label{S:Intro}
\setcounter{equation}{0}

The goal of this paper is to incorporate classical Nevanlinna-Pick 
interpolation into the the general setting of free noncommutative 
function theory as treated in the recent book \cite{KVV-book}. To 
set the results into a broader context, we first review developments 
in Nevanlinna-Pick interpolation theory, beginning with the classical 
version, continuing with more elaborate versions involving matrix- 
and operator-valued interpolants for  tangential-type interpolation 
conditions, then extensions to multivariable settings, and finally 
the free noncommutative setting.

We are now approaching the centennial of the Nevanlinna-Pick 
interpolation theorem which characterizes when there is a holomorphic 
map of the unit disk into its closure satisfying a finite collection 
of prescribed interpolation conditions:

\begin{theorem}   \label{T:NP} (See Pick (1916) \cite{Pick} and 
    Nevanlinna (1919) \cite{Nevanlinna}.) Given points $z_{1}, \dots, z_{N}$ in 
    the unit disk ${\mathbb D} = \{ z \in {\mathbb C} \colon |z| < 
    1\}$ and associated preassigned values $\lambda_{1}, \dots, 
    \lambda_{N}$ in the complex plane ${\mathbb C}$, there exists a 
    holomorphic function $s$ mapping the unit disk ${\mathbb D}$ into 
    the closed unit disk $\overline{\mathbb D}$ and satisfying the 
    interpolation conditions
    \begin{equation}   \label{interpolate1}
	s(z_{i}) = \lambda_{i} \text{ for } i = 1, \dots, N
\end{equation}
if and only if the $N \times N$ so-called Pick matrix ${\mathbb P}$ 
is positive-definite:
$$
  {\mathbb P} : = \left[ \frac{ 1 - \lambda_{i} 
  \overline{\lambda_{j}}}{1 - z_{i} \overline{z_{j}}}\right]_{i,j=1, 
  \dots, N} \succeq 0.
$$
\end{theorem}

Much later in the late 1960s, Sarason \cite{Sarason} introduced an operator-theoretic 
point of view to the problem which led to the Commutant Lifting 
approach to a variety of more general matrix- and operator-valued 
interpolation and moment problems.  We mention in 
particular the Fundamental Matrix Inequality approach (based on 
manipulation of positive operator-valued kernels) of Potapov
(see \cite{KY} and the references there), the detailed application 
of the Commutant Lifting approach in the books of Foias-Frazho
\cite{FF} and Gohberg-Foias-Frazho-Kaashoek \cite{FFGK}, as well as 
the state-space approach of Ball-Gohberg-Rodman \cite{BGR90} for 
rational matrix functions. Much of this work was stimulated by the 
connections with and needs of $H^{\infty}$-control, as also exposed 
in the books \cite{FF, BGR90, FFGK} which emphasized the connection 
between holomorphic functions and transfer functions of 
input/state/output linear systems.  The following is a sample theorem 
from this era.  For $\cX$ and $\cX^{*}$ any Hilbert spaces, we let 
$\cL(\cX, \cX^{*})$ denote the space of bounded linear operators from 
$\cX$ to $\cX^{*}$.  Let us use the notation $\cS(\cU, \cY)$ for the 
$\cL(\cU, \cY)$-valued {\em Schur class} consisting of holomorphic 
functions $S$ mapping the unit disk ${\mathbb D}$ into the unit ball
$\overline{\cB}\cL(\cU, \cY)$ of the space of operators $\cL(\cU, 
\cY)$ from $\cU$ into $\cY$.

\begin{theorem}   \label{T:tanNP}
    Assume that we are given a subset $\Omega$ of the unit disk 
    ${\mathbb D} \subset {\mathbb C}$, three coefficient Hilbert 
    spaces $\cE, \cU, \cY$, and functions $a \colon \Omega \to \cL(\cY, 
    \cE)$ and $b \colon \colon \Omega \to \cL(\cU, \cE)$. Then the 
    following conditions are equivalent:
    
    \begin{enumerate}
    \item
    There exists a Schur-class function $S \in \cS(\cU, \cY)$ such that $S$ 
    satisfies the set of left-tangential interpolation conditions:
   \begin{equation}   \label{tanint}
   a(z) S(z) = b(z) \text{ for each } z \in \Omega.
   \end{equation}
   
   \item
   The generalized de Branges-Rovnyak kernel 
   $$K^{\dBR}_{a,b}(z,w) : = \frac{ a(z) a(w)^{*} - b(z) b(w)^{*}}{ 1 - z 
   \overline{w}} 
   $$
   is a positive kernel on $\Omega$ (written as $K_{a,b} \succeq 0$), i.e., 
   for each finite set of 
   points $\{z_{1}, \dots, z_{N}\}$ in $\Omega$, the $N 
   \times N$ block matrix
   $$  \left[ \frac{ a(z_{i}) a(z_{j})^{*} - b(z_{i}) b(z_{j})^{*}}{ 
   1 - z_{i} \overline{z_{j}}} \right]_{i,j=1, \dots, N}
   $$
   is a positive semidefinite matrix.
   
   \item
   There is an auxiliary Hilbert space $\cX$ and a contractive (or 
   even unitary) colligation matrix
     $$
     \bU: = \begin{bmatrix} A & B \\ C & D \end{bmatrix} \colon 
     \begin{bmatrix} \cX \\ \cU \end{bmatrix} \to \begin{bmatrix} \cX 
	 \\ \cY \end{bmatrix}
 $$
 so that the $\cL(\cU, \cY)$-valued function $S$ given by
   \begin{equation}  \label{transfunc1}
     S(z) = D + z C (I - zA)^{-1} B
   \end{equation}
  satisfies the interpolation conditions \eqref{tanint} on $\Omega$.
 \end{enumerate}
 \end{theorem}
 
  We note that the form \eqref{transfunc1} for a holomorphic function 
  $S(z)$ on the disk is called a {\em transfer-function realization for 
  $S$} due to the following connection with the associated input/state/output 
  linear system
  $$
  \bU \colon \left\{ \begin{array}{ccc}
  x(n+1) & = & A x(n) + B u(n) \\
  y(n) & = & C x(n) + D u(n):
  \end{array} \right.
  $$
  {\em if one runs the system with an input string $\{u(n)\}_{n \ge 0}$ 
  and initial condition $x(0) = 0$, then the output string $\{y(n) 
  \}_{n \ge 0}$ recursively generated by the system equations is 
  given by
  $$
   \sum_{n=0}^{\infty} y(n) z^{n} = S(z) \cdot \sum_{n=0}^{\infty} 
   u(n) z^{n}
  $$
  where $S(z) = D + \sum_{n=1}^{\infty} C A^{n-1} B z^{n} = D + z C (I  
  - zA)^{-1} B$ is as in \eqref{transfunc1}.}
  Note also that equivalence (1) $\Leftrightarrow$ (2) for the 
  special case where $\Omega$ is a finite set $\{z_{1}, \dots, 
  z_{N}\}$, $\cE = \cU = \cY = {\mathbb C}$, $a(z_{i}) = 1$, 
  $b(z_{i}) = \lambda_{i}$ in Theorem \ref{T:tanNP} amounts to the 
  content of Theorem \ref{T:NP}.  Another salient special case is the 
  equivalence (1) $\Leftrightarrow$ (3) for the special case $\Omega 
  = {\mathbb D}$, $\cE = \cY$, $a(z) = I_{\cY}$, $b(z) = S(z)$: then 
  the content of Theorem \ref{T:tanNP} is the realization theorem for 
  the $\cL(\cU, \cY)$-valued Schur class:  any holomorphic function 
  $S \colon {\mathbb D} \mapsto \overline{\cB} \cL(\cU, \cY)$ can be 
  realized as the transfer function of a conservative  
  input/state/output linear system (i.e., as in \eqref{transfunc1} with $\bU$ 
  unitary).
 
 The extension to multivariable domains has several new ideas.  First 
 of all, we use a $s \times r$-matrix polynomial (or a holomorphic 
 operator-valued function in possible generalizations) $Q(z)$ in $d$ 
 variables to define a domain ${\mathbb D}_{Q}$ by 
 $$
 {\mathbb D}_{Q} = \{ z = (z_{1}, \dots, z_{d}) \in {\mathbb C}^{d} \colon 
 \| Q(z) \|_{{\mathbb C}^{s \times r}} < 1\}.
 $$
 Secondly, to get a theory parallel to the classical case, it is 
 necessary to replace the Schur class of the domain ${\mathbb D}_{Q}$
 (consisting of holomorphic functions $S \colon {\mathbb D}_{Q} \to 
 \overline{\cB}\cL(\cU, \cY)$) with functions having matrix or operator rather 
 than scalar arguments.  
Let $\cK$ be any fixed auxiliary infinite-dimensional separable Hilbert space.  
 For $T$ a commutative $d$-tuple $T = (T_{1}, \dots, T_{d})$ of 
 operators on $\cK$ with Taylor spectrum inside the 
 region ${\mathbb D}_{Q}$ and $S$ a holomorphic $\cL(\cU, \cY)$-valued 
 function on ${\mathbb D}_{Q}$, it is possible to use the Taylor functional 
 calculus (using Vasilescu's adaptation of the Bochner-Martinelli 
 kernel---see \cite{AT, BB04} for details) to make sense of the function 
 $S$ applied to the commutative operator-tuple $T$ to get an operator 
 $S(T) \in \cL(\cU \otimes \cK, \cY \otimes \cK)$.  We define the 
 {\em Schur-Agler class} $\mathcal{SA}_{Q}(\cU, \cY)$ to consist of 
 those holomorphic $\cL(\cU, \cY)$-valued functions on ${\mathbb D}_{Q}$ such 
 that $\| S(T) \|_{\cL(\cU \otimes \cK, \cY \otimes \cK)} \le 1$ 
 for all commutative operator $d$-tuples $T = (T_{1}, \dots, T_{d})$ such that $\| 
 Q(T) \| < 1$.  A result from \cite{AT} guarantees that $T$ has 
 Taylor spectrum inside ${\mathbb D}_{Q}$ whenever $\| Q(T) \| < 1$, so the 
 definition makes sense.  We then may state our $Q$-analogue of 
 Theorem \ref{T:tanNP} as follows.
 
 \begin{theorem}   \label{T:QtanNP}
     Assume that $Q$ is an $s \times r$-matrix-valued polynomial 
     defining a domain ${\mathbb D}_{Q} \subset {\mathbb C}^{d}$ as above.
    Assume that we are given a subset $\Omega$ of ${\mathbb D}_{Q}$,
    three coefficient Hilbert spaces $\cE, \cU, \cY$, and functions 
    $a \colon \Omega \to \cL(\cY, \cE)$ and $b \colon \colon \Omega \to \cL(\cU, \cE)$. 
    Then the following conditions are equivalent:
    
    \begin{enumerate}
    \item
    There exists a Schur-Agler-class function $S \in \cS_{Q}(\cU, \cY)$ such that $S$ 
    satisfies the set of left-tangential interpolation conditions:
   \begin{equation}   \label{Qtanint}
   a(z) S(z) = b(z) \text{ for each } z \in \Omega.
   \end{equation}
   
   \item
   There is an auxiliary Hilbert space $\cX$ so that the kernel on $\Omega$ given by 
   $$
   K_{a,b}(z,w) : = a(z) a(w)^{*} - b(z) b(w)^{*}
   $$
   has a factorization
   \begin{equation}   \label{QAglerdecom}
       K_{a,b}(z,w) = H(z)\left( (I_{s} - Q(z) Q(w)^{*}) \otimes I_{\cX} 
       \right) H(w)^{*}
   \end{equation}
   for some operator-valued function  $H \colon \Omega \to \cL({\mathbb C}^{r} 
   \otimes \cX, \cE)$.

   \item
   There is an auxiliary Hilbert space $\cX$ and a contractive (or 
   even unitary) colligation matrix
     $$
     \bU: = \begin{bmatrix} A & B \\ C & D \end{bmatrix} \colon 
     \begin{bmatrix} \cX^{s} \\ \cU \end{bmatrix} \to \begin{bmatrix} 
	 \cX^{r} 
	 \\ \cY \end{bmatrix}
 $$
 so that the $\cL(\cU, \cY)$-valued function $S$ given by
   \begin{equation}  \label{transfunc0}
S(z) = D +  C (I - (Q(z) \otimes I_{\cX}) A)^{-1} (Q(z) \otimes I_{\cX}) B
   \end{equation}
  satisfies the interpolation conditions \eqref{Qtanint}.
 \end{enumerate}
 \end{theorem}
 
 This theorem follows by pasting together various pieces from the 
 results of \cite{AT, BB04, BB05} (see also \cite{MP} for a somewhat 
 different setting and point of view).  The seminal work of Agler 
 \cite{Agler-Hellinger, Agler-pre} handled the special case where $Q$ is 
 taken to be 
$Q_{\rm diag}(z): =
 \sbm{ z_{1} & & \\ & \ddots & \\ & & z_{d}}$, so ${\mathbb D}_{Q}$ becomes 
 the polydisk
 $${\mathbb D}^{d} = \{ z = (z_{1}, \dots, z_{d}) \in 
 {\mathbb C}^{d} \colon |z_{k}| < 1 \text{ for } k=1, \dots, d\}.
 $$
 In this setting, one can work with global power series 
 representations rather than the more involved Taylor functional 
 calculus.
 For a thorough current (as of 2009) update of the ideas from the 
 Agler unpublished manuscript \cite{Agler-pre}, we recommend the paper of 
 Jury-Knese-McCullough \cite{JKMcC}.
Additional work on the polydisk case was obtained in \cite{AMcC99, BT}.
Another important special case is the case where $Q$ is taken to be 
$Q_{\rm row}(z) = \begin{bmatrix} z_{1} & \cdots & z_{d} 
\end{bmatrix}$.  Various pieces of Theorem \ref{T:QtanNP} for this 
special case appear in \cite{BTV, AriasPopescu, DP}.
For the case $Q = Q_{\rm row}$, the Schur-Agler class $\mathcal{SA}_{Q_{\rm 
row}}(\cU, \cY)$  coincides with the contractive multiplier class 
associated with the Drury-Arveson kernel---we say more about 
this special situation in Subsection \ref{S:NCvsLTOA} below.  The case 
where $Q(z) = \sbm{ Q_{1}(z)  & & \\ & \ddots & \\ & & Q_{N}(z) }$ 
where each $Q_{k}(z)$ has the
form $Q_{\rm row}(z)$ (of various sizes $d_{1}, 
\dots, d_{N}$) was considered by Tomerlin in \cite{Tomerlin}.
Note that $Q_{\rm diag}(z)$ is the special case of this where each 
$d_{k} = 1$.
One can also see the ideas of \cite{Agler-pre} as influencing the test-function 
approach of Dritschel-McCullough and collaborators, originating in
\cite{DMMcC, DMcC} with followup work in \cite{McCS, DritPick, JKMcC, BGH}.

The next development in our story is the extension to noncommutative 
variables.
With motivation from multidimensional system theory, 
Ball-Groenewald-Malakorn \cite{BGM2} introduced a noncommutative 
Schur-Agler class defined as follows.  We now let $z = 
(z_{1}, \dots, z_{d})$ be freely noncommuting  formal indeterminates.
We let $\free$ be the unital free semigroup (i.e., monoid in the 
language of algebraists) on $d$ generators, denoted here by the first 
$d$ natural numbers $\{1, \dots, d\}$.  Thus elements of $\free$ 
consist of words $\fa = i_{N} \cdots i_{1}$ where each letter 
$i_{k}$ is in the set $\{1, \dots, d\}$ with multiplication given by 
concatenation:
$$
\fa \cdot \fb = i_{N} \cdots i_{1} j_{M} \cdots j_{1} \text{ if 
} \fa = i_{N} \cdots i_{1} \text{ and } \fb = j_{M} \cdots j_{1}.
$$
Furthermore, we consider the empty word $\emptyset$ as an element of 
$\free$ and let this serve as the unit element for $\free$.
For $\fa \in \free$, we 
define the noncommutative monomial $z^{\fa}$ by
$$
 z^{\fa} = z_{i_{N}} \cdots z_{i_{1}} \text{ if } \fa = i_{N} 
 \cdots i_{1}, \quad z^{\emptyset} = 1.
$$
For $\cV$ a vector space, we let $\cV\langle \langle z \rangle \rangle$ be the space 
of formal power series
$$
  f(z) = \sum_{\fa \in \free} f_{\fa} z^{\fa}
$$
with coefficients $f_{\fa}$ in $\cV$. If all but 
finitely many of the coefficients $f_{\fa}$ vanish, we refer to 
$f(z)$ as a \textbf{noncommutative polynomial} with the notation $f \in \cV 
\langle z \rangle$.  Let us consider the special case where the 
noncommutative polynomial involves only linear terms 
$$
  Q(z) = \sum_{k=1}^{d} L_{k} z_{k}
$$
and where $\cV$ is taken to be the space ${\mathbb C}^{s \times r}$;
the paper \cite{BGM2} assumes some other structure on $Q(z)$ details 
of which we need not go into here. 
We define a noncommutative domain ${\mathbb D}_{Q} \subset \cL(\cK)^{d}$ consisting of 
operator tuples $T = (T_{1}, \dots, T_{d}) \in \cL(\cK)^{d}$ (now not 
necessarily commutative) such that
$Q(T) : = \sum_{k=1}^{d} L_{k} \otimes T_{k}$ has $\| Q(T) \| < 1$.
For $\fa$ a word in $\free$ and for $T = 
(T_{1}, \dots, T_{d}) \in \cL(\cK)^{d}$, we use the noncommutative functional 
calculus notation
$$
  T^{\fa} = T_{i_{N}} \cdots T_{i_{1}} \in \cL(\cK) \text{ if } 
  \fa = i_{N} \cdots i_{1} \in \free, \quad T^{\emptyset} = 
  I_{\cK}.
$$
where now the multiplication is operator composition rather than 
concatenation.  Given a formal series $S(z) = \sum_{\fa \in \free} 
S_{\fa} z^{\fa}$ with coefficients $S_{\fa} \in \cL(\cU, 
\cY)$ and given a operator $d$-tuple $T = (T_{1}, \dots, T_{d})$, we 
define 
\begin{equation}   \label{powerseriesrep}
  S(T) = \sum_{\fa \in \free} S_{\fa} \otimes T^{\fa} \in 
  \cL(\cU \otimes \cK, \cY \otimes \cK)
\end{equation}
whenever the series converges in some reasonable sense.
Then we define the noncommutative Schur-Agler class 
$\mathcal{SA}_{Q}(\cU, \cY)$ to consist of all formal power series
$S(z) = \sum_{\fa \in \free} S_{\fa} z^{\fa}$  in $\cL(\cU, 
\cY)\langle \langle z \rangle \rangle$ such that
$S(T)$ is defined and $\| S(T) \| \le 1$ for all $T = (T_{1}, \dots, T_{d}) \in 
\cL(\cK)^{d}$ such that $\| Q(T) \| < 1$.  Then the main result from 
\cite{BGM2} is the following realization theorem for the Schur-Agler 
class $\mathcal{SA}_{Q}(\cU, \cY)$ (where $Q(z) = \sum_{k=1}^{d} 
L_{k} z_{k}$ is a noncommutative linear function having some additional structure 
not discussed here).

\begin{theorem}   \label{T:BGM2}  Suppose that $Q$ is a linear pencil 
    as discussed above and suppose that $S \in \cL(\cU, \cY) \langle 
    \langle z \rangle \rangle$ is a given formal power series.  Then 
    the following conditions are equivalent.
    \begin{enumerate}
	\item $S$ is in the noncommutative Schur-Agler class $\mathcal{SA}_{Q}(\cU, \cY)$. 
	\item The noncommutative formal kernel $K_{I,S}(z,w) = I - S(z)S(w)^{*}$ 
	has a formal noncommutative Agler decomposition, i.e.: there exist an 
	auxiliary Hilbert space $\cX$ and a 
	formal power series $H(z) \in \cL({\mathbb C}^{r} \otimes \cX, \cY) \langle \langle z 
	\rangle \rangle$ so that
$$
  K_{I,S}(z,w) = H(z) \left( (I_{{\mathbb C}^{s}}  - Q(z)Q(w)^{*}) 
  \otimes I_{\cX} \right) H(w)^{*}.
 $$
 
 \item There is an auxiliary Hilbert space $\cX$ and a unitary 
 colligation matrix $\bU$
 $$
  \bU = \begin{bmatrix} A & B \\ C & D \end{bmatrix} \colon 
  \begin{bmatrix} \cX^{s} \\ \cU \end{bmatrix} \to \begin{bmatrix} 
      \cX^{r} \\ \cY \end{bmatrix}
 $$
 so that
 $$ 
 S(z) = D + C (I - (Q(z) \otimes I_{\cX}) A)^{-1} (Q(z) \otimes 
 I_{\cX}) B.
 $$
\end{enumerate}
\end{theorem}

Theorem \ref{T:BGM2} has a couple of limitations: (1) it is missing 
an interpolation-theoretic aspect and (2) it is tied noncommutative 
domains defined by linear pencils, thereby guaranteeing global formal 
power series representations for noncommutative functions on the 
associated domain ${\mathbb D}_{Q}$. 
While there has been some work on an interpolation theory for such 
domains ${\mathbb D}_{Q}$ (see \cite{P98, CJ03, MS2004, KVV2005, 
BB07, Bala1}) and 
thereby addressing the first limitation, the second limitation is
more fundamental and is the main inspiration for the present paper.

Motivation comes from a result of Alpay--Kaliuzhnyi-Verbovetskyi \cite{AKV} 
which says that that one need only plug in $T = (T_{1}, 
\dots, T_{d}) \in \cL({\mathbb C}^{n})^{d}$ (i.e., a $d$-tuple of $n 
\times n$ matrices) sweeping over all $n \in {\mathbb N}$ to 
determine if a given power series $S \in \cL(\cU, \cY)\langle \langle 
z \rangle \rangle$ is in the noncommutative Schur-Agler class 
$\mathcal{SA}_{Q}(\cU, \cY)$ for $Q(z)$ a linear nc function 
(specifically $Q(z) = \sbm{ z_{1} & & \\ & \ddots & \\ & & z_{d}}$) as 
in Theorem \ref{T:BGM2}.  
We may thus view $S$ as a function from
the disjoint union $({\mathbb D}_{Q})_{\rm nc} : = 
\amalg_{n=1}^{\infty}({\mathbb D}_{Q})_{n}$,
where we set 
$$
  ({\mathbb D}_{Q})_{n} = \{ Z = (Z_{1}, \dots, Z_{d}) \in ({\mathbb 
  C}^{n \times n})^{d} \colon \| Q(Z) \| < 1\},
$$
into the space $\cL(\cU, \cY)_{\rm nc} = \amalg_{n=1}^{\infty} 
\cL(\cU, \cY)^{ n \times n}$ (where we identify $\cL(\cU, \cY)^{n 
\times n}$ with $\cL(\cU^{n} , \cY^{n})$ when convenient).
It is easily checked that the such a function $S$, when given by a power series 
representation as in \eqref{powerseriesrep}, satisfies the following 
axioms:
\begin{enumerate}
    \item[(A1)] $S$ is \textbf{graded}, i.e., $S$ maps $({\mathbb 
    D}_{Q})_{n}$ into $\cL(\cU, \cY)^{n \times n}$,
    
    \item[(A2)] $S$ \textbf{respects direct sums}, i.e., if $Z = 
    (Z_{1}, \dots, Z_{d}) \in ({\mathbb D}_{Q})_{n}$ and $W = (W_{1}, 
    \dots, W_{d}) \in ({\mathbb D}_{Q})_{m}$ and we set
    $$
     \sbm{ Z & 0 \\ 0 & W } = 
	 \left( \sbm{ Z_{1} & 0 \\ 0 & W_{1} }, 
	 \cdots,  \sbm{ Z_{d} & 0 \\ 0 & W_{d}} \right),
   $$
   then 
   $$
 S\left( \sbm{ Z & 0 \\ 0 & W } \right) = \sbm{ S(Z) & 0 \\ 0 & S(W) }.
  $$
   \item[(A3)]  $S$ \textbf{respects similarities}, i.e., if $\alpha$ is an 
  invertible $n \times n$ matrix over ${\mathbb C}$, if $Z = (Z_{1}, 
  \dots, Z_{d})$ and $  \alpha^{-1} Z \alpha := ( \alpha^{-1} Z_{1} \alpha, \cdots, 
  \alpha^{-1} Z_{d} \alpha )$ are both in ${\mathbb D}_{Q}$, then it follows 
  that
  $$ 
   S( \alpha^{-1} Z \alpha) =  \alpha^{-1} S(Z) \alpha.
  $$
 \end{enumerate}
Such an axiom system for an operator-valued function $S$ defined on 
square-matrix tuples of all possible sizes was introduced by J.L.\ 
Taylor \cite[Section 6]{Taylor} in connections with representations 
of the free algebra and the quest for a functional calculus for 
noncommuting operator tuples.  Recent work of Kaliuzhnyi-Verbovetskyi 
and Vinnikov \cite{KVV-book} provides additional insight and 
completeness  into the work of Taylor; in particular, there it is 
shown that, under mild local boundedness conditions, any function $S$ 
satisfying the axioms (A1), (A2), (A3) is given locally by a power 
series representation of the sort in \eqref{powerseriesrep}.  Closely 
related approaches to and results on such a  ``free noncommutative function theory''  
can be found in the work of Voiculescu  \cite{Voiculescu1, 
Voiculescu2} and Helton-Klep-McCullough \cite{HKMcC1, HKMcC2, HMcC}.

The purpose of the present paper is to extend the Nevanlinna-Pick interpolation and 
transfer-function realization theory for the Schur-Agler class as presented in the 
progression of Theorems \ref{T:NP}, \ref{T:tanNP}, \ref{T:QtanNP}, 
\ref{T:BGM2} to the general setting of free noncommutative function 
theory.  There has already appeared results in this direction in the 
work of Agler-McCarthy \cite{AMcC-global, AMcC-Pick} (see also 
\cite{Bala2}).  Our main 
result (see Theorem \ref{T:ncInt}) presents a more unified setting 
for their results as well as extending their results to a more 
natural level of generality (see the comments immediately after 
Corollary \ref{C:AMcC-Pick} below).

The proof strategy for Theorems \ref{T:QtanNP}, \ref{T:BGM2} as well 
as Theorem \ref{T:ncInt} has a common skeleton which originates from 
the seminal 1990 paper of Agler \cite{Agler-Hellinger}.  Indeed, in 
all these Theorems (as well as in the assorted incremental versions 
done in \cite{AT, BB04, BGH, DMMcC, DMcC}), the proof of (1) 
$\Rightarrow$ (2) involves a cone separation argument; an exception 
is the work of Paulsen and collaborators \cite{LMP, MP} where the 
cone-separation argument is replaced by an operator-algebra approach 
with an appeal to the Blecher-Ruan-Sinclair characterization of 
operator algebras. A new feature 
for the noncommutative setting of Theorem \ref{T:ncInt}, first 
observed by Agler-McCarthy in \cite{AMcC-Pick}, is that the cone 
separation argument still applies with a localized weaker version 
(see condition (1$^{\prime}$) in Theorem \ref{T:ncInt}) of condition 
(1) as a hypothesis.  The implication (2) $\Rightarrow$ (3) (in Theorems 
\ref{T:QtanNP}, \ref{T:BGM2} as well as in \ref{T:ncInt} and an even 
larger assortment of closely related versions in the literature) is 
via what is now called the {\em lurking isometry argument}.  While 
this idea first appears in \cite{Agler-Hellinger} for the context of 
multivariable interpolation, it actually has much earlier 
manifestations already in the univariate theory:  we mention the 
early work of Liv\v{s}ic \cite{Livsic1, Livsic2} where the 
Characteristic Operator Function was first introduced, the proof 
of the Nevanlinna-Pick interpolation theorem due to Sz.-Nagy--Koranyi 
\cite{NK}, and the Abstract Interpolation Problem framework for 
interpolation and moment problems due to 
Katsnelson-Kheifets-Yuditskii \cite{KKY, Kheifets, KY}. The proof of 
(3) $\Rightarrow$ (1) is a straightforward application of a general 
principle on composing a contractive linear fractional map with a 
contractive load (we refer \cite{HZ} for a general formulation).
Here we also introduce the notion of \textbf{complete Pick kernel} 
(see \cite{AMcC00, Quiggen} as well as the book \cite{AMcC-book}) for 
the free noncommutative setting.

The paper is organized as follows.
Section \ref{S:prelim} collects preliminary material needed for the 
proof of Theorem \ref{T:ncInt} as well as pushing out the boundary of 
the free noncommutative function theory. Included here is a review of the basics 
of free noncommutative function theory from \cite{KVV-book} as well as some 
additional material relevant to the proof of Theorem \ref{T:ncInt}:  
some calculus and open questions concerning \textbf{full 
noncommutative envelopes} and \textbf{noncommutative-Zariski closure} of any finite subset 
$\Omega$ of a full noncommutative set $\Xi$, as well as a review of material from our 
companion paper \cite{BMV1} concerning \textbf{completely positive 
noncommutative kernels}, a notion needed for the very formulation of 
condition (2) in Theorem \ref{T:ncInt}. Section \ref{S:SAint} 
introduces the noncommutative Schur-Agler class $\mathcal{SA}_{Q}(\cU, \cY)$ for 
the free noncommutative setting, poses the \textbf{Left-Tangential Interpolation 
Problem}, and states the main result Theorem \ref{T:ncInt} along with some 
corollaries and remarks exploring various special cases and 
consequences.  Section 4 then presents the proof of Theorem 
\ref{T:ncInt} in systematic fashion one step at a time:  (1) 
$\Rightarrow$ (1$^{\prime}$) $\Rightarrow$ (2) $\Rightarrow$ (3) 
$\Rightarrow$ (1).  Section \ref{S:mult} reviews material from 
\cite{BMV1} concerning noncommutative reproducing kernel Hilbert spaces, 
identifies a class of kernels $k_{Q_{0}}$ for which the associated 
Schur-Agler class $\mathcal{SA}_{Q_{0}}(\cU, \cY)$ coincides with the 
contractive multiplier class $\overline{\cB} \cM(k_{Q_{0}} \otimes 
I_{\cU}, k_{Q_{0}} \otimes I_{\cY})$ (multiplication operators mapping 
the reproducing kernel Hilbert space $\cH(k_{Q_{0}} \otimes 
I_{\cU})$ contractively into $\cH(k_{Q_{0}} \otimes 
I_{\cY})$), thereby getting noncommutative versions of the Nevanlinna-Pick 
interpolation and transfer-function realization theory for contractive multipliers 
on the Drury-Arveson space (see \cite{Arv, BTV, EP}).

\section{Preliminaries}  \label{S:prelim}
We review some preliminary material on 
noncommutative functions and completely positive noncommutative kernels
from \cite{BMV1} which are needed in the sequel.  A comprehensive 
treatise on the topic of noncommutative functions (but with no discussion 
of noncommutative kernels which are first introduced in \cite{BMV1}) 
is the book of Kaliuzhnyi-Verbovetskyi and Vinnikov \cite{KVV-book}.
Henceforth we shall use the abbreviation {\em nc} for the term {\em 
noncommutative}.

\subsection{Noncommutative functions and completely positive 
noncommutative kernels}  \label{S:ncfuncker}

We suppose that we are given a vector space $\cV$. Thus $\cV$ is 
equipped with a scalar multiplication by complex numbers which makes 
$\cV$ a bimodule over ${\mathbb C}$.  We define the 
associated nc space $\cV_{\rm nc}$ to consist of the disjoint union 
over all $n \in {\mathbb N}$ of $n \times n$ matrices over $\cV$:
$$
 \cV_{\rm nc} = \amalg_{n=1}^{\infty} \cV^{n \times n}.
$$
A subset $\Omega$ of $\cV_{\rm nc}$ is said to be a 
\textbf{nc set} 
if $\Omega$ is closed under direct sums:
$$
Z = [Z_{ij}]_{i,j=1}^{n} \in \Omega_{n},\, W = 
[W_{ij}]_{i,j=1}^{m}
\in \Omega_{m} 
\Rightarrow
Z \oplus W = \left[ \begin{smallmatrix} Z & 0 \\ 0 & W 
\end{smallmatrix} \right] \in \Omega_{n+m}.
$$

Suppose next that $\Omega$ is a subset of $\cV_{\rm nc}$), and that
 $\cV_{0}$ is another vector space (i.e., a bimodule over ${\mathbb C}$).  
 For $\alpha \in {\mathbb C}^{n \times m}$, $V \in \cV_{0}^{m \times k}$, $\beta \in {\mathbb C}^{k 
 \times \ell}$, we can use the module structure of $\cV_{0}$ over 
 ${\mathbb C}$ to make sense of the matrix multiplication $\alpha V \beta \in 
 \cV_{0}^{n \times \ell}$ and similarly $\alpha Z \beta$ makes sense as an 
 element of  $\cV^{n \times \ell}$ for $Z \in \cV^{m \times k}$.
  Given a function $f \colon \Omega \to 
 \cV_{0,nc}$, we say that $f$ is a \textbf{noncommutative (nc) function} if
 \begin{itemize}
     \item $f$ is \textbf{graded}, i.e., $f \colon \Omega_{n} \to 
     \cV_{0,n} = (\cV_{0})^{n \times n}$, and
     \item $f$  \textbf{respects intertwinings}:  
     \begin{equation}  \label{funcintertwine}
	 Z \in \Omega_{n}, \, 
     \widetilde Z \in \Omega_{m}, \, \alpha \in {\mathbb C}^{m \times n} 
     \text{ with } \alpha Z = \widetilde Z \alpha \Rightarrow \alpha f(Z) = f(\widetilde Z) 
     \alpha.
   \end{equation}
 \end{itemize}
 An equivalent characterization of nc functions is (see  \cite[Section 
 I.2.3]{KVV-book}):
 $f$ is a nc function if and only if 
 \begin{itemize}
     \item $f$ is a graded, 
     
     \item $f$ \textbf{respects direct sums}, i.e.,
     \begin{equation}   \label{funcdirsum}
     Z,W \in \Omega \text{ such that also } \sbm{Z & 0 \\ 0 & W } \in 
     \Omega \Rightarrow f\left( \sbm{ Z & 0 \\ 0 & W} 
     \right) = \sbm{f(Z) & 0 \\ 0 & f(W)},
     \end{equation}
     and 
 \item $f$  \textbf{respects similarities}, i.e.: whenever $Z, \widetilde Z \in \Omega_{n}$, 
 $\alpha \in {\mathbb C}^{n \times n}$ with $\alpha$ invertible such that 
 $\widetilde Z = \alpha Z \alpha^{-1}$, then 
 \begin{equation}   \label{funcsim}
 f(\widetilde Z) = \alpha f(Z) \alpha^{-1}.
 \end{equation}
 \end{itemize}
Following \cite{KVV-book}, we denote the set of all nc functions from
$\Omega$ into $\cV_{0,{\rm nc}}$ by $\cT(\Omega; \cV_{0,{\rm nc}})$; 
we note that we do not require that the domain $\Omega$ for a nc 
function be a nc set as is done in \cite{KVV-book}.
  
We now suppose that we are  given two additional vector spaces $\cV_{0}$ and 
$\cV_{1}$. For $K$ a function from 
$\Omega \times \Omega$ into the nc space 
$$
\cL(\cV_{1}, \cV_{0})_{\rm nc}:=  \amalg_{n,m = 1}^{\infty} 
\cL(\cV_{1}^{n \times m}, \cV_{0}^{n \times m}),
$$
we say that $K$ 
is  a \textbf{nc kernel} if 
\begin{itemize}
\item $K$ is \textbf{graded} in the sense that
\begin{equation}   \label{kergraded}
Z \in \Omega_{n},\, W \in \Omega_{m} \Rightarrow K(Z,W) \in 
\cL(\cV_{1}^{n \times m}, \cV_{0}^{n \times m})
\end{equation}
and
\item $K$ \textbf{respects intertwinings} in the following sense:
\begin{align}
   & Z \in \Omega_{n},\, \widetilde Z \in \Omega_{\widetilde n},\, \alpha \in 
    {\mathbb C}^{\widetilde n \times n} \text{ such that } \alpha Z = \widetilde Z 
    \alpha,  \notag \\
 & W \in \Omega_{m},\, \widetilde W \in \Omega_{\widetilde m}, \, 
 \beta \in {\mathbb 
 C}^{\widetilde m \times m} \text{ such that } \beta W = \widetilde W \beta.  
 \notag \\
 & P \in \cV_{1}^{n \times m} \Rightarrow  \alpha \, K(Z,W)(P) \, \beta^{*} = 
 K(\widetilde Z, \widetilde W) (\alpha P \beta^{*}).
 \label{kerintertwine}
 \end{align}
\end{itemize}
An equivalent set of conditions is: 
\begin{itemize}
    \item $K$ is graded,
    
    \item $K$ \textbf{respects direct sums}: for $Z \in \Omega_{n}$ 
    and 
    $\widetilde Z \in \Omega_{\widetilde n}$ such that $\sbm{ Z & 0 
    \\ 0 & W } \in \Omega_{n+m}$, 
    $W \in \Omega_{m}$ and $\widetilde W \in \Omega_{\widetilde m}$ such that 
   $\sbm{W & 0 \\ 0 & \widetilde W} \in \Omega_{m + \widetilde m}$, 
   and
    $P = \sbm{ P_{11} & P_{12} \\ P_{21} & P_{22} }$ in  
    $\cV_{1}^{(n+m) \times (\widetilde n+ \widetilde m)}$, we have
   \begin{equation}  \label{kerdirsum} 
   K\left( \sbm{ Z & 0 \\ 0 & \widetilde Z}, \sbm {W & 0 \\ 0 & 
    \widetilde W} \right) \left( \sbm{ P_{11} & P_{12} \\ P_{21} & 
    P_{22} } \right) =
    \begin{bmatrix}  K(Z,W)(P_{11}) & K(Z, \widetilde W)(P_{12}) \\
	  K(\widetilde Z, W)(P_{21}) & K(\widetilde Z, \widetilde 
	  W)(P_{22}) \end{bmatrix}.
\end{equation}

\item $K$ \textbf{respects similarities}:  
\begin{align}
  &  Z, \widetilde Z \in \Omega_{n}, \, \alpha \in {\mathbb C}^{n \times n} 
    \text{ invertible with } \widetilde Z = \alpha Z \alpha^{-1}, \notag \\
  & W, \widetilde W \in \Omega_{m}, \beta \in {\mathbb C}^{m \times m} 
  \text{ invertible with } \widetilde W = \beta W \beta^{-1}, \notag \\
& P \in \cA^{n \times m} \Rightarrow  K(\widetilde Z, \widetilde W)(P) =
\alpha \, K(Z,W)(\alpha^{-1} P \beta^{-1*})\,  \beta^{*}.
\label{kersim}
\end{align}
\end{itemize}
We denote the class of all such nc kernels by $\widetilde 
\cT^{1}(\Omega; \cV_{1, {\rm nc}}, \cV_{0, {\rm nc}})$.  

For the next definition we need to impose an order structure on 
$\cV_{0}$ and $\cV_{1}$ so that an appropriate notion of positivity 
is defined for square matrices over $\cV_{0}$ and $\cV_{1}$.  Recall
(see \cite{Paulsen, ERbook}) that a normed linear space $\cW$ is said to 
be an \textbf{operator space} if it is equipped with a system of norms 
$\| \cdot \|_{n}$ on $n \times n$ matrices over $\cW$ so that there 
is a map $\varphi \colon \cW \to \cL(\cX)$ (where $\cX$ is some 
Hilbert space) so that, for each $n = 1, 2, \dots$, the map
$$
 \varphi^{(n)} = 1_{{\mathbb C}^{n \times n}} \otimes \varphi \colon
 \cW^{n \times n} \to \cL(\cX)^{n \times n}
$$
defined by
\begin{equation}  \label{varphi-n}
 {\rm id}_{{\mathbb C}^{n \times n}} \otimes \varphi \colon [ w_{ij}]_{i,j = 
 1, \dots, n} \mapsto [ \varphi(w_{ij}) ]_{i,j=1, \dots, n}
\end{equation}
is a linear isometry.  By the theorem of Ruan \cite[Theorem 
2.3.5]{ERbook}, such a situation is 
characterized by the system of norms satisfying 
\begin{align}
   &  \| X \oplus Y \|_{n + m}  = \max \{ \| X\|_{n}, \| Y \|_{m} \}, 
    \text{ and }  \notag \\  
 & \| \alpha X \beta \|_{m} \le \| \alpha \| \| X \|_{n} \| \beta \| \text{ for all }
 X \in \cW^{n \times n},\, \alpha \in {\mathbb C}^{m \times n}, 
 \beta \in {\mathbb C}^{n \times m}.
 \label{Ruan}
\end{align}
We say that $\cW$ is an \textbf{operator system} if $\cW$ is an 
operator space such that the image $\varphi(\cW)$ of $\cW$ under the 
map $\varphi$ given above is a unital subspace of $\cL(\cX)$ closed 
under taking adjoints.  Note that the set of selfadjoint elements of 
$\varphi(\cW)$ is nonempty since $\varphi(\cW)$ contains the identity 
element $1_{\cX}$.  Then the adjoint operation and the notion of 
positivity in $\cL(\cX)$ pulls back to $\cW$ (as well as to square 
matrices $\cW^{n \times n}$ over $\cW$). Operator systems also have 
an abstract characterization: {\em if $\cW$ is a 
matrix-ordered  $*$-vector space with an Archimedean matrix order unit 
$e$, then $\cW$ is completely order isomorphic (and hence also 
completely isometrically isomorphic) to a concrete operator system}
by a theorem of Choi and Effros (see \cite[Theorem 13.1]{Paulsen}).

We now specialize the setting for our nc kernels defined by 
\eqref{kergraded}, \eqref{kerintertwine}, \eqref{kerdirsum}, 
\eqref{kersim} by assuming that the vector spaces $\cV_{1}$ and 
$\cV_{0}$ are both operator systems, now denoted as $\cS_{1}$ and 
$\cS_{0}$ respectively. Given a nc kernel $K
 \in\widetilde \cT^{1}(\Omega; \cS_{0, {\rm nc}}, \cS_{1, {\rm nc}})$, 
we say that $K$ is \textbf{completely positive} (cp) if in addition, 
the map
\begin{equation}  \label{kercp'}
    [P_{ij}]_{i,j = 1, \dots, n} \mapsto
    \left[ K(Z^{(i)}, Z^{(j)}) ( P_{ij}) \right]_{i,j = 1, \dots, n}
\end{equation}
is a positive map from $\cS_{1}^{N \times N}$ to $\cS_{0}^{N \times 
N}$  ($N = \sum_{i=1}^{n} m_{i}$) for any choice of $Z^{(i)} \in 
\Omega_{m_{i}}$, $m_{i} \in {\mathbb N}$,  $i = 1, \dots, n$, $n$ 
arbitrary.  In case $\Omega$ is a nc subset of $\cV_{\rm nc}$, 
we can iterate the ``respects direct sums'' property 
\eqref{kerdirsum} of $K$ to see that
$$
K \left( \sbm{ Z^{(1)} & & \\ & \ddots & \\ & & Z^{(n)} },
\sbm{ Z^{(1)} & & \\ & \ddots & \\ & & Z^{(n)}} \right) \left( [ 
P_{ij} ] \right) = 
\left[ K(Z^{(i)}, Z^{(j)})(P_{ij}) \right]_{i,j = 1, \dots, n},
$$
for any choice of $Z^{(i)} \in \Omega_{n_{i}}$, $P_{ij} \in 
\cS_{1}^{n_{i} \times n_{j}}$, $i=1, \dots, n$.  Then the condition 
\eqref{kercp'} can be written more simply as
\begin{equation} \label{kercp}
    Z \in \Omega_{n},\,  P \succeq 0 \text{ in }
    \cS_{1}^{n \times n} \Rightarrow K(Z,Z)(P) \succeq 0 \text{ in } 
    \cS_{0}^{n \times n} \text{ for all } n \in {\mathbb N}.
\end{equation}
  If $\cV_{1} = \cA_{1}$ is a 
$C^{*}$-algebra, we can rewrite \eqref{kercp'} as
\begin{equation}   \label{kercp''}
\left[ K(Z^{(i)}, Z^{(j)}) (R_{i}^{*} R_{j}) \right]_{i,j = 1, \dots, 
n} \succeq 0 \text{ in } \cV_{0}^{N \times N}
\end{equation}
for all $R_{i} \in \cA_{1}^{m_{i} \times N}$, $Z^{(i)} \in 
\Omega_{m_{i}}$, $i = 1, \dots, n$ with $n, N \in {\mathbb N}$ arbitrary.
If $\cS_{0} = \cA_{0}$ is also a $C^{*}$-algebra, \eqref{kercp''} can turn can 
be equivalently expressed as
\begin{equation}   \label{kercp'''}
    \sum_{i,j=1}^{n} V_{i}^{*} K(Z^{(i)}, Z^{(j)})(R_{i}^{*} R_{j}) V_{j}
    \succeq 0 \text{ in } \cV_{0}
\end{equation}
for all $R_{i} \in \cA_{1}^{N \times m_{i}}$, $Z^{(i)} \in 
\Omega_{m_{i}}$, $V_{i} \in \cA_{0}^{m_{i} \times 1}$, $i = 1, \dots, 
n$ with $n \in {\mathbb N}$ arbitrary.  When we restrict to the case 
$m_{i} = 1$ for all $i$, the formulation \eqref{kercp'''} amounts to 
the notion of complete positivity of a kernel given by 
Barreto-Bhat-Liebscher-Skeida in  \cite{BBLS}.

The more concrete setting for nc kernels which we shall be interested 
in here is as follows.  We again take the ambient set of points to be
the nc space $\cV_{\rm nc}$ associated with a vector space $\cV$, while 
for the operator systems $\cS_{0}$ and $\cS_{1}$ we take
$$
  \cS_{0} = \cL(\cE), \quad \cS_{1} = \cA
$$
where $\cE$ is a coefficient Hilbert space and $\cA$ is a 
$C^{*}$-algebra.  Then we have have the following 
characterization of cp nc kernels for this setting from \cite{BMV1}.

\begin{theorem}   \label{T:cpker}  
    Suppose that $\Omega$ be a subset of $\cV_{\rm 
    nc}$,  $\cE$ is a Hilbert space, $\cA$ is a $C^{*}$-algebra and
    $K \colon \Omega \times \Omega \to \cL(\cA, \cL(\cE))_{\rm nc}$ is a 
    given function.  Then the following are equivalent.
\begin{enumerate}
   \item $K$ is a cp nc kernel from $\Omega \times \Omega$ to 
   $\cL( \cA, \cL(\cE))_{\rm nc}$.
   
   \item There is a Hilbert space $\cH(K)$ whose elements are nc 
   functions $f \colon \Omega \to \cL(\cA, \cE)_{\rm nc}$ such that:
   \begin{enumerate}
   \item For each $W \in \Omega_{m}$, $v \in \cA^{1 \times m}$, and 
   $y \in \cE^{m}$, the function 
 $$
 K_{W,v,y} \colon \Omega_{n} \to \cL(\cA, \cE)^{n \times n} \cong 
 \cL(\cA^{n}, \cE^{n})
 $$
 defined by
 \begin{equation}   \label{kerel}
 K_{W,v,y}(Z) u = K(Z,W)(uv) y
 \end{equation}
 for $Z \in \Omega_{n}$, $u \in \cA^{n}$ belongs to $\cH(K)$.
 
 \item The kernel elements $K_{W,v,y}$ as in \eqref{kerel} have the reproducing property:  for 
 $f \in \cH(K)$, $W \in \Omega_{m}$, $v \in \cA^{1 \times m}$,
 \begin{equation}   \label{reprod}
  \langle f(W)(v^{*}), y \rangle_{\cE^{m}} = \langle f, K_{W,v,y} 
  \rangle_{\cH(K)}.
 \end{equation}
 
  \item $\cH(K)$ is equipped with a unital $*$-representation  
  $\sigma$ mapping $\cA$ to  $\cL(\cH(K))$  such that
 \begin{equation}  \label{rep1}
     \left( \sigma(a) f\right)(W)(v^{*}) = f(W)(v^{*}a)
 \end{equation}
 for $a \in \cA$, $W \in \Omega_{m}$, $v \in \cA^{1 \times m}$, with 
 action on kernel elements $K_{W,v,y}$ given by
 \begin{equation}   \label{kerelaction}
     \sigma(a) \colon K_{W,v,y} = K_{W, av, y}.
  \end{equation}
  We then say that $\cH(K)$ is the \textbf{noncommutative Reproducing 
  Kernel Hilbert Space} (nc RKHS) associated with the cp nc kernel 
  $K$.
 \end{enumerate}
 
 \item $K$ has a \textbf{Kolmogorov decomposition}:  there is a Hilbert space 
 $\cX$ equipped with a unital $*$-representation $\sigma \colon \cA 
 \to \cL(\cX)$ together with a nc function $H \colon \Omega \to 
 \cL(\cX, \cE)_{\rm nc}$ so that
 \begin{equation}   \label{Koldecom}
     K(Z,W) (P) = H(Z) ( {\rm id}_{{\mathbb C}^{n \times m}} \otimes \sigma) (P) H(W)^{*}
 \end{equation}
 for all $Z \in \Omega_{n}$, $W \in \Omega_{m}$, $P \in \cA^{n \times 
 m}$.
\end{enumerate}
\end{theorem}

\begin{remark}  \label{R:Omega-gen}
Theorem 3.1 in \cite{BMV1} assumes that $\Omega$ is a 
nc subset of $\cV_{\rm nc}$ rather than an arbitrary subset.  
    However, as is explained in Proposition \ref{P:extcpncker} below, 
    one can always extend a nc function/nc kernel/nc cp kernel on 
    $\Omega$ uniquely to a nc function/nc kernel/nc cp kernel 
    respectively on the nc envelope $[\Omega]_{\rm nc}$.  With this 
    fact in hand, one can see that there is no harm done in taking $\Omega$ to be 
    an arbitrary subset of $\cV_{\rm nc}$ in Theorem \ref{T:cpker}.
    \end{remark}

\begin{example}  \label{E:linear}  Suppose that $\cV$,
    $\cV_{0}$ and $\cV_{1}$ are complex vector spaces, and $\varphi \in 
    \cL(\cV_{1}, \cV_{0})$ is a linear operator.  There are two distinct 
    procedures (at least) for relating $\varphi$ to our nc function 
    theory.
    
    \smallskip
    
    \textbf{(a)} Define a function $\bphi \colon \cV_{\rm nc}
    \times \cV_{\rm nc} \to \cL(\cV_{1}, \cV_{0})_{\rm nc}$ by
 \begin{equation}  \label{varphinm}
\bphi(Z,W) = {\rm id}_{{\mathbb C}^{n \times m}} \times \varphi \colon
[ a_{ij} ]_{\sm{ 1 \le i \le n; \\ 1 \le j \le m}} \mapsto 
[\varphi(a_{ij}) ]_{\sm{ 1 \le i \le n; \\ 1 \le j \le m}}
\end{equation}
 for  $Z \in \Omega_{n}$ and $W \in \Omega_{m}$ (so $\bphi(Z,W) \in 
 \cL(\cV_{1}^{n \times m}, \cV_{0}^{n \times m})$ for $Z \in \Omega_{n}$, $W 
 \in \Omega_{m}$).  Thus $\bphi(Z,W) \in \cL(\cV_{1}, \cV_{0})$ 
 depends on $Z,W$ only through the respective sizes $n$ and $m$:  $Z 
 \in \cV^{n \times n}$, $W \in \cV^{m \times m}$.
 Thus we really have
 $$
 \bphi(Z,W) = \varphi^{(n,m)} := {\rm id}_{{\mathbb C}^{n \times m}} 
 \otimes \varphi \colon \cV_{1}^{n \times m} \to \cV_{0}^{n \times m}.
$$
The computation, for $\alpha \in {\mathbb C}^{n \times N}$, $X \in 
\cV^{N \times M}$, $\beta \in {\mathbb C}^{M \times m}$,
\begin{align}
    \alpha \cdot \varphi^{(N,M)}(X) \cdot \beta & =
    [ \alpha_{ij}] \cdot [ \varphi(X_{ij}) ] \cdot [\beta_{ij}] 
    \notag \\
    & = [ \sum_{k,\ell} \alpha_{ik} \varphi(X_{ij}) \beta_{\ell 
    j} ] \notag \\
& = [ \varphi( \sum_{k,\ell} \alpha_{ik} X_{k \ell} 
\beta_{\ell j})] \notag \\
& = \varphi^{(n,m)}( \alpha \cdot X \cdot \beta)
\label{computation}
\end{align}
shows that $\bphi$ enjoys the \textbf{bimodule property}:
\begin{equation}   \label{bimodule1}
    \alpha \cdot \bphi(Z,W)(X)\cdot \beta = \bphi(\widetilde Z, 
    \widetilde W)( \alpha \cdot 
    X \cdot \beta)
\end{equation}
for $\alpha \in {\mathbb C}^{n \times N}$, $X \in \cV_{0}^{N \times 
M}$, $\beta \in {\mathbb C}^{M \times m}$, $Z \in \cV^{N \times N}$, 
$W \in \cV^{M \times M}$, $\widetilde Z \in \cV^{n \times n}$, 
$\widetilde W \in \cV^{m \times m}$.  The ``respects intertwining 
property'' \eqref{kerintertwine} for the kernel $\bphi$ says that 
\eqref{bimodule1} holds whenever $Z,W,\widetilde Z, \widetilde W$ are related via
\begin{equation}   \label{intertwine-bphi}
\alpha Z = \widetilde Z \alpha, \quad \beta^{*} W = \widetilde W 
\beta^{*}.
\end{equation}
We conclude that the bimodule property \eqref{bimodule1} is formally 
stronger than the ``respects intertwining property'' in that, for 
given $\alpha, \beta$, the bimodule property does not require one to 
search for points $Z,\widetilde Z, W, \widetilde W$ which satisfy the 
intertwining conditions \eqref{intertwine-bphi}. In any case, we 
conclude that $\bphi$ so defined is a nc kernel. Let us say that a nc 
kernel of this form is a  \textbf{nc constant kernel}.

In case $\cV_{1}$ and $\cV_{0}$ are operator systems and $\varphi$ 
is a completely positive map in the sense of the operator algebra literature 
(see \cite{ERbook, Paulsen}), the resulting kernel $\bphi$ is furthermore 
a completely positive 
nc kernel; this example is discussed in some detail in \cite[Section 
3.3]{BMV1}, 

\smallskip

\textbf{(b)}  For this construction we suppose that we have given 
only two vector spaces $\cV$ and $\cV_{0}$ and that $\varphi$ is a linear 
map from $\cV$ to $\cV_{0}$. Define a map $L \colon \cV_{\rm nc} \to 
\cV_{0, {\rm nc}}$ by
$$
  L_{\varphi}(Z) = \varphi^{(n,n)}(Z): = [ \varphi(Z_{ij})] \text{ 
  for } Z = [Z_{ij}] \in \cV^{n \times n}
$$
where we use the notation $\varphi^{(n,n)}$ as in Example 
\ref{E:linear} (a) above. From the definition we see that $L$ is 
graded, i.e., $L \colon \cV^{n \times n} \to \cV_{0}^{n \times n}$ 
for all $n \in {\mathbb N}$.  To check the ``respects intertwining'' 
property \eqref{funcintertwine}, we use the bimodule property 
\eqref{computation} to see that, if $\alpha \in {\mathbb C}^{m \times 
n}$, $Z \in \cV^{n \times n}$, $\widetilde Z \in \cV^{m \times m}$ 
are such that $\alpha Z = \widetilde Z \alpha$, then
\begin{align*}
\alpha \cdot L_{\varphi}(Z) & = \alpha \cdot \varphi^{(n \times n)}(Z) =
\varphi^{(m \times n)}( \alpha \cdot Z) \\
& = \varphi^{(m \times n)}( \widetilde Z \cdot \alpha) =
\varphi^{(m \times m)}(\widetilde Z) \cdot \alpha =
L_{\varphi}(\widetilde Z) \cdot \alpha
\end{align*}
and it follows that $L_{\varphi}$ is a nc function.  We shall say 
that a nc function of this form is a \textbf{nc linear map}.  In 
Section \ref{S:relativelyfull} we shall be particularly interested in 
the special case where $\cV_{0} = \cL(\cR, \cS)$, the space of 
bounded linear operators between two Hilbert spaces $\cR$ and $\cS$.
\end{example}

\subsection{Full noncommutative sets}    \label{S:ncdisk}
We shall be interested in the nc set $\Omega \subset \cV_{\rm nc}$ 
on which our noncommutative functions are defined having some additional structure.  
In all these examples we suppose that $\cV_{\rm nc}$ is the 
noncommutative set generated by a vector space $\cV$. 

\begin{definition}   \label{D:fullsubset}
We say that a subset $\Xi$ of $\cV_{\rm nc}$ is a  \textbf{full nc subset} 
of $\cV_{\rm nc}$ if the following conditions hold:
\begin{enumerate}
    \item $\Xi$ is \textbf{closed under direct sums}:  $Z \in \Xi_{n}$, $W \in 
    \Xi_{m}$ $\Rightarrow$ $\sbm{ Z & 0 \\ 0 & W} \in 
    \Xi_{n+m}$.
    
    \item $\Xi$ is \textbf{invariant under left injective 
    intertwinings}:  $Z \in \Xi_{n}$, $\widetilde Z \in \cV^{m \times 
    m}$ such that $\cI \widetilde Z = Z \cI$ for some injective $\cI 
    \in {\mathbb C}^{n \times m}$ (so $n \ge m$) $\Rightarrow$ 
    $\widetilde Z \in \Xi_{m}$.
 \end{enumerate}
\end{definition}

\noindent
An equivalent version of property (2) in Definition 
\ref{D:fullsubset} is that  $\Xi \subset \cV_{\rm nc}$ 
is \textbf{closed under
restriction to invariant subspaces}:  {\em whenever there is 
an invertible $\alpha \in {\mathbb C}^{n \times n}$ and a $Z \in \Xi$ 
of size $n \times n$ such that $\alpha^{-1} Z \alpha = \sbm{ 
\widetilde Z & Z_{12} \\ 0 & Z_{22}}$ with $\widetilde Z$ of size $m 
\times m$ ($m \le n$), then not only is $\sbm{ \widetilde Z & Z_{12} 
\\ 0 & Z_{22}}$ in $\Xi$ but also $\widetilde Z$ is in $\Xi$.}  Here 
we view the ${\mathbb C}$-linear space equal to the span of the first 
$m$ columns of $\alpha$ as an invariant subspace for $Z$ with matrix 
representation $\widetilde Z$ determined by
$$
 Z \left( \alpha \sbm{ I_{m} \\ 0 } \right) = \left( \alpha \sbm{ 
 I_{m} \\ 0 } \right) \widetilde Z.
$$
Note that here $\left( \alpha \sbm{ I_{m} \\0} \right)$ is an $n \times 
m$ matrix over ${\mathbb C}$ while $Z$ and $\widetilde Z$ are 
matrices of respective sizes $n \times n$ and $m \times m$ over $\cV$.
For a more concrete illustrative example, see Example \ref{E:ncpoly} 
below.

We next suppose that we are given two coefficient Hilbert spaces 
$\cS$ and $\cR$ and 
that $Q \colon \Xi \to \cL(\cR, \cS)$ is a nc function.
We associate with any such $Q$  the nc set ${\mathbb D}_{Q} \subset 
\Xi$ defined by
\begin{equation}
{\mathbb D}_{Q} =   \{ Z  \in \Xi  \colon \|Q(Z)\| < 1\}. \label{defDQ}
\end{equation}
Here the norm of $Q(Z)$ is taken in $\cL(\cR, \cS)^{n \times n} \cong 
\cL(\cR^{n}, \cS^{n})$ if 
$Z \in \Xi_{n}$.

The reader is welcome to have in mind the following examples as 
illustrative special cases of the general setup.

\begin{example}   \label{E:ncpoly} 
    We let $\cV$ be the vector space ${\mathbb C}^{d}$ of $d$-tuples 
    of complex numbers with $\Xi = \cV_{\rm nc}: = 
    \amalg_{n=1}^{\infty} ({\mathbb C}^{d})^{n \times n}$.  We 
    identify $({\mathbb C}^{d})^{n \times n}$ ($n \times n$ matrices 
    with entries from ${\mathbb C}^{d}$) with $({\mathbb C}^{n \times 
    n})^{d}$  ($d$-tuples of $n \times n$ complex matrices) and hence 
    we may view $\Xi$ as $\amalg_{n=1}^{\infty} ({\mathbb C}^{n 
    \times n})^{d}$.  Then we write an element $Z \in \Xi_{n}$ (the 
    elements of $\Xi \cap ({\mathbb C}^{n \times n})^{d})$ as a 
    $d$-tuple $Z = (Z_{1}, \dots, Z_{d})$ where each $Z_{i} \in 
    {\mathbb C}^{n \times n}$.  
Then an \textbf{invariant subspace} for a point $Z \in ({\mathbb C}^{n \times 
n})^{d}$ (as in the context of the reformulation of ``invariance under 
left injective intertwinings''  (see the discussion immediately after 
Definition \ref{D:fullsubset} above) amounts 
to a joint invariant subspace for the matrices $Z_{1}, \dots, Z_{n}$ in the classical
sense.

In the context of this example, we may define a notion of \textbf{noncommutative 
matrix polynomial}, by which we mean a formal expression of the form
\begin{equation}   \label{ncpoly}
  Q(z) = \sum_{{\mathfrak a} \in \free} Q_{\mathfrak a} z^{\mathfrak a}
\end{equation}
where the coefficients $Q_{\mathfrak a} \in {\mathbb C}^{s \times r}$ are complex matrices 
with all but finitely many equal to $0$,
as in the discussion above leading up to the statement of Theorem 
\ref{T:BGM2}.
 Such a formal expression 
$Q(z)$ defines a nc function from $({\mathbb C}^{d})_{\rm nc}$ to 
$({\mathbb C}^{s \times r})_{\rm nc}$ if, for $Z = (Z_{1}, \dots, 
Z_{d}) \in ({\mathbb C}^{n \times n})^{d}$ we define
$$
Q(Z) = \sum_{{\mathfrak a} \in \free} Q_{\mathfrak a} \otimes 
Z^{\mathfrak a} \in 
{\mathbb C}^{rn \times sn} \cong ({\mathbb C}^{r \times s})^{n\times n}
$$
as explained in the Introduction.
Then the associated $Q$-disk ${\mathbb D}_{Q}$ consists of all points 
$Z = (Z_{1}, \dots, Z_{d}) \in ({\mathbb C}^{n \times n})^{d}$ such 
that $\| Q(Z) \| < 1$. A set of this form can be thought of as a noncommutative analogue of 
a  semi-algebraic set as defined in real algebraic geometry (see 
e.g.\ \cite{BCR98}).  We mention some particular cases:
\begin{enumerate}
    \item If $Q(z) = \begin{bmatrix} z_{1} & \cdots & z_{d} 
\end{bmatrix}$ (a $1 \times d$ nc polynomial matrix), then the 
associated disk ${\mathbb D}_{Q}$ consists of $d$-tuples $Z = (Z_{1}, 
\dots, Z_{d})$ for which $\| \begin{bmatrix} Z_{1} & \cdots & Z_{d} 
\end{bmatrix}\| < 1$, or equivalently, for which $Z_{1} Z_{1}^{*} + 
\cdots + Z_{d} Z_{d}^{*} \prec I_{n}$.  We refer to this set as the 
\textbf{noncommutative ball}.

\item If $Q(z) = \sbm{ z_{1} & & \\ & \ddots & \\ & & z_{d} }$ (a $d 
\times d$ nc polynomial matrix), then the associated $Q$-disk 
${\mathbb D}_{Q}$ consists of $d$-tuples of $n \times n$ matrices
$Z = (Z_{1}, \dots, Z_{d})$ such that $\left\| \sbm{ Z_{1} & & \\ & 
\ddots & \\ & & Z_{d}} \right\| < 1$, or equivalently, for which 
$Z_{i}^{*} Z_{i} \prec I_{n}$ for each $i= 1, \dots, d$. We refer to 
this set as the \textbf{noncommutative polydisk}.
\end{enumerate}
\end{example}

\begin{example}  \label{E:polyhalfplane}
We next present an infinite-dimensional example. If $\cC$ is a 
$C^{*}$-algebra, let ${\mathbb H}^{+}(\cC)$ and ${\mathbb H}^{-}(\cC)$
be the upper and lower half-planes over $\cC$ given by
\begin{align*}
    & {\mathbb H}^{+}(\cC) = \left\{ a \in \cC \colon {\rm Im}\, a = 
    \frac{a - a^{*}}{2} > 0 \right\}, \\
    & {\mathbb H}^{-}(\cC) = \left\{ a \in \cC \colon {\rm Im}\, a = 
    \frac{a - a^{*}}{2} < 0 \right\}.
\end{align*}
We then define the  upper and lower fully matricial half-planes 
${\mathbb H}^{+}(\cA_{\rm nc})$ and ${\mathbb H}^{-}(\cA_{\rm nc})$ 
over a $C^{*}$-algebra $\cA$ by
$$
  {\mathbb H}^{\pm}(\cA_{\rm nc}) = \amalg_{n=1}^{\infty} {\mathbb 
  H}^{\pm}(\cA^{n \times n}).
$$
If we let the underlying vector space $\cV$ be equal to the 
$C^{*}$-algebra $\cA$, let $\Xi = \{ a \in \cA \colon i 1_{\cA} + 
a \text{ invertible}\}$ and then define a nc function $Q_{\pm}$ from $\Xi$ to 
$\cA_{\rm nc}$ by $Q_{+}(A) = (A + i 1_{\cA^{n \times 
n}})^{-1} (A - i  1_{\cA^{n \times n}})$ and $Q_{-}(A) = 
(A - i 1_{\cA^{n \times n}})^{-1} (A + i 1_{\cA^{n \times n}})$
for $A \in \cA^{n \times n}$, then we recover ${\mathbb H}^{\pm}(\cA_{\rm 
nc})$ as ${\mathbb H}^{\pm}(\cA_{\rm nc}) = {\mathbb D}_{Q_{\pm}}$.
This example comes up in free probability (see \cite{PV}).
\end{example}

\begin{example}  \label{E:comtuples}
Let $\cV = {\mathbb C}^{d}$ and $\cV_{\rm nc} \cong \amalg_{n=1}^{\infty} 
({\mathbb C}^{n \times n})^{d}$ as in Example \ref{E:ncpoly}.
Let $G$ be an open subset of ${\mathbb C}_{d}$ and
define a subset $\Xi$ of $({\mathbb C}^{d})_{\rm nc}$ to consist 
of commutative $d$-tuples of $n \times n$ matrices $Z = (Z_{1}, 
\dots, Z_{d})$ (so each $Z_{i} \in {\mathbb C}^{n \times n}$ and 
$Z_{i}Z_{j} = Z_{j} Z_{i}$ for all pairs of indices $i,j = 1, \dots, 
d$ for $n=1,2,\dots$) such that the joint spectrum $\sigma_{\rm joint}(Z)$ (i.e., the 
set of joint eigenvalues which is the same as the Taylor spectrum of 
$Z$ for this matrix case) is contained in $G$.  Now it is an easy exercise 
to verify that $\Xi$ so defined is a full nc subset of $({\mathbb 
C}^{d})_{\rm nc}$.  Suppose next that $\cR$ and $\cS$ are two 
coefficient Hilbert spaces and that $q$ is a $d$-variable 
holomorphic function on $G$ with values in $\cL(\cR, \cS)$ such that 
$\|q(z) \| < 1$ for all $z \in G$.  As the joint spectrum and the Taylor 
spectrum are the same for commutative matrix tuples, for $Z = (Z_{1}, 
\dots, Z_{d}) \in \Xi_{n}$, we can define $q(Z) \in \cL(\cR^{n}, 
\cS^{n})$ by the Taylor functional calculus or the Martinelli-Vasilescu 
functional calculus (see e.g. \cite{Curto}). If we define the 
associated $q$-disk ${\mathbb D}_{q}$ by
$$
  {\mathbb D}_{q} = \{ Z \in \Xi \colon \| q(Z) \| < 1\},
$$
then by construction we have ${\mathbb D}_{q,1} = \Xi_{q,1}$.
Elementary properties of the Taylor/Martinelli-Vasilescu functional calculus 
$f \mapsto f(Z)$ are that the ``respects direct sums'' and ``respects similarities'' 
properties are satisfied and thus $f$ so defined is a locally bounded nc function on 
${\mathbb D}_{q}$ (locally bounded referring to a neighborhood of a given point $Z^{(0)} = 
(Z_{1}^{(0)}, \dots, Z_{d}^{(0)}) \in {\mathbb D}_{q, n}$ 
in the standard Euclidean topology of  $({\mathbb C}^{n \times n})^{d}$).  
Conversely, any such locally bounded nc function is analytic (see \cite[Theorm 7.4]{KVV-book}), 
and its definition on ${\mathbb D}_{q, n}$ is determined by its 
definition on ${\mathbb D}_{q,1}$ via the 
Taylor/Martinelli-Vasilescu functional calculus. Thus the nc 
function theory can be used to get results for the standard theory of 
several (commuting) complex variables.  We note that domains of the type 
${\mathbb D}^{\infty}_{q}$ (where one uses commutative-operator 
$d$-tuples $Z = (Z_{1}, \dots, Z_{d})$ rather than commutative 
finite-matrix $d$-tuples as here) come up in the definition 
of the commutative Schur-Agler classes in \cite{AT, BB04, MP}.  We 
have more to say on this setup in Section \ref{S:commutative} 
and Remark \ref{R:AKV}.
\end{example}

\subsection{Noncommutative envelopes}
\label{S:fgncset}

We fix a noncommutative function $Q$ on the full nc subset $\Xi$ of 
$\cV_{\rm nc}$ as in Subsection \ref{S:ncdisk} and suppose that 
$\Omega_{0}$ is a subset of ${\mathbb D}_{Q}$.  The following 
three notions of nc subset generated by $\Omega_{0}$ will be useful 
in the sequel.

\begin{definition}  \label{D:gen}
    \begin{enumerate} 
	\item We say that $\Omega$ is the 
	\textbf{noncommutative envelope of $\Omega_{0}$} (or 
	\textbf{nc envelope of $\Omega_{0}$} for short) (notation: 
	$\Omega = [\Omega_{0}]_{\rm nc}$) if $\Omega$ 
	is equal to the smallest subset of $\Xi$ 
	containing $\Omega_{0}$ which is closed under direct sums:
	\begin{equation}   \label{prop1'}
	Z \in \Omega, \, W \in \Omega \, \Rightarrow \begin{bmatrix} Z & 0 \\ 0 & 
	W \end{bmatrix} \in \Omega.
	\end{equation}
Equivalently,
\begin{equation}  
[\Omega_{0}]_{\rm nc} = \bigcap \{ \Omega \subset \Xi \colon \Omega \supset \Omega_{0},\, 
\Omega \text{ satisfies 
\eqref{prop1'}}\}.
 \label{prop1''}
 \end{equation}
 
 We say that the set $[\Omega_{0}]_{\rm nc} \cap {\mathbb D}_{Q}$ is 
 the \textbf{${\mathbb D}_{Q}$-relative nc envelope} of $\Omega_{0}$.
	
	\item  We say that $\Omega$ is the \textbf{noncommutative 
	similarity-invariant envelope of $\Omega_{0}$} (or simply \textbf{nc 
	similarity envelope of $\Omega_{0}$}) (notation: $\Omega = 
	[\Omega_{0}]_{\rm nc, \, sim}$) if $\Omega$ is the 
	smallest subset of $\Xi$ containing $\Omega_{0}$ 
	which is closed under direct sums \eqref{prop1'} and under similarity 
	transforms:
	\begin{equation}   \label{prop2'}
 Z \in \Omega_{n}, \, \alpha \in {\mathbb C}^{n \times n} \text{ 
 invertible }  \Rightarrow \alpha Z \alpha^{-1} \in \Omega_{n}.
\end{equation}
Equivalently,
\begin{equation}  
[\Omega_{0}]_{\rm nc, \, sim} = \bigcap \{ \Omega \subset \Xi \colon \Omega \supset \Omega_{0},\, 
\Omega \text{ satisfies \eqref{prop1'} and \eqref{prop2'}}\}.
 \label{prop2''}
\end{equation}

We say that the set $[\Omega_{0}]_{\rm nc, sim} \cap {\mathbb D}_{Q}$ 
is the \textbf{${\mathbb D}_{Q}$-relative nc similarity envelope} of 
$\Omega_{0}$.

\item We say that $\Omega$ is the \textbf{full nc envelope of 
$\Omega_{0}$} (notation: $\Omega = [\Omega_{0}]_{\rm full}$) if $\Omega$  
is the smallest subset of $\Xi$ containing $\Omega_{0}$ which is closed under direct sums 
\eqref{prop1'} and  under left injective intertwinings
(see Definition \ref{D:fullsubset}).
Equivalently,
\begin{equation}   \label{prop3''}
    [\Omega_{0}]_{\rm full} =  \bigcap \{ \Omega \subset \Xi \colon 
    \Omega \supset \Omega_{0}, \, \Omega \text{ is a full nc subset}  
    \text{ as in Definition \ref{D:fullsubset}} \}.
\end{equation}

We say that the set $[\Omega_{0}]_{\rm full} \cap {\mathbb D}_{Q}$ is 
the \textbf{${\mathbb D}_{Q}$-relative full nc envelope} of $\Omega_{0}$.
\end{enumerate}
\end{definition}

We note that the properties \eqref{prop1'}, \eqref{prop2'}, as well 
as properties (1) and (2) in Definition \ref{D:fullsubset} for 
subsets $\Omega$ of $\Xi$ are invariant under intersection.  Hence it becomes clear that the 
intersection of all supersets of $\Omega_{0}$ satisfying some 
combination of properties \eqref{prop1'}, \eqref{prop2'}, (1)--(2) 
again  satisfies the same combination, and hence the 
three respective envelopes exist and are alternatively characterized 
by \eqref{prop1''}, \eqref{prop2''}, and \eqref{prop3''} respectively.

Alternatively, these envelopes can be described more constructively 
via a \textbf{bottom-up} (rather than \textbf{top-down} as in 
Definition \ref{D:gen}) procedure, as explained in the following 
Proposition.

\begin{proposition}   \label{P:bottom-up}
    Let $\Omega_{0}$ be a subset of a full nc set $\Xi \subset \cS_{\rm nc}$.
    Then:
\begin{enumerate}	
\item  The nc-envelope $[\Omega_{0}]_{\rm nc}$ consists of all 
matrices of the form
$$
    Z = \begin{bmatrix} Z^{(1)} & & \\ & \ddots & \\ & & Z^{(N)} 
\end{bmatrix}
$$
where each $Z^{(j)} \in \Omega_{0}$ where $N = 1,2,\dots$.

\item The nc similarity envelope $[\Omega_{0}]_{\rm nc, sim}$ 
consists of all matrices $Z$ in $\Xi$ such that 
there exists a square invertible complex matrix $\alpha$ so 
that $\alpha Z \alpha^{-1} \in [\Omega_{0}]_{\rm nc}$.

\item The full nc envelope $[\Omega_{0}]_{\rm full}$ consists of all 
$\widetilde Z \in \Xi$ for which there is a $Z \in [\Omega_{0}]_{\rm 
nc}$ and a left injective intertwiner $\cI$ so that 
$\cI \widetilde Z = Z \cI$.
\end{enumerate}
Consequently, we have the chain of containments
\begin{equation}   \label{chain}
    \Omega_{0} \subset [\Omega_{0}]_{\rm nc} \subset 
    [\Omega_{0}]_{\rm nc, sim} \subset [\Omega_{0}]_{\rm full}.
\end{equation}

\end{proposition} 

\begin{proof}  Let $[\Omega_{0}]^{o}_{\rm nc}$ be the 
    set of all matrices of the form $Z = \sbm{ Z^{(1)} & & \\ & 
    \ddots & \\ & & Z^{(N)} }$ as in statement (1). Then it is clear 
    that necessarily $[\Omega_{0}]^{o}_{\rm nc} \subset 
    [\Omega_{0}]_{\rm nc}$.  Similarly, if $[\Omega_{0}]^{o}_{\rm nc, 
    sim}$ and $[\Omega_{0}]^{o}_{\rm full}$ are the candidate sets 
    for $[\Omega_{0}]_{\rm nc, sim}$ and $[\Omega_{0}]_{\rm full}$ 
    described in statements (2) and (3) of the Proposition, it is 
    clear that $[\Omega_{0}]^{o}_{\rm nc, sim} \subset 
    [\Omega_{0}]_{\rm nc, sim}$ and that $[\Omega_{0}]^{o}_{\rm full} 
    \subset [\Omega_{0}]_{\rm full}$.  Thus it remains only to verify 
    the reverse containments.
    
    Suppose that $Z = \sbm{ Z^{(1)} & & \\ & \ddots & \\ & & 
    Z^{(N)}}$ where each $Z^{(j)}$ is in $[\Omega_{0}]^{o}_{\rm nc }$ 
    and  hence has in turn a direct-sum decomposition 
    $Z^{(j)} = \sbm{ Z^{(j,1)} & & \\ & \ddots & \\ & & 
    Z^{(j,n_{j})}}$ with each matrix $Z^{(j,\ell)}$ coming from 
    $\Omega_{0}$.  
    Then it is clear that
    $$
      Z = \sbm{ Z^{(1,1)} & & & & & & \\ & \ddots & & & & & \\
       & & Z^{(1,n_{1})} & & & & \\
      & & & \ddots & & & \\ & & & & Z^{(N,1)} & & \\
      & & & & & \ddots & \\  & & & & & & Z^{(N, n_{N})} }
   $$
   is the direct sum of $n_{1} + \cdots + n_{N}$ elements from 
   $\Omega_{0}$. i.e., $Z \in [\Omega_{0}]^{o}_{\rm nc}$.  Thus 
   $[\Omega_{0}]^{o}_{\rm nc}$ is a nc set containing $\Omega_{0}$ 
   and the reverse containment $[\Omega_{0}]_{\rm nc} \subset 
   [\Omega_{0}]^{o}_{\rm nc}$ follows.
   
   The fact that $[\Omega_{0}]^{o}_{\rm nc, sim}$ is a nc set 
   invariant under similarity is the content of Proposition A.1 from 
   \cite{KVV-book} and obviously $[\Omega_{0}]^{o}_{\rm nc, sim}$ 
   contains $\Omega_{0}$.  Hence the reverse containment 
   $[\Omega_{0}]_{\rm nc, sim} \subset [\Omega_{0}]^{o}_{\rm nc, sim}$ follows.
   
   Finally, since $[\Omega_{0}]^{o}_{\rm full}$ contains 
   $\Omega_{0}$, to show the reverse containment 
   $[\Omega_{0}]_{\rm full} \subset [\Omega_{0}]^{o}_{\rm full}$, it 
   suffices to show that $[\Omega_{0}]^{o}_{\rm full}$ is closed 
   under direct sums and left injective intertwinings.  
   
   \smallskip
   
   \noindent
   \textbf{Closure under direct sums:} Suppose that $\widetilde 
   Z^{(1)}$ and $\widetilde Z^{(2)}$  are in $[\Omega_{0}]^{o}_{\rm 
   full}$.  Then by part (1) of the Proposition already proved,
   there is a $Z^{(1)} \in [\Omega_{0}]^{o}_{\rm nc}$ and a $Z^{(2)} 
   \in [\Omega_{0}]^{o}_{\rm nc}$ together with injective 
   $\cI^{(1)}$ and $\cI^{(2)}$ so that $\cI^{(1)} \widetilde Z^{(1)} 
   = Z^{(1)} \cI^{(1)}$ and $\cI^{(2)} \widetilde Z^{(2)} = Z^{(2)} 
   \cI^{(2)}$.  Then
   $\sbm{ \cI^{(1)} & 0 \\ 0 & \cI^{(2)}}$ is also injective and
 $$
 \begin{bmatrix} \cI^{(1)} & 0 \\ 0 & \cI^{(2)} \end{bmatrix}
\begin{bmatrix} \widetilde Z^{(1)} & 0 \\ 0 & \widetilde Z^{(2)} 
\end{bmatrix}
= \begin{bmatrix} Z^{(1) }& 0 \\ 0 & Z^{(2)} \end{bmatrix} 
 \begin{bmatrix} \cI^{(1)} & 0 \\ 0 & \cI^{(2)} \end{bmatrix}
 $$
where $\sbm{ Z^{(1)} & 0 \\ 0 & Z^{(2)} }$ is in 
$[\Omega_{0}]^{o}_{\rm nc}$ since both $Z^{(1)}$ and $Z^{(2)}$ are in 
$[\Omega_{0}]^{o}_{\rm nc}$.   This enables us to conclude that 
$\sbm{ \widetilde Z^{(1)} & 0 \\ 0 & \widetilde Z^{(2)}}$ is in
$[ \Omega]^{o}_{\rm full}$.

\smallskip

\noindent
\textbf{Closure under left injective intertwinings:}
 Suppose that $\cI$ is injective, $Z \in [ \Omega_{0}]^{o}_{\rm 
 full}$ and $\widetilde Z \in \Xi$ satisfies $\cI \widetilde Z = Z 
 \cI$.  We wish to show that $\widetilde Z \in [\Omega_{0}]^{o}_{\rm 
 full}$.  Toward this end, observe first of all that $Z \in 
 [\Omega_{0}]^{o}_{\rm full}$ means that there is an injective 
 $\cI_{0}$ and a $W \in [ \Omega_{0}]^{o}_{\rm nc}$ so that $\cI_{0} 
 Z = W \cI_{0}$.  Then we see that
 $$
  W \cI_{0} \cI = \cI_{0} Z \cI = \cI_{0} \cI \widetilde Z.
 $$
 As $W \in [\Omega_{0}]^{o}_{\rm nc}$ and $\cI_{0} \cI$ is again 
 injective, it follows that $\widetilde Z \in [\Omega]^{o}_{\rm 
 full}$ as wanted.
 
 The chain of containments \eqref{chain} is an immediate consequence 
 of the respective envelope characterizations in parts (1), (2), (3) 
 of the Proposition.
\end{proof}

The key property of  finitely generated nc subsets is given 
by the following lemma.  

\begin{proposition}  \label{P:finitetype}
    Let $\Omega$ be any one of the three ${\mathbb D}_{Q}$-relative envelopes 
    $[\Omega_{0}]_{\rm nc} \cap 
    {\mathbb D}_{Q}$, $[\Omega_{0}]_{\rm nc,\, 
    sim} \cap {\mathbb D}_{Q}$, or $[\Omega_{0}]_{\rm full} \cap 
    {\mathbb D}_{Q}$.
    \begin{enumerate} 
	\item Suppose that $f \in \cT(\Omega; \cE_{\rm nc})$ is such 
	that $f|_{\Omega_{0}} \equiv 0$. Then also $f \equiv 0$.  
	Hence any function $f \in \cT(\Omega; \cE_{\rm nc})$ is 
	uniquely determined by its restriction ${\mathfrak R} f : = 
	f|_{\Omega_{0}}$ to $\Omega_{0}$.
	
	\item  Suppose also that $\Omega_{0}$ is a finite subset of ${\mathbb D}_{Q}$ 
   and that $\cE$ is a finite-dimensional Hilbert space.  Then the vector space
    $\cT(\Omega; \cL(\cE)_{\rm nc})$ of all $\cL(\cE)$-valued nc functions on 
    $\Omega$ has finite dimension:
    $$
      \dim \cT(\Omega; \cL(\cE)_{\rm nc}) < \infty.
    $$
    \end{enumerate}
  \end{proposition}
  
\begin{proof}  From the chain of containments \eqref{chain}, we see 
    that it suffices to consider the case where $\Omega = 
    [\Omega_{0}]_{\rm full}$.
    
    We let ${\mathfrak R}$ be the restriction map ${\mathfrak R} \colon f 
    \mapsto f|_{\Omega_{F}}$ for $f \in \cT(\Omega; \cL(\cE)_{\rm 
    nc})$.  Let $\cF_{\rm gr}(\Omega_{F}, \cL(\cE)_{\rm nc})$ be the linear space of all 
   graded  $\cL(\cE)_{\rm nc}$-valued functions on $\Omega_{F}$.  Note that
    $$
     {\mathfrak R} \colon \cT(\Omega; \cL(\cE)_{\rm nc}) \to \cF_{\rm 
     gr}(\Omega_{F}, \cL(\cE)_{\rm nc}).
    $$
   As ${\mathfrak R}$ is linear, to show that $f$ is 
    uniquely determined by ${\mathfrak R}f$ it suffices to show that 
    ${\mathfrak R}$ is injective, i.e.: ${\mathfrak R}f \equiv 0$ 
    $\Rightarrow$ $f \equiv 0$.
    Suppose therefore that $f \in \cT(\Omega; \cL(\cE)_{\rm nc})$ 
    vanishes on $\Omega_{0}$.  Let ${\mathfrak Z} = \{ Z \in \Omega 
    \colon f(Z) = 0\}$.  By assumption $\Omega_{0} \subset {\mathfrak 
    Z}$.  Since $f$ as a nc function on $\Omega$ respects direct 
    sums, it follows that ${\mathfrak Z}$ is closed under direct 
    sums and hence ${\mathfrak Z} \supset [\Omega_{0}]_{\rm nc}$.
    Suppose next that $\widetilde Z$ is a point in ${\mathbb D}_{Q, m}$ 
    such that there is an injective matrix $\cI \in {\mathbb C}^{n 
    \times m}$ and a $Z \in {\mathfrak Z}_{n}$  so that $\cI \widetilde Z = 
    Z \cI$.  Since $f$ respects intertwinings, it follows that
    $0 = f(Z) \cI = \cI f(\widetilde Z)$.  As $\cI$ is injective, it 
    follows that $f(\widetilde Z) = 0$, i.e., $\widetilde Z \in 
    {\mathfrak Z}$.
     Thus 
    ${\mathfrak Z}$ has the invariance properties required for the inclusion 
    $\Omega: = [\Omega_{0}]_{\rm full} \cap {\mathbb D}_{Q} 
    \subset {\mathfrak Z}$, and hence $f$ vanishes identically on 
    $\Omega$ and statement (1) of Proposition \ref{P:finitetype} 
    follows.
    
    We now suppose that $\Omega_{0}$ is a finite subset and that 
    $\cE$ is a finite-dimensional Hilbert space.  Enumerate 
    the elements of $\Omega_{0}$ as $\Omega_{0} = \{Z^{(1)}, \dots, Z^{(N)}\}$. 
   Observe that $\cF_{\rm gr}(\Omega_{0}, \cL(\cE)_{\rm nc})$ is finite-dimensional. To 
   see this, say $Z^{(i)} \in \Omega_{F, n_{i}}$ so that the value 
   $f(Z^{(i)})$ of a graded function $f$ at $Z^{ (i)}$ is in 
   $\cL(\cE)^{n_{i} \times n_{i}}$.  Let 
   $\{E^{(n_{i})}_{j,k} \colon 1 \le j, k \le n_{i} \cdot \dim \cE \}$
   be a (finite) basis for $\cL(\cE)^{n_{i} \times n_{i}}$.  For $1 
   \le i \le N$ and $1 \le j,k \le n_{i}\cdot \dim \cE$ set
   $$
    f_{i; jk}(Z) = \begin{cases} E^{(n_{i})}_{jk} & \text{if } Z = 
    Z^{(i)}, \\
    0 & \text{otherwise.}
    \end{cases}
    $$
    Then the collection $\{f_{i,jk} \colon 1 \le i \le N,\, 1 \le j,k 
    \le n_{i} \cdot \dim \cE \}$ is a finite basis for the linear 
    space $\cF_{\rm 
    gr}(\Omega_{0}, \cL(\cE)_{\rm nc})$.
    By part (1) we know that ${\mathfrak R}$ is injective.
    Thus ${\mathfrak R}$ is an injective mapping from the linear 
    space $\cT(\Omega; \cL(\cE)_{\rm nc})$ into the 
    finite-dimensional linear space $\cF_{\rm gr}(\Omega_{0}, 
    \cL(\cE)_{\rm nc})$. It now follows from the null-kernel theorem from Linear 
    Algebra that $\dim \cT(\Omega; \cL(\cE)_{\rm nc}) < \infty$.
\end{proof}

\subsection{Noncommutative Zariski closed sets}  \label{S:Zariski}
We include here some material not needed in the sequel, as it may be 
of independent interest.

Suppose that $\Omega$ is a subset of $\Xi$.  We define the nc-Zariski 
closure $\overline \Omega$ of $\Omega$ by
\begin{equation}  \label{Zariski-close}
    \overline{\Omega} = \{ Z \in \Xi \colon f \in \cT(\Xi; {\mathbb 
    C}_{\rm nc}) \text{ with } f|_{\Omega} = 0 \Rightarrow f(Z) = 0\}.
\end{equation}
We say that $\Omega$ is \textbf{nc-Zariski closed} if $\Omega = 
\overline{\Omega}$.  If $\Omega$ is a subset of ${\mathbb D}_{Q}$, we 
say that $\Omega$ is \textbf{${\mathbb D}_{Q}$-relative nc-Zariski 
closed} if $\Omega = \overline{\Omega} \cap {\mathbb D}_{Q}$.

The next result gives a relation between nc-Zariski closure and the 
full nc envelope of a given set $\Omega$. 

\begin{proposition}   \label{P:Zariski-close} For $\Omega$ any subset 
    of $\Xi$, we have the containment
\begin{equation}  \label{gen/Zariski}
	[\Omega]_{\rm full} \subset \overline{\Omega}.
  \end{equation}
\end{proposition}

\begin{proof} To show that $[\Omega]_{\rm  full} \subset 
    \overline{\Omega}$, it suffices to show:
    \begin{equation}   \label{toshow}
f \in \cT(\Xi; {\mathbb C}),\, f|_{\Omega} = 0,\,
Z \in [\Omega]_{\rm  full} \Rightarrow f(Z) = 0.
\end{equation}
Note that $f \in \cT(\Xi; {\mathbb C}) \Rightarrow f|_{[\Omega]_{\rm 
full}} \in \cT([\Omega]_{\rm  full}; {\mathbb C})$. Hence the 
implication \eqref{toshow} follows directly from Proposition 
\ref{P:finitetype}.
\end{proof}
  
  A natural question is whether the containment \eqref{gen/Zariski} 
  is in fact an equality.   For the case where $\cV = {\mathbb C}$, 
  $\Xi = \cV_{\rm nc} = {\mathbb C}_{\rm nc}$
(e.g. if $Q$ is a single-variable nc scalar polynomial 
$Q$), and $\Omega$ is taken to be a finite subset, the answer is in the affirmative.

\begin{proposition}  \label{P:Zariski-close-d=1}
   Suppose that $\Omega_{F}$ is a finite subset of ${\mathbb D}_{Q}$ 
   where $Q \in \cT({\mathbb C}_{\rm nc}; {\mathbb C}_{\rm nc})$ is a 
   nc scalar-valued function on ${\mathbb C}_{\rm nc}$ (e.g., a nc 
   single-variable polynomial).  Then we have the equality:
   \begin{equation}   \label{Zclosure-d=1}
       [\Omega_{F}]_{\rm  full} = \overline{\Omega_{F}}.
   \end{equation}
\end{proposition}

\begin{proof}  The containment $[\Omega_{F}]_{\rm  full} \subset \overline{\Omega_{F}}$ 
    follows from  Proposition \ref{P:Zariski-close} (without the 
    extra assumption that $\cV = {\mathbb C}$).
    
    For the reverse containment, we prove the contrapositive: {\em  if $Z 
    \in {\mathbb D}_{Q}$ is not in $[\Omega_{F}]_{\rm  full}$, then $Z$ is not in 
    $\overline{\Omega_{F}}$.}  Toward this end, it suffices to 
    produce a nc polynomial $p_{0}$ such that $p_{0}|_{\Omega_{F}} = 0$ but $p_{0}(Z) \ne 0$. 
     We first note the formula for the 
    evaluation of a single-variable polynomial $p$ on an $n \times n$
    Jordan cell:
    \begin{equation}   \label{polyfunccalc}
    p \left( \sbm{ \lambda & 1 & & & \\ & \lambda & 1 & & \\ & & 
    \ddots & \ddots & \\ & & & \ddots & 1 \\ & & & & \lambda}\right)
    = \sbm{ p(\lambda) & p'(\lambda) & \cdots & \cdots & 
    \frac{1}{(n-1)!}p^{(n-1)}(\lambda) 
    \\ & p(\lambda) & p'(\lambda) & \cdots & \frac{1}{(n-2)!} p^{(n-2)}(\lambda) \\
    & & \ddots & \ddots & \vdots \\ & & & \ddots & p'(\lambda) \\ & & & & p(\lambda)}
    \end{equation}
 Consequently, $p$ vanishing on $\Omega_{F}$ is characterized by the 
 condition:
 \begin{equation} \label{polyvanish}
     p^{(k)}(\lambda) = 0 \text{ for } 0 \le k < n_{\lambda}
     \text{ for all } \lambda \in \sigma(W) \text{ for all } W  \in 
     \Omega_{F}
 \end{equation}
 where $n_{\lambda}$ is the maximum length of a Jordan chain for 
 eigenvalue $\lambda \in \sigma(W)$ for some matrix $W \in \Omega_{F}$.
 Since $Z$ is not in $[\Omega_{F}]_{\rm full}$, it follows that 
 $Z$ is not similar to a matrix of the form $\sbm{ \widetilde 
 Z_{1} & X_{12} \\ 0 & X_{22}}$ with $\widetilde Z_{1}$ equal to a 
 direct sum of matrices in $\Omega_{F}$.  It follows that either (i) $Z$ has an 
 eigenvalue $\lambda_{0}$ distinct from the eigenvalues of all 
 matrices in $\Omega_{F}$, or (ii) all eigenvalues of $Z$ occur as 
 an eigenvalue of a matrix in $\Omega_{F}$ but there is at least one such 
 eigenvalue $\lambda_{0}$ for which the length of the Jordan chain 
 for the eigenvalue $\lambda_{0}$ for the matrix $Z$ is larger than the 
 length of the Jordan chain with eigenvalue $\lambda_{0}$ for any 
 matrix $W \in \Omega_{F}$.  In case (i), let $p_{0}$ be any polynomial 
 satisfying \eqref{polyvanish} but with $p_{0}(\lambda_{0}) \ne 0$.  In 
 case (ii), let $p_{0}$ be any polynomial satisfying 
 \eqref{polyvanish}  but with $p_{0}^{(n_{\lambda_{0}})}(\lambda_{0}) \ne 
 0$.  In either case this is a simple Hermite interpolation problem. 
 As a consequence of the functional calculus formula 
 \eqref{polyfunccalc}, we see that the construction yields a 
 polynomial $p_{0}$ with $p_{0}(Z) \ne 0$ while $p_{0}|_{\Omega} = 0$
 as needed.
 \end{proof}
 
 \begin{remark} \label{R:rightsur} In addition to the left injective 
     intertwinings introduced in Definition \ref{D:fullsubset}, there 
     is a symmetric notion of right surjective intertwinings:
     we say that a subset $\Omega$ of $\cV_{\rm nc}$ is invariant 
     under \textbf{right surjective intertwinings} if:  $\widetilde Z \in 
     \Xi_{n}$ such that there is a $Z \in \Omega_{m}$ together with 
     a surjective $\cI' \in {\mathbb C}^{n \times m}$  such that 
     $\widetilde Z \cI' = \cI' Z$, then $\widetilde Z \in \Omega_{n}$.  It is easily 
     checked that the Zariski closure $\overline{\Omega}$ of any 
     subset $\Omega$ is not only invariant under left injective 
     intertwinings but also under right surjective intertwinings. 
     If it is the case that the Zariski closure always equals the    
     full nc envelope $[\Omega]_{\rm full}$, then it would follow 
     that the full nc envelope $[\Omega]_{\rm  full}$ is in fact 
     also invariant under right surjective intertwinings.  Whether 
     this is the case in general, we leave as an open question.
  \end{remark}

\subsection{Tests for complete positivity of nc kernels}  
\label{S:testncker} 

An interesting fact is that, at least in some special cases, 
there is finite test for complete positivity of a nc kernel with 
domain equal to a finitely generated nc set.

\begin{proposition}   \label{P:fgncset}
   Let $\Omega$ be the full nc envelope $[\Omega_{F}]_{{\rm full}}$ of
the finite subset $\Omega_{F} = \{Z^{(1)}, \dots, Z^{(N)}\}$ of 
$\Xi$. Suppose that $K \colon \Omega \times \Omega \to \cL(\cA, \cL(\cE))_{\rm nc}$ is a 
    nc kernel.  Then $K$ is a cp nc kernel if and only if
    \begin{equation*} 
    K\left( \bigoplus_{i=1}^{N} Z^{(i)}, \bigoplus_{i=1}^{N} Z^{(i)}
\right)
    \end{equation*}
    is a completely positive map. 
    \end{proposition}
    
    \begin{proof}
    Necessity follows from the definition of a cp nc kernel. 
    
    For sufficiency, proceed as follows. 
    Let $n_{i}$ denote the size of $Z^{(i)}$ (so $Z^{(i)} \in 
    \cV^{n_{i} \times n_{i}}$).  For $1 \le i_{0} \le N$,
    let $E^{(i_{0})}$  be the 
    $(\sum_{i=1}^{N} n_{i}) \times n_{i_{0}}$ matrix of column-block structure 
    with $i$-th block column having size $n_{i} \times 
    n_{i_{0}}$ such that the $i_{0}$-block is equal to the  $i_{0} \times i_{0}$ 
    identity matrix $I_{i_{0}}$ and all other blocks equal to $0$:
    $$
       E^{(i_{0})} = \sbm{ 0 \\ \vdots \\ 0 \\ I_{i_{0}} \\ 0 \\ 
       \vdots \\ 0 }.
    $$
    From the intertwining relation $\left( \bigoplus_{i=1}^{N} 
    Z^{(i)} \right) E^{(i_{0})} = E^{(i_{0})} Z^{(i_{0})}$ and the 
    ``respects intertwinings'' property \eqref{kerintertwine} we see 
    that
    $$
    E^{(i_{0})*} K\left( \bigoplus_{i=1}^{N} Z^{(i)}, 
    \bigoplus_{i=1}^{N} Z^{(i)} \right) \left( \left[ P_{ij} \right] 
    \right) E^{(i_{0})} = K(Z^{(i_{0})}, Z^{(i_{0})} )(P_{i_{0} i_{0}}).
    $$
    We conclude that the map $K\left( \bigoplus_{i=1}^{N} Z^{(i)}, 
    \bigoplus_{i=1}^{N} Z^{(i)} \right)$ being positive implies that 
    the map $K(Z^{(i_{0})}, Z^{(i_{0})})$ is a positive map for each 
    $i_{0}$, $1 \le i_{0} \le N$.  A similar argument gives that 
    $K\left( \bigoplus_{i=1}^{N} Z^{(i)}, \bigoplus_{i=1}^{N} Z^{(i)} \right)$ 
    being completely positive implies that $K(Z^{(i_{0})}, 
    Z^{(i_{0})})$  is completely positive.  More generally, 
    $K\left( \bigoplus_{i=1}^{N} Z^{(i)}, 
    \bigoplus_{i=1}^{N} Z^{(i)} \right)$ being completely positive 
    implies that $K\left( \bigoplus_{1}^{M}\left(\bigoplus_{i=1}^{N} 
    Z^{(i)}\right ),  \bigoplus_{1}^{M} \left(\bigoplus_{i=1}^{N} 
    Z^{(i)} \right) \right)$ is completely positive for any $M \in 
    {\mathbb N}$.  Invoking a variant of the intertwining argument 
    once again, we see that $K( \bigoplus_{j=1}^{L} Z^{(i_{j})},
    \bigoplus_{j=1}^{L} Z^{(i_{j})})$ is a completely positive map, 
    where here $\{Z^{i_{1}}, \dots, Z^{i_{L}}\}$ is any subcollection of the 
  set of points $\{Z^{(1)}, \dots, Z^{(N)} \}$ with each allowed to 
  be repeated any number of times up to $M$ times. We conclude that 
  $K(Z,Z)$ is completely positive for any $Z$ in the nc envelope 
  $(\Omega_{F})_{\rm nc}$  of $\Omega_{F}$.  
  
  It remains to check that $K(\widetilde Z, \widetilde Z)$ is completely positive for any 
  $\widetilde Z \in [\Omega_{F}]_{\rm full}$. Suppose that such a 
  $\widetilde Z$ is in $\Omega_{n}$. 
  By definition there is an injective matrix $\cI \in {\mathbb C}^{n 
  \times m}$ and a $Z \in [\Omega_{F}]_{\rm nc}$ of size $m \times m$ 
   so that $\cI \widetilde Z = Z \cI$. Use the ``respects 
  intertwinings'' property \eqref{kerintertwine} of the nc kernel $K$ 
  to see that
  $$
    \cI K(\widetilde Z, \widetilde Z)(P) \cI^{*} = K(Z,Z)(\cI P \cI^{*}) 
    \succeq 0
  $$
  for any $P \succeq 0$ in $\cA^{m \times m}$.  As $\cI$ is 
  injective, we conclude that $K(\widetilde Z, \widetilde Z)(P) 
  \succeq 0$, i.e., $K(\widetilde Z, \widetilde Z)$ is a positive map.
  A similar  argument shows that in fact  $K(\widetilde Z, \widetilde Z)$ 
  is completely positive.
  \end{proof}
    
    \begin{corollary}   Suppose that $\Omega_{F} = \{Z^{(1)}, \dots,
Z^{(N)}\} $ is a finite subset of $\C_{\rm nc}$ and the nc set
$\overline{\Omega_F}$ is its $\C_{\rm nc}$-relative nc-Zariski
closure. 
    Suppose that $K 
    \colon \overline{\Omega_F} \times \overline{\Omega_F} \to
\cL(\cA, \cL(\cE))_{\rm nc}$ is a 
    nc kernel.  Then $K$ is a cp kernel if and only if
    $$
    K\left( \bigoplus_{i=1}^{N} Z^{(i)}, \bigoplus_{i=1}^{N} Z^{(i)}
\right)
    $$
    is a completely positive map.  
\end{corollary}	

\begin{proof}
The result follows from applying Proposition
\ref{P:Zariski-close-d=1} to Proposition \ref{P:fgncset}
\end{proof}

\begin{corollary} Let $\Omega$ and $K$ be as in  Proposition
\ref{P:fgncset} and consider the special case where $\cA=\mathbb C$.
Then $K$ is a cp nc kernel if and only if 
\begin{equation} \label{choi}
K\left( \bigoplus_{j=1}^{N'} \bigoplus_{i=1}^{N}
Z^{(i)},\bigoplus_{j=1}^{N'} \bigoplus_{i=1}^{N} Z^{(i)} \right)
\left( \mathfrak C_{N'} \right)
\end{equation}
is positive where $N'$ is set equal to the level of $\Omega$ containing
$\bigoplus_{i=1}^{N} Z^{(i)}$ (i.e., $\bigoplus_{i=1}^{N} Z^{(i)} \in 
    \Omega_{N'}$) and where ${\mathfrak C}_{N'}$ (the \textbf{Choi 
    matrix} at level $N'$) is the $(N')^{2} \times (N')^{2}$ matrix
written 
    out as a  block $N' \times N'$ matrix with $N' \times N'$ matrix 
    entries given by 
    $$  
    {\mathfrak C}_{N'} = [ E^{N'}_{i,j}]_{i,j=1, \dots, N'}
    $$
    where $E^{N'}_{i,j}$ is the $N' \times N'$ matrix with
$(i,j)$-entry 
    equal to $1$ and all other entries equal to $0$.
  \end{corollary}

\begin{proof} Necessity follows from the definition of a cp nc kernel
and the fact that $\mathfrak C_{N'}$ is a positive map.

For sufficiency, we assume that \eqref{choi} is positive. Since an nc
kernel respects direct sums, we have that 
$$K\left( \bigoplus_{j=1}^{N'} \bigoplus_{i=1}^{N}
Z^{(i)},\bigoplus_{j=1}^{N'} \bigoplus_{i=1}^{N} Z^{(i)} \right)
\left( \mathfrak C_{N'} \right) = \left[ K\left( \bigoplus_{i=1}^{N}
Z^{(i)}, \bigoplus_{i=1}^{N} Z^{(i)} \right) \left(E^{N'}_{i,j}
\right) \right]_{i,j=1, \dots, N'}
$$
By \cite[Theorem 3.14]{Paulsen}, the map $K\left( \bigoplus_{i=1}^{N}
Z^{(i)}, \bigoplus_{i=1}^{N} Z^{(i)} \right)$ is completely positive,
and the result follows from Proposition \ref{P:fgncset}.
\end{proof}

\subsection{Extensions of noncommutative functions and kernels}  \label{S:ncgendom}

 The following result clarifies the relation between nc functions, nc 
 kernels, and cp nc kernels defined on a subset $\Omega$ versus 
 defined on one of the envelopes $[\Omega]_{\rm nc}$, $[\Omega]_{\rm nc, sim}$, 
 $[\Omega]_{\rm full}$ of $\Omega$. 

\begin{proposition}   \label{P:extcpncker}  Suppose that $\Omega$ is 
    a subset (not necessarily a nc subset) of $\cV_{\rm nc}$.
    \begin{enumerate}  
	\item 
	Any nc function $f \colon \Omega \to 
	\cV_{0, {\rm nc}}$ extends uniquely to a nc function 
	$\widetilde f$ on the nc envelope $[\Omega]_{\rm nc}$ and  on the 
	nc similarity envelope $[\Omega]_{\rm nc, sim}$ but not 
	necessarily on the full nc envelope $[\Omega]_{\rm full}$.
	
	\item Any nc kernel $K \colon \Omega \to 
	\cL(\cV_{1}, \cV_{0})_{\rm nc}$ extends uniquely to 
	a nc kernel $\widetilde K$ on the nc envelope  $[\Omega]_{\rm nc}$ 
	and on the nc similarity envelope 
	$[\Omega]_{\rm nc, sim}$, but not necessarily on the full 
	envelope $[\Omega]_{\rm full}$. 
	
	\item Any cp nc kernel $K \colon \Omega \times \Omega 
    \to \cL(\cV_{1}, \cV_{0})_{\rm nc}$ extends uniquely to 
    a cp nc kernel on the nc envelope  $[\Omega]_{\rm nc}$ and on the nc 
    similarity envelope $[\Omega]_{\rm nc, sim}$, but not necessarily 
    on the full nc envelope $[\Omega]_{\rm full}$. 
\end{enumerate}
\end{proposition}

\begin{proof}[Proof of (1).] The positive assertions in item (1) essentially follow 
    from Pro\-positions A.1 and A.3 in \cite{KVV-book}.  
    
    To show that 
    the assertion can fail for the full nc envelope, consider the 
    following example.  We take $\cV = {\mathbb C}^{2}$ so $\cV^{n 
    \times n}$ can be identified with pairs $(Z_{1}, Z_{2})$ of $n 
    \times n$ matrices over ${\mathbb C}$.  Take $\Omega $ to be the  
    singleton set $\Omega = \{Z^{(0)} =  (Z^{(0)}_{1}, Z^{(0)}_{2}) \}$ where 
    we set
    $$
      Z^{0}_{1} = \begin{bmatrix} 0 & 1 \\ 0 & 0 \end{bmatrix}, \quad Z^{(0)}_{2} = 
      \begin{bmatrix} 0 & 0 \\ 0 & 1 \end{bmatrix}.
    $$
    The only constraint on $\Lambda_{0} \in {\mathbb C}^{2 \times 2}$ 
    required for the function $f \colon \Omega \to {\mathbb C}^{2 \times 
    2}$ defined by $f(Z^{(0)}) = \Lambda_{0}$ to be a nc function is 
    that the value $\Lambda_{0}$ be in the double commutant of 
    $Z^{(0)}$, i.e.:  $\alpha \in {\mathbb C}^{2 \times 2}$   such that 
    $\alpha Z^{0}_{1} = Z^{0}_{1} \alpha$ and $ \alpha Z^{0}_{2} = Z^{0}_{2} \alpha$ 
    $\Rightarrow$ $\alpha \Lambda_{0} = \Lambda_{0} \alpha$.  The commutant of 
    $Z_{1}^{(0)}$ consists of Toeplitz matrices $\{ \sbm{ a & b \\ 0 & a } 
    \colon a,b \in {\mathbb C}\}$ while the commutant of 
    $Z^{(0)}_{2}$ consists of diagonal matrices $\{ \sbm{ d_{1} &  0 
    \\ 0 & d_{2}} \colon d_{1}, d_{2} \in {\mathbb C}\}$.  The 
    intersection of these two commutants consists of scalar multiples 
    of the identity.  Hence the value $\Lambda_{0}$ is 
    unconstrained:  the function $f$ given by $f(Z^{(0)}) = 
    \Lambda_{0}$ is a nc function on $\Omega$ for any $\Lambda_{0} 
    \in {\mathbb C}^{2 \times 2}$.  If $f$ extends to a nc function 
    on the full nc envelope, then in particular $f( Z^{(0)})$ must 
    have the form
    $$
    \Lambda_{0} =  f(Z^{(0)}_{1}, Z^{(0)}_{2}) = \begin{bmatrix} f( 0,0) & * \\ 0 
      & f(0,1) \end{bmatrix}.
    $$
    Choosing $\Lambda_{0} = \sbm{ 0 & 0 \\ 1 & 0 }$ for example then 
    leads to a contradiction.  For additional information concerning 
    existence and construction of nc-function extensions to full nc 
    envelopes, we refer to Subsection \ref{S:Steindom} below.
    \end{proof}
    
    \begin{proof}[Proof of (2).]  Given a nc kernel $K$ on $\Omega \times \Omega$, 
we may define an extension $\widetilde K$ to the nc 
envelope $\Omega_{\rm nc}$ of $\Omega$ by 
\begin{equation}   \label{tildeK}
\widetilde K\left( \sbm{ Z^{(1)} & & \\ & \ddots & \\ & & Z^{(N)}},
 \sbm{ W^{(1)} & & \\ & \ddots & \\ & & W^{(M)}}\right)([P_{ij}]) =
 [ K(Z^{(i)}, W^{(j)})(P_{ij}) ]_{\begin{smallmatrix} i = 1, \dots, N; 
 \\ j = 1, \dots, M \end{smallmatrix}}
\end{equation}
for any $Z^{(1)}, \dots, Z^{(N)}, W^{(1)}, \dots, W^{(M)} \in \Omega$.  If it 
happens that $\sbm{ Z^{(1)} & & \\ & \ddots & \\ & & Z^{(N)}}$ 
and  $\sbm{ W^{(1)} & & \\ & \ddots & \\ & & W^{(M)}}$ are already in 
$\Omega$, then the formula \eqref{tildeK} is consistent with how $K$ 
is already defined by the localized ``respects direct sums'' condition.
Then one can check that $\widetilde K$ is a nc kernel on $[\Omega]_{\rm nc} \times 
[\Omega]_{\rm nc}$ which when restricted to $\Omega \times \Omega$ 
 agrees with $K$.  If $\widetilde Z$, $\widetilde W$ are in the nc 
 similarity envelope $[\Omega]_{\rm nc, sim}$, then there is an 
 invertible matrix $\alpha = \begin{bmatrix} \alpha_{1} & \cdots & 
 \alpha_{N} \end{bmatrix}$ over ${\mathbb C}$ and points $Z^{(1)}, \dots, 
Z^{(N)} \in \Omega$ as well as an invertible matrix $\beta = 
\begin{bmatrix} \beta_{1} & \cdots & \beta_{M} \end{bmatrix}$ over ${\mathbb 
    C}$ and  points $W^{(1)}, \dots, W^{(M)}$ in $\Omega$ so that
  $$
 \alpha^{-1} \widetilde Z \alpha = \begin{bmatrix} Z^{(1)} & & \\ & \ddots & \\ 
 & & Z^{(N)} \end{bmatrix}, \quad
 \beta^{-1} \widetilde W \beta = \begin{bmatrix} W^{(1)} & & \\ & \ddots & \\
& &  W^{(M)} \end{bmatrix}.
$$
We then define 
\begin{equation}  \label{tildeK-2}
 \widetilde K(\widetilde Z, \widetilde W)(P) =  \alpha\, \left[ K(Z^{(i)}, 
 W^{(j)})  ( [\alpha^{-1} P \beta^{-1 *} ]_{ij})\right]_{\begin{smallmatrix} i=1, \dots, N; \\
   j = 1, \dots, M \end{smallmatrix}} \, \beta^{*}.
\end{equation}
We leave it to the reader to verify that $\widetilde K$ given by 
\eqref{tildeK-2} is a well-defined nc kernel on $[\Omega]_{\rm nc, 
sim} \times [\Omega]_{\rm nc, sim}$ which extends $K$.

The result concerning lack of extension in general of a nc kernel on 
$\Omega \times \Omega$ to a nc kernel on $\Omega_{\rm nc, full} 
\times \Omega_{\rm nc, full}$ will 
follow from the result for the case of cp nc kernels discussed in 
the next part.
\end{proof}

\begin{proof}[Proof of (3).]   To verify item (3), it suffices to show that the construction of the 
nc-kernel extension in part (2) leads to a cp kernel on $[\Omega]_{\rm nc} \times [\Omega]_{\rm nc}$ 
and on $[\Omega]_{\rm nc, sim} \times [\Omega]_{\rm nc, sim}$ where we now assume that the original 
kernel $K$ was cp on $\Omega \times \Omega$. The fact that the kernel $\widetilde K$ given by 
\eqref{tildeK} is cp on $[\Omega]_{\rm nc} \times [\Omega]_{\rm nc}$ 
amounts to the definition of a cp kernel for the setting where $\Omega$ 
is not necessarily a nc set.  Similarly, one can verify by inspection that 
$\widetilde K$ given by \eqref{tildeK-2} on  $[\Omega]_{\rm nc, sim} \times 
[\Omega]_{\rm nc, sim}$ is cp if the original $K$ is 
cp.
 
 Finally, if we choose $H \colon  \Omega \to \cL(\cE)_{\rm nc}$ to be a nc function on 
 $\Omega$ which fails to have a nc-function extension to $[\Omega]_{\rm full}$ 
 (as in the example in the proof of part (1) above), then 
 $K(Z,W)(P) : = H(Z) (P \otimes I_{\cE}) H(W)^{*}$ is a cp nc kernel 
 on $\Omega$ which fails to have a nc-kernel (much less a cp 
 nc-kernel) extension to $[\Omega]_{\rm full}$.  This completes the 
 proof of Proposition \ref{P:extcpncker}.
\end{proof}
 
 In summary, given a nc kernel on a set of the form $\Omega \times \Omega$ with 
 $\Omega$ not necessarily a nc set, by Proposition \ref{P:extcpncker}
 we can always consider its extension to the nc envelope $[\Omega]_{\rm 
 nc} \times [\Omega]_{\rm nc}$ or to the nc similarity envelope of 
 $[\Omega]_{\rm nc, sim} \times [\Omega]_{\rm nc, sim}$ but not 
 necessarily to the full nc envelope $[\Omega]_{\rm full} \times 
 [\Omega]_{\rm full}$.  However, 
 when there is an extension to the full nc envelope $[\Omega]_{\rm full}$, 
 the extension necessarily is unique.  Similar remarks hold 
 for nc functions defined on a set $\Omega$ which is not necessarily a 
 nc set.     One trivial situation  when nc extension to the full nc 
 envelope is possible is when the nc function has the form $\bphi$ as 
 in Example \ref{E:linear} (b) on $\Omega$ or the nc kernel has the 
 form $\bphi$ as in \ref{E:linear} (a) on $\Omega \times \Omega$.
 
 \subsection{Internal tensor product of $C^{*}$-correspondences} 
 \label{S:tensorprod}
 We shall need a certain special case of a general $C^{*}$-correspondence internal 
 tensor product  and related constructions (see \cite{MS1998, RW}).  
 The following theorem 
 summarizes these results in the form needed for the sequel.
 
 \begin{theorem}   \label{T:tensorprod}  Suppose that $\cX$ is a 
     Hilbert space, $\cE$ and $\cF$ are two auxiliary Hilbert spaces 
     and $\cX$ is equipped with a $*$-representation $\pi$ of 
     $\cL(\cE)$:
 $$
    \pi \colon \cL(\cE) \to \cL(\cX).
 $$
 Then the following statements hold:
 \begin{enumerate}
     \item 
 Define an inner product $\langle \cdot, \cdot \rangle_{0}$ on the 
 algebraic tensor product $\cL(\cE, \cF) \otimes_{\rm alg} \cX$ by
 \begin{equation}   \label{0innerprod}
 \langle T \otimes x, S \otimes x' \rangle_{0} = \langle \pi(S^{*}T) 
 x, x' \rangle_{\cX}
 \end{equation}
 for $x,x' \in \cX$, $T,S \in \cL(\cE, \cF)$.  Then $\langle \cdot, 
 \cdot \rangle_{0}$ is positive semidefinite. Modding out by elements 
 of zero self inner-product gives rise to a pre-Hilbert space 
 $(\cL(\cE, \cF) \otimes \cX)_{0}$.  We 
 let $\cL(\cE, \cF) \otimes_{\pi} \cX$ denote the Hilbert-space completion 
 of  $(\cL(\cE, \cF) \otimes \cX)_{0}$ in the $\langle, \cdot, \cdot 
 \rangle_{0}$ inner product.  This Hilbert 
 space completion satisfies the balancing law
 \begin{equation}   \label{balance}
   T E \otimes  x = T \otimes \pi(E) x
 \end{equation}
 for $T \in \cL(\cE, \cF)$, $E \in \cL(\cE)$ and $x \in \cX$.
 
 \item 
 Furthermore, an operator $T \in \cL(\cE, \cF)$ induces an operator 
 $L_{T}$ mapping $\cX$ into  $\cL(\cE, \cF) \otimes_{\pi} \cX$ given by
 \begin{equation}   \label{LT}
     L_{T} \colon x \mapsto T \otimes x
  \end{equation}
  with adjoint action on elementary tensors given by
  \begin{equation}   \label{LT*}
      L_{T}^{*} \colon S \otimes x \mapsto  \pi(T^{*}S) x,
  \end{equation}
  such that
  $$
    \| L_{T} \|_{\cL(\cX, \cL(\cE, \cF) \otimes_{\pi} \cX)} \le \| 
    T\|_{\cL(\cE, \cF)}
  $$
  with equality in case $\pi$ is a faithful representation.
  \end{enumerate}
  \end{theorem}

  \begin{proof}
      The construction of the space $\cL(\cE, \cF) 
     \otimes_{\pi} \cX$ is a special case of a more general 
     construction called the {\em inner tensor product} for 
     $C^{*}$-correspondences (also called imprimitivity bimodules).
     For a  proof of the positive-semidefiniteness of the inner 
     product \eqref{0innerprod} in this more general setting,
     we refer to Proposition 3.16 in \cite{RW}.  In any case we shall 
     do a different more general version of this computation in the proof of 
     Theorem \ref{T:nctensorprod} below.
     
     As a nice exercise we go ahead here with verifying explicitly the balance law \eqref{balance}:
  \begin{align*}
   &   \langle TE \otimes x - T \otimes \pi(E) x, TE \otimes x - T 
      \otimes \pi(E) x \rangle_{0}  \\
  & \quad =  \langle  \pi(E^{*}T^{*}T E) x, x \rangle_{\cX} - 
  \langle \pi(E^{*} T^{*} T) \pi(E) x, x \rangle_{\cX} \\
  & \quad \quad \quad \quad - \langle \pi(T^{*}T E) x, \pi(E) x 
  \rangle_{\cX}+ \langle \pi(T^{*}T) \pi(E) x, \pi(E) x \rangle_{\cX} 
  \\
  & \quad = 0.
  \end{align*}
This completes the proof of statement (1).
  
  To verify the properties of $L_{T}$, first note that
  \begin{align*}
\| L_{T} x \|^{2}_{\cL(\cE, \cF) \otimes_{\pi} \cX} & = \| T 
      \otimes x \|^{2}_{\cL(\cE, \cF) \otimes_{\pi} \cX}   
       = \langle \pi(T^{*}T) x, x \rangle_{\cX} \\ 
&  \le \| \pi((T^{*}T)^{1/2}) \|^{2} \| x \|^{2} \le \| 
 (T^{*}T)^{1/2}\|^{2} \| x \|^{2} = \| T \|^{2} \|x \|^{2}
\end{align*}
with equality throughout in case $\pi$ is faithful.  Thus $L_{T}$ is 
well defined with $\| L_{T} \| \le \| T \|$ and with equality in case 
$\pi$ is faithful.  Finally note that
$$
\langle L_{T}^{*} (S \otimes x), x' \rangle_{\cX}  = \langle S \otimes 
x, T \otimes x' \rangle_{\cL(\cE, \cF) \otimes _{\pi} \cX}
= \langle \pi(T^{*}S) x, x' \rangle_{\cX}
$$
and the formula \eqref{LT*} follows.  This completes the proof of 
statement (2).
 \end{proof}
 
 We shall actually need the following extension of Theorem 
 \ref{T:tensorprod} to our nc setting.

  \begin{theorem}  \label{T:nctensorprod}  Suppose that $\cX$, $\cE$, 
      $\cF$ are Hilbert spaces with $\cX$ equipped with a 
      $*$-representation $\pi \colon \cL(\cE) \to \cL(\cX)$ as in 
      Theorem \ref{T:tensorprod}.Then the following statements hold:
      \begin{enumerate}
\item Fix a positive integer $k \in {\mathbb N}$ and define an inner product $\langle 
\cdot , \cdot \rangle_{0}$ on the disjoint union of algebraic tensor product 
      spaces
      $$
      (\cL(\cE, \cF^{k}) \otimes_{\rm alg} \cX)_{\rm nc} : =
      \amalg_{n=1}^{\infty} \cL(\cE^{n} \otimes \cF^{k}) \otimes_{\rm 
      alg} \cX^{n}
      $$
 by
 \begin{equation}  \label{nc0innerprod}
 \langle T \otimes x, T' \otimes x' \rangle_{0} = 
 \langle ({\rm id}_{{\mathbb C}^{n' \times n}} \otimes \pi) 
 (T^{\prime *} T) x, x' \rangle_{\cX^{n'}}
 \end{equation}
 for $T \in \cL(\cE^{n}, \cF^{k})$, $T'  \in  \cL(\cE^{n'}, 
 \cF^{k})$,  $x \in \cX^{n}$, $x' \in \cX^{n'}$.  Then the 
 inner-product $\langle \cdot, \cdot \rangle_{0}$ is 
 positive-semidefinite on $(\cL(\cE, \cF^{k}) \otimes_{\rm alg} 
 \cX)_{\rm nc}$.  
 
 We let $(\cL(\cE, \cF^{k}) \otimes_{\pi} 
 \cX)_{\rm nc}$ denote the Hilbertian completion of
 $(\cL(\cE, \cF^{k}) \otimes_{\rm alg} \cX)_{\rm nc}$ in the 
 $0$-inner product (the completion of 
 the positive-definite inner product obtained by identifying 
 elements of self inner-product equal to zero with the zero element 
 of the space).  Then elements of   $(\cL(\cE, \cF^{k}) \otimes_{\rm alg} \cX)_{\rm nc}$
 satisfy the balancing law
 \begin{equation}   \label{ncbalance}
     TS \otimes x = T \otimes ({\rm id}_{{\mathbb C}^{n \times n'}} 
     \otimes \pi)(S) x
 \end{equation}
 for $T \in \cL(\cE^{n}, \cF^{k})$, $S \in \cL(\cE^{n'}, \cE^{n})$, 
 $x \in \cX^{n'}$.
 
 \item Suppose that $T$ is an operator in $\cL(\cE^{n}, \cF^{k})$.  
 Then
 \begin{equation}  \label{ncLT}
    L_{T} \colon x \mapsto T \otimes x
 \end{equation}
 defines a bounded linear operator from $\cX_{\rm nc}$ into $(\cL(\cE, 
 \cF^{k}) \otimes_{\pi} \cX)_{\rm nc}$ with adjoint action on 
 elementary 
 tensors given by
 \begin{equation}   \label{ncLT*}
 L_{T}^{*} \colon T' \otimes x' \mapsto 
 ({\rm id}_{{\mathbb C}^{n \times n'}} \otimes \pi) (T^{*} T') x'
 \end{equation}
 for $T' \in \cL(\cE^{n'}, \cF^{k})$, $x' \in \cX^{n'}$.
 
 \item   The order-$k$ tensor product space $(\cL(\cE, 
 \cF^{k}) \otimes_{\pi} \cX)_{\rm nc}$  can be 
 identified with the $k$-fold orthogonal direct sum of the order-1 
 tensor product space $(\cL(\cE, \cF) \otimes_{\pi} \cX)_{\rm nc}$
 via an identification map 
  $$
    \iota_{k} \colon  (\cL(\cE, \cF^{k}) \otimes_{\pi} \cX)_{\rm nc}
    \mapsto \bigoplus_{1}^{k} (\cL(\cE, \cF) 
    \otimes_{\pi} \cX)_{\rm nc} 
 $$
 with action on elementary tensors given as follows.  For $T$ an 
 operator in $\cL(\cE^{n}, \cF^{k})$, we may decompose $T$ as a block-column 
 operator matrix $T = \sbm{ T_{1} \\ \vdots \\ T_{k}}$ where each 
 $T_{i} \in \cL(\cE^{n}, \cF)$ for $i=1, \dots, k$.  Then for $T 
 \in \cL(\cE^{n}, \cF^{k})$ and $x \in \cX^{n}$ so that $T \otimes x$ 
 is an elementary tensor in 
 $(\cL(\cE, \cF^{k}) \otimes_{\pi} \cX)_{\rm nc}$, we define
 \begin{equation}   \label{idk}
     \iota_{k} \colon  T \otimes x = \begin{bmatrix} T_{1} \\ \vdots 
     \\ T_{k} \end{bmatrix} \otimes x \mapsto 
     \begin{bmatrix} T_{1} \otimes x \\ 
     \vdots \\ T_{k} \otimes x \end{bmatrix} 
  \in  (\cL(\cE, \cF) \otimes_{\pi} \cX)_{\rm nc})^{k}
 \end{equation}
 \end{enumerate}
 \end{theorem}
 
 \begin{proof}
     To prove that the $0$-inner product \eqref{nc0innerprod} is 
     positive semidefinite, choose any positive integers $n_{j}$ ($j 
     = 1, \dots, n$) along with operators $T_{j} \in \cL(\cE^{n_{j}} \otimes 
     \cF_{k})$ and vectors $x_{j} \in \cX^{n_{j}}$, and compute
  \begin{align*}
   &   \langle \sum_{j=1}^{n} T_{j} \otimes x_{j}, \, \sum_{i=1}^{n} 
      T_{i} \otimes x_{i} \rangle_{0} =  \sum_{i,j=1}^{n} \langle ({\rm 
      id}_{{\mathbb C}^{n_{i} \times n_{j}}} \otimes \pi)(T_{i}^{*} 
      T_{j}) x_{j}, x_{i} \rangle_{\cX^{n_{i}}}  \\
 & \quad = \left\langle ({\rm id}_{{\mathbb C}^{N \times N}} \otimes 
 \pi) \left( [T_{i}^{*} T_{j}] \right)
 \begin{bmatrix} x_{1} \\ \vdots \\ x_{n} \end{bmatrix},
      \begin{bmatrix} x_{1} \\ \vdots \\ x_{n} \end{bmatrix} \right 
	  \rangle_{\cX^{N}}
\end{align*}
where $N = \sum_{i=1}^{n} n_{i}$.   The block $n \times n$ matrix 
$[T_{i}^{*} T_{j}] = \sbm{ T_{1}^{*} \\ \vdots \\ T_{n}^{*} } \sbm{ 
T_{1} & \cdots & T_{n} }$ is a positive element of $\cL(\cE^{N})$ and 
hence can be factored in the form
$$
   [ T_{i}^{*} T_{j} ] = A^{*} A
$$
where $A \in \cL(\cE^{N})$.  Hence the preceding calculation can be 
continued as
\begin{align*}
   &  \left\langle ({\rm id}_{{\mathbb C}^{N \times N}} \otimes 
 \pi) \left( [T_{i}^{*} T_{j}] \right)
 \begin{bmatrix} x_{1} \\ \vdots \\ x_{n} \end{bmatrix},
      \begin{bmatrix} x_{1} \\ \vdots \\ x_{n} \end{bmatrix} 
\right  \rangle_{\cX^{N}} = \langle ({\rm id}_{{\mathbb C}^{N \times 
	  N}} \otimes \pi) (A^{*} A) x, x \rangle_{\cX^{N}} \\
	  & \quad \| ({\rm id}_{{\mathbb C}^{N \times N}} \otimes 
	  \pi)(A) x \|^{2} \ge 0.
\end{align*}

The verification of the balancing law \eqref{ncbalance} proceeds as 
in the verification of \eqref{balance} in the proof of Theorem 
\ref{T:tensorprod}.  This completes the proof of statement (1).

To verify the formula \eqref{ncLT*}, observe that, for $x \in 
\cX^{n}$,
\begin{align*}
\langle L_{T}^{*} (T' \otimes x'), x \rangle_{\cX^{n}} & =
\langle T' \otimes x', T \otimes x \rangle_{(\cL(\cE, \cF^{k}) 
\otimes_{\pi} \cX)_{\rm nc}}  \\
& = \langle ({\rm id}_{{\mathbb C}^{n \times n'}} \otimes \pi) 
(T^{*}T') x', x \rangle_{\cX^{n}}.
\end{align*}
This completes the proof of statement (2).

To verify that the map $\iota_{k}$ given by \eqref{idk} is an 
isometry, we compute, for $T_{i} \in \cL(\cE^{n}, \cF)$, $x
\in \cX^{n}$, $T'_{i} \in \cL(\cE^{n'}, \cF)$, $x' \in 
\cX^{n'}$ for $i = 1, \dots, k$,
\begin{align*}
  & \left\langle \begin{bmatrix} T_{1} \\ \vdots \\ T_{k} 
\end{bmatrix} \otimes x, \begin{bmatrix} T'_{1} \\ \vdots \\ T'_{k} 
\end{bmatrix} \otimes x' \right\rangle_{(\cL(\cX, \cF^{k}) \otimes 
\cX)_{\rm nc}} = \sum_{i=1}^{k} \langle ({\rm id}_{{\mathbb C}^{n' 
\times n}} \otimes \pi)(T^{\prime *}_{i} T_{i}) x, x' 
\rangle_{\cX^{n'}} \\
& \quad = \left\langle \begin{bmatrix} T_{1} \otimes x \\ \vdots \\ 
T_{k} \otimes x \end{bmatrix}, \, \begin{bmatrix} T'_{1} \otimes x' 
\\ \vdots \\ T'_{k} \otimes x' \end{bmatrix} \right \rangle_{((\cL(\cE, 
\cF) \otimes \cX)_{\rm nc})^{k}}.
\end{align*}
We conclude that $\iota_{k}$ can be extended via linearity to an isometry 
from the space
$$
 \cD = {\rm span} \left\{ \sbm{ T_{1}  \\  \vdots  \\  
T_{k}}\otimes x \colon T_{i} \in 
    \cL(\cE^{n}, \cF), x \in \cX^{n}, n \in {\mathbb N} 
    \right\} \subset (\cL(\cE, \cF^{k}) \otimes_{\pi} \cX)_{\rm nc}
$$
onto the space
$$
\cR = {\rm span} 
\left\{ \sbm{ T_{1} \otimes x \\ \vdots \\ T_{k} \otimes x} \colon
T_{i} \in \cL(\cE^{n}, \cF), x \in \cX^{n}, n \in {\mathbb N} 
\right\} \subset ((\cL(\cE, \cF) \otimes_{\pi} \cX)_{\rm nc})^{k}.
$$
It now suffices to verify that the spaces $\cD$ and $\cR$ are dense 
in their respective ambient spaces $(\cL(\cE, \cF^{k}) \otimes_{\pi}  
\cX)_{\rm nc}$ and $\oplus_{i=1}^{k} (\cL(\cE, \cF) 
\otimes_{\pi}  \cX)_{\rm nc}$.   The fact that $\cD$ is dense is 
clear since a generic elementary tensor is in $\cD$.  As for 
$\cR$, we note that a more general tensor $\sbm{ T_{1} \otimes x_{1} 
\\ \vdots \\ T_{k} \otimes x_{k}}$ can be obtained as a linear 
combination of generic elements of $\cR$:
$$
\begin{bmatrix} T_{1} \otimes x_{1} \\ \vdots \\ T_{k} \otimes x_{k} 
    \end{bmatrix} =
 \begin{bmatrix} T_{1} \otimes x_{1} \\ 0 \otimes x_{1} \\ \vdots \\
     0 \otimes x_{1} \end{bmatrix} + \begin{bmatrix} 0 \otimes x_{2} 
     \\ T_{2} \otimes x_{2} \\ \vdots \\ 0 \otimes x_{2} 
 \end{bmatrix} + \cdots + \begin{bmatrix} 0 \otimes x_{k} \\ \vdots 
 \\ 0 \otimes x_{k} \\ T_{k} \otimes x_{k} \end{bmatrix}.
 $$
\end{proof}
   
In case the representation $\pi$ is a multiple of the identity 
representation, we have the following simplification of the space 
$(\cL(\cE, \cF) \otimes_{\pi} \cX)_{\rm nc}$ and of the operators 
$L_{T}$ and $L_{T}^{*}$.

\begin{theorem}   \label{T:ncfindim}
  Suppose that $\cX$, $\cE$, $\cF$ and $\pi$  are as in Theorem \ref{T:tensorprod}
  with $\pi \colon \cL(\cE) \to \cL(\cX)$ a multiple of the identity 
  representation, i.e., there is a Hilbert space $\cX_{0}$ so that 
  $\cX$ has the form $\cX = \cE \otimes \cX_{0}$ and
  $$
     \pi(E) = E \otimes I_{\cX_{0}} \text{ for } E \in \cL(\cE).
  $$
  Then the space $\cL(\cE, \cF) \otimes_{\pi} (\cE \otimes \cX_{0})$ 
  can be identified with $\cF \otimes \cX_{0}$ with identification map 
  on elementary tensors given by
  \begin{equation}   \label{identification}
  \iota \colon T \otimes (e \otimes x_{0}) \mapsto Te \otimes x_{0}
  \end{equation} 
  for $T \in \cL(\cE, \cF)$, $e \in \cE$, $x_{0} \in \cX_{0}$.  
  Moreover 
  \begin{equation} \label{opidentification} 
 \iota  \cdot L_{T} =  T \otimes I_{\cX_{0}}. 
 \end{equation}
 \end{theorem}
 
 \begin{proof}
     For $T,T' \in \cL(\cE, \cF)$ and $e \otimes x_{0}, e' \otimes 
     x'_{0} \in \cE \otimes \cX_{0}$, we compute
 \begin{align*}
  &   \langle T \otimes (e \otimes x_{0}), T' \otimes (e' \otimes 
  x_{0}')
  \rangle_{(\cL(\cE, \cF) 
     \otimes_{\pi} (\cE \otimes \cX_{0}))_{\rm nc}} \\
 & \quad =
 \langle (T^{\prime *} T \otimes I_{\cX_{0}}) (e \otimes x_{0}), (e' 
 \otimes x'_{0}) \rangle_{\cE \otimes \cX_{0}}  \\
& \quad = \langle T e \otimes x_{0}, T' e' \otimes x'_{0} 
\rangle_{\cF \otimes \cX_{0}}.
\end{align*}
We conclude that the map $\iota$ given by \eqref{identification} 
defines an isometry from
$$
 \cD = {\rm span}_{T,e,x_{0}} \{ T \otimes (e \otimes x_{0})\} 
 \subset (\cL(\cE, \cF) \otimes_{\pi} (\cE \otimes \cX_{0}))_{\rm nc}
$$
onto
$$
\cR = {\rm span}_{T,e,x_{0}} \{ Te \otimes x_{0} \} \subset \cF 
\otimes \cX_{0}.
$$
It is now a matter of checking that $\cD$ and $\cR$ each is dense in 
its respective ambient space to see that \eqref{identification} 
establishes a unitary identification between $(\cL(\cE, \cF) 
\otimes_{\pi}  (\cE \otimes \cX_{0}))_{\rm nc}$ and $\cF \otimes 
\cX_{0}$.  

Finally, the  identity \eqref{opidentification} follows immediately from the definitions \eqref{LT} 
and  \eqref{identification}.
 \end{proof}
 
 \begin{remark}   \label{R:findim}  It is well known that  in case 
     $\cE$ is finite-dimensional, then any representation $\pi \colon 
     \cL(\cE) \to \cL(\cX)$ has the special form
$$
  \cX = \cE \otimes \cX_{0}, \quad \pi(E) = E \otimes I_{\cX_{0}}
$$
assumed in Theorem \ref{T:ncfindim}  (see e.g.\ \cite[Corollary 1 
page 20]{Arv76},
\end{remark}

 \section{Schur-Agler class interpolation theorems}  \label{S:SAint}
 We fix a full nc subset $\Xi$ of nc envelope $\cV_{\rm nc}$ of a 
 vector space $\cV$ and we suppose that $Q \in \cT(\Xi; \cL(\cR, 
 \cS)_{\rm nc})$ is a nc function from $\Xi$ to $\cL(\cR, \cS)_{\rm 
 nc}$ as in Subsection \ref{S:ncdisk}.  Then we define the nc disk 
 ${\mathbb D}_{Q}$ as in \eqref{defDQ}.
 Let $\cU$ and $\cY$ be two additional coefficient Hilbert spaces. 
We  define the  \textbf{nc Schur-Agler class} $\cSA_{Q}(\cU, \cY)$ by
\begin{equation}  \label{ncSchurAgler} 
\cSA_{Q}(\cU, \cY) = \{ S \in \cT({\mathbb D}_{Q}; \cL(\cU, \cY)_{\rm 
nc}) \colon \| S(Z) \| \le 1 \text{ for all } Z \in {\mathbb D}_{Q} \}.
\end{equation}
We note that the special case where $\Xi = ({\mathbb C}^{d})_{\rm nc}$, $\cR = 
{\mathbb C}^{r}$, $\cS = {\mathbb C}^{s}$ and $Q(z)$ is linear 
amounts to the setting of \cite{BGM2} discussed in the Introduction, 
where one can work with a globally defined power series representation 
for the Schur-Agler-class function $S$. 

We consider the following left-tangential interpolation problem for 
the class $\mathcal{SA}_{Q}(\cU, \cY)$.  We suppose that $\cE$ is 
another coefficient Hilbert space and that we are given 
nc functions $a \in \cT(\Omega; \cL(\cY, \cE)_{\rm nc})$ and  
$b \in \cT(\Omega; \cL(\cU, \cE)_{\rm nc})$.
We consider the following problem:

\smallskip

\noindent
\textbf{Left-Tangential Interpolation Problem:}  Given
a set of points $\Omega \subset {\mathbb D}_{Q}$ and nc functions
 $a \in \cT(\Omega'; \cL(\cY, \cE)_{\rm nc})$ and 
$b \in \cT(\Omega'; \cL(\cU, \cE)_{\rm nc})$  where we set
\begin{equation}   \label{Omega'}
 \Omega' = [\Omega]_{\rm  full} \cap {\mathbb D}_{Q}
 \end{equation}
 is the ${\mathbb D}_{Q}$-relative full nc envelope of $\Omega$, 
find $S$ in the Schur-Agler class $\mathcal{SA}_{Q}(\cU, \cY)$ such that
\begin{equation}  \label{int}
   a(Z) S(Z) = b(Z) \text{ for all } Z \in \Omega.
\end{equation}

\smallskip

We are now ready to state the main result concerning interpolation 
in the Schur-Agler class $\mathcal{SA}_{Q}(\cU, \cY)$. 
We note that statement (3) in the following theorem uses the 
notations and results from Theorem \ref{T:nctensorprod}.

\begin{theorem}  \label{T:ncInt} Suppose that $Q$ and ${\mathbb 
    D}_{Q}$ are as in \eqref{defDQ} and that we are given a data set
    (a set of points $\Omega \subset {\mathbb D}_{Q}$ and nc 
 functions $a \in \cT(\Omega'; \cL(\cY, \cE)_{\rm nc})$ and 
 $b \in \cT(\Omega'; \cL(\cU, \cE)_{\rm nc})$ with $\Omega'$ as  in \eqref{Omega'}) 
  for a Left-Tangential Interpolation Problem as above.
 Then the following conditions are equivalent.
 \begin{enumerate}
     \item[(1)] The Left-Tangential Interpolation Problem has a solution, 
     i.e., there exists an $S \colon {\mathbb D}_{Q} 
     \to \cL(\cU, \cY)_{\rm nc}$ in the nc Schur-Agler class 
     $\cSA_{Q}(\cU, \cY)$ satisfying the left-tangential 
     interpolation condition \eqref{int} on $\Omega$.
     
\item[(1$^{\prime}$)] The inequality
   \begin{equation} \label{ineq}
       a(Z)  a(Z)^{*} -  b(Z)  b(Z)^{*} \succeq 0
   \end{equation}
   holds for each $Z \in \Omega'$.

 \item[(2)] The pair $(a,b)$ has a left-tangential nc Agler decomposition over $\Omega$, 
   i.e., there exists a cp nc kernel 
    $\Gamma \colon \Omega \times \Omega \to 
    \cL(\cL(\cS), \cL(\cE))_{\rm nc}$ (as defined in Subsection 
    \ref{S:ncgendom}) so that
    \begin{align}    
& a(Z)( P \otimes I_{\cY}) a(W)^{*} - b(Z) (P \otimes I_{\cU})) 
b(W)^{*} \notag \\
& \quad = 
 \Gamma(Z,W) \left( (P \otimes I_{\cS}) - Q(Z) (P \otimes I_{\cR}) Q(W)^{*} 
 \right)
 \label{Aglerdecom}	
 \end{align}
 for all $Z \in \Omega_{n}$, $W \in \Omega_{m}$, $P \in {\mathbb 
C}^{n \times m}$.

\item[(3)] There exists an auxiliary Hilbert space $\cX$ equipped 
with a unitary $*$-representation $\pi \colon \cL(\cS) \mapsto 
\cL(\cX)$ and a contractive 
(even unitary) colligation matrix 
\begin{equation}   \label{bU}
\bU = \begin{bmatrix} A & B \\ C & D \end{bmatrix} \colon 
\begin{bmatrix}  \cX \\ 
\cU \end{bmatrix} \to \begin{bmatrix}  (\cL(\cS, \cR) \otimes_{\pi} 
\cX)_{\rm nc} \\ \cY \end{bmatrix}
\end{equation}
so that the function  $S(Z)$ defined by
\begin{equation}   \label{transfuncreal}
    S(Z) = D^{(n)} + C^{(n)} (I - L_{Q(Z)^{*}}^{*} 
    A^{(n)})^{-1} L_{Q(Z)^{*}}^{*} B^{(n)}
\end{equation}
for $Z \in \Omega_{n}$ (where here we use the notation \eqref{ncLT*} with $T 
= Q(Z)^{*}$)
where
\begin{equation}   \label{coln}
\begin{bmatrix} A^{(n)} & B^{(n)} \\ C^{(n)} & D^{(n)} \end{bmatrix}
= \begin{bmatrix} I_{n} \otimes A & I_{n} \otimes B \\ I_{n} \otimes 
C & I_{n} \otimes D \end{bmatrix} \colon \begin{bmatrix} 
\cX^{n} \\ \cU^{n} \end{bmatrix} \to \begin{bmatrix} 
\left((\cL(\cS,\cR) \otimes \cX)_{\rm nc} \right)^{n} \\ \cY^{n} \end{bmatrix}.
\end{equation}
satisfies the Left-Tangential Interpolation condition \eqref{int} on 
$\Omega$.
\end{enumerate}
Moreover, the implications (2) $\Rightarrow$ (3) and (3) 
$\Rightarrow$ (1) hold under the 
weaker assumption that $a$ and $b$ are only graded (rather than nc) 
functions defined only from $\Omega$ to $\cL(\cY, \cE)_{\rm nc}$ and $\cL(\cU, \cE)_{\rm nc}$ 
respectively.
\end{theorem}

We take up the proof of Theorems \ref{T:ncInt} in 
Section \ref{S:proofs}. We explore now  
various immediate corollaries which fall out as special cases.

\subsection{Agler decomposition and transfer-function realization for 
nc Schur-Agler class} \label{S:decom&real}
     If we look at the implication (1) 
     $\Leftrightarrow$ (3)  in Theorem \ref{T:ncInt} for the special 
     case where $\cE = \cY$, $\Omega = {\mathbb D}_{Q}$, $a(Z) = 
     I_{\cY^{n}}$ for $Z \in {\mathbb D}_{Q,n}$ and we set $S(Z) = 
     b(Z)$, we arrive at the following result.

\begin{corollary}  \label{C:AMcC-global}
    Let $Q$ and ${\mathbb D}_{Q}$ are as in 
    \eqref{ncpoly} and \eqref{defDQ} and $S$ is a graded function 
    from ${\mathbb D}_{Q}$ into $\cL(\cU, \cY)_{\rm nc}$.  Then the 
    following are equivalent:
 \begin{enumerate}
     \item $S \in \mathcal{SA}_{Q}(\cU, \cY)$, i.e. $S$ is a 
     contractive nc function on ${\mathbb D}_{Q}$.
     
  \item $S$ has a nc Agler decomposition over $\Omega$, i.e., there 
  exists a cp nc kernel $\Gamma \colon \Omega \times \Omega \to 
  \cL(\cS, \cL(\cY))_{\rm nc}$ so that
  \begin{align}   
    &  P \otimes I_{\cY} - S(Z) (P \otimes I_{\cU}) S(W)^{*} \\
    & \quad =\Gamma(Z,W)\left( (P \otimes I_{\cS}) - Q(Z) ( P \otimes 
    I_{\cR} ) Q(W)^{*} \right).
      \label{Aglerdecom-a=I}
      \end{align}
  \item $S$ has a transfer-function realization as in 
     \eqref{transfuncreal}.
 \end{enumerate}
\end{corollary}

  It has been observed by 
    Agler-McCarthy (see \cite{AMcC-global}) that Corollary 
    \ref{C:AMcC-global} can be used to prove nc analogues of the 
    Oka-Weil approximation theory for holomorphic functions in 
    several complex variables.  We present here a somewhat more 
    general setting for this type of result.  
    
    We let $\Xi$ be a  full nc subset of
    $\cV_{\rm nc}$ for some vector space 
    $\cV$.   We suppose that $\bA$ is an algebra of nc functions 
    contained in $\cT(\Xi; {\mathbb C}_{\rm nc})$.
    If $Q = [ q_{ij} ]$ is a finite matrix with entries 
    $q_{ij}$ in the algebra $\bA$, we define a subset ${\mathbb D}_{Q} 
    \subset \Xi$ as in \eqref{defDQ}:  
    \begin{equation}   \label{basic}
    {\mathbb D}_{Q} = \{ Z \in \Xi \colon \| Q(Z) \| < 1\}.
    \end{equation}
    A set of the form \eqref{basic} is said to be a \textbf{basic 
    $\bA$-free open set}.   Note that the intersection of two basic
    $\bA$-free open sets  is again a basic $\bA$-free open set since
    ${\mathbb D}_{Q_{1}} \cap {\mathbb D}_{Q_{2}}  = {\mathbb D}_{Q_{1} \oplus Q_{2}}$.
    Hence one can define a topology on $\Xi$, hereby called the
    \textbf{$\bA$-free topology}, by declaring that the basic $\bA$-free 
    open sets form a basis for this topology. 
 We can now state our version of a nc Oka-Weil theorem.

\begin{theorem} \label{T:OW} \textbf{(nc Oka-Weil Theorem)}
 With notation and definitions as above, 
suppose that $\Omega$ is an $\bA$-free open subset of $\Xi$ and 
suppose that $f$ is a nc function in
$\cT(\Omega; {\mathbb C}_{\rm nc})$ which is 
$\bA$-free locally bounded on $\Omega$ (i.e., for each $Z \in \Omega$, 
there is a $\bA$-free open subset $U$ of $\Omega$ containing $Z$
on which $f$ is uniformly bounded ($\| f(Z) \| \le M$ for all $Z \in 
U$ for some $M < \infty$).  Suppose that $K$ is a nc subset of $\Omega$ which is compact in 
the $\bA$-free topology.  Then there exists a sequence 
$\{q_{N}\}_{N=1}^{\infty}$ of 
functions from $\bA$ such that $q_{N}$ converges to $f$ uniformly on $K$.
\end{theorem}

\begin{proof}
Given a point $Z$ in $K$, there is an $\bA$-matrix $Q_{Z}$ with 
associated  $\bA$-free basic open set ${\mathbb D}_{Q_{Z}}$ such that $Z 
\in {\mathbb D}_{Q_{Z}} \subset \Omega$ and $f|_{{\mathbb 
D}_{Q_{Z}}}$ is bounded.  Then
$$
  {\mathfrak U} = \{ {\mathbb D}_{Q_{Z}} \colon Z \in K \}
$$
is an open cover of $K$.  As $K$ is compact in the $\bA$-free 
topology, there is a finite subcover:
$$
  K \subset {\mathbb D}_{Q_{Z^{(1)}}} \cup \cdots \cup {\mathbb 
  D}_{Q_{Z^{(N)}}}
$$
for some finitely many points $Z^{(1)}, \dots, Z^{(N)}$ in $K$.  We 
claim that
\begin{equation}   \label{claim-OW}
    \min_{j=1, \dots, N} \max_{W \in K}\{ \| Q_{Z^{(j)}}(W) \| \} < 1.
\end{equation}
If \eqref{claim-OW} fails to hold, then for each $j=1, \dots, N$ 
there is a $W^{(j)} \in K$ with
\begin{equation}  \label{Qjineq}
    \| Q_{Z^{(j)}}(W^{(j)}) \| \ge 1.
\end{equation}
Set $W^{(0)} = \sbm{ W^{(1)} & & \\ & \ddots & \\ & & W^{(N)}}$.  As 
we are assuming that $K$ is a nc set, it follows that $W^{(0)} \in 
K$.  As a consequence of \eqref{Qjineq} we also have $\| 
Q_{Z^{(j)}}(W^{(0)}) \| \ge 1$ for all $j$ implying that $W^{(0)}$ 
fails to be in $\cup_{j=1}^{N} {\mathbb D}_{Q^{(j)}}$ giving us the 
contradiction that $W^{(0)}$ is not in $K$ after all. We conclude 
that indeed the claim \eqref{claim-OW} must hold.

We therefore can find an index $j_{0}$ so that
$$
  r: = \max_{W \in K} \| Q_{Z^{(j_{0})}} (W) \|  < 1.
$$
To simplify notation, let us write simply $Q_{0}$ rather than 
$Q_{Z^{(j_{0})}}$.  Choose $t$ with $1 < t < 1/r$.  In this way we 
arrive at a single $\bA$-matrix $Q_{0}$ so that $K \subset {\mathbb 
D}_{t  Q_{0}}$.

By construction, the nc function $f$ is bounded on 
${\mathbb D}_{Q_{0}}$, i.e., there is an $M < \infty$ so that
$$
  \frac{1}{M} \cdot f \in \mathcal{SA}_{Q_{0}}({\mathbb C}_{\rm nc}).
$$
Hence, by Theorem \ref{T:ncInt} (specifically Corollary 
\ref{C:AMcC-global}), $f$ has a $Q_{0}$-realization of the form in 
\eqref{transfuncreal} valid on all of ${\mathbb D}_{Q_{0}}$. By 
construction $Q_{0}$ is a finite matrix over $\bA$, say of size $s 
\times r$, so we may take the spaces $\cR$ and $\cS$ in Theorem 
\ref{T:ncInt} to be 
$$
  \cR = {\mathbb C}^{r}, \quad \cS = {\mathbb C}^{s}.
$$
As a consequence of Remark \ref{R:findim} and Theorem 
\ref{T:ncfindim},  the realization formula \eqref{transfuncreal} for 
$f(Z)$ assumes the simpler form
\begin{equation}  \label{findimreal}
\frac{1}{M} \cdot  f(Z) = D^{(n)} + C^{(n)} (I -( Q_{0}(Z) \otimes I_{\cX_{0}}) 
  A^{(n)})^{-1} ( Q_{0}(Z) \otimes I_{\cX_{0}}) B^{(n)}
\end{equation}
where the contractive (or even unitary) colligation matrix $\sbm{A & 
B \\ C & D}$ has the form
$$
  \begin{bmatrix}  A^{(n)} & B^{(n)} \\ C^{(n)} & D^{(n)} \end{bmatrix} \colon
      \begin{bmatrix} ({\mathbb C}^{s} \otimes \cX_{0})^{n} \\ {\mathbb C}^{n} 
      \end{bmatrix}  \mapsto \begin{bmatrix} ({\mathbb C}^{r} \otimes 
      \cX_{0})^{n} \\ {\mathbb C}^{n} \end{bmatrix}
$$
(where we assume $Z \in {\mathbb D}_{Q_{0},n}$).
As $\| Q_{0}(Z) \| < 1$ for $Z \in {\mathbb D}_{Q_{0}}$,  we may 
expand out the representation \eqref{findimreal} for $f(Z)$ as an 
infinite series
\begin{equation}   \label{findimrealpartial}
\frac{1}{M} \cdot  f(Z) = 
D^{(n)} + \sum_{j=0}^{\infty} C^{(n)} \left( ( Q_{0}(Z) \otimes 
  I_{\cX_{0}}) A^{(n)} \right)^{j}  ( Q_{0}(Z) \otimes I_{\cX_{0}}) 
  B^{(n)}.
\end{equation}
As $K \subset {\mathbb D}_{t Q_{0}}$, we have that
$ \| Q_{0}(Z) \| \le \frac{1}{t} < 1$ for all $Z \in K$, and hence 
the above series converges uniformly on $K$.  As the matrix entries 
of $Q_{0}$ are all in the algebra $\bA$, it follows that each partial sum 
$q_{N}$
of the infinite series \eqref{findimrealpartial} is in $\bA$.  
We conclude that $f$ is the uniform limit of a sequence of functions 
$\{M \cdot q_{N}\}$ from $\bA$ as wanted, and Theorem \ref{T:OW} follows.
\end{proof}

We note that the special case of Theorem \ref{T:OW} where one takes $\Xi = 
({\mathbb C}^{d})_{\rm nc}$ and $\bA$ equal to the algebra of nc scalar polynomials
as in Example \ref{E:ncpoly}  amounts to the nc 
Oka-Weil theorem of Agler-McCarthy \cite[Theorem 9.7]{AMcC-global}.

\subsection{Relatively full nc set of interpolation nodes} \label{S:relativelyfull} 
The following result is just a reformulation of the 
   equivalence (1) $\Leftrightarrow$ (1$^{\prime}$) in Theorem 
   \ref{T:ncInt} for the case where $a(Z) = I_{n}$ (for $Z \in 
   \Omega_{n}$), $b(Z) = S_{0}(Z)$, and $\Omega = \Omega'$.

   \begin{corollary}  \label{C:relfull}  
  Suppose that $Q$ and ${\mathbb D}_{Q}$ are as in Theorem 
      \ref{T:ncInt} and that $\Omega \subset {\mathbb D}_{Q}$ is a 
      relatively full nc subset of ${\mathbb D}_{Q}$, i.e., 
 $$
     \Omega = [\Omega]_{\rm full} \cap {\mathbb D}_{Q}.
 $$
 Then any nc function $S_{0}$ on $\Omega$ can be extended to a nc 
 function $S$ on ${\mathbb D}_{Q}$ without increasing norm, i.e., 
 given a nc function $S_{0} \colon \Omega \to \cL(\cU, \cY)_{\rm 
 nc}$, there exists a nc function $S \colon {\mathbb D}_{Q} \to 
 \cL(\cU, \cY)_{\rm nc}$ such that
$$
  S|_{\Omega} = S_{0} \text{ and } \sup_{Z \in {\mathbb D}_{Q}} \| 
  S(Z) \| = \sup_{Z \in \Omega} \| S_{0}(Z) \|.
$$
\end{corollary}

This corollary motivates defining a nc subset $\cD$ of a full nc 
subset $\Xi$ to be a \textbf{(norm-preserving) nc interpolation domain} 
if $\cD$ has the same property as ${\mathbb D}_{Q}$ in the above 
corollary, namely:  {\em if $\Omega$ is a relatively full nc subset 
of $\cD$, i.e., if
$$
  \Omega = [\Omega]_{\rm full} \cap \cD,
$$
and if $S_{0} \colon \Omega \to \cL(\cU, \cY)_{\rm nc}$ is any nc 
function on $\Omega$, then there is a nc function $S \colon \cD \to 
\cL(\cU, \cY)_{\rm nc}$ such that}
$$
  S|_{\Omega} = S_{0} \text{ and } \sup_{Z \in \cD} \| S(Z) \| = 
  \sup_{Z \in \Omega} \| S_{0}(Z) \|.
$$
The content of Corollary \ref{C:relfull} is that any $Q$-disk ${\mathbb D}_{Q}$ 
is a nc interpolation domain.  We note that one can think of sets of 
the form ${\mathbb D}_{Q}$ as the nc version of analytic polyhedra 
as occurring in several-complex-variable theory which in turn are 
closely connected with domains of holomorphy (see e.g.\ 
\cite{Hormander}). As a first challenge we pose the 
following problem: {\em Find an intrinsic characterization of sets 
of the nc analytic polyhedra ${\mathbb D}_{Q}$, or more generally, 
interpolation domains.}

We mention several rather obvious necessary conditions for a domain 
$\cD \subset \cV_{\rm nc}$ to be a nc analytic polyhedron ${\mathbb 
D}_{Q}$ for some $Q$, namely:
\begin{enumerate}
\item  $\cD$ is a  nc set.
\item $\cD$ is $\Xi$-relatively locally closed with respect to similarity in the 
following sense: if $Z$ is an element of $\cD_{n}$, then there is a number $\delta > 0$ 
so that, whenever $\alpha$ is an invertible $n \times n$ matrix such 
that  $\| \alpha^{\pm 1} - I_{n}\| < \delta$, then $\alpha Z 
\alpha^{-1}$ is also in $\cD_{n}$.
\item $\cD$ is closed under restriction to invariant subspaces in the 
following strict sense:  whenever $Z \in \cD_{n}$, $\widetilde Z \in 
\Xi_{m}$ and $\alpha$ is an isometric $n \times m$ matrix (so $\alpha^{*} 
\alpha = I_{m}$) such that $\alpha \widetilde Z = Z \alpha$, then $\widetilde 
Z \in \cD$.
\end{enumerate}

While characterizing nc analytic polyhedra ${\mathbb D}_{Q}$ may be 
difficult, as we shall now show the case where $Q = L_{\varphi}$ is a nc linear 
map as in Example \ref{E:linear} (b) turns out to 
be tractable.   Helton, Klep, and McCullough in \cite{HKMcC2} introduced 
such nc domains ${\mathbb D}_{L}$ in the finite-dimensional setting, and studied various aspects of 
the associated nc function theory; following these authors, we shall refer to any such 
set ${\mathbb D}_{L}$ as a \textbf{pencil ball}. Our goal is to characterize intrinsically 
which nc subsets $\cD$ of $\cV_{\rm nc}$ can have the form of a 
pencil ball ${\mathbb D}_{L}$.  

Suppose first that $\cD = {\mathbb D}_{L_{\varphi}}$ for a nc linear 
map $L_{\varphi} 
\colon \cV^{n \times m} \to \cL(\cV_{1}^{n \times m}, \cV_{0}^{n 
\times m})$ as in Example \ref{E:linear} (a), where we assume that 
$\varphi$ is a linear map from the vector space $\cV$ into the 
operator space $\cL(\cR, \cS)$ of bounded linear operators between 
two Hilbert spaces $\cR$ and $\cS$. To simplify the notation, we 
write simply $L$ rather than $L_{\varphi}$.  Define a seminorm $\| \cdot \|_{n}$ 
on $\cV^{n \times n}$ by
$$
  \| Z \|_{n} = \| L(Z) \|_{\cL(\cR^{n}, \cS^{n})}.
$$
Using the bimodule property \eqref{computation} of $L$, it is easy to check that this 
system of norms $\{ \| \cdot \|\}_{n}$ satisfies the Ruan axioms:
\begin{enumerate}
    \item $\left\| \sbm{ Z & 0 \\ 0 & W} \right\|_{n} = \max \{ | Z 
    \|, \| W \| \}$ for $Z \in \cV^{n \times n}$, $W \in \cV^{m 
    \times m}$
    \item $\| \alpha \cdot Z \cdot \beta\| \le \| \alpha \| \| Z 
    \|_{n} \| \beta \|$ for $\alpha \in {\mathbb C}^{m \times n}$, $Z 
    \in \cV^{n \times n}$, $\beta \in {\mathbb C}^{n \times m}$.
\end{enumerate}
Indeed, these follow easily from the following properties of $L$:
\begin{enumerate}
    \item[(1$^{\prime}$)] $L\left( \sbm{ Z & 0 \\ 0 &  W } \right) = 
    \sbm{ L(Z) & 0 \\ 0 & L(W) }$ (i.e., $L$ respects direct sums),
    \item[(2$^{\prime}$)] $L(\alpha \cdot Z \cdot \beta) = \alpha 
    \cdot L(Z) \cdot \beta$ (the bimodule property \eqref{computation}).
 \end{enumerate}
 Given a nc subset $\cD$, if we can construct a system of norms (or 
 more generally just seminorms) $\| \cdot \|_{n}$ satisfying the Ruan 
 axioms so that the $n$-th level $\cD_{n}$ is the unit ball of 
 $\|  \cdot \|_{n}$
 $$
  \cD_{n} = \{ Z \in \cV^{n \times n} \colon \| Z \|_{n} < 1 \},
 $$
 then by Ruan's Theorem \cite[Theorem 2.3.5]{ERbook}, there is a completely isometric isomorphism 
 $\varphi$ from $\cV$ into a subspace of $\cL(\cH)$ for some Hilbert 
 space $\cH$  which we can without loss of generality take to have the form 
 $\cL(\cR, \cS)$; thus  $\| Z \|_{n} = \| L(Z) \|_{\cL(\cR^{n}, \cS^{n})}$) (where 
 $L(Z)  = \varphi^{(n \times n)}(Z) :=
 ({\rm id}_{{\mathbb C}^{n \times n}} \otimes \varphi)(Z)$ for 
 $Z \in \cV^{n \times n}$) (or only 
 a completely coisometric coisomorphism in the seminorm case), and it 
 then follows that $\cD = {\mathbb D}_{L}$.  If we fix a level $n$, 
 it is well known which sets $\cD_{n}$ can be the open unit ball (at 
 least up to boundary points) for some seminorm $\| \cdot \|_{n}$, i.e., so that
 there is a seminorm $\| \cdot \|_{n}$ so that
 $$
   \{ Z \in \cD_{n} \colon \| Z \|_{n} < 1 \} \subset \cD_{n} \subset 
   \{ Z \in \cD_{n} \colon \| Z \|_{n} \le 1\}
 $$
 (see \cite[Theorem  1.35]{Rudin}): namely, $\cD_{n}$ should be (i) 
 \textbf{convex} ($Z^{(1)}, Z^{(2)} \in \cD_{n}$, $0 < \lambda < 1$ 
 $\Rightarrow$ $\lambda Z^{(1)} + (1 - \lambda) Z^{(2)} \in 
 \cD_{n}$), (ii) \textbf{balanced} ($Z \in \cD_{n}$, $\lambda \in 
 {\mathbb C}$ with $|\lambda| \le 1$ $\Rightarrow$ $\lambda Z \in 
 \cD_{n}$), and (iii) \textbf{absorbing} ($Z \in \cV^{n \times n}$ 
 $\Rightarrow$  $\exists$ $t > 0 t$ in ${\mathbb R}$ so that $\frac{1}{t} 
 Z \in \cD_{n}$). It remains to understand what additional properties 
 are needed to get the system of norms $\{ \| \cdot \|_{n}\}$ so 
 constructed to also satisfy the Ruan axioms.
 
 There is a notion of convexity for this nc setting which has already 
 been introduced and used in a number of applications in the 
 literature (see e.g.\ \cite{EW} and the references there): {\em a nc 
 subset $\cD$ of $\cV_{\rm nc}$ is said to be \textbf{matrix-convex} 
 if}
 $$
 Z \in \cD_{n},\, \alpha \in {\mathbb C}^{n \times m} \text{ with } \alpha^{*} 
 \alpha = I_{m} \Rightarrow \alpha^{*} \cdot Z  \cdot \alpha \in 
 \cD_{m}.
 $$
 We shall need a nc extension of {\em balanced} defined as follows: 
 {\em a nc subset $\cD$ of $\cV_{\rm nc}$ is said to be 
 \textbf{matrix-balanced} if}
 $$
  Z \in \cD_{n}, \, Œ\alpha \in {\mathbb C}^{m \times n} \text{ with } \| 
 \alpha \| \le 1, \,
  \beta \in {\mathbb C}^{n \times m} \text{ with } \| \beta \| \le 1
 \Rightarrow \alpha \cdot Z \cdot \beta \in \cD_{m}.
 $$
 Unlike as in the non-quantized setting, $\cD$ being matrix-balanced trivially 
 implies that $\cD$ is matrix-convex.
 
 We can now state our characterization of which nc sets $\cD \subset 
 \cV_{\rm nc}$ have the form ${\mathbb D}_{L}$ of a pencil ball, at 
 least up to boundary points.

\begin{proposition}   \label{P:pencilball}  Given a nc set $\cD 
    \subset \cV_{\rm nc}$, for $Z \in \cV^{n \times n}$ define $\| 
    \cdot \|_{n}$ as the Minkowski functional associated with the set 
    $\cD_{n} \subset \cV^{n \times n}$:
    $$
    \| Z \|_{n} = \inf\{ t > 0 \colon t^{-1} Z \in \cD_{n} \}
    \text{ for } Z \in \cD_{n}.
    $$
    Then the following are equivalent:
    \begin{enumerate}
	\item
	$\{ \| \cdot \|_{n}\}$ is a system of norms satisfying the 
    Ruan axioms, and hence, from the preceding discussion, there is 
    nc linear map $L \colon \cV_{\rm nc} \to \cL(\cR, \cS)_{\rm nc}$ 
    as in Example \ref{E:linear} (a) so that $\cD = {\mathbb D}_{L}$ 
    up to boundary, i.e.,
   $$
   \{ Z \in \cV^{n \times n} \colon \| L(Z) \| < 1\} \subset \cD_{n} \subset   
 \{ Z \in \cV^{n \times n} \colon \| L(Z) \| \le  1\}.
   $$
   \item $\cD$ is a matrix-balanced nc subset of $\cV_{\rm 
   nc}$ such that each $\cD_{n}$ is absorbing.
   \end{enumerate}
 \end{proposition}
 
 \begin{proof} Suppose that $\{ \| \cdot \|_{n}\}$ satisfies the Ruan 
     axioms.  Then each $\| \cdot \|_{n}$ is absorbing since $\| \cdot 
     \|_{n}$ is a (finite-valued) seminorm on $\cD_{n}$.  By the 
     first Ruan axiom $\left\| \sbm{ Z & 0 \\ 0 & W } \right\| = 
     \max \{ \| Z \|, \| W \| \}$, we see that $\left\| \sbm{ Z & 0 
     \\ 0 & W} \right\| < 1$ $\Leftrightarrow$ both $\| Z\|_{n} < 1$ 
     and $\| W \|_{m} < 1$.  We conclude that $\cD$ is a nc set.  By 
     the second Ruan axiom $\| \alpha \cdot Z \cdot \beta \|_{m} = \| 
     \alpha \| \| Z \|_{n} \| \beta \|$,  we conclude that
 $$
 \| \alpha \| \le 1,\, \| \beta \| \le 1, \, \| Z \|_{n} \le 1 
 \Rightarrow \| \alpha \cdot Z \cdot \beta \|_{m} \le 1,
 $$
 i.e., $\alpha \cdot Z \cdot \beta \in \cD_{m}$ if $\| \alpha \| \le 
 1$, $\| \beta \| \le 1$, $Z \in \cD_{n}$.  Hence $\cD$ is 
 matrix-balanced.
 
 Conversely, assume that each $\cD_{n}$ is absorbing and that $\cD$ is 
 a matrix-balanced nc set. Define $\| \cdot \|_{n}$ as the Minkowski 
 functional of $\cD_{n}$:
 $$
 \| Z \|_{n} = \inf \{ t > 0 \colon \frac{1}{t} Z \in \cD_{n}\}
 \text{ for } Z \in \cV^{n \times n}.
 $$
 Then $\| Z \|_{n} < \infty$ since $\cD_{n}$ is absorbing.  Since 
 $\cD$ being matrix-balanced implies that each $\cD_{n}$ is also 
 convex and balanced, we see that $\| \cdot \|_{n}$ is a seminorm on 
 $\cD_{n}$ such that
 $$
 \{ Z \in \cV^{n \times n} \colon \| Z \|_{n} < 1\} \subset \cD_{n} 
 \subset \{ Z \in \cV^{n \times n} \colon \| Z \|_{n} \le 1\}.
 $$
 Since $\cD$ is a nc set, we see that
 $$
   Z \in \cD_{n}, W \in \cD_{m} \Rightarrow \sbm{ Z & 0 \\ 0 & W } 
   \in \cD_{n+m}.
 $$
 Since $\cD$ is also matrix-balanced, we see that
 \begin{align*}
 \begin{bmatrix} Z & 0 \\ 0 & W \end{bmatrix} \in \cD_{n+m} 
     \Rightarrow &
 \begin{bmatrix} I_{n} & 0 \end{bmatrix} \begin{bmatrix} Z & 0 \\ 0 & 
     W \end{bmatrix} \begin{bmatrix}  I_{n} \\ 0 \end{bmatrix} = Z \in 
 \cD_{n} \text{ and } \\
 & 
  \begin{bmatrix} 0 & I_{m} \end{bmatrix} \begin{bmatrix} Z & 0 \\ 0 
      & W \end{bmatrix}  \begin{bmatrix} 0 \\ I_{m} \end{bmatrix} = W \in 
 \cD_{m}
  \end{align*}
 Thus we have
 $$
 \begin{bmatrix} Z & 0 \\ 0 & W \end{bmatrix} \in \cD_{n+m} 
     \Leftrightarrow Z \in \cD_{n} \text{ and } W \in \cD_{m}.
 $$
 From this property we deduce that the first Ruan axiom holds: 
 $\left\| \sbm{ Z & 0 \\ 0 & W} \right\| = \max \{ \| Z \|, \, \| W 
 \| \}$.
 
 Finally, we use the matrix-balanced property of $\cD$ to deduce
 \begin{align*}
    &  \| \alpha Z \beta \|_{m}  = \inf\{t>0 \colon t^{-1} \alpha Z 
      \beta  \in \cD_{m} \}  \\
    &  \quad   =  \| \alpha \| \cdot \inf \{ t> 0  \colon \frac{\alpha}{\| \alpha \|}
    \cdot (t^{-1} Z ) \cdot  \frac{\beta}{\| \beta \|} \in \cD_{m} \} 
    \cdot \| \beta \| \\
    & \quad  \le \| \alpha \| \cdot \inf\{ t > 0 \colon t^{-1} Z \in 
    \cD_{n} \} \cdot \| \beta \| \text{ (since $\cD$ is 
    matrix-balanced)} \\
    & \quad = \| \alpha \| \cdot \|Z\|_{n}  \cdot \| \beta \|
 \end{align*}
 and the second Ruan axiom is verified.  
 \end{proof}

\subsection{The nc corona theorem}
If we look at the equivalence (1) $\Leftrightarrow$ 
(1$^{\prime}$) in Theorem \ref{T:ncInt} for the special case where $\Omega = {\mathbb D}_{Q}$,
we arrive at the following.

\begin{corollary}   \label{C:AMcC-global'}
     {\em Suppose that we are given nc functions $a \in 
     \cT({\mathbb D}_{Q}; \cL(\cY, \cE)_{\rm nc})$ and 
     $b \in \cT({\mathbb D}_{Q}; \cL(\cU, \cE)_{\rm nc})$.  Then the following are equivalent:
     \begin{enumerate}
	 \item $a(Z) a(Z)^{*}  - b(Z) b(Z)^{*} \succeq 0$ for all $Z 
	 \in {\mathbb D}_{Q}$.
	 
	 \item There exists a Schur-Agler class function $S \in 
	 \mathcal{SA}_{Q}(\cU, \cY)$ so that $a(Z) S(Z) = b(Z)$ for 
	 all $Z \in {\mathbb D}_{Q}$.
\end{enumerate}}
\end{corollary}
  
  We recall that the Carleson corona theorem (see \cite{Carleson}) 
 asserts than an $N$-tuple of bounded holomorphic functions on the unit 
 disk is not contained in a proper ideal if and only if the functions 
 are jointly bounded below by a positive constant.  A special case 
 of Corollary \ref{C:AMcC-global'} yields a free version of this 
 result.

\begin{corollary}  \label{C:Carleson}  Let $Q$ and ${\mathbb D}_{Q}$ 
     be as in \eqref{defDQ} and suppose that we are given
   $N$  scalar nc functions on ${\mathbb D}_{Q}$:  $\psi_{1}, \dots, \psi_{N}$ in 
   $\cT({\mathbb D}_{Q}; {\mathbb C}_{\rm nc})$. Assume that the family $\{ 
   \psi_{i} \colon i=1, \dots, N\}$ is uniformly bounded below in the 
   sense that there exist an $\epsilon > 0$ so that
   $$
        \sum_{i=1}^{N} \psi_{i}(Z) \psi_{i}(Z)^{*}  \succeq 
	\epsilon^{2} I_{{\mathbb C}^{n}}
$$
for all $Z \in {\mathbb D}_{Q, n}$ for all $n \in {\mathbb N}$.
Then there exist uniformly bounded nc functions $\phi_{1}, \dots \phi_{N}$ in 
$\cT({\mathbb D}_{Q}; {\mathbb C}_{\rm nc})$ so that the corona 
identity
$$  
\sum_{i=1}^{N} \psi_{i}(Z) \phi_{i}(Z) = I_{{\mathbb C}^{n}} \text{ 
for all } Z \in {\mathbb D}_{Q,n} \text{ for all } n \in {\mathbb N}
$$
holds.  In fact one can choose $\{ \phi_{i} \colon i=1, \dots, N\}$ so that
$$
 \sum_{i=1}^{N}  \phi_{i}(Z)^{*} \phi_{i}(Z) \preceq 
 (1/ \epsilon^{2}) I_{{\mathbb C}^{n}}
$$
for all $Z \in {\mathbb D}_{Q,n}$ for all $n \in {\mathbb N}$.
\end{corollary}

 \begin{proof}  This result amounts to the special case of Corollary 
     \ref{C:AMcC-global'} where $a(Z) = \begin{bmatrix} \psi_{1}(Z) & 
     \cdots & \psi_{N}(Z) \end{bmatrix}$ and $b(Z) =  \epsilon I_{{\mathbb 
     C}^{n}}$ for $Z \in {\mathbb D}_{Q,n}$ and one seeks to solve 
     for $S$ of the form $S(Z) =  \epsilon \sbm{ \phi_{1}(Z) \\ \vdots \\ 
     \phi_{N}(Z) }$.
 \end{proof}
    
    We note that Corollary \ref{C:AMcC-global'} with
$\Xi$ specialized to $\Xi = {\mathbb C}^{d}_{\rm nc}$ as in Example 
\ref{E:ncpoly} amounts to (a corrected version of) 
Theorem 8.1 in \cite{AMcC-global} (where the special case giving the nc Carleson 
  corona theorem is also noted).  By considering the special case $\cE = \cY$, $a(Z) = 
  I_{\cY^{n}}$ for $Z \in {\mathbb D}_{Q,n}$, and hence $S(Z) = b(Z)$, 
  one can see that it is not enough to assume only that 
  $b(Z) = S(Z)$ is  graded as in \cite{AMcC-global} since it is 
  easy to write down contractive graded functions which are not nc 
  functions.  As indicated in the statement of Theorem \ref{T:ncInt}, the 
  implication (2) $\Rightarrow$ (3) does hold under the weaker 
  assumption that $a$ and $b$ are only graded (not necessarily nc) 
  functions defined only on the subset $\Omega$ where the 
  interpolation conditions are specified.

  \subsection{Finite set of interpolation nodes} \label{S:finiteset}
  We next focus on the special case of Theorem \ref{T:ncInt} where the 
 set of interpolation nodes $\Omega$ is a finite set.  As already
 observed by Agler and McCarthy in \cite{AMcC-Pick}, 
 taking a singleton set $\{ Z^{(0)}\}0$ as the set of interpolation nodes 
 is equivalent to taking a finite set $\{ Z^{(1)}, \dots, Z^{(N)}\}$ 
 since one can use the nc function structure to get an equivalent 
 problem with the singleton set $\{ Z^{(0)}\}$ with $Z^{(0)} =
\sbm{ Z^{(1)} & & \\ & \ddots & \\ & & Z^{(N)}}$.  Hence we focus on 
the setting where the interpolation node set is a singleton.

\begin{corollary}  \label{C:AMcC-Pick}
 Let $Q$ and ${\mathbb D}_{Q}$ be as in Theorem \ref{T:ncInt} and suppose that $Z^{(0)}$ is 
    one particular point in ${\mathbb D}_{Q,n}$ and $\Lambda_{0}$ is a 
    particular operator in $\cL(\cU \otimes {\mathbb C}^{n}, \cY 
    \otimes {\mathbb C}^{n})$.  Then the following are equivalent:
    \begin{enumerate}
	\item There exists a function $S$ in the Schur-Agler class 
	$\mathcal{SA}_{Q}(\cU, \cY)$ so that $S(Z^{(0)}) = 
	\Lambda_{0}$.
	
	\item There exists a nc function $S_{\rm full}$ on the 
 ${\mathbb D}_{Q}$-relative  full nc envelope $\{Z^{(0)}\}_{\rm nc,full} 
 \cap {\mathbb D}_{Q}$ of the singleton set $\{Z^{(0)}\}$ such that $S_{\rm 
	full}(Z^{(0)}) = \Lambda_{0}$ and
	$\|S_{\rm full}(Z)\| \le 1$ for all $Z \in 
	\{Z^{(0)}\}_{\rm nc,full} \cap {\mathbb D}_{Q}$.
    \end{enumerate}
\end{corollary}
   
 \begin{proof}   This amounts to the equivalence (1) 
     $\Leftrightarrow$ (1$^{\prime}$) in Theorem \ref{T:ncInt} for the special 
     case where $\Omega$ is the singleton set $\Omega = \{Z^{(0)}\}$ 
     with $a(Z^{(0)}) = I_{n}$  and  $b(Z^{(0)}) =  \Lambda_{0}$.
 \end{proof}
 
 We note that Corollary \ref{C:AMcC-Pick} implies Theorem 1.3 
 in \cite{AMcC-Pick}, apart from the added content in 
 \cite{AMcC-Pick} that, in the case where 
 $\Xi = ({\mathbb C}^{d})_{\rm nc}$ and $Q$ are taken as in Example 
 \ref{E:ncpoly}, then one can take $S_{\rm full}$ in statement (2) to be a 
 nc polynomial. 
 The formulation in \cite{AMcC-Pick} is in terms of the nc-Zariski 
 closure $\overline{\{Z^{(0)}\}}$ rather than in terms of the full nc 
 envelope $\{ Z^{0}\}_{\rm nc,full}$.  However a consequence of the 
 containment \eqref{gen/Zariski} is that the Agler-McCarthy 
 hypothesis with nc-Zariski closure implies the hypothesis here with 
 full nc envelope.  Of course, whenever it is the case that the 
 containment \eqref{gen/Zariski} is actually an equality (as is the 
 case when $\Xi = {\mathbb C}_{\rm nc}$ by Proposition 
 \ref{P:Zariski-close-d=1}), then the hypothesis here and the 
 hypothesis in Theorem 1.3 from \cite{AMcC-Pick} are the same.

 \subsection{Commutative Schur-Agler class} \label{S:commutative}  
 As our next illustrative special case, we indicate how the commutative results of 
     \cite{AT, BB04, MP} follow from the general theory for the nc 
     case. Specifically, the
     Agler-decomposition as well as
     transfer-function realization and interpolation (at least for 
     the left-tangential case) results of \cite{AT, BB04, MP} for the 
     commutative Schur-Agler class determined by a matrix polynomial 
     $q$ in $d$ (commuting) complex variables (or more generally an 
     operator-valued holomorphic function on ${\mathbb C}^{d}$ as in 
     \cite{MP}) follow as a corollary of Theorem \ref{T:ncInt}.  
     Indeed, an application of
 Corollary \ref{C:AMcC-global} to the special case where $\Xi$ and 
 $Q=q$ are as in Example \ref{E:comtuples} gives all these results 
 for the now commutative Schur-Agler class $\mathcal{SA}_{q}(\cU, \cY)$ 
     (the specialization of the general noncommutative theory to the 
     commutative setup of Example \ref{E:comtuples}).  One can then 
     use the observations  made in the discussion of \ref{E:comtuples} 
     that the Taylor/Martinelli-Vasilescu functional calculus extends a 
     holomorphic function defined on ${\mathbb D}_{q,1}$ (i.e., $d$ 
     scalar arguments) to a nc function (i.e., a function respecting 
     intertwinings) defined on ${\mathbb D}_{q}$.  Putting all this together, 
     we see that the formally noncommutative Schur-Agler class 
     $\mathcal{SA}_{q}(\cU, \cY)$ for this special case is the same 
     as the commutative Schur-Agler class $\cC \mathcal{SA}_{q}(\cU, 
     \cY)$ as defined in \cite{AT, BB04, MP}, and the results of 
     \cite{AT, BB04, MP} on Agler decomposition, transfer-function 
     realization and interpolation follow as a special case of 
     Theorem \ref{T:ncInt}. 
 We give further discussion of this setting in Remark \ref{R:AKV} below.

\subsection{Unenhanced Agler decompositions}  \label{S:unenhanced}
    In Theorem \ref{T:ncInt} suppose that the  $\Omega$ is a nc 
    subset which is open 
    in the finite topology of $\cV_{\rm nc}$ (see 
    \cite[page 83]{KVV-book}), i.e., that the intersection of $\Omega_{n}$ 
    with any finite-dimensional subspace $\cO$ of $\cV^{n \times n}$ 
    is open in the Euclidean topology of $\cO$ for each $n = 1, 2, 
    \dots$, and impose the standing assumption in Theorem 
    \ref{T:ncInt} that $a$ and $b$ are nc functions on the full nc 
    envelope of $\Omega$.  Then statement (2) in the Theorem 
    \ref{T:ncInt} can be weakened to the requirement that the Agler 
    decomposition \eqref{Aglerdecom} holds only for $Z,W$ both in 
    $\Omega_{n}$ and $P = I_{n}$ for each $n=1,2,\dots$, i.e., the 
    Agler decomposition \eqref{Aglerdecom} can be weakened to the 
    ``unenhanced'' form
    \begin{equation}  \label{unenhanced}
a(Z) a(W)^{*} - b(Z) b(W)^{*} = \Gamma(Z,W)(I - Q(Z) Q(W)^{*})
\end{equation}
for $Z, W \in \Omega_{n}$ for $n=1,2,\dots$. 
To see 
    this, let $Z,W$ be any two points in $\Omega_{n}$.  Since 
    $\Omega$ is now assumed to be finitely open, there is an 
    $\epsilon > 0$ so that
whenever $ \alpha$ and $\beta$ are invertible $n \times n$ matrices with
$\| \alpha^{\pm 1} - I_{n} \| < \epsilon$ and $\| \beta^{\pm 1} - 
I_{n}\| < \epsilon$, then it follows that both $\widetilde Z : = \alpha Z 
\alpha^{-1}$  and  $\widetilde W: = \beta W \beta^{-1}$ are in $\Omega_{n}$.
Consequently, for $P = \alpha^{-1} \beta^{-1*}$ we have
\begin{align*}
    & a(Z) P a(W)^{*} - b(Z) P  b(W)^{*} \\
    & =
  a(\alpha^{-1} \widetilde Z \alpha) \alpha^{-1} \beta^{-1*} a((\beta^{-1} \widetilde W \beta)^{*} - 
  b(\alpha^{-1} \widetilde Z \alpha) \alpha^{-1} \beta^{-1*} b(\beta^{-1} \widetilde W \beta)^{*} \\
    & =\alpha^{-1} a(\widetilde Z) \alpha \cdot \alpha^{-1} \beta^{-1*} \cdot  
    \beta^{*} a(\widetilde W)^{*} \beta^{-1*} \\
    & \quad \quad \quad \quad  -
   \alpha^{-1} b(\widetilde Z) \alpha  \cdot \alpha^{-1} \beta^{-1*} 
     \cdot \beta^{*} b(\widetilde W)^{*} \beta^{-1*}  \\ 
  & \quad \quad \quad \quad \text{ (since $a$ and $b$ respect  intertwinings)} \\
  & = \alpha^{-1} \left( a(\widetilde Z) a(\widetilde W)^{*} - 
  b(\widetilde Z) b(\widetilde W)^{*} \right) \beta^{-1*} \\
    &  = \alpha^{-1} \, \Gamma(\widetilde Z, \widetilde W)
    \left(I_{\cS^{n}} - Q(\widetilde Z) Q(\widetilde W)^{*} \right) \, 
    \beta^{-1*}  \text{ (by \eqref{unenhanced})} \\
    &  = \Gamma(Z,W) \left( \alpha^{-1} \beta^{-1*} - \alpha^{-1} Q(\widetilde 
    Z) Q(\widetilde W)^{*} \beta^{-1*} \right) \\
    & \quad \quad \quad \text{ (since $\Gamma$ 
    respects intertwinings)} \\
    &  = \Gamma(Z,W) \left( \alpha^{-1} \beta^{-1*} - Q(Z) \alpha^{-1} \beta^{-1*} 
    Q(W)^{*}\right) \\ 
    & \quad \quad \quad \text{ (since $Q$ respects intertwinings) } \\
    & \quad =  \Gamma(Z,W)(P - Q(Z) P Q(W)^{*}).
    \end{align*}
    Hence, for given $Z,W \in \Omega_{n}$, \eqref{Aglerdecom} holds 
    for all $P \in {\mathbb C}^{n \times n}$ in an open set around 
    $I_{n}$.  But both sides of \eqref{Aglerdecom} are holomorphic in 
    the entries of $P$. Hence by the uniqueness of analytic 
    continuation off an open set it follows that \eqref{Aglerdecom} 
    holds for all $P \in {\mathbb C}^{n \times n}$.
    
    For $Z \in \Omega_{n}$ and $W \in \Omega_{m}$ with possibly $n 
    \ne m$, apply the preceding result with $\sbm{Z & 0 \\ 0 & W } 
    \in (\Omega)_{{\rm nc}, n+m}$ in place of $Z$ and $W$ and with 
    $\sbm{0 & P \\ 0 & 0} \in {\mathbb C}^{(n+m) \times (n+m)}$ in 
    place of $P$.  Then the resulting identity
    \begin{align*}
    &	a\left( \sbm{Z & 0 \\ 0 & W} \right) \sbm{ 0 & P \\ 
	0 & 0 } a\left( \sbm{ Z & 0 \\ 0 & W} \right)^{*}
	- 	b\left( \sbm{Z & 0 \\ 0 & W} \right) \sbm{0 & P \\ 
	0 & 0 } b\left( \sbm{ Z & 0 \\ 0 & W} \right)^{*} \\
	& \quad = \Gamma \left( \sbm{ Z & 0 \\ 0 & W}, \sbm{ Z & 0 \\ 
	0 & W} \right) \left( \sbm{ 0 & P \\ 0 & 0} - Q\left( \sbm{ Z 
	& 0 \\ 0 & W} \right) \sbm{ 0 & P \\ 0 & 0} Q\left( \sbm{ Z & 
	0 \\ 0 & W} \right)^{*} \right)
\end{align*}
combined with the ``respects direct sums'' property of $a,b,\Gamma, Q$ 
leads to the identity
$$
\begin{bmatrix} 0 & a(Z) P a(W)^{*} - b(Z) P a(W)^{*} \\ 0 & 0 
\end{bmatrix} = \begin{bmatrix} 0 & \Gamma(Z,W)(P - Q(Z) P Q(W)^{*}) 
\\ 0 & 0\end{bmatrix},
$$
and hence the identity \eqref{Aglerdecom} holds for $Z \in 
\Omega_{n}$, $W \in \Omega_{m}$, $P \in {\mathbb C}^{n \times m}$ 
with $n \ne m$ as well.

\section{Proofs of Schur-Agler class interpolation theorems}  \label{S:proofs}

We shall prove (1) $\Rightarrow$ (1$^{\prime}$) $\Rightarrow$ (2) $\Rightarrow$ 
(3) $\Rightarrow$ (1) in Theorem \ref{T:ncInt}.

\begin{proof}[Proof of (1) $\Rightarrow$ (1$^{\prime}$) in Theorem 
    \ref{T:ncInt}:]  Suppose that the left-tangential interpolation 
    condition \eqref{int} holds on $\Omega$ for a Schur-Agler class 
    function $S \in \mathcal{SA}_{Q}(\cU, \cY)$. It is easily checked 
    that the pointwise product $a \cdot S$ is again a nc function 
    whenever each of $a$ and $S$ is a nc function.  By the uniqueness 
    of a nc-function extension from $\Omega$ to $\Omega' = \Omega_{\rm nc, full} 
    \cap {\mathbb D}_{Q}$ (see Proposition \ref{P:finitetype}),
   the identity $a(Z) S(Z) = b(Z)$ holding on 
    $\Omega$ implies that it continues to hold on $\Omega' = \Omega_{\rm nc, full} 
    \cap {\mathbb D}_{Q}$.
    By assumption, $S(Z)$ is contractive for all $Z \in {\mathbb 
    D}_{Q}$, and hence in particular for all $Z \in 
    \Omega_{\rm nc, full}$.  We conclude that
    $$
    a(Z)  a(Z)^{*} -  b(Z) b(Z)^{*} = a(Z)(I - S(Z) S(Z)^{*}) a(Z)^{*}
    \succeq 0 \text{ for all } Z \in \Omega_{\rm full} \cap {\mathbb D}_{Q}
    $$
    and statement (2) of the theorem follows.
    \end{proof}
  
    \begin{proof}[Proof of (1$^{\prime}$) $\Rightarrow$ (2) in Theorem 
    \ref{T:ncInt};]
 We subdivide the proof into two cases.
    
    \smallskip
\noindent
\textbf{Case 1: $\Omega$ and $\dim \cE$ finite.}
For this case we assume that both $\Omega$ and $\dim \cE$ are finite.
 We first need a few additional preliminaries.
 
 \smallskip
 
\noindent
\textbf{The finite point set $\Omega$.}
Recall the underlying framework from Subsection \ref{S:ncdisk}.  We 
are given a vector space $\cV$, a full nc subset $\Xi \subset 
\cV_{\rm nc}$, a nc function $Q$ from $\Xi$ to $\cL(\cR, \cS)_{\rm 
nc}$ with associated nc $Q$-disk ${\mathbb D}_{Q} \subset \Xi$.  For 
the present Case 1, we are assuming that $\Omega$ is a finite subset 
of ${\mathbb D}_{Q}$.  Therefore the subspace $\cV^{0}$ of $\cV$ spanned by all 
the matrix entries of elements $Z$ of $\Omega$ is finite dimensional, 
say $\dim \cV^{0} = d$.  For this Case 1 part of the proof, it is 
only vectors in $\cV^{0}$ which come up, so without loss of 
generality we assume that $\cV = \cV^{0}$.  By choosing a basis we 
identify $\cV$ with ${\mathbb C}^{d}$, and thus each point $Z \in 
\Omega_{n}$ is identified with an element of $({\mathbb C}^{d})^{n 
\times n} \cong ({\mathbb C}^{n \times n})^{d}$.  For $Z \in 
\Omega_{n}$, we therefore view $Z$ as a $d$-tuple $Z = (Z_{1}, \dots, 
Z_{d})$ of complex $n \times n$ matrices ($Z_{k} \in {\mathbb C}^{n 
\times n}$ for $k=1, \dots, d$).  

We  next define nc functions $\chi_{k}$ defined on all of $({\mathbb 
C}^{d})_{\rm nc}$ with values in 
${\mathbb C}_{\rm nc}$ by 
$$
  \chi_{k}(Z) = Z_{k} \text{ if } Z = (Z_{1}, \dots, Z_{d}).
$$
We stack these into a block row matrix to define a nc function 
from $\Xi \subset ({\mathbb C}^{d})_{\rm nc}$ into 
$\cL({\mathbb C}^{d}, {\mathbb C})_{\rm nc}$ by
$$
  \chi(Z) = \begin{bmatrix} \chi_{1}(Z) & \cdots & \chi_{d}(Z) 
\end{bmatrix} = \begin{bmatrix} Z_{1} &  \cdots & Z_{d} \end{bmatrix}.
$$
We view each such $Z_{k}$ as an operator $Z_{k}^{r}$ acting on row vectors via right multiplication:
thus
$$
Z_{k}^{r} \colon x^{*} \mapsto x^{*} Z_{k} \text{ for } x^{*} \in 
{\mathbb C}^{1 \times n}.
$$
Thus, for $Z \in ({\mathbb C}^{n \times n})^{d}$, we identify 
$\chi(Z)$ with the operator $\chi(Z)^{r}$ acting from ${\mathbb C}^{1 
\times n}$ to  ${\mathbb C}^{1 \times nd}$ by
$$
 \chi(Z)^{r} \colon x^{*} \mapsto \begin{bmatrix} x^{*} Z_{1} & 
 \cdots & x^{*} Z_{d} \end{bmatrix} \text{ for } x^{*} \in {\mathbb 
 C}^{1 \times n}.
 $$

\smallskip

\noindent
\textbf{The linear space ${\mathfrak X}$ and its 
cone $\cC$.}
We let ${\mathfrak X}$ be the linear space of all nc kernels $K \in 
\widetilde \cT^{1}(\Omega; \cL(\cE)_{\rm nc}, {\mathbb C}_{\rm nc})$ with norm given by
$$
   \| K \|_{{\mathfrak X}} = \max \{ \| K(Z,W) \| \colon Z, W \in \Omega\}.
$$
We define a subset $\cC$ of ${\mathfrak X}$ by
\begin{align}
    \cC = & \{K \in {\mathfrak X} \colon \exists \text{ a 
    cp nc kernel } \Gamma \in \widetilde 
    \cT^{1}(\Omega; \cL(\cE)_{\rm nc}, \cL(\cS)_{\rm nc}) \text{ so 
    that } \notag \\
    & K(Z,W)(P) = \Gamma(Z,W) \left( P \otimes I_{\cS} - Q(Z) (P 
    \otimes I_{\cR}) Q(W)^{*}\right)  \label{Aglerdecom'} \\
    & \text{for all } Z \in \Omega_{n}, \,
    W \in \Omega_{m}, \, P \in {\mathbb C}^{n \times m} 
    \}. \notag
\end{align}
Key properties of $\cC$ are given by the following lemma.

\begin{lemma}  \label{L:cone}
    \begin{enumerate}
	\item The subset $\cC$ is a 
  closed cone in ${\mathfrak X}$. 
  
  \item For $f \in \cT(\Omega; \cE_{\rm nc})$, define 
  $D_{f,f} \in \widetilde \cT^{1}(\Omega; 
  \cL(\cE)_{\rm nc}, {\mathbb C}_{\rm nc})$ by
  \begin{equation} \label{Deltaff}
    D_{f,f}(Z,W)(P) = f(Z) P f(W)^{*}.
  \end{equation}
  Then $D_{f,f} \in \cC$. 
  \end{enumerate}
 \end{lemma}
 
 \begin{proof}[Proof of Lemma \ref{L:cone} part (1)]
One easily verifies from the definitions that $\tau K \in \cC$ 
whenever $K \in \cC$ and $\tau >0$ and that $K_{1} + K_{2} \in \cC$ whenever 
$K_{1}$ and $K_{2}$ are in $\cC$, i.e., $\cC$ is a cone.  

It remains to show that $\cC$ is closed in the norm topology 
of ${\mathfrak X}$.   Toward this end suppose that $\{K_{N} \colon N \in {\mathbb N}\}$ 
is a sequence of elements of $\cC$ such that $\| K - K_{N} \|_{\mathfrak X} \to 0$ as 
$N \to \infty$ for some  $K \in {\mathfrak X}$.   By definition, for each $N$ there is 
a cp nc kernel $\Gamma_{N} \in \widetilde \cT^{1}(\Omega; \cL(\cE)_{\rm nc}, 
     \cL(\cS)_{\rm nc})$ so that 
     $$
     K_{N}(Z,W)(P) = \Gamma_{N}(Z,W)( (P \otimes I_{\cS}) - Q(Z) (P 
     \otimes I_{\cR}) Q(W)^{*})
     $$
     for all $Z \in \Omega_{n}$, $W \in  
     \Omega_{m}$, $P 
     \in {\mathbb C}^{n \times m}$ for all $m,n \in {\mathbb N}$.
     The goal is to produce a cp nc kernel $\Gamma$ in $\widetilde 
     \cT^{1}(\Omega; \cL(\cE)_{\rm nc}, \cL(\cS)_{\rm nc})$
     so that $K$ can be expressed in the form \eqref{Aglerdecom'}.
     
 Define a number $\rho_{0}$ by
 \begin{equation}   \label{rho0}
   \rho_{0} = \max \{ \| Q(Z) \| \colon Z \in \Omega\}.
 \end{equation}
 As $\Omega$ is a finite subset of ${\mathbb D}_{Q}$, we see that 
 $\rho_{0} < 1$.  Since $\Gamma_{N}(Z,Z)$ is a positive map for each $N$ 
 and each $Z \in \Omega$,  we get the estimate
  \begin{align*}
     K_{N}(Z,Z)(I_{n}) &  = \Gamma_{N}(Z,Z)\left(I_{\cS^{n}} - Q(Z) 
     Q(Z)^{*}\right) \notag \\
   & \ge (1 - \rho_{0}^{2}) \Gamma_{N}(Z,Z)(I_{\cS^{n}}). 
 \end{align*}
 Consequently, we get
 \begin{equation} \label{est1}
  \|\Gamma_{N}(Z,Z)(I_{\cS^{n}}) \|_{\cL(\cE^{n})} \le \frac{1}{1 - \rho_{0}^{2}} 
   \|K_{N}(Z,Z)\|_{\cL(\cL({\mathbb C}^{n}), \cL(\cE^{n}))}.
 \end{equation}
 As a consequence of  $\| K_{N} - K \|_{\mathfrak X} \to 0$ as $N \to 
 \infty$, it follows that in particular $\| K_{N}(Z,Z) - 
 K(Z,Z)\|_{\cL(\cL({\mathbb C}^{n}), \cL(\cY^{n}))} \to 0$ as $N \to \infty$.  Thus
 $\| K_{N}(Z,Z)\|$ is uniformly bounded in $\cL(\cL({\mathbb C}^{n}), \cL(\cY^{n}))$.  
 As a consequence of \eqref{est1} we then see that  
 $\|\Gamma_{N}(Z,Z)(I_{\cS^{n}}) \|_{\cL(\cE^{n})}$ is uniformly 
 bounded in $N=1,2,\dots$.
 Since $\Gamma_{N}(Z,Z)$ is completely positive, we have
 $$
 \| \Gamma_{N}(Z,Z)\|_{\cL(\cL(\cS^{n}), \cL(\cE^{n}))} = \| 
 \Gamma_{N}(Z,Z)(I_{\cS^{n}})\|_{\cL(\cE^{n})}.
 $$
 Hence $\| \Gamma_{N}(Z,Z)\|_{\cL(\cL(\cS^{n}), \cL(\cE^{n}))}$ is uniformly bounded with 
 respect to $N \in {\mathbb N}$ for $Z \in \Omega$. Moreover, since 
 $\Gamma_{N}$ is a cp nc kernel, $\Gamma_{N}(Z,Z)$ is completely 
 positive as a map from $\cL(\cS)^{n \times n}$ to $\cL(\cE)^{n 
 \times n}$ (where $Z \in \Omega_{n}$).
 
 Let now $Z$ and $W$ be two points in $\Omega$.  Then $\sbm{ Z & 0 \\ 
 0 & W} $ is a point in the nc envelope $\Omega_{\rm nc}$ of $\Omega$ 
 and $\Gamma$.  By Proposition \ref{P:extcpncker}, both $K$ and 
 $\Gamma$ can be extended as nc and cp nc kernels respectively to 
 $\Omega_{\rm nc}$, and hence to the finite set $\Omega \cup \left\{ 
 \sbm{ Z & 0 \\ 0 & W} \right\}$.  Then the estimate \eqref{est1} and 
 the analysis there with $\sbm{ Z & 0 \\ 0 & W}$ in place of $Z$,
 we see that
 $\left\| \Gamma_{N}\left( \sbm{ Z & 0 \\ 0 & W} \right) \right\|$ is 
 uniformly bounded in $N$.  Let us note that
 $$
  \Gamma_{N} \left( \sbm{Z & 0 \\ 0 & W }, 
  \sbm{Z & 0 \\ 0 & W } \right) 
      \left( \sbm{ 0 & P \\ 0 & 0 } \right) =
      \sbm{ 0 & \Gamma_{N}(Z,W)(P) \\ 0 & 0 }.
   $$
   We conclude that $\| \Gamma_{N}(Z,W) \|$ is uniformly bounded in 
   norm with respect to $N$ for each $Z$ and $W$ in $\Omega$.
   
  Note that $\Gamma_{N}(Z,W) \in \cL(\cL(\cS^{m}, \cS^{n}), 
  \cL(\cE^{m}, \cE^{n}))$ if $Z \in \Omega_{n}$ and $W \in 
  \Omega_{m}$. A key point at this stage is that the Banach space 
  $\cL(\cL(\cS^{m}, \cS^{n}), \cL(\cE^{m}, \cE^{n}))$ has a predual 
  $\cL(\cL(\cS^{m}, \cS^{n}), \cL(\cE^{m}, \cE^{n}))_{*}$ such that 
  on bounded sets the weak-$*$ topology is the same as the pointwise 
  weak-$*$ topology: a bounded net $\{\Phi_{\lambda}\}$ converges to  $\Phi$ means that
  $\Phi_{\lambda}(T) \to \Phi(T)$ in the ultraweak (or weak-$*$) topology 
  of $\cL(\cE^{m}, \cE^{n})$ for each fixed $T \in \cL(\cS^{m}, 
  \cS^{n})$ and the topology on the whole space is defined to be the 
  strongest topology which agrees with this topology on bounded 
  subsets.  Since the weak and weak-$*$ topologies agree on bounded 
  subsets, this topology is sometimes also called the {\em 
  BW-topology} (for ``bounded-weak topology'').  
  In fact there is a more general result:  {\em if $\cX$ and 
  $\cZ$ are Banach spaces, then the space $\cL(\cX, \cZ^{*})$ is 
  isometrically isomorphic to the dual of the Banach projective tensor-product space 
  $\cX \widehat \otimes \cZ$; moreover a bounded net $\Phi_{\lambda}$ in
  $\cL(\cX, \cZ^{*})$ converges to $\Phi \in \cL(\cX, \cZ^{*})$ in the associated
  weak-$*$ topology if and only  if the ${\mathbb C}$-valued net 
  $\left(\Phi_{\lambda}(x)\right)({\mathfrak z})$ converges to 
  $\left(\Phi(x)\right)({\mathfrak z})$ 
  for each fixed $x \in \cX$ and ${\mathfrak z} \in \cZ$} (see \cite[Corollary 2 page 230]{DU} 
  as well as \cite[pages 84--85]{Paulsen} and \cite[Section IV.2]{Takesaki}). 
  We apply this result with $\cX = \cL(\cS^{m}, \cS^{n})$ and 
  $\cZ^{*} = \cL(\cE^{m}, \cE^{n})$ (so 
  $\cZ$ can be taken to be the trace-class operators 
  $\cC_{1}(\cE^{n}, \cE^{m})$ from $\cE^{n}$ to $\cE^{m}$).  Note 
  that for our application with $\dim \cE < \infty$ this result is actually more elementary 
  than the general case as described above.
  
  In  any case,  by the Banach-Alaogl\u u Theorem \cite[page 68]{Rudin}, 
  norm-closed and bounded subsets of $\cL(\cL(\cS^{m}, \cS^{n}), \cL(\cE^{m}, \cE^{n}))$
  are compact in the weak-$*$ topology. Since we established above 
  that $\{\Gamma_{N}(Z,W)\}$ is uniformly bounded in norm as $N \to 
  \infty$ for each of the finitely many $Z,W \in \Omega$,  it follows 
  that we can find a subnet $\{\Gamma_{\lambda}\}$ of the sequence $\{ \Gamma_{N}\}$ so 
 that $\Gamma_{\lambda}(Z,W)$ converges weak-$*$ to an element 
 $\Gamma(Z,W)$ in  $\cL(\cL(\cS^{m}, \cS^{n}), \cL(\cE^{m}, 
 \cE^{n}))$.  We need to check that $\Gamma$ so defined is a 
 cp nc kernel on $\Omega$, i.e., we must check that the limiting 
 $\Gamma$ satisfies \eqref{kerintertwine} and \eqref{kercp'''}  given that each 
 $\Gamma_{\lambda}$ does, or, given that 
 \begin{align*}
   & Z \in \Omega_{n},\, \widetilde Z \in \Omega_{\widetilde n},\, \alpha \in 
    {\mathbb C}^{\widetilde n \times n} \text{ such that } \alpha Z = \widetilde Z 
    \alpha,  \notag \\
 & W \in \Omega_{m},\, \widetilde W \in \Omega_{\widetilde m}, \, \beta \in {\mathbb 
 C}^{\widetilde m \times m} \text{ such that } \beta W = \widetilde W \beta, 
 \quad    P \in \cL(\cS)^{n \times m},
 \end{align*}
 we must show that
 \begin{align}  
  &   \alpha \Gamma_{\lambda}(Z,W)(P) \beta^{*} = \Gamma_{\lambda}(\widetilde 
     Z, \widetilde W) (\alpha P \beta^{*}) \text{ for all } \lambda   \notag \\
  & \quad \Rightarrow
      \alpha \Gamma(Z,W)(P) \beta^{*} = \Gamma(\widetilde 
     Z, \widetilde W) (\alpha P \beta^{*})
      \label{intertwine-check}
 \end{align}
as well as
\begin{align} 
  &  \sum_{i,j=1}^{n} V_{i}^{*} \Gamma_{\lambda}(Z^{(i)}, Z^{(j)})(R_{i}^{*} R_{j}) V_{j}
    \succeq 0  
    \Rightarrow  \sum_{i,j=1}^{n} V_{i}^{*} \Gamma(Z^{(i)}, Z^{(j)})(R_{i}^{*} R_{j}) V_{j}
    \succeq 0. 
    \label{kercp'''-check}
\end{align}

We now use the fact that weak-$*$ convergence on bounded sets is the 
same as pointwise weak-$*$ convergence as explained above. 
To verify \eqref{intertwine-check}, we fix a trace-class operator $X$ from 
$\cE^{\widetilde n}$ to $\cE^{\widetilde m}$. From the 
assumption in \eqref{intertwine-check} we then get 
$$
{\rm tr} \left( \alpha \Gamma_{\lambda}(Z,W)(P) \beta^{*} X \right) =
{\rm tr} \left( \Gamma_{\lambda}(\widetilde Z, \widetilde W)(\alpha P 
\beta^{*}) X \right) \text{ for all } \lambda.
$$
Since $\Gamma_{\lambda}(Z,W) \to \Gamma(Z,W)$ in the pointwise 
weak-$*$ topology for each fixed $Z,W \in \Omega$, we may take the limit 
with respect to the net $\lambda$ in this last expression to arrive at
$$
{\rm tr} \left( \alpha \Gamma(Z,W)(P) \beta^{*} X \right) =
{\rm tr} \left( \Gamma(\widetilde Z, \widetilde W)(\alpha P 
\beta^{*}) X \right).
$$
Since $X \in \cC_{1}(\cE^{\widetilde n}, \cE^{\widetilde m})$ is arbitrary, 
we may peel $X$ and the trace off to arrive at the desired conclusion 
in \eqref{intertwine-check}.

To verify \eqref{kercp'''-check}, we let $X$ be an arbitrary positive 
semidefinite trace-class operator in $\cL(\cE)$.  Then the hypothesis 
in \eqref{kercp'''-check} gives us
$$
{\rm tr}\left( \left(\sum_{i,j=1}^{n} V_{i}^{*} \Gamma_{\lambda}(Z^{(i)}, 
Z^{(j)})(R_{i}^{*} R_{j}) V_{j} \right) X \right) \ge 0.
$$
Again using the pointwise weak-$*$ convergence of 
$\Gamma_{\lambda}(Z,W)$ to $\Gamma(Z,W)$ for each $Z,W \in 
\Omega$, we may take the limit of this last expression to get
$$
{\rm tr}\left( \left(\sum_{i,j=1}^{n} V_{i}^{*} \Gamma(Z^{(i)}, 
Z^{(j)})(R_{i}^{*} R_{j}) V_{j} \right) X \right) \ge 0.
$$
As $X$ is an arbitrary positive semidefinite trace-class operator on 
$\cE$, we arrive at the conclusion of \eqref{kercp'''-check} as 
required.

It remains only to check that the kernel $\Gamma$ so constructed 
provides an Agler decomposition  \eqref{Aglerdecom'} for the limit 
kernel $K$.  Since $\{K_{N}(Z,W)\}$ 
is converging to $K(Z,W)$ in $\cL(\cL({\mathbb C}^{m}, {\mathbb C}^{n}),  
\cL(\cE^{m}, \cE^{n}))$-norm 
(where $Z \in \Omega_{n}$, $W \in \Omega_{m}$), it follows that 
the subnet $\{K_{\lambda}(Z,W)\}$ converges weak-$*$ (and hence 
pointwise weak-$*$ as well) to $K(Z,W)$.  This 
together with the 
pointwise weak-$*$ convergence of $\Gamma_{\lambda}$ and the fact 
that $\Gamma_{\lambda}$ provides an Agler decomposition for 
$K_{\lambda}$ for each $\lambda$ leads to the conclusion that indeed 
$\Gamma$ provides an Agler decomposition \eqref{Aglerdecom'}.  This completes the proof of part 
(1) of Lemma  \ref{L:cone}.
 \end{proof}
 
 \begin{proof}[Proof of Lemma \ref{L:cone} part (2)]
  Let $f \in \cT(\Omega; \cE_{\rm nc})$ and suppose $Z \in 
\Omega_{n}$, $W \in \Omega_{m}$, $P \in {\mathbb C}^{n \times m}$.  We must 
     produce a cp nc kernel $\Gamma$ lying in the class $\cT^{1}(\Omega; 
     \cL(\cE)_{\rm nc}, \cL(\cS)_{\rm nc})$ so that
 \begin{equation}  \label{posAglerdecom}
     f(Z) P f(W)^{*} = \Gamma(Z,W) \left(P \otimes I_{\cS} - Q(Z) (P \otimes 
     I_{\cR}) Q(W)^{*}\right).
 \end{equation}
 
 Toward this goal, we need to introduce some auxiliary operators and 
 spaces.  Let $\cR^{\otimes 2} : = \cR \otimes \cR$ be the 
 Hilbert-space tensor product of $\cR$ with itself and inductively 
 set $\cR^{ \otimes k} = \cR^{\otimes (k-1)} \otimes \cR$.  The {\em Fock 
 space} associated with $\cR$ is then defined to be the 
 Hilbert-space orthogonal direct sum
 $$
  {\mathbb F}(\cR) = \bigoplus_{k=0}^{\infty} \cR^{\otimes k}
 $$
 where we let $\cR^{\otimes 0}$ be the space of scalars ${\mathbb C}$.
 Fix a  linear functional $\ell$ on $\cS$ of unit norm.  For $Z \in \Omega_{n}$ define 
 $Q_{0}(Z) \in \cL(\cR^{n}, {\mathbb C}^{n})$ by
 \begin{equation}   \label{defQ0}
   Q_{0}(Z) = (\ell \otimes I_{n}) Q(Z).
 \end{equation}
  One can check that $Q_{0}$ is a nc function, i.e., $Q_{0} 
 \in \cT^{1}(\Omega; \cL(\cR, {\mathbb C})_{\rm nc})$.  For $W \in 
 \Omega_{m}$, the operator $Q_{0}(W)^{*} \in \cL({\mathbb C}^{m}, 
 \cR^{m})$.  We identify $\cR^{m}$ with the tensor product space 
 ${\mathbb C}^{m} \otimes \cR$ and then define an operator 
 $L_{Q_{0}(W)^{*}} \in \cL(\cR^{m}, \cR^{m} \otimes \cR)$ on an 
 elementary tensor 
 $c \otimes r$ ($c \in {\mathbb C}^{m}$ and $r \in \cR$) by 
 \begin{equation}  \label{LQ0W} 
 L_{Q_{0}(W)^{*}} \colon c \otimes r \mapsto Q_{0}(W)^{*}c \otimes r.
 \end{equation}
  We note the identifications
 $$
  \cR^{m} \otimes \cR \cong ({\mathbb C}^{m} \otimes \cR) \otimes \cR
  \cong {\mathbb C}^{m} \otimes (\cR \otimes \cR) \cong 
  (\cR \otimes \cR)^{m}.
 $$
 Then $L_{Q_{0}(W)^{*}}$, being of the form $Q_{0}(W)^{*} \otimes 
 I_{\cR}$, extends to a bounded operator from ${\mathbb C}^{m} 
 \otimes \cR \cong \cR^{m}$ into $(\cR \otimes \cR)^{m}$ with 
 $\|L_{Q_{0}(W)^{*}}\| = \| Q_{0}(W)^{*}\|$.  
 
 We next define the generalized power  
 $(L_{Q_{0}(W)^{*}})^{(k)} \colon {\mathbb C}^{m} \to (\cR^{\otimes 
 (k+1)})^{m}$ by
 $$
 (L_{Q_{0}(W)^{*}})^{(k)} =
 (Q_{0}(W)^{*} \otimes I_{\cR^{\otimes (k-1)}})
 \cdots (Q_{0}(W)^{*} \otimes I_{\cR}) Q_{0}(W)^{*}.
 $$
 Note that $\|  (L_{Q_{0}(W)^{*}})^{(k)} \| \le \| Q_{0}(W)^{*} 
 \|^{k}$ where $\| Q_{0}(W)^{*} \| < 1$ and hence
 \begin{equation}   \label{summable}
 \sum_{k=0}^{\infty} \|  (L_{Q_{0}(W)^{*}})^{(k)} \| < \infty.
 \end{equation}
 Then the adjoint  $ (L_{Q_{0}(W)^{*}})^{(k)*}$ maps $(\cR^{\otimes 
 k})^{m}$ into ${\mathbb C}^{m}$.
 
 For $Z \in \Omega_{n}$, we define $H(Z) \colon {\mathbb F}(\cR)^{n} 
 \to \cE^{n}$ by
 $$
 H(Z) = f(Z) \, \rm{row}_{k \ge 0}[ (L_{Q_{0}(Z)^{*}})^{(k)*}]
 $$
 Note that $H(Z)$ is bounded as an operator from ${\mathbb 
 F}(\cR)^{n}$ to $\cE^{n}$ due to the validity of the summability 
 condition \eqref{summable}.   From the fact that both $f$ and 
 $Q_{0}$ are nc functions,  one can check that $H$ so defined is a nc 
 function ($H \in \cT(\Omega; \cL({\mathbb F}(\cR), \cE)_{\rm nc})$). 
 For $W \in \Omega_{m}$,  then $H(W)^{*}$ is given by
 $$
 H(W)^{*} = {\rm col}_{k \ge 0}[ (L_{Q_{0}(W)^{*}})^{(k)}] f(W)^{*} 
 \colon \cE^{m} \to {\mathbb F}(\cR)^{m}.
 $$
 Define a representation from $\cL(\cS)$ to $\cL({\mathbb F}(\cR))$ by
 $$
  \pi(X) = (\ell X \ell^{*})I_{{\mathbb F}(\cR)}.
  $$
  Define $\Gamma$ from $\Omega \times \Omega$ to $\cL(\cL(\cS), 
  \cL(\cY))_{\rm nc}$ by
  \begin{equation}   \label{GammaKol}
  \Gamma(Z,W)(X) = H(Z) ({\rm id}_{{\mathbb C}^{n \times m}} \otimes 
  \pi(X)) H(W)^{*}.
  \end{equation}
  Then $\Gamma$ is a cp nc kernel ($\Gamma \in \widetilde 
  \cT^{1}(\Omega; \cL(\cE)_{\rm nc}, \cL(\cS)_{\rm nc})$ since \eqref{GammaKol} exhibits 
  a Kolmogorov decomposition \eqref{Koldecom} for $\Gamma$.  It 
  remains to check that $\Gamma$ provides an Agler decomposition 
  \eqref{posAglerdecom} for $D_{f,f}(Z,W)(P) = f(Z) P f(W)^{*}$.  To 
  this end we compute
  \begin{align*}
    &  \Gamma(Z,W)\left(P \otimes I_{\cS} - Q(Z) (P \otimes I_{\cR}) 
      Q(W)^{*}\right)  \\
      & = f(Z)\left( \sum_{k=0}^{\infty} \left(L_{Q_{0}(Z)^{*}} \right)^{(k)*} 
      \left( P \otimes I_{\cR^{\otimes k}} \right) 
      \left(L_{Q_{0}(W)^{*}} \right)^{(k)} \right)f(W)^{*} \\
  & \quad - f(Z)\left( \sum_{k=0}^{\infty} \left(L_{Q_{0}(Z)^{*}} \right)^{(k)*} 
     Q(Z) \left( P \otimes I_{\cR^{\otimes k}} \right) Q(W)^{*}
      \left(L_{Q_{0}(W)^{*}} \right)^{(k)} \right)f(W)^{*} \\
  & = f(Z)\left( \sum_{k=0}^{\infty} \left(L_{Q_{0}(Z)^{*}} \right)^{(k)*} 
      \left( P \otimes I_{\cR^{\otimes k}} \right) 
      \left(L_{Q_{0}(W)^{*}} \right)^{(k)} \right)f(W)^{*} \\
      &  - f(Z)\left( \sum_{k=0}^{\infty} \left(L_{Q_{0}(Z)^{*}} \right)^{(k)*} 
(Q(Z) \otimes I_{\cR^{\otimes k}}) \cdot  \right.\\
& \quad \quad \quad \quad \cdot \left. \left( P \otimes 
I_{\cR^{\otimes k+1}} \right) (Q(W)^{*} \otimes I_{\cR^{\otimes k}})
     \left(L_{Q_{0}(W)^{*}} \right)^{(k)} \right)f(W)^{*} \\
 & =  f(Z)\left( \sum_{k=0}^{\infty} \left(L_{Q_{0}(Z)^{*}} \right)^{(k)*} 
      \left( P \otimes I_{\cR^{\otimes k}} \right) 
      \left(L_{Q_{0}(W)^{*}} \right)^{(k)} \right)f(W)^{*} \\      
 & \quad
 -f(Z) \left( \sum_{k=1}^{\infty} \left(L_{Q_{0}(Z)^{*}} \right)^{(k)*} 
      \left( P \otimes I_{\cR^{\otimes k}} \right) 
      \left(L_{Q_{0}(W)^{*}} \right)^{(k)} \right)f(W)^{*}  \\
      & \quad = f(Z) P f(W)^{*}
   \end{align*}
   as wanted.
 \end{proof}

 \textbf{The cone $\cC_{\rho}$ for $\rho_{0} < \rho < 1$.}
 We shall actually need the following adjustment of the cone $\cC$.  
Let $\rho_{0}$ be defined as in \eqref{rho0} and let $\rho$ be any positive real number 
with $\rho_{0} < \rho < 1$.  

Define a subset $\cC_{\rho}$ of ${\mathfrak X}$ 
to consist of all kernels $K \in {\mathfrak X}$ such that 
there exists cp nc kernels 
\begin{align*}
& \Gamma_{0} \in \widetilde \cT^{1}(\Omega; \cL(\cE)_{\rm nc}, 
 \cL(\cS)_{\rm nc}), \quad \Gamma_{1} \text{ and } \Gamma_{2} \in 
 \widetilde \cT^{1}(\Omega; \cL(\cE)_{\rm nc}, \cL(\cE)_{\rm nc}),  \\
&    \Gamma_{3} \in \widetilde \cT^{1}(\Omega; \cL(\cE)_{\rm nc}, 
 {\mathbb C}_{\rm nc})
\end{align*}
 which induce a $\rho$-refined Agler decomposition for $K$:
 \begin{align} 
   K(Z,W)(P) =  &
 \Gamma_{0}(Z,W)\left( P \otimes I_{\cS}  - \frac{1}{\rho^{2}} Q(Z)(P 
 \otimes I_{\cR}) 
 Q(W)^{*}\right) \notag \\
 & \quad + \Gamma_{1}(Z,W) \left(P \otimes I_{\cE}  - (1 - \rho)^{2} 
 a(Z)(P \otimes I_{\cY})  a(W)^{*}\right)  \notag \\
& \quad  + \Gamma_{2}(Z,W)  \left(P \otimes I_{\cE}  - (1 - \rho)^{2} 
 b(Z)(P \otimes I_{\cU})  b(W)^{*}\right)  \notag \\
& \quad + \Gamma_{3}(Z,W)  \left(P   - (1 - \rho)^{2} 
 \chi(Z)(P \otimes I_{d})  \chi(W)^{*}\right) 
 \label{rhoAglerdecom'} 
 \end{align}
 
 Salient properties of the subset $\cC_{\rho}$ are summarized in the 
 next lemma.
 
 \begin{lemma}  \label{L:Crho}  Let $\rho_{0} < \rho < 1$ with 
     $\rho_{0}$ as in \eqref{rho0}.  Then
     the subset $ \cC_{\rho}$ of ${\mathfrak X}$ 
     has the following properties:
     \begin{enumerate}
\item  $\cC_{\rho}$ is a closed cone in ${\mathfrak X}$.

\item Suppose that $K \in {\mathfrak X}$ has the property that 
$K$ is in $\cC_{\rho}$ for 
all $\rho$ sufficiently close to $1$ with $ \rho_{0} < \rho < 1$. Then  $K \in 
\cC$.

\item  The positive kernels $D_{f,f}$ \eqref{Deltaff} are in 
$\cC_{\rho}$ for all $\rho$ with $\rho_{0} < \rho < 1$.
\end{enumerate}
\end{lemma}

\begin{proof}[Proof of Lemma  \ref{L:Crho} part 1:]
    That $\cC_{\rho}$ 
    is invariant under positive rescalings and taking of sums is elementary; we conclude that 
    indeed $\cC_{\rho}$  is a cone.
    
    That $\cC_{\rho}$ is closed in ${\mathfrak 
    X}$ once $\rho < 1$ is chosen sufficiently close to 1 can be proved 
    in much the same way used to show that $\cC$ is closed 
    (part (1) of Lemma \ref{L:cone}). By choosing $\rho < 1$ 
    sufficiently close to 1, we can guarantee that
    \begin{align}
&	I_{\cS^{n}} - \frac{1}{\rho^{2}} Q(Z) Q(Z)^{*} \succeq 
	\epsilon_{0}^{2} I_{\cS^{n}}, \notag \\
& I_{\cE^{n}} - (1 - \rho)^{2} a(Z) a(Z)^{*} \succeq \epsilon_{0}^{2} 
I_{\cE^{n}}, \notag  \\
&  I_{\cE^{n}} - (1 - \rho)^{2} b(Z) b(Z)^{*} \succeq \epsilon_{0}^{2} 
I_{\cE^{n}}, \notag \\
& I_{n} - (1 - \rho)^{2} \chi(Z) \chi(Z)^{*} \succeq \epsilon_{0}^{2} 
I_{n} 
\label{preest}
\end{align}
for all of the finitely many points $Z \in \Omega$ (where $n = n_{Z}$ 
is chosen so that $Z \in \Omega_{n}$).  
We then see that 
\begin{align}
 &  K_{N}(Z,Z)(I_{n})  = 
    \Gamma_{N,0}(Z,Z)(I_{\cS^{n}} - (1/\rho^{2}) Q(Z) Q(Z)^{*})   
    \notag\\
 & \quad  \quad \quad \quad  
 + \Gamma_{N,1}(Z,Z) (I_{\cE^{n}} - (1 - \rho)^{2} a(Z) a(Z)^{*})  
 \notag\\
 & \quad  \quad \quad \quad
 + \Gamma_{N,2}(Z,Z)  (I_{\cE^{n}} - (1 - \rho)^{2} b(Z) b(Z)^{*}) \\
& \quad \quad \quad \quad + \Gamma_{N,3}(Z,Z) (I_{n} - (1 - \rho)^{2} \chi(Z) \chi(Z)^{*}) 
\succeq \notag \\
 & \epsilon_{0}^{2} \left( \Gamma_{N,0}(Z,Z)(I_{\cS^{n}})
 + \Gamma_{N,1}(Z,Z)(I_{\cE^{n}}) + \Gamma_{N,2}(Z,Z)(I_{\cE^{n}}) + 
 \Gamma_{N,3}(Z,Z)(I_{n}) \right)
 \label{est1'}
 \end{align}
 and hence each of the quantities 
 \begin{align}
& \| \Gamma_{N,0}(Z,Z)(I_{\cS^{n}})\|_{\cL(\cE^{n})}, \quad 
   \| \Gamma_{N,1}(Z,Z)(I_{\cE^{n}})\|_{\cL(\cE^{n})}, \notag \\
 & \| \Gamma_{N,2}(Z,Z)(I_{\cE^{n}})\|_{\cL(\cE^{n})}, \quad 
     \| \Gamma_{N,3}(Z,Z)(I_{n})\|_{\cL(\cE^{n})} 
     \label{kernels}
 \end{align}
 is bounded above by  $\frac{1}{\epsilon_{0}^{2}} \| K_{N}(Z,Z) (I_{n}) 
 \|_{\cL(\cE)}$.  As this last quantity is uniformly bounded with 
 respect to $N$, each of the cp nc kernels in the list is uniformly bounded in 
the appropriate operator norm and hence each sequence 
$\{\Gamma_{N,k}\}$  has a weak-$*$ convergent subnet $\{ 
\Gamma_{\alpha, k}\}$ converging to some cp nc kernel $\Gamma_{k}$  
($k = 0,1,2,3$).  The fact that the foursome $\{ \Gamma_{N,k} \colon 
k=0,1,2,3\}$ provides a $\rho$-refined Agler decomposition 
\eqref{rhoAglerdecom'} for $K_{N}$ implies that the limiting foursome 
$\{ \Gamma_{k} \colon k=0,1,2,3\}$ provides a $\rho$-refined Agler 
decomposition for the limiting kernel $K$ now proceeds as in the 
proof of Lemma \ref{L:cone} part (1), and hence the limiting kernel 
$K$ still has a $\rho$-refined Agler decomposition as wanted.
\end{proof}
    
 \begin{proof}[Proof of Lemma  \ref{L:Crho} part 2:]
 Suppose that $K \in {\mathfrak X}$ 
 is in $\cC_{\rho}$ for 
 all $\rho < 1$ subject to $\rho_{00} < \rho < 1$ for some 
 $\rho_{00}$ with $\rho_{0} \le \rho_{00} < 1$.  Hence for each such 
 $\rho$ there are cp nc kernels $\Gamma_{\rho,0}$, $\Gamma_{\rho,1}$, 
 $\Gamma_{\rho,2}$, $\Gamma_{\rho,3}$ so that $K$ has 
 a decomposition as in \eqref{rhoAglerdecom'} (with $\Gamma_{\rho,k}$ 
 in place of $\Gamma_{k}$ for $k = 0,1,2,3$).  
 The estimate \eqref{est1'} 
 is uniform with respect to $\rho$ for $\rho_{0} < \rho <1$ since the 
 estimates \eqref{preest} are uniform in $\rho$ as $\rho$ approaches 
 1 once $\rho$ is sufficiently close to 1.
 By again following 
 the proof of Lemma \ref{L:cone} part (1), we can find a subnet 
 $\rho_{\alpha}$ of an increasing sequence 
 $\{\rho_{N}\}_{N \in {\mathbb N}}$ of positive numbers with 
 limit equal to $1$ so that each
 $\Gamma_{\rho_{\alpha},k}(Z,W)$ converges in the appropriate 
 operator BW-topology to an operator 
 $\Gamma_{k}(Z,W)$ for each $Z \in 
 \Omega_{n}$, $W \in \Omega_{m}$ for each $k=0,1,2,3$.  Since the convergence is with 
 respect to the pointwise weak-$*$ topology, one can check as in the 
 proof of Lemma \ref{L:cone} part (1) that the fact that the kernel 
 $K$ has a $\rho$-refined Agler decomposition provided by the 
 foursome $\{\Gamma_{\rho,0},\, \Gamma_{\rho,1}, \, 
 \Gamma_{\rho,2},\, \Gamma_{\rho,3}\}$ implies, upon taking a limit as 
 $\rho_{\alpha} \to 1$, that $K$ also has the limiting representation
 \begin{align*}
    &  K(Z,W)(P) =   \Gamma_{0}(Z,W)( P \otimes I_{\cS} - Q(Z) (P \otimes 
 I_{\cR}) Q(W)^{*}) \\
 & \quad + \Gamma_{1}(Z,W) (P \otimes I_{\cE}) + 
 \Gamma_{2}(Z,W)(P \otimes I_{\cE}) + \Gamma_{3}(Z,W)(P).
 \end{align*}
 The last three terms on the right side of the equality are all cp nc 
 kernels and hence have a standard Agler decomposition by Lemma 
 \ref{L:cone} part (3) while the first term is in the form of an 
 Agler decomposition.  Thus each term is in $\cC$ and, as $\cC$ is a 
 cone, the sum is again in $\cC$, so $K$ has a standard Agler 
 decomposition as claimed.  This concludes the proof of
 part (2) of Lemma  \ref{L:Crho} follows.
\end{proof}
     
\begin{proof}[Proof of Lemma  \ref{L:Crho} part 3:]
  Apply part (2) of Lemma \ref{L:cone} with 
  $\frac{1}{\rho} Q$ in place of $Q$ to see that there is a cp nc 
  kernel $\Gamma_{0}$ so that
  $$
  D_{f,f}(Z,W)(P) = \Gamma_{0}(Z,Z)\left( I_{\cS} \otimes P - 
  \frac{1}{\rho^{2}} Q(Z) (I_{\cR} \otimes P) Q(W)^{*} \right).
  $$
  Then $D_{f,f}$ has a decomposition  of the form 
  \eqref{rhoAglerdecom'} (with $\Gamma_{1}$, $\Gamma_{2}$ and 
  $\Gamma_{3}$ all taken equal to zero) as required.
 \end{proof}

 \noindent
 \textbf{The cone separation argument.}
 We now have all the preliminaries needed to complete the proof of 
 (1$^{\prime}$) 
 $\Rightarrow$ (2) in Theorem \ref{T:ncInt} for Case 1. In this part 
 we are assuming that $a,b$ are nc functions on $\Omega_{\rm nc, 
 full} \cap {\mathbb D}_{Q}$.
 To show that $a,b$ has an Agler decomposition \eqref{Aglerdecom} on 
 $\Omega$, it  suffices to show that the kernel
  \begin{equation}   \label{Kab}
 K_{a,b}(Z,W)(P): = a(Z) (P \otimes I_{\cY}) a(W)^{*} - b(Z) (P 
 \otimes I_{\cU}) b(W)^{*}
 \end{equation}
 is in the cone $\cC_{\rho}$ for all $\rho$ subject to $\rho_{0} < 
 \rho <1$ for some $\rho_{00} \le \rho_{0}$.
 Since the cone $\cC_{\rho}$ is closed (by Lemma \ref{L:Crho} part 
 (1)), by the contrapositive formulation of the Hahn-Banach separation 
 theorem adapted to the case of 
point/cone pair (see \cite[Theorem 3.4 part (b)]{Rudin}), to show 
that $K_{a,b}$ is in the cone $\cC_{\rho}$ it suffices to show:
{\em if ${\mathbb L}$ is any continuous linear functional on the normed linear 
space of kernels ${\mathfrak X}$  such that 
\begin{equation}   \label{separation}
 {\rm Re}\, {\mathbb L}(K) \ge 0 \text{ for all } K \in \cC_{\rho},
 \end{equation}
 then also}
\begin{equation} \label{goal}
{\rm Re}\, {\mathbb L}(K_{a,b}) \ge 0.
\end{equation}
 
 In general for $K \in {\mathfrak X}$ we define $\widehat K \in 
 {\mathfrak X}$ by
 $$
   \widehat K(Z,W)(P) = \left( K(W, Z)(P^{*}) \right)^{*}.
 $$
 Let us say that the kernel $K \in {\mathfrak X}$ is 
 \textbf{Hermitian} if it is the case that $K = \widehat K$.
 Given a continuous linear functional ${\mathbb L}$ on ${\mathfrak X}$, 
 define ${\mathbb L}_{1} \colon {\mathfrak X} \to {\mathbb C}$ by
 $$
{\mathbb L}_{1}(K) =\frac{1}{2}({\mathbb L}(K) + \overline{{\mathbb L}(\widehat K)}).
 $$
 Note that if $K$ is a Hermitian kernel, then
 $$ 
 {\mathbb L}_{1}(K) = {\rm Re}\, {\mathbb L}(K).
 $$
 In particular, if $K$ is cp, then $K$ is Hermitian and 
 ${\mathbb L}_{1}(K) = {\rm Re} \, {\mathbb L}(K)$.

 Now assume that ${\mathbb L}$ is chosen so that \eqref{separation} 
 holds.
 Let $\cH^{\circ}_{{\mathbb L}}$ be the vector space 
 $\cT(\Omega; \cE_{\rm nc}) $ of all 
 $\cE$-valued nc functions on $\Omega$.   Introduce a 
 sesquilinear form on $\cH^{\circ}_{\mathbb L}$ by
 \begin{equation}   \label{innerprod}
   \langle f, \, g \rangle_{\cH^{\circ}_{\mathbb L}} = 
   {\mathbb L}_{1}(D_{f,g})
 \end{equation}
 where $D_{f,g} \in {\mathfrak X}$ is given by
 $$
   D_{f,g}(Z,W)(P) = f(Z)\, P \, g(W)^{*}
 $$
 for $Z \in \Omega_{n}$, $W \in \Omega_{m}$, 
 $P \in {\mathbb C}^{n\times m}$.
 We have observed in part (3) of Lemma \ref{L:Crho} that cp 
 kernels $D_{f,f}$ are in $\cC_{\rho}$ and hence 
 ${\rm Re}\, {\mathbb L}(D_{f,f}) \ge 0$ by the construction  
 \eqref{separation}.  But for any $f$, $D_{f,f} = \widehat
 D_{f,f}$, so by the above remarks we have $0 \le {\rm Re}\, {\mathbb 
 L}(D_{f,f}) = {\mathbb L}_{1}(D_{f,f})$.  We conclude that 
 the inner product is positive semidefinite.  By modding out by any 
 functions having zero self inner product and considering 
 equivalence classes, we get a finite dimensional Hilbert space which 
 we denote by $\cH_{{\mathbb L}}$.  
 
For $f \in \cH^{\circ}_{\mathbb L}$, we let $[f]$ be the equivalence 
class of $f$ in $\cH_{\mathbb L}$.   Since $\Omega$ is  
 finite and $\cE$ is finite-dimensional, we know by Lemma 
 \ref{P:finitetype} that $\cH^{\circ}_{\mathbb L} = \cT(\Omega; 
 \cE_{\rm nc})$ is finite-dimensional.   Denote by $\cB: = \{f_{1}, 
 \dots, f_{K}\}$ any basis for $\cH^{\circ}_{\mathbb L}$.  Then 
 certainly the collection of equivalence classes $\{ [f_{1}], \cdots, 
 [f_{K}]\}$ is a spanning set for $\cH_{\mathbb L}$.  By standard 
 Linear Algebra, we can choose a subset, denoted after possible reindexing as 
 $\cB': = \{[f_{1}], \dots, [f_{K'}] \}$ for some $K' \le K$, as a basis for 
 $\cH_{\mathbb L}$.  Suppose that $f = \sum_{k=1}^{K'}c_{k} f_{k}$ has 
 zero self-inner product in $\cH^{\circ}_{\mathbb L}$, so
 $$
 \left[\sum_{k=1}^{K'} c_{k}  f_{k} \right] = \sum_{k=1}^{K'} c_{k} [f_{k}] = 0.
 $$
 Since $\cB'$ is linearly independent in $\cH_{\mathbb L}$, it 
 follows that $c_{k} = 0$ for $1 \le k \le K'$, from which it follows 
 that $f = \sum_{k=1}^{K'} c_{k} f_{k} = 0$ in $\cT(\Omega; \cE_{\rm nc})$.
 We conclude that the ${\mathbb L}$-inner product \eqref{innerprod} 
 is in fact positive definite when restricted to
 $\cH^{\circ \prime}_{\mathbb L} : = {\rm span} \{f_{1}, \dots, 
 f_{K'}\}$.  We therefore may view the space $\cH_{\mathbb L}$ as 
 the space of bona fide functions $\cH_{\mathbb L} \cong 
 \cH^{\circ'}_{\mathbb L}$ with inner product given by 
 \eqref{innerprod} and with orthonormal basis given by
 \begin{equation}   \label{HLbasis}
     \cB' = \{ f_{1}, \dots, f_{K'}\}.
 \end{equation}
 With this convention in force, it follows that 
 the point-evaluation maps 
 \begin{equation}  \label{pointeval}
  \ev_{Z} \colon f \mapsto f(Z)
 \end{equation}
 are well-defined as operators from $\cH_{\mathbb L}$ into $\cE^{n}$ 
 for each $Z \in \Omega_{n}$.  As $\cH_{\mathbb L}$ is also 
 finite-dimensional, it follows that each such map $\mathbf{ev}_{Z}$
 is bounded as an operator from $\cH_{\mathbb L}$ into $\cE^{n}$ as 
 well.

 For $\cX$ an arbitrary separable coefficient Hilbert space, define a space 
 $\cH_{{\mathbb L}, \cX}$ to be the space of nc functions 
 $\cT(\Omega ; \cL(\cX, \cE))$ with inner product given by
 \begin{equation} \label{HLXinnerprod}
     \langle \boldf, \bg \rangle_{\cH_{{\mathbb L}, \cX}} = {\mathbb 
     L}_{1}(D_{\boldf,\bg})
  \end{equation}
  where $D_{\boldf,\bg}$ is the kernel in $\widetilde \cT^{1}(\Omega; 
  \cL(\cE)_{\rm nc}, {\mathbb C}_{\rm nc})$ given by
  \begin{equation}   \label{Dfg-ker}
  D_{\boldf,\bg}(Z,W)(P) = \boldf(Z) (P \otimes I_{\cX}) \bg(W)^{*}.
  \end{equation}
  For $x^{*} \in \cX^{*} = \cL(\cX, {\mathbb C})$ and $Z \in \Omega$, 
  define a function $\bx^{*} \in \cT(\Omega; \cL(\cX, {\mathbb C})_{\rm nc})$ as in 
  Example \ref{E:linear} part (b) by
  $$
   \bx^{*}(Z) = {\rm id}_{{\mathbb C}^{n}} \otimes x^{*} 
   \text{ if } Z \in \Omega_{n}.
 $$
 For $f \in \cT(\Omega ; \cL({\mathbb C}, \cE)_{\rm nc})$ and $x^{*} \in 
 \cX^{*}$, it then follows that $Z \mapsto \boldf(Z) := f(Z) \cdot
 \bx^{*}(Z)$, as the pointwise composition of nc functions,
 is itself a nc function in $\cT(\Omega; \cL(\cX, \cE)_{\rm nc})$;
 this is explained in  \cite[Section 4.1]{BMV1} in the context of Schur multipliers
 or can be easily verified directly.  Furthermore, if we have two 
 elements $f, g$ of $\cH_{{\mathbb L}} = \cT(\Omega ; 
 \cL({\mathbb C}, \cE)_{\rm nc})$ and two elements $x^{*},y^{*}$ of $\cX^{*}$ and 
 set $\boldf = f \circ \bx^{*}$ and $\bg = g \circ \by^{*}$ as above, then $\boldf$ and 
 $\bg$ are in $\cH_{{\mathbb L}, \cX}$  with inner product given by
 $$
     \langle \boldf, \bg \rangle_{\cH_{{\mathbb L}, \cX}}  =
     {\mathbb L}_{1}\left( D_{\boldf, \bg} \right)
  $$
 where, for $Z \in \Omega_{n}$ and $W \in \Omega_{m}$, $D_{\boldf, 
 \bg}$ is given by
 \begin{align*}
 D_{\boldf, \bg}(Z,W)(P) & =
 f(Z) \left({\rm id}_{{\mathbb C}^{n}} \otimes x^{*} \right) \left( P 
 \otimes I_{\cX} \right) \left( {\rm id}_{{\mathbb C}^{m}} \otimes y 
 \right) g(W)^{*} \\
 & = f(Z) P g(W)^{*} \otimes x^{*} y \\
 & = (f(Z) P g(W)^{*}) \cdot x^{*}y \text{ (since $x^{*}y \in 
 {\mathbb C}$).}
 \end{align*}
 We conclude that 
 $$
  \langle \boldf, \bg \rangle_{\cH_{{\mathbb L}, \cX}} =
  \langle f, g \rangle_{\cH_{{\mathbb L}}} \cdot \langle x^{*}, y^{*} 
  \rangle_{\cX^{*}}.
$$
It follows that the map
$$ 
\iota \colon f \otimes x^{*} \mapsto f \cdot \bx^{*}
$$
extends to an isometry from the Hilbert-space tensor product 
$\cH_{{\mathbb L}} \otimes \cX^{*}$ into $\cH_{{\mathbb L}, \cX}$.

To see that this isometry is onto, we proceed as follows.
Let $\{x_{\beta} \colon \beta \in {\mathfrak B} \}$ be an orthonormal 
basis for $\cX$ (so ${\mathfrak B} = \{1, \dots, \dim \cX\}$ 
in case $\cX$ is finite-dimensional, and ${\mathfrak B} = {\mathbb 
N}$ (the natural numbers) otherwise), and let $\boldf$ be 
an arbitrary element of $\cH_{{\mathbb L}, \cX}$.  We let 
$\bx_{\beta}$ be the  nc function in $\cT(\Omega, \cL({\mathbb C}, 
\cX)_{\rm nc})$ given by
$$
  \bx_{\beta}(Z) = {\rm id}_{{\mathbb C}^{n}} \otimes x_\beta \text{ if } Z 
  \in \Omega_{n}
$$
(again as in Example \ref{E:linear} part (b)).  Then the pointwise 
composition $f_{\beta}: = \boldf \cdot \bx_{\beta}$ is a nc function in 
$\cH_{{\mathbb L}}$. If we apply the construction of the previous 
paragraph to the pair $f_{\beta}$ and $x_{\beta}$, we get the 
elementary-tensor nc function 
$$
f_{\beta} \cdot \bx_{\beta}^{*} = (\boldf \cdot \bx_{\beta})\cdot \bx_{\beta}^{*}
$$
in $\iota (\cH_{{\mathbb L}} \otimes \cX^{*})$. We claim that
\begin{equation}  \label{claim}
   \boldf = \sum_{\beta \in {\mathfrak B}} (\boldf \cdot \bx_{\beta}) \cdot 
   \bx_{\beta}^{*}
\end{equation}
with convergence of the series in $\cH_{{\mathbb L},\cX}$-norm in 
case $\dim \cX = \infty$.
It then follows that the span of the images of elementary tensors 
$\iota(f \otimes \bx^{*})$ is dense in $\cH_{{\mathbb L}, \cX}$ and 
hence that the map $\iota$ extends to a unitary identification of 
$\cH_{{\mathbb L}} \otimes \cX^{*}$ with $\cH_{{\mathbb L}, \cX}$.

To verify the claim \eqref{claim}, note first that since $\{ 
\bx_{\beta}^{*} \colon \beta \in {\mathfrak B}\}$ is an orthonormal 
set in $\cX^{*}$, it follows that, for any $f,g \in \cH_{{\mathbb 
L}}$, 
$$
  \langle f \cdot \bx_{\beta}^{*},\,  g \cdot \bx_{\beta'}^{*} 
  \rangle_{\cH_{{\mathbb L}, \cX}} = \langle f, g 
  \rangle_{\cH_{{\mathbb L}}} \cdot \delta_{\beta, \beta'}
$$
(where $\delta_{\beta, \beta'}$ is the Kronecker delta).  Hence
$$
\langle (\boldf \cdot \bx_{\beta}) \cdot \bx_{\beta}^{*},\,
(\boldf \cdot \bx_{\beta'}) \cdot \bx_{\beta'}^{*} 
\rangle_{\cH_{{\mathbb L}, \cX}} = \delta_{\beta, \beta'} {\mathbb 
L}_{1}\left( D_{\boldf \cdot \bx_{\beta} ,
\boldf \cdot \bx_{\beta} } 
\right)
$$
where
\begin{align*}
    D_{\boldf \cdot \bx_{\beta} ,\, \boldf\cdot 
    \bx_{\beta} }(Z,W)(P)
& = \boldf(Z) \left( P \otimes x_{\beta} x_{\beta}^{*} \right) 
f(W)^{*} \\
& = D_{\boldf,\,  \boldf \cdot \bx_{\beta} \cdot 
\bx_{\beta}^{*}}(Z,W)(P)
 \end{align*}
 and hence
 $$
 \langle (\boldf \cdot \bx_{\beta}) \cdot \bx_{\beta}^{*}, \,
 (\boldf \cdot \bx_{\beta'}) \cdot \bx_{\beta'}^{*}  
 \rangle_{\cH_{{\mathbb L}, \cX}} = \delta_{\beta, \beta'} \langle 
 \boldf,\, (\boldf \cdot \bx_{\beta}) \cdot \bx_{\beta}^{*} 
 \rangle_{\cH_{{\mathbb L}}, \cX}.
 $$
 Let $K = \dim \cX$ if $\dim \cX$ is finite and $K \in {\mathbb N}$ 
 arbitrary otherwise.  Then it follows that
 $$
 \left\| \boldf - \sum_{\beta = 1}^{K} \boldf \cdot \bx_{\beta} \cdot 
 \bx_{\beta}^{*} \right\|^{2}_{\cH_{{\mathbb L}, \cX}} =
 \| f \|^{2}_{\cH_{{\mathbb L}, \cX}} - \left \langle 
 \boldf,  \, \sum_{\beta = 1}^{K} \boldf \cdot \bx_{\beta} \cdot \bx_{\beta}^{*} 
 \right\rangle_{\cH_{{\mathbb L}, \cX}}.
 $$
 Note next that
$$
     D_{\boldf,\, \sum_{\beta = 1}^{K} \boldf \cdot \bx_{\beta} \cdot 
     \bx_{\beta}^{*}}(Z,W)(P)  =
 \boldf(Z) \left( P \otimes \sum_{\beta = 1}^{K} x_{\beta} 
 x_{\beta}^{*} \right) \boldf(W)^{*}.
 $$
 In case $K = \dim \cX < \infty$, we have $P \otimes \sum_{\beta=1}^{K} 
 x_{\beta} x_{\beta}^{*} = P \otimes I_{\cX}$ and we conclude that 
 $ D_{\boldf,\, \sum_{\beta = 1}^{K} \boldf \cdot \bx_{\beta} \cdot 
     \bx_{\beta}^{*}} = D_{\boldf,\, \boldf}$ from which the claim 
    \eqref{claim} follows as wanted.  In case $\dim \cX = \infty$,
    we use that $\sum_{\beta = 1}^{\infty} x_{\beta} 
 x_{\beta}^{*}$ converges strongly in $\cL(\cX)$ to $I_{\cX}$ from 
 which it follows that 
 $$
 \lim_{K \to \infty} D_{\boldf,\, \sum_{\beta = 1}^{K} \boldf \cdot \bx_{\beta} \cdot 
     \bx_{\beta}^{*}} = D_{\boldf, \, \boldf}
  $$
  with convergence in the BW-topology on ${\mathfrak X}$.  Due to the 
  continuity of ${\mathbb L}_{1}$ with respect to the BW-topology on 
  ${\mathfrak X}$, it then follows that
  $$
   \lim_{K \to \infty} \| \boldf - \sum_{\beta = 1}^{K} (\boldf \cdot \bx_{\beta}) \cdot 
    \bx_{\beta}^{*} \|^{2}_{\cH_{{\mathbb L}, \cX}}  = 0
  $$
  and the claim \eqref{claim} follows in this case as well. In the 
  sequel we shall freely use the resulting identification $\iota$ 
  between $\cH_{{\mathbb L}} \otimes \cX^{*}$ and $\cH_{{\mathbb L}, 
  \cX}$ without explicit mention of the map $\iota$.

 For $Z \in \Omega$, we let $n_{Z}$ denote the natural number $n$ so that 
 $Z \in \Omega_{n}$.  We next define a Hilbert space $\bcH$ by
 $$
 \bcH = \bigoplus_{Z \in \Omega} \cL({\mathbb C}^{n_{z}}, \cE^{n_{z}})
 $$
 Here $\cL({\mathbb C}^{n_{z}}, \cE^{n_{z}})$ is given the 
 Hilbert-Schmidt operator norm and the direct sum is a Hilbert-space 
 direct sum.   More generally, let $\bcH_{\cX}$ be the Hilbert space
 $$
   \bcH_{\cX} = \bigoplus_{Z \in \Omega} \cL(\cX^{n_{z}}, \cE^{n_{z}})
 $$
  As $\cE$ is finite-dimensional, any operator in $\cL(\cX^{n_{z}}, 
  \cE^{n_{z}})$ has finite Hilbert-Schmidt norm and we can again view 
  $\bcH_{\cX}$ as a Hilbert space.  The relation between $\bcH$ and 
  $\bcH_{\cX}$ is analogous to that derived above between 
  $\cH_{{\mathbb L}}$ and $\cH_{{\mathbb L}, \cX}$.  Specifically, if
  $h = \bigoplus_{Z \in \Omega} h_{Z} \in \bcH$ and $x^{*} \in 
  \cX^{*}$, then
  $$
  h \cdot \bx^{*} : = \bigoplus_{Z \in \Omega} h_{Z} ({\rm 
  id}_{{\mathbb C}^{n_{z}}} \otimes x^{*}) \in 
  \bcH_{\cX}.
  $$
  Moreover, given two pairs $h,k \in \bcH$ and $x^{*}, y^{*} \in 
  \cX^{*}$, 
  \begin{align*}
      \langle h \cdot \bx^{*}, k \cdot \by^{*} \rangle_{\bcH_{\cX}} & 
 =  \sum_{Z \in \Omega} {\rm tr}\left(({\rm id}_{{\mathbb C}^{n_{z}}} 
 \otimes y) k_{Z}^{*} h_{Z}({\rm id}_{{\mathbb C}^{n_{z}}} \otimes x^{*})\right) \\
 & = \sum_{Z \in \Omega} {\rm tr}\left( k_{Z}^{*} h_{Z} ({\rm 
 id}_{{\mathbb C}^{n_{z}}} \otimes x^{*} y) \right) \\
& = \left( \sum_{Z \in \Omega} {\rm tr}(k_{Z}^{*} h_{Z}) \right) \cdot 
      x^{*}y \text{ (since $x^{*} y \in {\mathbb C}$)} \\
      & = \langle h, \, k \rangle_{\bcH} \cdot \langle x^{*}, y^{*} 
      \rangle_{\cX^{*}}.
   \end{align*}
  Thus the map $\biota$ defined on elementary tensors by
  $$
  \biota \colon h \otimes x^{*} \mapsto h \cdot \bx^{*}
  $$
  extends to an isometry from the Hilbert-space tensor product $\bcH 
  \otimes \cX^{*}$ into $\bcH_{\cX}$.  To see that $\biota$ is onto, 
  we again work with an orthonormal basis $\{ x_{\beta} \colon \beta 
  \in {\mathfrak B}\}$ for $\cX$.  Given any $\bh = \bigoplus_{Z \in 
  \Omega} \bh_{Z} \in \bcH_{\cX}$, then $h_{\beta} = \bh \cdot 
  \bx_{\beta} : = \bigoplus_{Z \in \Omega} \bh(Z) \left( {\rm 
  id}_{{\mathbb C}^{n_{z}}} \otimes x_{\beta} \right)$ is in $\bcH$.
  We claim that we can recover $\bh$ as
  \begin{equation}   \label{claim2}
      \bh = \sum_{\beta \in {\mathfrak B}} h_{\beta} \cdot 
      \bx_{\beta}^{*} = \sum_{\beta \in {\mathfrak B}} (\bh \cdot 
      \bx_{\beta}) \cdot \bx_{\beta}^{*}
 \end{equation}
 with convergence of the series in the norm of $\bcH_{\cX}$ in the 
 case that $\dim \cX = \infty$.  As a consequence it then follows 
 that the span of images of elementary tensors $\biota(h \otimes 
 x^{*})$ ($h \in \bcH$ and $x^{*} \in \cX^{*}$) is dense in 
 $\bcH_{\cX}$ and hence $\biota $ is onto.  The claim \eqref{claim2} 
 in turn can be seen as an immediate consequence of the identity
 $$
   \sum_{\beta \in {\mathfrak B}} x_{\beta} x_{\beta}^{*} = I_{\cX}
 $$
 (with convergence n the strong topology of $\cL(\cX)$ in case $\dim 
 \cX = \infty$).  We hence identify $\bcH_{\cX}$ with the 
 Hilbert-space tensor product $\bcH \otimes \cX^{*}$ without explicit 
 mention of the identification map $\biota$.
 
 We shall also find it convenient to have a version of these spaces 
 where the target space in the $Z$-slice is simply ${\mathbb C}$ rather than 
 $\cE^{n_{z}}$.  Toward this end, note that we can choose an 
 orthonormal basis $\{ e_{1}, \dots, e_{\dim \cE}\}$ to represent any 
 operator $T \in \cL({\mathbb C}^{n_{z}}, \cE^{n_{z}} \cong {\mathbb 
 C}^{n_{z} \cdot (\dim \cE)})$ as an 
 $k_{Z} \times n_{Z}$ matrix over ${\mathbb C}$,
 where we set 
 \begin{equation}  \label{kZ}
  k_{Z} = (\dim \cE) \cdot n_{Z}
 \end{equation}
  (with $k_{Z} \cdot n_{Z}$ entries) while an operator 
 $\widetilde T \in \cL({\mathbb C}^{k_{z} \cdot n_{z}}, 
 {\mathbb C})$ is represented by a $1 \times 
 k_{Z} \cdot n_{Z}$ matrix over $\cE$ (again with $k_{Z} \cdot n_{Z}$ entries).
 In general, to map $\cL({\mathbb C}^{n}, {\mathbb C}^{k}) \cong 
 {\mathbb C}^{k \times n}$ bijectively to $\cL({\mathbb C}^{k \cdot n 
 }, {\mathbb C}) \cong {\mathbb C}^{1 \times k \cdot n}$, 
 we introduce the operator $\rowvec_{k}$ ($\rowvec$ for {\em row-vectorization})  which 
 reorganizes a $k \times n$ matrix into a $1 \times (k \cdot n)$ matrix 
 as follows:  if $T = [t_{ij}] = \sbm{ T_{1} \\ 
 \vdots \\ T_{k} }$ where $T_{i} = \begin{bmatrix} 
 t_{i,1} & \cdots & t_{i, n} \end{bmatrix}$ is the $i$-th row of 
 $T$ for $1 \le i \le k$, we 
 define 
 $$
 \rowvec_{k} \colon
 \cL({\mathbb C}^{n}, {\mathbb C}^{k}) \cong 
 {\mathbb C}^{k \times n} \to \cL({\mathbb 
 C}^{k \cdot n}, {\mathbb C}) \cong {\mathbb C}^{1 \times k 
 \cdot n}
 $$
 by
$$ \rowvec_{k} \colon T = \begin{bmatrix} T_{1} \\ \vdots \\ 
 T_{k} \end{bmatrix} \mapsto \begin{bmatrix} T_{1} & \cdots & 
 T_{k} \end{bmatrix}.
 $$
 We now introduce yet another Hilbert space $\widetilde \bcH$ by
 \begin{equation}  \label{tbcH}
   \widetilde \bcH = \bigoplus_{Z \in \Omega} \cL({\mathbb 
   C}^{k_{z} \cdot n_{z}}, {\mathbb C}) \cong \bigoplus_{Z \in \Omega} 
   {\mathbb C}^{1 \times k_{z} \cdot n_{z}}
 \end{equation}
 ($k_{Z}$ as in \eqref{kZ})
 and define an identification map $\bi \colon \bcH \to \widetilde 
 \bcH$ by
 $$
   \bi = [{\rm row}_{Z \in \Omega}\,  \rowvec_{k_{z}} ] \colon
   \bigoplus_{Z \in \Omega} h_{Z} \mapsto {\rm row}_{Z \in \Omega} 
  [ \rowvec_{k_{z}}(h_{Z})].
 $$
 For $\cX$ a coefficient Hilbert space as above, we also introduce 
 the companion Hilbert space
 \begin{equation}  \label{tbcHX}
 \widetilde \bcH_{\cX} = \bigoplus_{Z \in \Omega} \cL(\cX^{k_{z} 
 \cdot n_{z}}, {\mathbb C}) \cong \widetilde \cH \otimes \cX^{*}
 \end{equation}
 and extend the identification map $\bi \colon \bcH \to \widetilde 
 \bcH$ to an identification map between $\bcH_{\cX}$ and $\widetilde 
 \bcH_{\cX}$ by
 $$
 \bi_{\cX}  = \bi \otimes I_{\cX^{*}} \colon \bcH_{\cX} \cong \bcH 
 \otimes \cX^{*} \to \widetilde \bcH_{\cX} \cong \widetilde \bcH 
 \otimes \cX^{*}.
 $$
 Explicitly, the operator $\bi_{\cX}$ is constructed as follows.  
 We use an orthonormal basis for $\cE$ to view an element 
 $T \in \cL(\cX^{n_{z}}, \cE^{n_{z}})$  as an operator in 
 $\cL(\cX^{n_{z}}, {\mathbb C}^{k_{z}}) \cong (\cL(X, {\mathbb 
 C})^{k_{z} \times n_{z}})$.  We then apply the operator 
 $\rowvec_{k_{z}}$ to this $k_{Z} \times n_{Z}$ matrix over $\cL(\cX, 
 {\mathbb C})$ to get a $1 \times (k_{Z} \cdot n_{Z}) $ matrix over 
 $\cL(\cX, {\mathbb C})$.  This last matrix in turn can be 
 interpreted as an operator in $\cL(\cX^{k_{z} \times n_{z}}, 
 {\mathbb C})$. Concretely we view the direct sum in \eqref{tbcH} and 
 in \eqref{tbcHX} as a row direct sum; then we view elements of 
 $\widetilde \bcH$ and of $\widetilde \bcH_{\cX}$ as long row vectors:
 \begin{equation}   \label{tbcH-inden}  
  \widetilde \bcH  = {\mathbb C}^{1 \times N}, \quad
    \widetilde \bcH_{\cX} = (\cX^{*})^{1 \times N} \text{ where } 
   N = \sum_{Z \in \Omega} k_{Z}\cdot n_{Z}.
  \end{equation}

 We next introduce operators $\cI \colon \cH_{{\mathbb L}} \to 
 \bcH$ and more generally $\cI_{\cX} \colon \cH_{{\mathbb L}, \cX} 
 \to \bcH_{\cX}$ by
 \begin{align}   
 &  \cI \colon f \mapsto \bigoplus_{Z \in \Omega} f(Z) \text{ for } f 
   \in \cH_{{\mathbb L}} \\
& \cI_{\cX} \colon \boldf  \mapsto  
 \bigoplus_{Z \in \Omega} \boldf(Z) \text{ for } \boldf \in \cH_{{\mathbb L}, 
 \cX}.
 \label{defcI}
 \end{align}
 We first note that both {\em $\cI$ and $\cI_{\cX}$ are injective}; 
 indeed if  
 an element $f \in \cH_{{\mathbb L}}$ has the property that $f(Z) = 
 0$ for all $Z \in \Omega$, then necessarily $f$ is zero as an element of 
 $\cH_{{\mathbb L}}$, and a similar statement applying to an element $\boldf$ of 
 $\cH_{{\mathbb L}, \cX}$.
 Note next that when $\boldf \in \cH_{{\mathbb L}, \cX}$ has the form of an 
 elementary tensor $\boldf = f \cdot \bx^{*} = \iota (f \otimes x^{*})$ 
 for an $f \in \cH_{{\mathbb L}}$ and $x^{*} \in \cX^{*}$, then
 $$
 \cI_{\cX} \colon \iota( f \otimes x^{*}) \mapsto \biota (\cI f 
 \otimes x^{*}).
 $$
 We conclude that when $\iota$ and $\biota$ are used to identify 
 $\cH_{{\mathbb L}, \cX}$ and $\bcH_{\cX}$ with the respective 
 tensor-product spaces $\cH_{{\mathbb L}} \otimes \cX^{*}$ and $\bcH 
 \otimes \cX^{*}$, then the map $\cI_{X}$ assumes the operator 
 elementary-tensor form
 \begin{equation}  \label{optensor}
   \cI_{\cX} = \cI \otimes I_{\cX^{*}}.
 \end{equation}
  It is natural then to also define $\widetilde \cI 
 \colon \cH_{{\mathbb L}} \to \widetilde \bcH$ and $\widetilde 
 \cI_{\cX} \colon \cH_{{\mathbb L}, \cX} \to \widetilde \bcH_{\cX}$ by
 \begin{align*}
   &  \widetilde \cI = \bi \circ \cI \colon f \mapsto {\rm row}_{Z \in 
     \Omega}[ \rowvec_{k_{z}}(f(Z))], \\
   & \widetilde \cI_{\cX} = \widetilde \cI \otimes I_{\cX^{*}} \colon 
   \boldf \mapsto {\rm row}_{Z \in \Omega} 
   [\rowvec_{k_{z}}(\boldf(Z))].
  \end{align*}
   We shall be interested in all these constructions for the particular 
 cases where $\cX = \cR$, $\cX = \cS$, $\cX = \cU$, and $\cX = \cY$.
 Note the the case $\cX = {\mathbb C}$ has already appeared 
 explicitly:
 $\cH_{{\mathbb L}, {\mathbb C}} = \cH_{{\mathbb L}}$,
 $\bcH_{{\mathbb C}} = \bcH$, $\widetilde \bcH_{{\mathbb C}} = 
 \widetilde \bcH$, $\cI_{{\mathbb C}} = \cI$, $\widetilde 
 \cI_{{\mathbb C}} = \widetilde \cI$, etc.

We define right multiplication operators
\begin{align}
  &  M_{Q}^{r} \colon \cH_{{\mathbb L}, \cS}  \to \cH_{{\mathbb 
    L},\cR}, \quad M_{a}^{r} \colon \cH_{{\mathbb L}, \cE} \to 
    \cH_{{\mathbb L}, \cY}, \notag \\
    & M_{b}^{r} \colon \cH_{{\mathbb L}, \cE} \to 
    \cH_{{\mathbb L},\cU}, \quad
 M_{\chi}^{r} \colon \cH_{\mathbb L}  \to {\rm row}_{1 \le i \le d} 
 [\cH_{\mathbb L}] 
 \label{rightmultops}
\end{align}
by
\begin{align*}
& M_{Q}^{r} \colon \boldf(Z) \mapsto \boldf(Z) Q(Z), \quad M_{a}^{r} 
 \colon \boldf(Z) \mapsto \boldf(Z) a(Z), \\
 & M_{b}^{r} \colon \boldf(Z) \mapsto \boldf(Z) b(Z), \quad
 M_{\chi}^{r} \colon f(Z) \mapsto f(Z) \chi(Z).
\end{align*}
From the defining form \eqref{rhoAglerdecom'} for a kernel to be in 
$\cC_{\rho}$ and the fact that ${\rm Re}\, {\mathbb L}$ is 
nonnegative on $\cC_{\rho}$, we read off that
\begin{align}
&    \| M_{Q}^{r}\boldf\|^{2}_{\cH_{{\mathbb L},\cR}} \le \rho^{2} 
    \| \boldf \|^{2}_{\cH_{{\mathbb L},\cS}}, \quad
  \| M_{a}^{r} \boldf \|^{2}_{\cH_{{\mathbb L},\cE}} \le 
 \frac{1}{(1 - \rho)^{2}} \| \boldf \|^{2}_{\cH_{{\mathbb L}, \cY}}, 
 \notag \\
 & \| M_{b}^{r} \boldf\|^{2}_{\cH_{{\mathbb L},\cE}} \le 
  \frac{1}{(1 - \rho)^{2}}\| \boldf \|^{2}_{\cH_{{\mathbb L}, \cU}}, \quad
 \| M_{\chi}^{r}f\|^{2}_{{\rm row}_{1 \le i \le d}\cH_{\mathbb L}}\le 
  \frac{1}{(1 - \rho)^{2}} \| f \|^{2}_{\cH_{\mathbb L}}
  \label{norms}
 \end{align}
 for all $\boldf$ or $f$ in the space $\cH_{{\mathbb L},\cX}$ for the 
 appropriate space $\cX$, and hence each of the 
 operators $M_{Q}^{r}$, $M_{a}^{r}$, $M_{b}^{r}$, 
 $M_{\chi}^{r}$ is well-defined and bounded with $\| M_{Q}^{r} 
 \| \le \rho < 1$ and $\|M_{a}^{r}\|$, $\| M_{b}^{r}\|$, 
 $\| M_{\chi}^{r} \|$ all at most $\frac{1}{1-\rho} < \infty$ for $0 
 < \rho < 1$ with $\rho$ sufficiently close to $1$.
Then we have the following intertwining relations
\begin{align}
    & \cI_{\cR} M^{r}_{Q} = \bQ^{r} \cI_{\cS},  \notag  \\
    & \cI_{\cY} M^{r}_{a} = \ba^{r} \cI_{\cE}, \notag \\
    & \cI_{\cU} M^{r}_{b} = \bb^{r} \cI_{\cE}, \notag \\
    & \cI   M^{r}_{\chi_{k}} = \bchi^{r}_{k}\cI \text{ for } k = 1, 
    \dots, d
 \label{intertwine}
 \end{align}
 where the operators $\bQ^{r}$, $\ba^{r}$, $\bb^{r}$, 
 $\bchi^{r}$ are given by
 \begin{align*}
 & \bQ^{r} = {\rm diag}_{Z \in \Omega}[ Q(Z)] \colon 
  \bigoplus_{Z \in \Omega} \bh_{Z} \mapsto
  \bigoplus_{Z \in \Omega}  \bh_{Z} Q(Z), \\
& \ba^{r} = {\rm diag}_{Z \in \Omega} [a(Z)^{r}] \colon 
\bigoplus_{Z \in \Omega} \bh_{Z} \mapsto \bigoplus_{Z \in \Omega} 
 \bh_{Z} a(Z),  \\
& \bb^{r} = {\rm diag}_{Z \in \Omega} [b(Z)^{r}] \colon 
\bigoplus_{Z \in \Omega} \bh_{Z} \mapsto \bigoplus_{Z \in \Omega} 
\bh_{Z} b(Z),   \\
& \bchi_{k}^{r} = {\rm diag}_{Z \in \Omega} [\chi_{k}(Z)^{r}] \colon 
\bigoplus_{Z \in \Omega} h_{Z} \mapsto \bigoplus_{Z \in \Omega} 
h_{Z} Z_{k} \text{ for } k = 1, \dots, d.
\end{align*}
From the right-multiplier form of all these operators, one can deduce 
the next layer of intertwining relations:
\begin{align}
& \widetilde \cI_{\cR} M^{r}_{Q} = ({\rm diag}\, \bQ)^{r} \widetilde 
\cI_{\cS},   \notag \\
&  \widetilde \cI_{\cY} M^{r}_{a} = ({\rm diag}\, \ba)^{r} \widetilde 
\cI_{\cE}, \notag \\
&  \widetilde \cI_{\cU} M^{r}_{b} = ({\rm diag}\, \bb)^{r} \widetilde 
\cI_{\cE}, \notag \\
&  \widetilde \cI_{\cY} M^{r}_{\chi_{k}} = ({\rm diag}\, \bchi_k)^{r} 
\widetilde \cI. 
   \label{intertwine'}
    \end{align}
Here we view elements of $\widetilde \bcH_{\cX}$ as row matrix of 
length $N$ over $\cX^{*}$ as in \eqref{tbcH-inden}    
of the form 
$\widetilde \bh = {\rm row}_{Z \in \Omega} [ \bh(Z) ]$ where $\bh(Z) 
\in (\cX^{*})^{1 \times k_{z}}$ for $\cX$ equal to any of $\cS, \cR, \cE, 
\cU, \cY,{\mathbb C}$ and then the operators $({\rm diag}\, \bQ)^{r}$,
$({\rm diag}\, \ba)^{r}$, $({\rm diag}\, \bb)^{r}$, $({\rm diag}\, 
\bchi_{k})^{r}$ are given by
\begin{align}
    & ({\rm diag}\, \bQ)^{r} \colon {\rm row}_{Z \in \Omega} [\widetilde 
    \bh(Z)] \mapsto {\rm row}_{Z \in \Omega}[ \widetilde \bh(Z) \cdot 
    \diag Q(Z)], \notag \\
 & ({\rm diag}\, \ba)^{r} \colon {\rm row}_{Z \in \Omega} [\widetilde 
    \bh(Z)] \mapsto {\rm row}_{Z \in \Omega} [\widetilde \bh(Z) \cdot 
    \diag a(Z)], \notag \\
 & ({\rm diag}\, \bb)^{r} \colon {\rm row}_{Z \in \Omega} [\widetilde 
    \bh(Z)] \mapsto {\rm row}_{Z \in \Omega}[ \widetilde \bh(Z) \cdot 
    \diag b(Z)], \notag \\
  & ({\rm diag}\, \bchi_{k})^{r} \colon {\rm row}_{Z \in \Omega}[ \widetilde 
    h(Z)] \mapsto {\rm row}_{Z \in \Omega}[ \widetilde h(Z) \cdot 
    \diag \chi_{k}(Z)].
    \label{matrixops}
\end{align}

To get a matrix representation for $\widetilde \cI$ and $\widetilde 
\cI_{\cX}$, we wish to also represent elements of $\cH_{{\mathbb L}}$ 
and $\cH_{{\mathbb L}, \cX}$ as a space of row vectors.  Toward this 
end, we use the orthonormal basis $\cB' = \{f_{1}, \dots, f_{K'}\}$  
\eqref{HLbasis} for $\cH_{{\mathbb L}}$ to identify $\cH_{{\mathbb 
L}}$ with the space of row vectors
$$
   \cH_{{\mathbb L}} \cong {\mathbb C}^{1 \times K'}.
$$
Then the space $\cH_{{\mathbb L}, \cX} \cong \cH_{{\mathbb L}} 
\otimes \cX^{*}$ can be identified with row vectors over $\cX^{*}$:
$$
  \cH_{{\mathbb L}, \cX} \cong (\cX^{*})^{1 \times K'}.
$$
Let us write $\vec f$ and $\vec \boldf$ for elements of 
$\cH_{{\mathbb L}}$ and $\cH_{{\mathbb L}, \cX}$ when viewed as a 
row matrix in ${\mathbb C}^{1 \times K'}$ (respectively, $(\cX^{*})^{1 
\times K'}$).
With these identifications, the operator $\cI \in \cL(\cH_{{\mathbb 
L}}, \widetilde \bcH)$ is represented as right multiplication by a 
matrix $[\widetilde \cI] \in {\mathbb C}^{K' \times N}$:
$$
  [\widetilde \cI]^{r} \colon \vec f \mapsto \vec f \cdot [\widetilde 
  \cI].
$$

Then $\widetilde \cI \colon \cH_{{\mathbb L}} \to \widetilde \bcH$ has a matrix 
representation $[ \widetilde \cI] \in {\mathbb C}^{K' \times N}$ 
induced by these respective bases.  Moreover the operator 
$M^{r}_{\chi_{k}} \in \cL(\cH_{{\mathbb L}})$ has a matrix 
representation $[M^{r}_{\chi_{k}}] \in {\mathbb C}^{K' \times K'}$ 
induced by the given basis for $\cH_{{\mathbb L}}$ and similarly
the operator $ ({\rm diag}\, \bchi_{k})^{r}$ has a matrix representation $[{\rm diag} \, \bchi_{k}]$ of the 
form ${\rm diag}_{Z \in \Omega} [\diag [\bchi_{k}]]$ induced by the given basis for $\widetilde 
\bcH$.  The last of the intertwining relations \eqref{intertwine'} 
implies the matricial intertwining relation
\begin{equation}   \label{intertwining1}
[ \widetilde \cI]^{r} [M^{r}_{\chi_{k}}]^{r} =[ {\rm diag}\,
\bchi_{k}]^{r}
 [\widetilde \cI]^{r},
\end{equation}
or, directly in terms of the matrices
$$
 [M^{r}_{\chi_{k}}] [ \widetilde \cI] = [\widetilde \cI] [ {\rm 
 diag}\,  \bchi_{k}].
$$
A consequence of the operator elementary-tensor form of $\widetilde 
\cI_{\cX}$ \eqref{optensor} is that the matrix representation for 
$\widetilde \cI_{\cX}$ is given by the same matrix $[\cI_{\cX}]$
(here we use that $\cX^{*}$ as a vector space is a ${\mathbb C}$-module
so multiplication of a matrix over $\cX^{*}$ by a matrix over 
${\mathbb C}$ makes sense as long as the sizes fit):
$$
[\widetilde \cI_{\cX}]^{r} \colon \vec \boldf \mapsto \vec \boldf 
\cdot [\widetilde \cI].
$$

When we identify $\cH_{{\mathbb L}, \cX}$ with $(\cX^{*})^{1 \times 
K'}$, the operators $M_{Q}^{r}$, $M^{r}_{a}$, $M^{r}_{b}$, 
$M^{r}_{\chi_{k}}$ in \eqref{rightmultops} then have representations 
as multiplication on the right by appropriate matrices
\begin{align*}
  &  [M^{r}_{Q}] \in \cL(\cR, \cS)^{K' \times K'}, \quad
    [M^{r}_{a}] \in \cL(\cY, \cE)^{K' \times K'}, \\
 & [M^{r}_{b}] \in \cL(\cU, \cE)^{K' \times K'}, \quad
 [M^{r}_{\chi_{k}} ] \in {\mathbb C}^{K' \times K'}.
\end{align*}

Note that the spaces $\widetilde \bcH_{\cX}$ are already represented 
as spaces of row vectors, namely $(\cX^{*})^{1 \times N}$; the operators on these 
spaces are represented as right multiplication by matrices as in  
\eqref{matrixops}.  Then the intertwining relations 
\eqref{intertwine'} assume the following matricial form:
\begin{align}
    &  [\widetilde \cI]^{r} [ M^{r}_{Q}]^{r} = [{\rm diag}\, \bQ]^{r} 
    [ \widetilde \cI]^{r}, \notag \\
 &  [\widetilde \cI]^{r} [ M^{r}_{a}]^{r} = [{\rm diag}\, \ba ]^{r} 
    [ \widetilde \cI]^{r}, \notag \\
  &  [\widetilde \cI]^{r} [ M^{r}_{b}]^{r} = [ {\rm diag}\, \bb ]^{r} 
    [ \widetilde \cI]^{r}, \notag \\
  & [ \widetilde \cI]^{r} [ M^{r}_{\chi_{k}}]^{r} = [ {\rm diag}\, 
  \bchi_{k}]^{r} [ \widetilde \cI]^{r},
  \label{matricialform}
\end{align}
or, in terms of the matrices themselves,
\begin{align*}
    & [M^{r}_{Q}] [ \widetilde \cI] = [ \widetilde \cI] [{\rm diag}\, 
    \bQ], \\
  & [M^{r}_{a}] [ \widetilde \cI] = [ \widetilde \cI] [{\rm diag}\, 
    \ba], \\
   & [M^{r}_{b}] [ \widetilde \cI] = [ \widetilde \cI] [{\rm diag}\, 
    \bb], \\
  & [M^{r}_{\chi_{k}}] [ \widetilde \cI] = [ \widetilde \cI] [{\rm diag}\, 
    \bchi_{k}].
    \end{align*}

Let us now define a point $\bZ$ in the nc envelope $(\Omega)_{\rm 
nc}$ of $\Omega$ by 
$$
\bZ = {\rm diag}_{Z \in \Omega} [ {\rm diag}_{1, \dots, k_{Z}} [ Z ]].
$$
Then the operator $\bZ^{r}$ of right multiplication by $\bZ$ on the 
row space ${\mathbb C}^{1 \times N}$ is the same as the operator 
$[{\rm diag}\, \bchi]^{r} = \begin{bmatrix} {\rm diag}\, \bchi_{1}^{r} 
& \cdots & {\rm diag}\, \bchi_{d}^{r} \end{bmatrix}$, i.e., 
the matrix $\bchi = \begin{bmatrix} {\rm diag}\, \bchi_{1} & 
\cdots & {\rm diag}\, \bchi_{d}\end{bmatrix}$ is in $(\Omega)_{\rm nc}$.  
As the right multiplication operator $[\widetilde \cI]^{r}$ is injective and 
this is a matrix over ${\mathbb C}$, a consequence of the 
intertwining relation \eqref{intertwining1} is that the matrix
$$
 [M^{r}_{\chi}] = \begin{bmatrix} [M^{r}_{\chi_{1}}] & \cdots &
 [M^{r}_{\chi_{d}}] \end{bmatrix}
$$
is in the full nc envelope $(\Omega)_{\rm nc, full}$.  As $Q$ is a nc 
function on $\Xi \supset \Omega_{\rm full}$, it follows that 
$Q(\bchi)$ and $Q([M^{r}_{\chi}])$ are defined. Since nc functions 
respect intertwinings, a consequence of \eqref{intertwining1} is the 
intertwining relation
\begin{equation}   \label{intertwining2}
[\widetilde \cI]^{r} Q([M_{\chi}^{r}])^{r} = Q([{\rm diag}\, 
\bchi])^{r} [\widetilde \cI]^{r}.
\end{equation}
since nc functions respect direct sums, it follows from the 
definitions that
$$
   Q([{\rm diag}\, \bchi]) = [{\rm diag}\, \bQ].
$$
Substitution of this relation back into \eqref{intertwining2} then 
gives
$$
[\widetilde \cI]^{r} \, Q([M_{\chi}^{r}])^{r} = [{\rm diag}\, \bQ]\,
[\widetilde \cI]^{r}.
$$
Comparison of this with the first of relations \eqref{matricialform} 
then shows that 
$$
 [\widetilde \cI]^{r} \, Q([M_{\chi}^{r}])^{r} =
 [\widetilde \cI]^{r} \, [M_{Q}^{r}]^{r}.
$$
As $[\widetilde \cI]^{r}$ is injective, we conclude that
$$
Q([M_{\chi}^{r}]) = [M_{Q}^{r}].
$$
We have already observed (see \eqref{norms}) that $\| M_{Q}^{r} \| < 
1$.  We conclude that the point $[M_{\chi}^{r}]$ is not only in the 
full nc envelope $(\Omega)_{\rm nc,full}$ of $\Omega$, but also is in the 
${\mathbb D}_{Q}$-relative full nc envelope $\Omega'$ of $\Omega$.

A quite similar analysis identifies the evaluation of the nc functions 
$a$, $b$ at the point $[M_{\chi}^{r}]$:
\begin{equation} \label{Mab-iden}
a([M_{\chi}^{r}]) = [M_{a}^{r}], \quad b([M_{\chi}^{r}]) = 
[M_{b}^{r}].
\end{equation}
As we saw above that $[M_{\chi}^{r}]$ is in $\Omega'$, our standing 
hypothesis tells us that
$$ 
a([M_{\chi^{r}}]) a([M_{\chi^{r}}])^{*} - b([M_{\chi^{r}}]) 
b([M_{\chi^{r}}])^{*} \succeq 0,
$$
From the identities \eqref{Mab-iden}, this translates to
$$
[M_{a}^{r}]  [M_{a}^{r}]^{*} -  [M_{b}^{r}] [M_{b}^{r}]^{*} \succeq  0.
$$
In terms of   matrix right-multiplication  operators, this becomes
$$
[M_{a}^{r}]^{r *}[M_{a}^{r}]^{r} - [M_{b}^{r}]^{r*} [M_{b}^{r}]^{r} 
\succeq 0.
$$
In terms of the original coordinate-free presentation of 
$\cH_{{\mathbb L}, \cE}$, this last expression translates to the 
statement
$$
   \| M_{a}^{r} \boldf \|^{2}_{\cH_{{\mathbb L}, \cY}} - \| M_{b}^{r} 
   \boldf \|^{2}_{\cH_{{\mathbb L}, \cU}} \ge 0
$$
for all $\boldf \in \cH_{{\mathbb L}, \cE}$, or, equivalently,
$$
 {\mathbb L}_{1}\left(D_{\boldf a, \boldf a} - D_{\boldf b, \boldf b} 
 \right) \ge 0 \text{ for all } \boldf \in \cH_{{\mathbb L}, \cE},
 $$
 where
 $$
 \left( D_{\boldf a, \boldf a} - D_{\boldf b, \boldf b} \right) 
 (Z,W)(P) = \boldf(Z) K_{a,b}(Z,W)(P) \boldf(W)^{*}
 $$
 (the kernel $K_{a,b}$ defined as in \eqref{Kab}).  In particular we 
 may choose
 $$
   \boldf(Z) = {\rm id}_{{\mathbb C}^{n}}\otimes I_{\cE} \text{ for } Z \in \Omega_{n}
 $$
 to conclude that ${\mathbb L}_{1}(K_{a,b}) \ge 0$.  We have thus 
 arrived at condition \eqref{goal} as required to complete the proof 
 of Case 1.

\smallskip

\noindent
\textbf{Case 2: Reduction of the case of a general subset 
$\Omega$ and coefficient space $\cE$ to the case of a finite  subset 
$\Omega_{\alpha}$ and finite-dimensional coefficient  Hilbert 
space $\cE_{\alpha}$.} 
The general case can be reduced to the special situation of Case 1 as 
an application of a theorem of Kurosh  that says that the limit of 
an inverse spectrum of nonempty compacta is a nonempty compactum 
(see \cite[Theorem 2.56]{AMcC-book}).  For the sake of completeness 
we go through the proof presented in \cite[pages 73--75]{AP} adapted 
to the special case here.

We assume condition (1$^{\prime}$) in the statement of the theorem, but now 
the point subset $\Omega$ of ${\mathbb D}_{Q}$ is not necessarily 
finite and the coefficient Hilbert space $\cE$ is not 
necessarily finite-dimensional.  

\smallskip
\textbf{The directed set $\fA$.} We let $\fA$ be the set of all pairs 
$(\Omega_{0}, \cE_{0})$ where $\Omega_{0}$ is a finite subset of $\Omega$ 
and $\cE_{0}$ is a finite-dimensional subspace of 
$\cE$.  It will be convenient to use lower case Greek letters, e.g. 
$\alpha$, for a generic element of $\fA$.  If $\alpha = (\Omega_{0}, 
\cE_{0}) \in \fA$, let us write $\alpha_{\Omega} = \Omega_{0}$ and 
$\alpha_{\cE} = \cE_{0}$, so $\alpha = (\alpha_{\Omega}, 
\alpha_{\cE})$.  We define a partial order on $\fA$ as follows:  
given two elements $\alpha$ and $\beta$ of $\fA$, we say that 
$\alpha \preceq \beta$ if both $\beta_{\Omega} \subset 
\alpha_{\Omega}$ (as finite subsets of $\Omega$) and $\alpha_{\cE} 
\subset \beta_{\cE}$ (as finite-dimensional subspaces of $\cE$). Then $\preceq$ 
satisfies reflexivity ($\alpha \preceq \alpha$) and transitivity 
($\alpha \preceq \beta$, $\beta \preceq \gamma$ $\Rightarrow$ $\alpha 
\preceq \gamma$).  

We claim that $\fA$ is a directed set, i.e.:
\begin{itemize}
    \item[Claim:] {\em Given any $\alpha, \beta \in \fA$, one can always find a $\gamma \in 
\fA$ so that both $\alpha \preceq \gamma$ and $\beta \preceq \gamma$}.
\end{itemize}
To verify the Claim, let $\alpha = (\alpha_{\Omega}, \alpha_{\cY})$ and $\beta = 
(\beta_{\Omega}, \beta_{\cY})$ be in $\fA$.  Then choose a finite
set $\gamma_{\Omega}$ so that both $\alpha_{\Omega} \subset 
\gamma_{\Omega}$ and $\beta_{\Omega} \subset \gamma_{\Omega}$. 
  Similarly, let $\gamma_{\cY}$ 
be any finite dimensional subspace containing both $\alpha_{\cY}$ and 
$\beta_{\cY}$.  Then $\gamma = (\gamma_{\Omega}, \gamma_{\cY}) \in 
\fA$ has the property that both $\alpha \preceq \gamma$ and $\beta \preceq \gamma$.

\smallskip

\textbf{The compact set ${\mathbb K}_{\alpha}$ for $\alpha \in \fA$.}
For each $\alpha \in \fA$, we let ${\mathbb K}_{\alpha}$ be the set 
of all cp nc kernels $\Gamma \colon \alpha_{\Omega} \times 
\alpha_{\Omega} \to \cL(\cL(\cS), \cL(\alpha_{\cE}))_{\rm nc}$ 
so that
\begin{align}  
   & (I_{n} \otimes \Pi^{\cE}_{\alpha_{\cE}} \left( a(Z) (P \otimes I_\cY) a(W)^{*} - 
 b(Z) (P \otimes I_{\cU}) b(W)^{*} \right) |_{\alpha_{\cE}^{m}}  
    \notag \\
    & \quad =
    \Gamma(Z,W) \left( (P \otimes I_{\cS}) - Q(Z) (P \otimes I_{\cR}) 
    Q(W)^{*} \right)
     \label{localAgdecom}
 \end{align}
 for all $Z \in \alpha_{\Omega,n}$, $W \in \alpha_{\Omega,m}$ and $P 
 \in {\mathbb C}^{n \times m}$, where $\Pi^{\cE}_{\alpha_{\cE}} \colon \cE 
 \to \alpha_{\cE}$ is the orthogonal projection of $\cE$ onto 
 $\alpha_{\cE}$.  The accomplishment of Case 1 of the proof is 
 to confirm that ${\mathbb K}_{\alpha}$ is nonempty for all $\alpha 
 \in \fA$.  Furthermore, one can use the estimate \eqref{est1} with
 $(I_{n} \otimes \Pi^{\cE}_{\alpha_{\cE}}) \big( a(Z) (P \otimes I_\cY) a(W)^{*} - 
 b(Z) (P \otimes I_{\cU}) b(W)^{*} \big) |_{\alpha_{\cE}^{m}}$  
 in place of $K_{N}(Z,Z)(I)$ to see 
 that each ${\mathbb K}_{\alpha}$ is compact in the norm
 topology of ${\mathfrak X}$.
 
 \smallskip
 
 \textbf{The restriction maps $\pi^{\alpha}_{\beta}$.}
 For $\alpha$ and $\beta$ two elements of $\fA$ with $\beta \preceq 
 \alpha$, we define a map $\pi^{\alpha}_{\beta} \colon {\mathbb 
 K}_{\alpha} \to {\mathbb K}_{\beta}$ as follows:  for $\Gamma \in 
 {\mathbb K}_{\alpha}$, define $\pi^{\alpha}_{\beta} \Gamma \in 
 {\mathbb K}_{\beta}$ by
 $$
\left(\pi^{\alpha}_{\beta} \Gamma\right)(Z,W)(R) =
(I_{n} \otimes \Pi^{\alpha}_{\beta}) \Gamma(Z,W)(R)|_{{\mathbb C}^{m} \otimes 
\beta_{\cY}}
$$
for $Z \in \beta_{\Omega,n}$, $W \in \beta_{\Omega, m}$, $R \in 
\cS^{n \times m}$, where $\Pi^{\alpha}_{\beta} \colon 
\alpha_{\cY} \to \beta_{\cY}$ is the orthogonal projection.  Then it 
is easily verified that the family of maps  $\{ \pi^{\alpha}_{\beta} 
\colon \beta \preceq \alpha\}$ satisfies the following properties:
\begin{align}
   & \pi^{\alpha}_{\alpha} = {\rm id}_{{\mathbb K}_{\alpha}},  
   \label{prop1} \\
  &  \gamma \preceq \beta,\,  \beta \preceq \alpha  \,
    \Rightarrow \, \pi^{\alpha}_{\gamma} = \pi^{\beta}_{\gamma} \circ 
    \pi^{\alpha}_{\beta}.  \label{prop2}
\end{align}

\smallskip

\textbf{The sets ${\mathbb K}$, ${\mathbb K}^{\beta \alpha}$, and 
${\mathbb K}^{\alpha}$.}  We let ${\mathbb K}$ be the 
Cartesian-product space 
$$
   {\mathbb K} = \Pi_{\alpha \in \fA} \, {\mathbb K}_{\alpha}.
$$
We denote an element of ${\mathbb K}$ by $\bGamma = \{ 
\Gamma_{\alpha} \}_{\alpha \in \fA}$ where each $\Gamma_{\alpha} \in 
{\mathbb K}_{\alpha}$.  We endow ${\mathbb K}$
with the standard Cartesian-product topology (the weakest topology 
such that each projection $p_{\alpha} \colon \bGamma \mapsto 
\Gamma_{\alpha} \in {\mathbb K}_{\alpha}$ is continuous).  Since each 
fiber ${\mathbb K}_{\alpha}$ is nonempty and compact by the preceding 
discussion, it is a consequence of Tikhonov's Theorem (and the Axiom 
of Choice) that ${\mathbb K}$ is nonempty and compact.  Given a pair 
of elements $(\beta, \alpha)$ from $\fA$ with $\beta \preceq \alpha$, 
we let ${\mathbb K}^{\beta \alpha}$ be the subset of ${\mathbb K}$ 
consisting of all elements $\bGamma = \{ \Gamma_{\gamma} \}_{\gamma 
\in \fA}$ such that $\pi^{\alpha}_{\beta} \Gamma_{\alpha} = 
\Gamma_{\beta}$.  We can construct an element of ${\mathbb K}^{\beta 
\alpha}$ as follows:  choose any kernel $\Gamma_{\alpha}$ in 
${\mathbb K}_{\alpha}$ and define $\bGamma = \{ 
\Gamma_{\alpha'}\}_{\alpha' \in \fA} \in {\mathbb K}$ by:
$$
 \Gamma_{\alpha'} = \begin{cases} \Gamma_{\alpha} & \text{if } 
 \alpha' = \alpha, \\
\pi^{\alpha}_{\beta} \Gamma_{\alpha} & \text{if } \alpha' = \beta,  \\
   \text{any element of } {\mathbb K}_{\alpha'} & \text{otherwise.}
   \end{cases}
$$
In this way we see that ${\mathbb K}^{\beta \alpha}$ is nonempty for 
any pair of indices $\alpha, \beta \in \fA$ with $\beta \preceq 
\alpha$.  

For $\alpha$ a single index in $\fA$, we define a subset ${\mathbb K}^{\alpha}$
by 
$$
    {\mathbb K}^{\alpha} = \cap_{\beta \in \fA \colon \beta \preceq 
    \alpha} {\mathbb K}^{\beta \alpha}.
$$
A key point is that ${\mathbb K}^{\alpha}$ is always nonempty.  This 
follows by a construction similar to that used to show that ${\mathbb 
K}^{\beta \alpha}$ is nonempty:  choose any $\Gamma_{\alpha} \in 
{\mathbb K}_{\alpha}$ and define $\bGamma = \{ \Gamma_{\alpha'} 
\}_{\alpha' \in \fA}$ by
$$
   \Gamma_{\alpha'} = \begin{cases} 
\pi^{\alpha}_{\alpha'} \Gamma_{\alpha} & \text{if } \alpha' \preceq \alpha,  \\
   \text{any element of } {\mathbb K}_{\alpha'} & \text{otherwise.}
   \end{cases}
$$
Then $\bGamma$ so constructed is in ${\mathbb K}^{\alpha}$.

The next step is to argue that $\cap_{\alpha \in \fA} {\mathbb 
K}^{\alpha}$ is nonempty.  To do this, we argue:
\begin{enumerate}  
    \item Each ${\mathbb K}^{\alpha}$ is a closed subset of the 
    compact space ${\mathbb K}$, and 
    \item the collection of sets $\{ {\mathbb K}^{\alpha} \colon 
    \alpha \in \fA\}$ has the Finite Intersection Property:  given 
    any finitely many $\alpha^{(1)}, \dots, \alpha^{(N)} \in \fA$, 
    then 
 $$
   {\mathbb K}^{\alpha^{(1)}} \cap \cdots \cap {\mathbb 
   K}^{\alpha^{(N)}} \ne \emptyset.
 $$
 \end{enumerate}
It then follows from a standard compactness argument that 
$\bigcap_{\alpha \in \fA} {\mathbb K}^{\alpha} \ne \emptyset$.

To verify (1), we note that each ${\mathbb K}^{\beta \alpha}$ is 
closed by the definition of the Cartesian product topology and the 
fact that each map $\pi^{\alpha}_{\beta} \colon {\mathbb K}_{\alpha} 
\to {\mathbb K}_{\beta}$ is continuous.

To verify (2), suppose that we are given finitely many indices 
$\alpha^{(1)} , \dots, \alpha^{(N)}$ in $\fA$.  We noted above  that
$\fA$ is a directed set; hence  a simple induction argument using 
the transitivity property of $\fA$ enables us to find $\alpha^{(0)} \in \fA$ so 
that $\alpha^{(j)} \preceq \alpha^{(0)}$ for $j = 1, \dots, N$. 
Then we can use property \eqref{prop2} of the maps 
$\pi^{\alpha}_{\beta}$ to see that 
 ${\mathbb K}^{\alpha^{(0)}} \subset 
\cap_{j=1}^{N} {\mathbb K}^{\alpha^{(j)}}$:
if $\beta \preceq \alpha^{(j)}$, then since also $\alpha^{(j)} 
\preceq \alpha^{(0)}$ we have, for $\bGamma = \{ 
\Gamma_{\alpha'}\}_{\alpha' \in \fA}$ in ${\mathbb 
K}^{\alpha^{(0)}}$,
$$
\pi^{\alpha^{(j)}}_{\beta} \Gamma_{\alpha^{(j)}} =
\pi^{\alpha^{(j)}}_{\beta} \pi^{\alpha^{(0)}}_{\alpha^{(j)}} 
\Gamma_{\alpha^{(0)}} = \pi^{\alpha^{(0)}}_{\beta} 
\Gamma_{\alpha^{(0)}} = \Gamma_{\beta}
$$
verifying that $\bGamma \in {\mathbb K}^{\alpha^{(j)}}$ for each $j = 
1, \dots, N$.  As we verified above that 
each ${\mathbb K}^{\alpha}$ is nonempty, in particular 
${\mathbb K}^{\alpha^{(0)}}$ is nonempty and hence $\bigcap_{j=1}^{N} 
{\mathbb K}^{\alpha^{(j)}} \ne \emptyset$.

\smallskip

\textbf{Construction of $\Gamma$ giving an Agler decomposition on all 
of $\Omega$.} 
As a result of the preceding paragraph, we can find an element 
$\bGamma = \{\Gamma_{\alpha'}\}_{\alpha' \in \fA}$ in $\cap_{\alpha 
\in {\mathfrak A}} {\mathbb K}^{\alpha}$, i.e., an element $\bGamma$ of ${\mathbb K}$ 
such that 
\begin{equation}  \label{defprop}
\beta \preceq \alpha \text{ in } {\mathfrak A} \Rightarrow \pi^{\alpha}_{\beta} 
\Gamma_{\alpha} = \Gamma_{\beta}.
\end{equation}
 We use this $\bGamma$ to construct 
a cp nc kernel $\Gamma$ defined on all of $\Omega \times \Omega$ giving the 
desired Agler decomposition \eqref{Aglerdecom} for $K_{a,b}$ on all of 
$\Omega$ as follows.
Given $Z \in \Omega_{n}$, $W \in \Omega_{m}$, $e = \sbm{ e_{1} \\ 
\vdots \\ e_{m}} \in \cE^{m}$, $e' = \sbm{e'_{1} \\ \vdots \\ e'_{n}} 
\in \cE^{n}$ and $R \in \cL(\cS)^{n \times m}$, choose an 
element $\alpha \in \fA$ so that $Z,W \in \alpha_{\Omega}$ and 
$e_{j}, e'_{i} \in \alpha_{\cE}$ for $1 \le j \le m$ and $1 \le i \le 
n$.  We wish to define an operator $\Gamma(Z,W)(R) \in \cL(\cE^{m}, 
\cE^{n})$ so that
\begin{equation}   \label{defGamma}
  \langle \Gamma(Z, W)(R) e,\, e' \rangle_{\cE^{n}} =
  \langle \Gamma_{\alpha}(Z,W)(R) e, \, e' \rangle_{(\alpha_{\cE})^{n}}.
\end{equation}
There are various issues to check.  

First of all, we can always find 
an $\alpha \in \fA$ so that a given $Z \in \Omega_{n}$ and $W \in 
\Omega_{m}$ are in $\alpha_{\Omega}$ and $e_{j} \in \alpha_{\cE}$, 
$e'_{i} \in \alpha_{\cE}$ for all $j=1, \dots, m$ and $i = 1, \dots, 
n$.  Set $\alpha_{\Omega}$ equal to any finite subset 
containing $Z$ and $W$, 
and choose $\alpha_{\cE}$ to be any finite dimensional subspace of 
$\cE$ containing the finite set of vectors $\{e_{j},\,  e'_{i} \colon 1 
\le j \le m; \,  1 \le i \le n\}$.  

Secondly the right-hand side of 
\eqref{defGamma} is independent of the choice of index $\alpha \in 
\fA$.  Indeed, if $\alpha'$ is another such index, choose 
$\alpha^{(0)} \in \fA$ with $\alpha \preceq \alpha^{(0)}$ and $\alpha' 
\preceq \alpha^{(0)}$. Then a simple application of the property 
\eqref{defprop} combined with the property \eqref{prop2} of the 
restriction maps $\pi^{\alpha}_{\beta}$ enables us to see that
\begin{align*}
\langle \Gamma_{\alpha}(Z,W)(R) e, e' \rangle_{(\alpha_{\cE})^{n}}
& = \langle \Gamma_{\alpha^{(0)}}(Z,W)(R) e, e' 
\rangle_{(\alpha^{(0)}_{\cE})^{n}} \\
& = \langle \Gamma_{\alpha'}(Z,W)(R) e, e' \rangle_{(\alpha'_{\cE})^{n}}
\end{align*}
as required.

One can now use standard techniques from Hilbert space theory 
(Principle of Uniform Boundedness and Riesz-Frechet Theorem---see 
\cite{Rudin}) to see 
that the well-defined sesquilinear form given by \eqref{defGamma} 
arises from an operator $\Gamma(Z,W)(R) \in \cL(\cE^{m}, \cE^{n})$.
Since the form is linear in $R$, the operator $\Gamma(Z,W)(R)$ is 
also linear in $R$.  Since cp nc kernel properties involve only 
finitely many points $Z,W$ at a time, it is easily verified that the 
fact that each $\Gamma_{\alpha}$ is a cp nc kernel on 
$\alpha_{\Omega} \times \alpha_{\Omega}$ implies that $\Gamma$ is a cp nc 
kernel on $\Omega \times \Omega$.  Finally, since each 
$\Gamma_{\alpha}$ gives rise to a  local Agler decomposition for 
$K_{a,b}$ on $\alpha_{\Omega}$ compressed to $\alpha_{\cE}$, it is 
easily argued that $\Gamma$ gives rise to a global Agler 
decomposition \eqref{Aglerdecom} on all of $\Omega$.

We have now finally completed the proof of (1$^{\prime}$) $\Rightarrow$ (2) in 
Theorem \ref{T:ncInt}.
 \end{proof}

 \begin{proof}[Proof of (2) $\Rightarrow$ (3) in Theorem \ref{T:ncInt}]
     For this part we assume only that $a,b$ are graded functions 
     from $\Omega$ into $\cL(\cY, \cE)_{\rm nc}$ and $\cL(\cU, 
     \cE)_{\rm nc}$  respectively.
 We assume that we are given an Agler 
decomposition for $K_{a,b}$ as in \eqref{Aglerdecom}.  Let 
$\Gamma(Z,W)(T) = H(Z) \left( ({\rm id}_{{\mathbb C}^{n \times m}} 
\otimes \pi)(T) \right) H(W)^{*}$ be the Kolmogorov decomposition
\eqref{Koldecom} for $\Gamma$.  Then \eqref{Aglerdecom} becomes
\begin{align}
&  a(Z)( P \otimes I_{\cY}) a(W)^{*} - b(Z) ( P \otimes  I_{\cU}) b(W)^{*} 
\notag \\
& \quad = H(Z)
\left[ ({\rm id}_{{\mathbb C}^{n \times m}} \otimes \pi) 
( P \otimes I_{\cS} - Q(Z) (P \otimes I_{\cR}) Q(W)^{*}) \right] H(W)^{*}.
\label{Aglerdecom''}
\end{align}
Since $\pi$ is a unital $*$-representation, we have the simplification
$$
({\rm id}_{{\mathbb C}^{n \times m}} \otimes \pi) 
( P \otimes I_{\cS}) = P \otimes I_{\cX}.
$$
If we rebalance the identity \eqref{Aglerdecom''} to eliminate all 
minus signs, we arrive 
at
\begin{align}  
& H(Z) [ ({\rm id}_{{\mathbb C}^{n \times m}} \otimes \pi) 
( Q(Z) ( P \otimes I_{\cR}) Q(W)^{*}) ]  H(W)^{*} + 
a(Z)( P \otimes I_{\cY}) a(W)^{*} \notag \\
 & \quad = H(Z)  ( P \otimes I_{\cX}) H(W)^{*} + b(Z) (P \otimes 
 I_{\cU}) b(W)^{*}.
\label{balanced}
\end{align}
We next write this as an inner-product identity
\begin{align}
&\langle H(Z) [ ({\rm id}_{{\mathbb C}^{n \times m}} \otimes \pi) 
( Q(Z) \cdot P  \cdot Q(W)^{*}) ]  H(W)^{*}e, e' \rangle + 
\langle a(Z) \cdot P \cdot a(W)^{*}e, e' \rangle \notag \\
 & \quad =\langle  H(Z) \cdot P \cdot H(W)^{*}e, e' \rangle
 +\langle b(Z)  \cdot P \cdot b(W)^{*}e, e' \rangle
\label{balanced'}
\end{align}
where $e$ and $e'$ are arbitrary vectors in $\cE^{m}$ and $\cE^{n}$ 
respectively, and the various inner 
products are in the space $\cE^{n}$.  We also simplify the notation 
by viewing the operator $P \otimes I_{\cR}$ (as well as $P \otimes 
I_{\cY}$ and $P \otimes I_{\cX}$)  as multiplication by the scalar 
matrix $P$ (which makes sense due the vector-space structure of 
$\cR$, $\cY$, and $\cX$).

Suppose next that $P$ has a factorization
$$
  P = \beta^{*} \alpha \text{ where } \alpha \in {\mathbb C}^{k 
  \times m}, \, \beta \in {\mathbb C}^{k \times n}.
$$
Then, by the definition \eqref{nc0innerprod} of the $(\cL(\cS, \cR^{k}) 
\otimes_{\pi} \cX)_{\rm nc}$-inner product (see Theorem 
\ref{T:nctensorprod}),
we may rewrite \eqref{balanced'} as  an alternative inner-product identity: 
\begin{align}
   & \left\langle \begin{bmatrix} \alpha \cdot Q(W)^{*} \otimes H(W)^{*}e \\
    \alpha \cdot a(W)^{*} e \end{bmatrix},
 \begin{bmatrix} \beta \cdot Q(Z)^{*} 
    \otimes H(Z)^{*}e' \\ \beta \cdot a(Z)^{*} e' \end{bmatrix} 
    \right\rangle \notag \\
    & \quad = \left \langle \begin{bmatrix} \alpha \cdot H(W)^{*}e \\ \alpha \cdot 
    b(W)^{*} e \end{bmatrix}, 
    \begin{bmatrix} \beta \cdot H(Z)^{*} e' 
    \\ \beta \cdot b(Z)^{*} e' \end{bmatrix} \right \rangle,
\label{longinnerprod'}
\end{align}
or equivalently (by recalling the definition \eqref{ncLT} of $L_{T}$) as
\begin{align}
& \left\langle \begin{bmatrix}   L_{\alpha \cdot 
Q(W)^{*}}  H(W)^{*} \\  \alpha \cdot 
a(W)^{*} \end{bmatrix} e, \,
\begin{bmatrix}  L_{ \beta \cdot Q(Z)^{*}} H(Z)^{*} \\ 
    \beta \cdot a(Z)^{*} \end{bmatrix} e' \right\rangle \notag \\ 
& \quad = \left\langle \begin{bmatrix} \alpha \cdot H(W)^{*} \\ 
 \alpha \cdot b(W)^{*} \end{bmatrix} e, 
\, \begin{bmatrix}  \beta \cdot H(Z)^{*} \\  \beta \cdot b(Z)^{*} 
\end{bmatrix} e' \right\rangle,
\label{longinnerprod}
\end{align}
where the first inner product is in $(\cL(\cS, \cR^{k}) \otimes_{\pi}
\cX)_{\rm nc}  \oplus \cY^{k}$ and the 
second inner product is in $\cX^{k} \oplus \cU^{k}$.
Furthermore, by part (3) of Theorem \ref{T:nctensorprod}, we can 
identify the space $(\cL(\cS, \cR^{k}) \otimes_{\pi} \cX)_{\rm nc}$ 
with the $k$-fold direct sum space $\left( (\cL(\cS, \cR) \otimes_{\pi} 
\cX)_{\rm nc} \right)^{k}$.

For $\alpha \in {\mathbb C}^{k \times m}$, $W \in \Omega_{m}$, $y \in 
\cY^{m}$, $m,k \in {\mathbb N}$ arbitrary, let us introduce the 
notation
\begin{align*}
    \widehat D^{(k)}_{W, \alpha,e} &  = 
\begin{bmatrix} \alpha \cdot Q(W)^{*} \otimes  H(W)^{*}e \\  
\alpha \cdot a(W)^{*} e\end{bmatrix}  \in
\begin{bmatrix}  \cL(\cS, \cR^{k}) \otimes_{\pi}\cX)_{\rm nc} \\ \cY^{k} 
\end{bmatrix}, \\
 R^{(k)}_{W,\alpha, e} &  = 
\begin{bmatrix} \alpha \cdot H(W)^{*} e \\ \alpha \cdot b(W)^{*}e \end{bmatrix}  
 \in \begin{bmatrix}  \cX^{k} \\ \cU^{k} \end{bmatrix}.
\end{align*}
We wish to use the identification map $\iota_{k}$ of the form 
\eqref{idk} to identify the space $(\cL(\cS, \cR^{k}) 
\otimes_{\pi}\cX)_{\rm nc}$ with the space
$\left( (\cL(\cS, \cR) \otimes_{\pi} \cX)_{\rm nc}\right)^{k}$.  
Toward this end we note that
$$
\iota_{k} \colon \alpha \cdot Q(W)^{*} \otimes H(W)^{*} e \mapsto
\begin{bmatrix} \alpha_{1 \cdot} \cdot Q(W)^{*} \otimes H(W)^{*} e \\ 
    \vdots \\ \alpha_{k \cdot} \cdot Q(W)^{*} \otimes H(W)^{*} e 
\end{bmatrix}
$$
where $\alpha_{i \cdot} = \begin{bmatrix} \alpha_{i1} & \cdots & 
\alpha_{im} \end{bmatrix}$ is the $i$-th row of $\alpha$ for $i = 1, 
\dots, k$.  In place of $\widehat D^{(k)}_{W, \alpha, e}$ we now work 
with
$$
 D^{(k)}_{W, \alpha, e} = \begin{bmatrix} \sbm{ \alpha_{1 \cdot} 
 \cdot Q(W)^{*} \otimes H(W)^{*} e \\ \vdots \\ \alpha_{k \cdot} 
 \cdot Q(W)^{*} \otimes H(W)^{*} e } \\ \alpha \cdot a(W)^{*} e 
\end{bmatrix} \in \begin{bmatrix} \left( (\cL(\cS, \cR) \otimes 
\cX)_{\rm nc} \right)^{k} \\ \cY^{k} \end{bmatrix},
$$
or equivalently (by \eqref{LT})
$$
 D^{(k)}_{W, \alpha, e} = \begin{bmatrix} \sbm{ L_{\alpha_{1 \cdot} 
 \cdot Q(W)^{*}}  H(W)^{*} e \\ \vdots \\ L_{\alpha_{k \cdot} 
 \cdot Q(W)^{*}} H(W)^{*} e } \\ \alpha \cdot a(W)^{*} e 
\end{bmatrix} \in \begin{bmatrix} \left( (\cL(\cS, \cR) \otimes 
\cX)_{\rm nc}  \right)^{k} \\ \cY^{k} \end{bmatrix},
$$
and let $\cD^{(k)}$ and $\cR^{(k)}$ be the corresponding sets
\begin{align*}
& \cD^{(k)} = \{  D^{(k)}_{W,\alpha,e} \colon W \in \Omega_{0,m}, 
\alpha \in {\mathbb C}^{k \times m}, e \in \cE^{m}, \, m \in {\mathbb N}\}, \\
& \cR^{(k)} = \{  R^{(k)}_{W,\alpha,e} \colon W \in \Omega_{0,m}, 
\alpha \in {\mathbb C}^{k \times m}, e \in \cE^{m}, \, m \in {\mathbb N}\}.
\end{align*}
Then the content of \eqref{longinnerprod} is that
\begin{equation}   \label{shortinnerprod}
    \left\langle D^{(k)}_{W, \alpha,e}, D^{(k)}_{Z,\beta, e'} 
    \right\rangle_{\left((\cL(\cS, \cR) \otimes_{\pi} \cX)_{\rm 
    nc}\right)^{k} \oplus \cY^{k}} =
    \left\langle R^{(k)}_{W, \alpha,e},\, R^{(k)}_{Z, \beta, e'} 
    \right\rangle_{ \cX^{k} \oplus \cU^{k}}
\end{equation}
for any pairs of vectors $D^{(k)}_{W, \alpha,e}, D^{(k)}_{Z, \beta,e'} \in \cD^{(k)}$ 
and  $R^{(k)}_{W,\alpha,e}, R^{(k)}_{Z, \beta,e'} \in 
\cR^{(k)}$.  It follows that the mapping $\bV^{(k)} \colon 
D^{(k)}_{W,V^{*},e} \mapsto R^{(k)}_{W,V^{*},e}$ extends by linearity 
and taking of limits to a well-defined isometry from
$\bcD^{(k)}: = \overline{\rm span}\, \cD^{(k)}$ onto $\bcR^{(k)} 
:=\overline{\rm span} \,\cR^{(k)}$.

For the moment we are interested only in the special case $k=1$:  
$\bV^{(1)} \colon \bcD^{(1)} \to \bcR^{(1)}$, where
$$
  \bcD^{(1)} \subset \begin{bmatrix} (\cL(\cS, \cR) \otimes_{\pi} 
  \cX)_{\rm nc} \\ \cY \end{bmatrix}, 
  \quad  \bcR^{(1)} \subset \begin{bmatrix} \cX \\ \cU \end{bmatrix}.
$$
Suppose for the moment that
$\bcD^{(1) \perp} = \sbm{ (\cL(\cS, \cR) \otimes_{\pi} \cX)_{\rm nc} \\ \cY } \ominus \bcD^{(1)}$ 
and $\bcR^{(1) \perp} = \sbm{  \cX \\ \cU} \ominus \bcR^{(1)}$ have the same 
dimension. Then we can extend $\bV^{(1)}$ to a unitary operator 
$$
\bU^{*}  = \begin{bmatrix} A^{*} & C^{*} \\ B^{*} & D^{*} 
\end{bmatrix} \colon \begin{bmatrix} (\cL(\cS, \cR) \otimes_{\pi} 
\cX)_{\rm nc} \\ \cY \end{bmatrix}
\to  \begin{bmatrix} \cX \\ \cU \end{bmatrix}.
 $$
 by letting 
$\bU^{*}|_{\bcD^{(1)}} = \bV^{(1)} \colon \bcD^{(1)} \to \bcR^{(1)}$, 
letting $\bU^{*}|_{\bcD^{(1) \perp}}$ be any unitary transformation of 
$\bcD^{(1) \perp}$ onto $\bcR^{(1) \perp}$, and  then using linearity 
to extend to a unitary map $\bU^{*}$ from all of $(\cL(\cS, \cR) 
\otimes_{\pi} \cX)_{\rm nc} \oplus \cY$ 
onto $\cX \oplus \cU$.
Even if the dimensions do not match, we can always artificially 
add on an infinite-dimensional direct-sum 
component $\cX_{0}$ to $\cX$ and to $\cL(\cS, \cR) \otimes_{\pi} \cX$
to make the two codimensions have common 
dimension infinity, and then apply the construction just sketched 
above.  Alternatively, one can get a contractive extension $\bU^{*}$ 
by simply defining $\bU^{*}|_{\bcD^{(1) \perp}} = 0$ and extending by 
linearity; with this last construction one loses conservation of energy 
but avoids inflation of the state-space dimension.

In any case, we have now arrived at a contractive (possibly even 
unitary) colligation matrix $\bU$ as in \eqref{bU} so that
$$
  \bU^{*} =\left.  \begin{bmatrix} A^{*} & C^{*} \\ B^{*} & D^{*} 
\end{bmatrix} \right|_{\bcD^{(1)}} = \bV^{(1)}.
$$
For $k \in {\mathbb N}$ any positive integer, let $\bU^{(k)} = I_{k} \otimes \bU$ as in 
\eqref{coln} (so $\bU^{(1)} = \bU$). Note that here we use the canonical identifications
\begin{equation}   \label{identifications}
\begin{bmatrix} ((\cL(\cS, \cR) \otimes_{\pi} \cX)_{\rm nc})^{k}\\ \cY^{k}\end{bmatrix} \cong 
    \begin{bmatrix} (\cL(\cS, \cR) \otimes_{\pi} \cX)_{\rm nc} \\ \cY \end{bmatrix}^{k}, \quad
\begin{bmatrix}  \cX^{k} \\  \cU^{k} \end{bmatrix} \cong 
\begin{bmatrix} \cX \\ \cU \end{bmatrix}^{k}
\end{equation}
when convenient; the meaning should be 
clear from the context.
Then it is clear that, for each $k \in {\mathbb N}$,
$\bU^{(k)Œ*}$ is a contractive (or even 
unitary) linear transformation of $((\cL(\cS, \cR) \otimes_{\pi} 
\cX)_{\rm nc})^{k} \oplus \cY^{k}$ to 
$\cX^{k} \oplus \cU^{k}$.  

The next goal is to verify that in addition
$\bU^{(k)*}$ extends $\bV^{(k)}$:
\begin{equation}   \label{claim'}
    \bU^{(k)*}|_{\bcD^{(k)}} = \bV^{(k)}.
\end{equation}
We are given that $\bU^{*} \colon D^{(1)}_{W,v^{*},e} \to 
R^{(1)}_{W, v^{*},e}$ for each $W \in \Omega_{m}$, 
$v \in \cA^{m \times 1}$ and $e \in \cE^{m}$.  We must show that
\begin{equation}  \label{claim''}
I_{k} \otimes \bU^{*} \colon  D^{(k)}_{W, \alpha,e} \to 
R^{(k)}_{W, \alpha,e}
\end{equation}
for each $W \in \Omega_{m}$, $\alpha \in {\mathbb C}^{k \times m}$, 
$e \in \cE^{m}$.  Up to the identifications \eqref{identifications}, note that 
we have the identities
$$
 D^{(k)}_{W, \alpha, e} = \begin{bmatrix} D^{(1)}_{W, \alpha_{1 
    \cdot}, e} \\ \vdots \\ D^{(1)}_{W, \alpha_{k \cdot}, e} 
\end{bmatrix}, \quad
R^{(k)}_{W, \alpha, e} = \begin{bmatrix} R^{(1)}_{W, \alpha_{1 
\cdot}, e} \\ \vdots \\ R^{(1)}_{W, \alpha_{k \cdot}, e} 
\end{bmatrix}.
$$
Then by definition we have
\begin{align*}
   & \bU^{(k)*} = I_{k} \otimes \bU^{*} \colon D^{(k)}_{W, \alpha, e} =
    \begin{bmatrix} D^{(1)}_{W, \alpha_{1 \cdot}, e} \\ \vdots \\ 
	D^{(1)}_{W, \alpha_{k \cdot}, e} \end{bmatrix}  \mapsto
\begin{bmatrix} \bU^{*} D^{(1)}_{W, \alpha_{1 \cdot}, e} \\ \vdots \\ 
    \bU^{*} D^{(1)}_{W, \alpha_{k \cdot}, e} \end{bmatrix} \\
    & \quad  = \begin{bmatrix} R^{(1)}_{W, \alpha_{1 \cdot}, e} \\ 
    \vdots \\ R^{(1)}_{W, \alpha_{k \cdot}, e} \end{bmatrix} = 
    R^{(k)}_{W, \alpha, e}
\end{align*}
and hence \eqref{claim''} as well as \eqref{claim'} follow as wanted.

We now look at \eqref{claim'} for the special case where $k=m$ and 
$\alpha = I_{m}$.  
We therefore have
\begin{equation}   \label{Um*extends}
    \begin{bmatrix} A^{(m)*} & C^{(m)*} \\ B^{(m)*} & D^{(m)*} 
    \end{bmatrix} \begin{bmatrix} L_{Q(W)^{*}} H(W)^{*} e \\  a(W)^{*}e 
\end{bmatrix} = 
    \begin{bmatrix} H(W)^{*}e 
     \\ b(W)^{*} e \end{bmatrix}
\end{equation}
where we use the notation in \eqref{coln}.
  From the first block row of 
\eqref{Um*extends} we get
$$
A^{(m)*} L_{Q(W)^{*}} H(W)^{*} e + C^{(m)*} a(W)^{*}e = H(W)^{*}  e.
$$
We rearrange this to
$$
  (I - A^{(m)*}L_{ Q(W)^{*}}) H(W)^{*} e = C^{(m)*}a(W)^{*} e.
$$
As $\bU^{(m)*}$ is contractive, it follows that $A^{(m)*}$ is 
contractive. As $W \in \Omega_{m} \subset {\mathbb D}_{Q}$, 
it follows that $Q(W)$ is strictly contractive; since $\| 
L_{Q(W)^{*}} \| = \| Q(W)^{*}\|$, we see that $L_{Q(W)^{*}}$ is 
strictly contractive as well.    
Hence $I - A^{(m)*} L_{Q(W)^{*}}$ is invertible.  We may then 
solve for $H(W)^{*}e$ to get
$$
  H(W)^{*} e =  (I - A^{(m)*} L_{Q(W)^{*}})^{-1} C^{(m)*}a(W)^{*} e.
$$
Substituting this back into the second block row of 
\eqref{Um*extends} gives us
$$
b(W)^{*} e = B^{(m)*} L_{Q(W)^{*}}  (I - A^{(m)*}L_{ Q(W)^{*}})^{-1} C^{(m)*} 
a(W)^{*} e + D^{(m)*}a(W)^{*} e,
$$
Upon taking adjoints and replacing $W$ with $Z$, we arrive at $b(Z) = 
a(Z) S(Z)$ with $S(Z)$ given by \eqref{transfuncreal}.
\end{proof}
 
 \begin{proof}[Proof of (3) $\Rightarrow$ (1) in Theorem \ref{T:ncInt}:]  
 Let us now assume that we are given graded functions $a,b$ from 
 $\Omega$ to $\cL(\cY, \cE)_{\rm nc}$ and $\cL(\cU, \cE)_{\rm nc}$ 
 respectively along with the contractive (or even unitary) colligation matrix \eqref{bU} 
 so that $S$ given by \eqref{transfuncreal} satisfies the 
 interpolation conditions \eqref{int}.  A key point is that the formula 
 \eqref{transfuncreal} actually makes sense for any $Z \in {\mathbb  D}_{Q,n}$, 
 not just $Z \in \Omega_{n}$.  From the tensored form of the formula 
\eqref{transfuncreal}, one can use the general results from 
\cite{KVV-book} or check directly that $S$ so defined on all of ${\mathbb D}_{Q}$ is 
indeed a nc function.  Moreover, by general principles concerning feedback loading 
of a passive 2-port by a strictly dissipative 1-port  (see e.g.\ \cite{HZ}), it follows 
that the feedback connection is well-posed with resulting closed-loop input/output 
map $S(Z)$ having norm at most 1:
$$
\| S(Z) \| \le 1 \text{ for any } Z \in {\mathbb D}_{Q}
$$
It follows that $S \in \mathcal{SA}_{Q}(\cU, \cY)$.  This completes 
the proof of (3) $\Rightarrow$ (1) in Theorem \ref{T:ncInt}.
\end{proof}

\begin{remark}  \label{R:weakhy} 
    Let us consider the special case where $\Omega = \{ Z^{(0)}\}$ 
    consists of a single point in ${\mathbb D}_{Q}$ say of size $n 
    \times n$ and where $\cE$ is finite-dimensional.  A close look at the 
    proof of (1$^{\prime}$)  $\Rightarrow$ (2)  (in particular, look at
    \eqref{intertwining1}) 
    in the proof of Theorem \ref{T:ncInt} shows that the hypothesis 
    (1$^{\prime}$) can be weakened to the following:
    \begin{enumerate}
\item[(1$^{\prime}_{0}$)] {\em Let $\Omega'_{0} \subset \Omega'$
(where $\Omega' = [\{ Z^{(0)}\}]_{\rm full} \cap {\mathbb D}_{Q}$) 
consist of those points $\widetilde Z \in \Omega'_{m}$ (where $m$ is any 
natural number between $1$ and $n \cdot \dim \cE$) such that there is an injection 
$\cI \in {\mathbb C}^{k \times m}$ so that 
$$
   \cI \widetilde Z = \left( \bigoplus_{1}^{n \cdot \dim \cE} Z^{(0)} \right) \cI.
$$
Assume that $a \in \cT(\Omega'_{0}; \cL(\cY, \cE))$ and $b \in 
\cT(\Omega'_{0}; \cL(\cU, \cE)$ are such that 
$$
   a(Z) a(Z)^{*} - b(Z) b(Z)^{*} \succeq 0
$$
for all $Z \in \Omega'_{0}$.}
\end{enumerate}
Then it is still the case that (1$^{\prime}_{0}$) $\Rightarrow$ (2) 
$\Rightarrow$ (3) $\Rightarrow$ (1) and we can conclude with only 
(1$^{\prime}_{0}$) as the hypothesis that 
there is a Schur-Agler class solution $S \in \mathcal{SA}_{Q}(\cU, 
\cY)$ of the single-point interpolation condition $a(Z^{(0)}) S(Z^{(0)}) = b(Z^{(0)})$.
 \end{remark}

\begin{remark} \label{R:otherimps} The proof of Theorem \ref{T:ncInt} 
verified the equivalence of the statements (1), (1$^{\prime}$), (2), (3) 
by verifying the implications (1) $\Rightarrow$ (1$^{\prime}$) 
$\Rightarrow$ (2) $\Rightarrow$ (3) $\Rightarrow$ (1).  It is instructive to see 
to what extent some of the 
non-adjacent implications can be seen directly. We mention a couple 
of these possibilities.

\smallskip

\noindent
   \textbf{(3) $\Rightarrow$ (2):} We assume that $a,b$ are nc 
   functions on the ${\mathbb D}_{Q}$-relative full nc envelope 
   $\Omega': = [\Omega]_{\rm full} \cap {\mathbb D}_{Q}$ and that 
   $S(Z)$ is given by the transfer-function realization formula 
   \eqref{transfuncreal} on all of ${\mathbb D}_{Q}$.  Then we claim 
   that the kernel $K_{a,b}$ \eqref{Kab} has the Agler decomposition 
   \eqref{Aglerdecom} with cp nc kernel $\Gamma$ given by
   $$
   \Gamma(Z,W)(R) = H(Z) (R \otimes I_{\cX}) H(W)^{*}
   $$
   with $H$ given by
   $$
     H(Z) =a(Z) C^{(n)} (I - L_{Q(Z)^{*}}^{*} A^{(n)})^{-1} \text{ for } Z \in 
     \Omega_{n}.
   $$
   To verify the claim, we first use the coisometry property of 
   $\sbm{A & B \\ C & D}$ to get
   \begin{equation} \label{coisom1}
   \begin{bmatrix} A^{(n)} & B^{(n)} \\ C^{(n)} & D^{(n)} \end{bmatrix}
       \begin{bmatrix} A^{(n)*}  & C^{(n)*} \\ B^{(n)*} & D^{(n)*} 
	   \end{bmatrix} = \begin{bmatrix} I_{((\cL(\cS, \cR) 
	   \otimes_{\pi}\cX)_{\rm nc})^{n}} & 0 \\ 0 & I_{\cY^{n}} \end{bmatrix}.
\end{equation}
Now let $P \in {\mathbb C}^{n \times n}$.  From the tensor-product 
structure of $\sbm{A^{(n)} & B^{(n)} \\ C^{(n)} & D^{(n)}}$ we get 
the intertwining relation
$$
 \begin{bmatrix} P & 0 \\ 0 & P \end{bmatrix}  \cdot \begin{bmatrix} A^{(n)} & B^{(n)} \\ 
     C^{(n)} & D^{(n)} \end{bmatrix} =
 \begin{bmatrix} A^{(n)} & B^{(n)} \\ 
     C^{(n)} & D^{(n)} \end{bmatrix} 
\cdot\begin{bmatrix} P  & 0 \\ 0 & P  \end{bmatrix}.
$$
Multiplication of the identity \eqref{coisom1} on the left by
$\sbm{  P  & 0 \\ 0 & P }$ then leads  to 
\begin{equation}   \label{isom2}
 \begin{bmatrix} A^{(n)} & B^{(n)} \\ C^{(n)} & D^{(n)} \end{bmatrix}
  \cdot    \begin{bmatrix} P  & 0 \\ 0 & P  \end{bmatrix} \cdot
      \begin{bmatrix} A^{(n)*}  & C^{(n)*} \\ B^{(n)*} & D^{(n)*} 
	   \end{bmatrix}  =
 \begin{bmatrix} P  & 0 \\ 0 & P \end{bmatrix} \cdot
     \begin{bmatrix} I_{(( \cL(\cS, \cR) \otimes_{\pi} \cX)_{\rm 
	 nc})^{n}} 
	 & 0 \\ 0 & I_{\cY^{n}} \end{bmatrix}.
 \end{equation}
 In the sequel to save space we suppress the $\cdot$ in the notation for the 
 multiplication of a scalar matrix times an operator matrix of 
 compatible block size.
  From the identity \eqref{isom2} we get the set of relations
\begin{align*}
  &  A^{(n)} P A^{(n)*} + B^{(n)} P  B^{(n)*} = P \otimes I_{ (\cL(\cS, 
    \cR) \otimes_{\pi} \cX)_{\rm nc}}, \\
 &  A^{(n)} P  C^{(n)*} + B^{(n)}  P D^{(n)*} = 0, \\
& C^{(n)} P  A^{(n)*} + D^{(n)}  P  B^{(n)*} = 0, \\
& C^{(n)}  P  C^{(n)*} + D^{(n)} P \cdot D^{(n)*} = P  \otimes I_{\cY}.
\end{align*}
We make use of these identities to then compute, for $Z,W \in \Omega_{n}$, $P \in {\mathbb C}^{n 
\times n}$,
\begin{align*}
  &  P \otimes I_{\cY}   - S(Z) P  S(W)^{*} =  
P \otimes I_{\cY} - (D^{(n)} + C^{(n)} (I - L_{Q(Z)^{*}}^{*} A^{(n)})^{-1}
L_{Q(Z)^{*}}^{*} B^{(n)})   P   \cdot \\
& \quad \quad \quad \quad \cdot  (D^{(n)*} + B^{(n)*} L_{Q(W)^{*}} (I - A^{(n)*} 
L_{Q(W)^{*}})^{-1} C^{(n)*})\\
& = P \otimes I_{\cY} - D^{(n)}   P   D^{(n)*} - 
C^{(n)} (I - L_{Q(Z)^{*}}^{*} A^{(n)})^{-1} L_{Q(Z)^{*}}^{*} B^{(n)} 
 P  D^{(n)*} \\
& \quad  -D^{(n)} P B^{(n)*}L_{Q(W)^{*}}(I - A^{(n)*} 
P_{Q(W)^{*}})^{-1} C^{(n)*}  \\
& \quad - C^{(n)} (I - L_{Q(Z)^{*}}^{*} A^{(n)})^{-1} L_{Q(Z)^{*}}^{*} B^{(n)} P B^{(n)*} 
L_{Q(W)^{*}}(I - A^{(n)*} L_{Q(W)^{*}})^{-1} C^{(n)*}  \\
& = C^{(n)} P C^{(n)*} + C^{(n)} (I - L_{Q(Z)^{*}}^{*} A^{(n)})^{-1} 
L_{Q(Z)^{*}}^{*} A^{(n)} P C^{(n)*} \\
&   + C^{(n)} P A^{(n)*} L_{Q(W)^{*}} (I - A^{(n)*} L_{Q(W)^{*}})^{-1} 
C^{(n)*} +  \\ 
&   C^{(n)} (I - L_{Q(Z)^{*}}^{*} A^{(n)})^{-1}L_{ Q(Z)^{*}}^{*}( A^{(n)} P A^{(n)*} - P) 
L_{Q(W)^{*}} (I - A^{(n)*}L_{Q(W)^{*}})^{-1} C^{(n)*} \\
& = C^{(n)} (I - L_{Q(Z)^{*}}^{*} A^{(n)})^{-1} X (I - A^{(n)*} L_{Q(W)^{*}})^{-1} C^{(n) *}
\end{align*}
where
\begin{align*}
    X  & = (I - L_{Q(Z)^{*}}^{*} A^{(n)}) P (I - A^{(n)*} L_{Q(W)^{*}}) +
    L_{Q(Z)^{*}}^{*} A^{(n)}  P (I 
    - A^{(n)*} L_{Q(W)^{*}}) \\
    & \quad \quad  + (I - L_{Q(Z)^{*}}^{*} A^{(n)}) P A^{(n)*} L_{Q(W)^{*}}
    + L_{Q(Z)^{*}}^{*} (A^{(n)} P A^{(n)*} - P) L_{Q(W)^{*}}  \\
    & = P \otimes I_{\cX} - L_{Q(Z)^{*}}^{*} P L_{Q(W)^{*}},
\end{align*}
and it follows that 
\begin{equation}  \label{isom3}
  P \otimes I_{\cY} - S(Z) P S(W)^{*} =  H_{0}(Z) (P \otimes I_{\cX} 
  - L_{Q(Z)^{*}}^{*} P L_{Q(W)^{*}})H_{0}(W)^{*}
\end{equation}
with 
$$
  H_{0}(Z) =  C^{(n)} (I - L_{Q(Z)^{*}}^{*} A^{(n)})^{-1}.
$$
As a consequence of the identity \eqref{ncLT*}, we see that
$$
 L_{Q(Z)^{*}}^{*} P L_{Q(W)^{*}} = 
 ({\rm id}_{{\mathbb C}^{n \times n}} \otimes \pi)(Q(Z) P Q(W)^{*})
$$
and hence \eqref{isom3} can be rewritten as
\begin{equation}  \label{isom4}
  P \otimes I_{\cY} - S(Z) P S(W)^{*} =  H_{0}(Z) ({\rm id}_{{\mathbb 
  C}^{n \times n}} \otimes \pi)(P \otimes I_{\cX} 
  - Q(Z) PQ(W)^{*})H_{0}(W)^{*}.
\end{equation}

If the interpolation conditions \eqref{int} hold on $\Omega$, then 
they continue to hold on $\Omega'$ by uniqueness of nc-function 
extensions to full nc envelopes (see Proposition \ref{P:extcpncker}).
If we multiply \eqref{isom4} by $a(Z)$ on the left and by $a(W)^{*}$ 
on the right and use the interpolation condition \eqref{int} we 
arrive at the desired Agler decomposition \eqref{Aglerdecom} for the 
kernel $K_{a,b}$
$$
K_{a,b}(Z,W) = H(Z) ({\rm id}_{{\mathbb C}^{n \times n}} \otimes \pi)
(P \otimes I_{\cS} - Q(Z) (P \otimes I_{\cR}) Q(W)^{*})H(W)^{*}
$$
with $H(Z): = a(Z) H_{0}(Z)$, 
at least for the square case where $Z \in \Omega'_{n}$, $W \in 
\Omega'_{n}$ and $P \in {\mathbb C}^{n \times n}$.  To then arrive at 
the same identity with $Z \in \Omega_{n}$, $W \in \Omega_{m}$, $P \in 
{\mathbb C}^{n \times m}$ with possibly $n \ne m$,  one can apply the 
square case to the situation where $\sbm{ Z & 0 \\ 0 & W} \in 
\Omega'_{n+m}$ replaces both $Z$ and $W$ and $\sbm{ 0 & P \\ 0 & 0}$ 
replaces $P$, and then proceed as is done at the end of Section 
\ref{S:unenhanced}. 

\smallskip 

\noindent
  \textbf{(2) $\Rightarrow$ (1$^{\prime}$):}  If we assume that (2) 
  holds in the strengthened form obtained in the previous paragraph, 
  namely that there is a cp nc kernel $K \colon \Omega' \times 
  \Omega' \to \cL(\cL(\cS), \cL(\cE))_{\rm nc}$ so that 
  \eqref{Aglerdecom} holds for $Z, W \in \Omega'$, then we can obtain 
  condition (1$^{\prime}$) easily by setting $Z = W \in \Omega'$ in 
  \eqref{Aglerdecom} and using that  
    $\| Q(Z) \| < 1$   along with the fact that $\Gamma$ is a cp nc 
    kernel on $\Omega' \times \Omega'$.  It is 
    not at all obvious how to get (1$^{\prime}$) directly from (2) if 
    one only assumes that the Agler decomposition \eqref{Aglerdecom} 
    holds on $\Omega$ rather than on all of $\Omega'$; this is part 
    of the content of (2) $\Rightarrow$ (3) $\Rightarrow$ (1) 
    $\Rightarrow$ (1$^{\prime}$).
  \end{remark}

\begin{remark}  \label{R:AKV}
    Consider the special case described in Example \ref{E:ncpoly}. We 
    assume that the nc function $Q$ is a matrix polynomial (a nc 
    function with values in $\cL(\cR, \cS)_{\rm nc} = \cL({\mathbb 
    C}^{s}, {\mathbb C}^{r})_{\rm nc}$).
   We identify an associated Schur-Agler class 
    function $S$ with its power series representation:
\begin{equation}   \label{series}
S(Z) = \sum_{\fa \in \free} S_{\fa} \otimes Z^{\alpha}.
\end{equation}
Here the coefficients $S_{\alpha}$ are operators from $\cU$ to $\cY$ 
and the point $Z = (Z_{1}, \dots, Z_{d}) \in {\mathbb D}_{Q}$, where 
${\mathbb D}_{Q}$ consists of $d$-tuples of $n \times n$ matrices 
over ${\mathbb C}$  such that $\| Q(Z) \| < 1$. 
   
In the original work of \cite{BGM2} the nc Schur-Agler 
class $\cSA_{Q}(\cU, \cY)$ (for the special case of linear $Q(z) = 
\sum_{k=1}^{d} L_{k} z_{k}$ with coefficient matrices $L_{k}$ 
satisfying some additional conditions which we need not go into here)
was defined to consist of noncommutative functions $S$ such that $\| 
S(Z) \| \le 1$ whenever $Z = (Z_{1}, \dots, Z_{d})$ is a $d$-tuple of 
operators on a fixed separable infinite-dimensional Hilbert space 
$\cK$ such that $\| Q(Z) \| < 1$ (where one can use the assumed nc 
power series representation for $S$ to define the functional calculus 
$Z \mapsto S(Z)$).  In general let us define the class 
$\cSA^{\infty}_{Q}(\cU, \cY)$ to consist of those functions $S \in 
\cSA_{Q}(\cU, \cY)$ such that $\| S(Z) \| \le 1$ whenever 
$Z \in \cL(\cK)^{d}$ with $\| Q(Z) \| < 1$.  Note that a 
priori from the definitions we have the containment 
$\cSA^{\infty}_{Q}(\cU, \cY) \subset \cSA_{Q}(\cU, \cY)$ (identify a 
matrix tuple $Z = (Z_{1}, \dots, Z_{d}) \in ({\mathbb C}^{n \times 
n})^{d}$ with an operator-tuple in $(\cL(\cK))^{d}$ by 
using some fixed orthonormal basis of $\cK$ to embed ${\mathbb 
C}^{n}$ into $\cK$).  We note that 
in the proof of (4) $\Rightarrow$ (1) in Theorem \ref{T:ncInt}, the 
function $S$ given by the realization formula \eqref{transfuncreal} 
is actually in the a priori smaller class $\cSA^{\infty}_{Q}(\cU, \cY)$. 
We hence arrive at the following corollary of the proof of Theorem \ref{T:ncInt}.  
A direct proof of 
this result for the special case where $Q(z) = \sbm{ z_{1} & & \\ & \ddots & \\ & & z_{d}}$ is 
given in \cite[Theorem 6.6]{AKV}.  

\smallskip

\noindent
\textbf{Corollary 4.5.1}    {\sl The classes $\cSA^{\infty}_{Q}(\cU, \cY)$ and 
    $\cSA_{Q}(\cU, \cY)$ coincide.}
    
\smallskip 

It appears that this result is new even for the commutative case. 
More precisely, let $G$ be an open subset of ${\mathbb C}^{d}$,
take the set $\Xi^{\infty}$ to consist of commutative $d$-tuples  $Z  = 
(Z_{1}, \dots, Z_{d})$ of operators on some fixed 
infinite-dimensional separable Hilbert space $\cK$ such that the 
Taylor spectrum $\sigma(Z)$ of $Z$ lies in $G$, let $\cR$ to be a 
collection $\{F_{k} \colon k \in I\}$ of matrix valued functions 
$F_{k}(z) = [f_{k,ij}(z)]$ (where the row and 
column indices for each $F_{k}$ have finite ranges, say $1 \le i \le 
m_{k}$ and $1 \le j \le n_{k}$) such that $\| F_{k}(z) \| < 1$ for $z 
\in G$, and set 
\begin{align*}
 &  q(z) = {\rm diag}_{k \in I} [F_{k}(z) ], \\
& {\mathbb D}^{\infty}_{q} = \{ Z \in \Xi_{\infty} \colon \| q(Z) 
\| < 1\}, \\ 
& \cC {\mathcal SA}^{\infty}_{q}(\cU, \cY) = \{ S \colon 
  {\mathbb D}_{q} \underset{\rm holo} \to \cL(\cU, \cY) \colon \|S(Z) \| \le 1 
  \text{ for } Z \in {\mathbb D}^{\infty}_{q} \},
\end{align*}
where $Z \mapsto S(Z)$ is defined via the Taylor functional calculus.
Then, as is explained in \cite{MP}, restricting now to the scalar 
case $\cU = \cY = {\mathbb C}$, 
$\cC \mathcal{SA}^{\infty}_{q}({\mathbb C})$ is the unit ball of the 
 dual operator algebra $\widetilde \cA_{\cR}$ associated 
with $\cR$ as in \cite{MP}.  Then results of \cite{MP} (using 
operator-algebra techniques quite different from ours) establish the 
commutative Agler-decomposition \eqref{Aglerdecom-a=I} (called {\em 
Agler factorization} in \cite{MP}) for the class
$\cC \mathcal{SA}^{\infty}_{q}({\mathbb C})$.  The results of \cite{MP} 
also establish that the algebra $\widetilde \cA_{\cR}$ is 
\textbf{residually finite-dimensional}, i.e.,  the norm of 
an $n \times n$-matrix $[a_{ij}]$ over $\widetilde A$ is given by
$ \| [a_{ij}] \|_{\widetilde \cA^{n \times n}} = \sup\{\| [ \pi(a_{ij}] 
\| \}$ where the supremum is over all finite-dimensional representations 
$\pi$ of $\cA$.  While this proves that 
\begin{equation}  \label{c-infin=c-fin}
\cC \mathcal{SA}^{\infty}_{q}({\mathbb C}) = \cC \mathcal{SA}_{q}({\mathbb C})
\end{equation}
in many cases, there are examples of 
$\cR$  for which not every finite-dimensional representation $\pi$ is 
of the form $f \mapsto f(Z)$ for a commutative finite-matrix tuple $Z 
= (Z_{1}, \dots, Z_{d}) \in {\mathbb D}_{q}$ (see \cite[Example 
5.5]{MP}).  A question asked in \cite{MP} is whether the identity  
\eqref{c-infin=c-fin} holds in general.  As a result of  Corollary 
4.5.1, we conclude that the answer is indeed yes.
\end{remark}

\section{Multipliers between nc reproducing kernel 
Hilbert spaces}  \label{S:mult}

\subsection{Review of nc reproducing kernel Hilbert spaces and 
contractive multipliers}

Suppose that we are given two cp nc kernels $K'$ and 
$K$, both defined on the Cartesian product ${\mathbb D} \times {\mathbb D}$ of 
a nc set ${\mathbb D} \subset \cV_{\rm nc}$ where $K'$ has values in $\cL(\cA, \cL(\cU))_{\rm nc}$ 
while $K$ has values in $\cL(\cA, \cL(\cY))_{\rm nc}$ (here $\cA$ is a 
$C^{*}$-algebra and $\cU$ and $\cY$ are two 
auxiliary Hilbert spaces).  Suppose next that $S$ is a nc function 
on ${\mathbb D}$ with values in $\cL(\cU, \cY)_{\rm nc}$.  We say that $S$ 
is a \textbf{multiplier} from $\cH(K')$ to $\cH(K)$, 
written as $S \in \cM(K',K)$, if the operator $M_{S}$ given by
$$
  (M_{S} f)(W) = S(W) f(W) \text{ for }  W \in \Omega_{m}
$$
maps $\cH(K')$ boundedly into $\cH(K)$. We can always multiply such 
an $S$ by a positive scalar to arrange that $\| M_{S} \| \le 1$ in 
which case we say that $S$ is a {\em contractive multiplier} and write
$S \in \overline{\cB}\cM(K',K)$.

The following result appears in \cite{BMV1}.

\begin{theorem}  \label{T:contmult}
    Given nc kernels $K'$ and $K$ from ${\mathbb D} \times {\mathbb D}$ 
    to  $\cL(\cA, \cL(\cU))_{\rm nc}$ and $\cL(\cA, 
    \cL(\cY))_{\rm nc}$ respectively and given a nc function 
    $S$ from ${\mathbb D}$ to $\cL(\cU, \cY)_{\rm nc}$, the following are 
    equivalent:
    \begin{enumerate}
	\item $S \in \overline{\cB}\cM(K',K)$.
	
	\item The de Branges-Rovnyak kernel $K^{\dBR}_{S}$ from ${\mathbb D}$ to $\cL(\cA, 
	\cL(\cY))_{\rm nc}$ associated with $S$ given by
\begin{equation}   \label{KS}
  K^{\dBR}_{S}(Z,W)(P) = K(Z,W)(P) - S(Z) \, K'(Z,W)(P)\, S(W)^{*}
\end{equation}
is a cp nc kernel on ${\mathbb D}$.
\end{enumerate}
\end{theorem}

\subsection{Interpolation by contractive multipliers}

 Let us define the \textbf{nc Pick interpolation problem} as follows.  
 
 \smallskip
 
 \noindent
 \textbf{Noncommutative Pick Interpolation Problem:}
  {\sl  Given cp nc kernels $K'$ and $K$ on ${\mathbb D} \times  
  {\mathbb D}$ to respectively $\cL(\cA, \cL(\cU))_{\rm nc}$ and $\cL(\cA, 
  \cL(\cY))_{\rm nc}$, and given a subset 
 $\Omega$ of ${\mathbb D}$ together with a nc function $S_{0}$ from $\Omega$ 
 into $\cL(\cU, \cY)_{\rm nc}$, find a nc function $S$ from all of 
 ${\mathbb D}$ into 
 $\cL(\cU, \cY)_{\rm nc}$ so that 
 \begin{enumerate}  
     \item $S|_{\Omega} = S_{0}$, and
     \item  $S$ is in the 
 contractive multiplier class $\overline{\cB}\cM(K',K)$.
 \end{enumerate}}
 
 \smallskip 
  
  An immediate consequence of Theorem \ref{T:contmult} is the following result 
 giving a necessary condition for solvability of the nc Pick 
 interpolation problem.
 
 \begin{theorem} \label{T:ncP-nec}
     Suppose that $K$, $K'$, $\Omega \subset {\mathbb D}$ and $S_{0}$ 
     form the data set for a solvable nc Pick interpolation problem. Then the 
     associated de Branges-Rovnyak kernel $K^{\dBR}_{S_{0}}$ from 
     $\Omega \times \Omega$ to $\cL(\cA, \cL(\cY))_{\rm nc}$ is a  cp 
     nc kernel on $\Omega$.
 \end{theorem}
 
 \begin{proof} By Theorem \ref{T:contmult}.
If $S$ is a solution of the nc Pick interpolation problem, then 
$K^{\dBR}_{S}$ is a cp nc 
kernel on ${\mathbb D}$, whence it follows that its restriction 
     $K^{\dBR}_{S_{0}}$ to the nc subset $\Omega \times \Omega$ 
     is a cp nc kernel on $\Omega$.
 \end{proof}
 
  Let us now restrict our considerations to the case where $K' = 
 I_{\cU} \otimes k$ and $K = I_{\cY} \otimes k$ where $k$ is a scalar 
 cp nc kernel
 $$
 k \colon \Omega \times \Omega \to \cL(\cA, 
 \cL({\mathbb C}))_{\rm nc}.
 $$
Inspired by the commutative case (see 
 \cite{AMcC-book}) we make the following definition.
 
 \begin{definition} \label{D:completePickker} We say that the cp 
    nc scalar kernel $k$ on  ${\mathbb D} \times {\mathbb D}$ is a 
     \textbf{complete Pick nc kernel} if:  
     given any coefficient Hilbert spaces $\cU$ and $\cY$ together with a 
     subset $\Omega$ of ${\mathbb D}$ and a graded function $S_{0}$ from 
     $\Omega$ to $\cL(\cU, \cY)_{\rm nc}$, the necessary 
     condition for existence of a solution of the associated nc Pick 
     interpolation problem given by Theorem \ref{T:ncP-nec} is also 
     sufficient, i.e., {\em the nc Pick interpolation problem 
     has a solution if and only if the associated kernel
\begin{equation}   \label{KS0}
 K_{S_{0}}(Z,W)(P) = 
  k(Z,W)(P) \otimes I_{\cY} - S_{0}(Z) \left( k(Z,W)(P)   \otimes 
  I_{\cU} \right)  S_{0}(W)^{*}
\end{equation}
is a cp kernel on $\Omega$.}
\end{definition}

We now give a general example of a scalar kernel which turns out to 
be a complete Pick nc kernel. Specialize the general situation 
described in Subsection \ref{S:ncdisk} to the case where
 the target space $\cS$ for the nc function $Q$ is taken to be $\cS = 
 {\mathbb C}$.  Let us now use the notation $Q_{0}$ for the nc function on 
$\Xi$ with values in $\cL(\cR, {\mathbb C}_{\rm nc})$ defining the nc 
domain ${\mathbb D}_{Q_{0}}$ as in \eqref{defDQ}. We mention that the $Q_{0}$ 
as in \eqref{defQ0} appearing in the proof of Lemma \ref{L:cone} part (2) above is an 
example of such a $Q_{0}$.  With notation as in the proof of Lemma \ref{L:cone} part (2), 
define a kernel $k_{Q_{0}}$ on ${\mathbb D}_{Q_{0}}$ by
\begin{equation}   \label{kQ0}
    k_{Q_{0}}(Z,W)(P) = \sum_{k=0}^{\infty} (L_{Q_{0}(Z)^{*}})^{(k)*} 
    (P \otimes I_{\cR^{\otimes k}})  (L_{Q_{0}(W)^{*}})^{(k)}.
\end{equation}
The following identities will be quite useful.

\begin{proposition}   \label{P:identities}
    The scalar kernel $k_{Q_{0}}$  \eqref{kQ0} satisfies the following identities:
 \begin{align}
   &  k_{Q_{0}}(Z,W)(P) - Q_{0}(Z) \big( k_{Q_{0}}(Z,W)(P) \otimes 
     I_{\cR} \big)  Q_{0}(W)^{*} = P, \label{id1} \\
 & k_{Q_{0}}(Z,W) \big( P - Q_{0}(Z) (P \otimes I_{\cR}) 
 Q_{0}(W)^{*} \big) = P.  \label{id2}
 \end{align}
 \end{proposition}
 
 \begin{proof}  To verify \eqref{id1}, note that
 \begin{align*} & Q_{0}(Z) \left( k_{Q_{0}}(Z,W)(P) \otimes I_{\cR} 
     \right) Q_{0}(W)^{*}   \\
& = \sum_{k=0}^{\infty} Q_{0}(Z) \left( 
     \left[ (L_{Q(Z)^{*}})^{(k)*} (P \otimes I_{\cR^{\otimes k}}) 
     (L_{Q(W)^{*}})^{(k)} \right] \otimes I_{\cR} \right) 
     Q_{0}(W)^{*} \\
 & = \sum_{k=0}^{\infty} Q_{0}(Z) \left( (L_{Q_{0}(Z)^{*}})^{(k)*} 
 \otimes I_{\cR}\right) \left( P \otimes I_{\cR^{\otimes k+1}} 
 \right) \left( (L_{Q_{0}(W)^{*}})^{(k)} \otimes I_{\cR} \right) 
 Q_{0}(W)^{*}  \\
 & = \sum_{k=0}^{\infty} (L_{Q_{0}(Z)^{*}})^{(k+1)*} (P \otimes 
 I_{\cR^{\otimes k+1}}) (L_{Q_{0}(W)^{*}})^{(k+1)}
 = k_{Q_{0}}(Z,W)(P) - P
     \end{align*}
  and \eqref{id1} follows.
  
  Similarly note that
\begin{align*}
 &   k_{Q_{0}}(Z,W)\big( Q_{0}(Z) (P \otimes I_{\cR}) Q_{0}(W)^{*} 
    \big) \\
    & = \sum_{k=0}^{\infty} (L_{Q_{0}(Z)^{*}})^{(k)*} 
 \left( Q_{0}(Z) (P \otimes I_{\cR}) Q_{0}(W)^{*} \otimes 
 I_{\cR^{\otimes k}} \right) (L_{Q_{0}(W)^{*}})^{(k)} = \\
&  \sum_{k=0}^{\infty} (L_{Q_{0}(Z)^{*}})^{(k)*} ( Q_{0}(Z) 
\otimes I_{\cR^{\otimes k}} ) ( P \otimes I_{\cR^{\otimes 
k+1}})( Q_{0}(W)^{*} \otimes I_{\cR^{\otimes k}} ) 
(L_{Q_{0}(W)^{*}})^{(k)}  \\
& = \sum_{k=0}^{\infty} \left( L_{Q_{0}(Z)^{*}} \right)^{(k+1)*} (P 
\otimes I_{\cR^{\otimes k+1}} )  \left( L_{Q_{0}(W)^{*}}\right)^{(k+1)}
= k_{Q_{0}}(Z,W)(P) - P
 \end{align*}
verifying \eqref{id2}
  \end{proof}
  
Since the additional work is minimal, we shall consider the nc Pick 
interpolation problem for a pair of kernels $(k_{Q_{0}} \otimes 
I_{\cU}, k_{Q_{0}} \otimes I_{\cY})$ in a more general 
left-tangential formulation:

\smallskip

\noindent
\textbf{Left-tangential nc Pick Interpolation Problem:} 
{\sl Given a nc function $Q_{0}$ from $\Xi$ into $\cL(\cR, {\mathbb 
C})_{\rm nc}$ defining the nc domain 
${\mathbb D}_{Q_{0}} \subset \Xi$ as above along with 
a subset $\Omega$ of ${\mathbb D}_{Q_{0}}$
together with nc functions $a \in \cT(\Omega'; \cL(\cY, \cE)_{\rm nc})$ and 
$b \in \cT(\Omega'; \cL(\cU, \cE)_{\rm nc})$ for auxiliary Hilbert 
spaces $\cU$, $\cY$, $\cE$, where we set $\Omega'$ equal to the ${\mathbb 
D}_{Q_{0}}$-relative full nc envelope of $\Omega$
\begin{equation}   \label{Omega'0}
\Omega' = \Omega_{\rm nc, full} \cap {\mathbb D}_{Q_{0}},
 \end{equation}
find a function $S$ in the contractive multiplier class 
$\overline{\cB}\cM(k_{Q_{0}} \otimes I_{\cU}, k_{Q_{0}} \otimes 
I_{\cY})$ so that}
\begin{equation} \label{multtanint}
    a(Z) S(Z) = b(Z) \text{ for all } Z \in \Omega.
\end{equation}

\smallskip

\noindent
Note that the special choice $a(Z) = I_{\cE}$ and $b(Z) = S_{0}(Z)$, 
reduces the Left-Tangential nc Pick Interpolation Problem   to 
the nc Pick Interpolation Problem.

By combining Proposition \ref{P:identities} with Theorem 
\ref{T:ncInt}, we arrive at the following solution of the 
Left-Tangential nc Pick Interpolation Problem.

\begin{theorem}  \label{T:multtanint}
    Suppose that $Q_{0}$, $\Omega$, $a$, $b$ are data for a 
    Left-Tangential nc Pick Interpolation Problem as above.  Then the 
    following are equivalent:
    \begin{enumerate}
	\item[(1)] The Left-Tangential nc Pick Interpolation Problem has a 
	solution, i.e., there exists an $S \colon {\mathbb D}_{Q_{0}} 
	\to \cL(\cU, \cY)_{\rm nc}$ in the contractive multiplier 
	class  $\overline{\cB} \cM(k_{0} \otimes I_{\cU}, k_{0}\otimes I_{\cY})$ 
	satisfying the left tangential interpolation conditions \eqref{multtanint}.
	
	\item[(1$^{\prime}$)] The inequality 
\begin{equation}   \label{ineq-mult}
     a(Z)  a(Z)^{*} -  b(Z) b(Z)^{*} \succeq 0
\end{equation}
holds for all $Z \in \Omega'$.

\item[(2)]  The generalized de Branges-Rovnyak kernel $K^{\dBR}_{a,b}$ given by
\begin{align}
&  K^{\dBR}_{a,b}(Z,W)(P)  =\notag \\
& \quad a(Z) \left( k_{Q_{0}}(Z,W)(P) \otimes I_{\cY} 
 \right) a(W)^{*} - b(Z) \left( k_{Q_{0}}(Z,W)(P) \otimes I_{\cU} 
 \right) b(W)^{*}
 \label{KgenSz}
\end{align}
is a cp nc kernel on $\Omega$.

\item[(3)] There exists an auxiliary Hilbert space $\cX$ and a contractive 
(even unitary) colligation matrix 
\begin{equation}   \label{bU-mult}
    \bU = \begin{bmatrix} A & B \\ C & D \end{bmatrix} \colon 
    \begin{bmatrix}   \cX  \\ \cU \end{bmatrix} \to 
 \begin{bmatrix} \cR \otimes \cX \\ \cY \end{bmatrix}
\end{equation} so that the function $S$ defined by
\begin{equation}   \label{transfunc-mult}
    S(Z) = D^{(n)} + C^{(n)} (I - (Q_{0}(Z) \otimes I_{\cX}) 
    A^{(n)})^{-1} (Q_{0}(Z) \otimes I_{\cX}) B^{(n)}
\end{equation}
for $Z \in \Omega_{n}$,
where
$$
\begin{bmatrix} A^{(n)} & B^{(n)} \\ C^{(n)} & D^{(n)} \end{bmatrix}
= \begin{bmatrix} I_{n} \otimes A & I_{n} \otimes B \\ I_{n} \otimes 
C & I_{n} \otimes D \end{bmatrix} \colon \begin{bmatrix} 
\cX^{n} \\ \cU^{n} \end{bmatrix} \to \begin{bmatrix} 
(\cR \otimes \cX )^{n} \\ \cY^{n} \end{bmatrix}.
$$
satisfies the Left-Tangential Interpolation condition \eqref{multtanint} on 
$\Omega$.
\end{enumerate}

Moreover, the implications (2) $\Rightarrow$ (3) and (3) 
$\Rightarrow$ (1) hold under the 
weaker assumption that $a$ and $b$ are only graded (rather than nc) 
functions defined only from 
$\Omega$ to $\cL(\cY, \cE)_{\rm nc}$ and $\cL(\cU, \cE)_{\rm nc}$ 
respectively.
\end{theorem}

\begin{proof}  When Theorem \ref{T:ncInt} is specialized to the case 
    $Q = Q_{0}$, it is clear that condition (1$^{\prime}$) in Theorem 
    \ref{T:ncInt} coincides exactly with condition (1$^{\prime}$) in Theorem \ref{T:multtanint}. 
    Furthermore, as a consequence of Remark \ref{T:ncfindim} and 
    Theorem \ref{T:ncfindim}, condition (3) in Theorem \ref{T:ncInt} 
    simplifies exactly to condition (3) in Theorem \ref{T:multtanint} 
    (as also observed in the proof of  Theorem \ref{T:OW} above).
    We shall now show that 
   (i) conditions (2) in Theorem \ref{T:ncInt} (for the case $Q = 
    Q_{0}$) and in Theorem \ref{T:multtanint} coincide and (ii)  
    conditions (1) in  Theorem \ref{T:ncInt} (for the case $Q = 
    Q_{0}$) and in Theorem \ref{T:multtanint} coincide.  Once these 
    correspondences are established, Theorem \ref{T:multtanint} 
    follows as an immediate consequence of Theorem \ref{T:ncInt}.
    
    \smallskip
    
 \noindent
 \textbf{Equivalence of conditions (2).}
    A consequence of the identity \eqref{id2} is that
\begin{align*}
  &  K^{\dBR}_{a,b}(Z,W) \left( P - Q_{0}(Z) (P \otimes I_{\cR}) 
    Q_{0}(W)^{*} \right)\\
    & \quad = a(Z) (P \otimes I_{\cY}) a(W)^{*} - b(Z) 
    (P \otimes I_{\cU}) b(W)^{*} \\
 & \quad = K_{a,b}(Z,W)(P).
\end{align*}
Thus $K^{\dBR}_{a,b}$ being a cp nc kernel implies that $K_{a,b}$ 
has an Agler decomposition \eqref{Aglerdecom}  (with
$\Gamma = K^{\dBR}_{a,b}$ and with $Q_{0}$ in place of $Q$). 
Conversely, suppose that $K_{a,b}$ has an Agler decomposition 
\eqref{Aglerdecom} (with respect to $Q_{0}$ rather than  $Q$).
Then there is a cp nc kernel $\Gamma$ so that
\begin{align}
   & a(Z) (P \otimes I_{\cY}) a(W)^{*} - b(Z) (P \otimes I_{\cU}) 
    b(W)^{*} \notag \\
& \quad = \Gamma(Z,W)\left( P - Q_{0}(Z) (P \otimes 
I_{\cR}) Q_{0}(W)^{*} \right).
\label{Gamma-prop}
\end{align}
Then 
\begin{align*}
 &  K^{\dBR}_{a,b}(Z,W)(P)  \\
& =  a(Z) \left(k_{Q_{0}}(Z,W)(P) \otimes I_{\cY}) \right) a(W)^{*} - 
     b(Z) \left( k_{Q_{0}}(Z,W)(P) \otimes I_{\cU} \right) b(W)^{*} \\ 
 &  =  \Gamma(Z,W) \left( k_{Q_{0}}(Z,W)(P)  -
 Q_{0}(Z) (k_{Q_{0}}(Z,W)(P) \otimes I_{\cR}) Q_{0}(W)^{*} \right) \\
 & \quad \text{(by \eqref{Gamma-prop} with $K_{Q_{0}}(Z,W)(P)$ in 
 place of $P$)} \\
 &  = \Gamma(Z,W)(P) \text{ (by \eqref{id1})}
 \end{align*}
implying that $K^{\dBR}_{a,b} = \Gamma$ is a cp nc kernel. 
The equivalence of the respective conditions (2) now follows.

\smallskip

\noindent
 \textbf{Equivalence of conditions (1):}  Note the interpolation 
 condition \eqref{int} in Theorem \ref{T:ncInt} coincides with the 
 interpolation condition \eqref{multtanint} required in Theorem 
 \ref{T:multtanint}.  To establish the equivalence of the respective 
 conditions (1), it suffices to show that the Schur-Agler class 
 ${\mathcal SA}_{Q_{0}}(\cU, \cY)$ coincides with the contractive 
 multiplier class $\overline{\cB}\cM(k_{Q_{0}} \otimes I_{\cU}, 
 k_{Q_{0}} \otimes I_{\cY})$.

By Theorem \ref{T:contmult}, the nc function $S$ is in the 
contractive multiplier class $\overline{\cB}\cM(k_{Q_{0}} \otimes 
I_{\cU}, k_{Q_{0}} \otimes I_{\cY})$ if and only if the associated 
de Branges-Rovnyak kernel
$$
 K^{\dBR}_{S}(Z,W)(P): =  k_{Q_{0}}(Z,W)(P) \otimes I_{\cY} - S(Z)
  (k_{Q_{0}}(Z,W)(P) \otimes I_{\cU}) S(W)^{*}
$$
is a cp nc kernel on ${\mathbb D}_{Q_{0}}$.  On the other hand, by 
the equivalence of (1) and (2) in Theorem \ref{T:ncInt} for the 
special case $\Omega =  {\mathbb D}_{Q_{0}}$, $\cE = \cY$, $a(Z) = 
I_{\cY^{n}}$ for $Z \in {\mathbb D}_{Q_{0},n}$ and $b(Z) = S(Z)$, we see 
that $S$ is in the Schur-Agler class ${\mathcal SA}_{Q_{0}}(\cU, 
\cY)$ if and only if the kernel $K_{S}$ given by
$$
  K_{S}(Z,W)(P): = P \otimes I_{\cY} - S(Z) (P \otimes I_{\cU}) S(W)^{*}
$$
has an Agler decomposition \eqref{Aglerdecom} (with $Q_{0}$ in place 
of $Q$).  By the calculation in the proof of the equivalence of 
conditions (2) (now specialized to the setting where  
$\Omega =  {\mathbb D}_{Q_{0}}$, $\cE = \cY$, $a(Z) = 
I_{\cY^{n}}$ for $Z \in {\mathbb D}_{Q_{0},n}$ and $b(Z) = S(Z)$)  
these latter two conditions are equivalent.  We conclude that 
${\mathcal SA}_{Q_{0}}(\cU, \cY) = \overline{\cB}\cM(k_{Q_{0}} 
\otimes I_{\cU}, k_{Q_{0}} \otimes I_{\cY})$ as required. 

This completes the proof of Theorem \ref{T:multtanint}.
\end{proof} 

Specializing the equivalence (1) $\Leftrightarrow$ (2) in Theorem 
\ref{T:multtanint} to the case $\cE = \cY$, $a(Z) = I_{\cY^{n}}$ for 
$Z \in \Omega_{n}$ yields the following corollary.

\begin{corollary}   \label{C:Pickker}
    The kernel $k_{Q_{0}}$ given by \eqref{kQ0} is a complete Pick nc 
    kernel.
\end{corollary}

 \begin{remark}  It is well known (see \cite{Quiggen, AMcC00} as well 
    as \cite[Theorem 7.6]{AMcC-book}) that in the classical case the 
    Szeg\H{o} kernel has a certain universal property with respect to 
    complete Pick kernels.  An open question which we leave for 
    future work is to find the analogue of this result for the 
    noncommutative-function setting.
\end{remark}

Another corollary is the transfer-function realization 
characterization of contractive multipliers in this nc 
Szeg\H{o}-kernel setting. 
 
 \begin{corollary} \label{C:realization}
A graded function $S$ mapping ${\mathbb B}_{nc}^{d}$ to
$\cL(\cU, \cY)_{\rm nc}$
is a nc function which is in the 
contractive-multiplier class $\overline{\cB}\cM(I_{\cU} \otimes 
k_{\dBR}, I_{\cY} \otimes k_{\dBR})$ if and only if $S$ can be 
realized in the transfer-function realization form 
\eqref{transfunc-mult} for $Z \in {\mathbb D}_{Q_{0},n}$.
 \end{corollary}
 
  \begin{proof}  The assertion of this corollary amounts to the 
     equivalence (1) $\Leftrightarrow$ (3) in Theorem 
     \ref{T:multtanint} for the special case where $\Omega = 
     {\mathbb D}_{Q_{0}}$, $\cE = \cY$, $a(Z) = I_{\cY^{n}}$ for $Z 
     \in {\mathbb D}_{Q_{0},n}$, $b(Z) = S(Z)$.
 \end{proof}

\section{Examples and applications of the Schur-Agler class 
interpolation and realization theorems}

In this final section we collect some additional illustrative 
special cases and examples of the general theory. 

\subsection{Single-point interpolation and Stein domination of interpolation node 
by interpolation value}  \label{S:Steindom}
The following proposition amounts to a more explicit formulation of 
Corollary \ref{C:AMcC-Pick}.  To better formulate the result, it is 
convenient to introduce the following definition. For the second 
definition below, just as in Case 1 of the proof of Theorem 
\ref{T:ncInt}, we assume that the ambient vector space $\cV$ for 
the set of points $\Xi$ is ${\mathbb C}^{d}$ (with $d$ chosen 
sufficiently large); then $\cV$ carries an operator-space structure via 
the identification ${\mathbb C}^{d} \cong \cL({\mathbb C}^{d}, 
{\mathbb C})$. Given a point $Z^{(0)} \in {\mathbb D}_{Q,n}$ and an operator $\Lambda_{0} \in 
\cL(\cU, \cY)^{n \times n} \cong \cL(\cU^{n}, \cY^{n})$, let us say that 
\begin{enumerate}
    \item 
    \textbf{$\Lambda_{0}$ 
dominates $Q(Z^{(0)})$ in the sense of Stein}, written as $\Lambda_{0} 
\succeq_{\cS} Q(Z^{(0)})$, if 
\begin{align}
&   0 \preceq P \in {\mathbb C}^{n \times n},
 \,    P \otimes I_{\cS} -  Q(Z^{(0)})  ( P  \otimes I_{\cR}) 
  Q(Z^{(0)})^{*} \succeq 0 \notag \\
 & \quad  \Rightarrow 
 P \otimes I_{\cY} - 
 \Lambda_{0} ( P \otimes I_{\cU})  \Lambda_{0}^{*}  \succeq 0,
 \label{Stein}
\end{align}
and
\item 
\textbf{$\Lambda_{0}$ 
dominates $Q(Z^{(0)})$ in the strict-Stein sense}, written as $\Lambda_{0} 
\succeq_{s\cS} Q(Z^{(0)})$, if, whenever it is the case that there is 
a $\delta > 0$ so that for any $P \in {\mathbb C}^{n \times n}$ with $P 
\succeq 0$ such that the two inequalities
\begin{align}
& P  -   (1 - \delta^{2}) Z^{(0)} ( P \otimes I_{{\mathbb C}^{d}})  Z^{(0)*} \succeq 0,  \notag \\
& (1 - \delta^{2}) P \otimes I_{\cS} -  
 Q(Z^{(0)})( P \otimes I_{\cR})  Q(Z^{(0)})^{*} \succeq 0
 \label{strictStein1}
\end{align}
hold, it follows that 
\begin{equation}   \label{strictStein2}
 P \otimes I_{\cY} - 
 \Lambda_{0} ( P \otimes I_{\cU})  \Lambda_{0}^{*}  \succeq 0.
\end{equation}
\end{enumerate}

Then we have the following refinement of Corollary \ref{C:AMcC-Pick}.

\begin{proposition}  \label{P:singlepoint}
    Let $Q$ and ${\mathbb D}_{Q}$ be as in Theorem \ref{T:ncInt} and 
    suppose that $Z^{(0)} \in {\mathbb D}_{Q,n}$ and $\Lambda_{0} \in 
    \cL(\cU^{n}, \cY^{n})$ for some $n \in {\mathbb N}$.  Consider 
    the following two statements: 
\begin{enumerate}
\item  There exists a function $S$ in the Schur-Agler class 
$\mathcal{SA}_{Q}(\cU, \cY)$ satisfying the interpolation condition
\begin{equation}   \label{int-singlepoint}
S(Z^{(0)}) = \Lambda_{0}.
\end{equation}
\item  For $k = n \cdot \dim \cY$, $\bigoplus_{1}^{k} \Lambda_{0}$ dominates 
$\bigoplus_{1}^{k} Q(Z^{(0)})$ 
in the strict-Stein sense \eqref{strictStein1}--\eqref{strictStein2}. 
\end{enumerate}
Then (1) $\Rightarrow$ (2). If we assume in addition that there exists a nc function $S_{0}$ (not 
necessarily contractive) on $\Omega'_{0}$ (defined as in Remark \ref{R:weakhy} with 
$\Omega = \{ Z^{(0)}\}$)  with $S_{0}(Z^{(0)}) = \Lambda_{0}$, then
(2) $\Rightarrow$ (1). 
\end{proposition}

\begin{proof}  
Suppose that the interpolation problem 
\eqref{int-singlepoint} has a Schur-Agler class solution $S \in 
\mathcal{SA}_{Q}(\cU, \cY)$ and that  $P \in {\mathbb 
C}^{k\cdot n \times k \cdot n}$ satisfies the hypotheses 
\eqref{strictStein1} for the strict-Stein dominance condition 
$\bigoplus_{1}^{k} \Lambda_{0} \succeq_{s \cS} \bigoplus_{1}^{k} 
Q(Z^{(0)})$, namely:
 \begin{equation}   \label{ineq0} 
    P  - (1 - \delta^{2})\left( \bigoplus_{1}^{k} Z^{(0)} \right) \,( P 
    \otimes I_{{\mathbb C}^{d}}) \, \left( 
    \bigoplus_{1}^{k} Z^{(0)} \right)^{*} \succeq 0
    \end{equation}
 and
 \begin{equation} \label{ineq1}
    (1 - \delta^{2}) P \otimes I_{\cS} - \left( \bigoplus_{1}^{k} 
	Q(Z^{(0)}) \right) ( P \otimes I_{\cR}) 
\left(	\bigoplus_{1}^{k}Q(Z^{(0)}) \right)^{*}  \succeq 0.
\end{equation}
 Let $m$ be the rank of $P$ and let $P = \cI \cI^{*}$ be a 
 full-rank factorization of $P$ with an injective $\cI$ of size 
 $k \cdot n \times m$. 
 An application of the Douglas lemma (see \cite{Douglas}) applied to the 
 condition \eqref{ineq0} with $\cI \cI^{*}$ substituted for $P$ gives 
 the existence of a $\widetilde Z \in \cV^{m \times m}$ with norm-squared at most 
 $(1 - \delta^{2})^{-1}$ solving the 
 factorization problem
 \begin{equation}   \label{have0}
   \cI  \widetilde Z = \left( \bigoplus_{1}^{k} 
    Z^{(0)} \right) (\cI \otimes I_{{\mathbb C}^{d}}).
  \end{equation}
  By the ``respects intertwinings''  property of the nc function $Q$, 
  we then have
 \begin{equation}   \label{have1}
 ( \cI  \otimes I_{\cS}) Q(\widetilde Z) =  \left( \bigoplus_{1}^{k} 
    Q(Z^{(0)}) \right) (\cI \otimes I_{\cR}).
 \end{equation}
 A similar application of the Douglas lemma applied to the inequality
 \eqref{ineq1} assures us that there is an operator $Y \in 
 \cL(\cR^{m}, \cS^{m})$  with $\| Y \|^{2} \le 1- \delta^{2} < 1$ such that
 \begin{equation}  \label{have2}
   (\cI \otimes I_{\cS})  Y = \left( \bigoplus_{1}^{k}Q(Z^{(0)}) 
   \right) (\cI \otimes I_{\cR}).
 \end{equation}
 As $\cI$ is injective, we see from \eqref{have1} and \eqref{have2} 
 that $Y = Q(\widetilde Z)$.  As $\| Q(\widetilde Z) \| = \| Y \| < 1$, we conclude from 
 \eqref{have0} that $\widetilde Z \in \Omega'_{0} \subset {\mathbb 
 D}_{Q}$.  As $S$ by assumption is in the nc Schur-Agler class 
 $\mathcal{SA}_{Q}(\cU, \cY)$, we conclude that $\| S(\widetilde Z) 
 \| \le 1$.  From the intertwining relation \eqref{have0} and the 
 fact that the nc function $S$ respects direct sums and 
 intertwinings, we see that we must have
 $$
   (\cI \otimes I_{\cY}) S(\widetilde Z) = \left( \bigoplus_{1}^{k} 
   S(Z^{(0)} )\right) (\cI \otimes I_{\cU}).
 $$
 From the fact that $\| S(\widetilde Z) \| \le 1$, i.e., $I_{\cY^{m}} 
 - S(\widetilde Z) S(\widetilde Z)^{*} \succeq 0$, we get
 \begin{align*} & P \otimes I_{\cY} - \left( \bigoplus_{1}^{k} 
     \Lambda_{0} \right) (P \otimes I_{\cU})  \left( \bigoplus_{1}^{k} 
     \Lambda_{0} \right)^{*} \\
     & \quad = \cI \cI^{*} \otimes I_{\cY} - \left( \bigoplus_{1}^{k} 
     S(Z^{(0)}) \right) (\cI \cI^{*} \otimes I_{\cU})  \left(\bigoplus_{1}^{k} 
     S(Z^{(0)}) \right)^{*}  \\
     & \quad = (\cI \otimes I_{\cY}) \left( I_{\cY^{m}} - S(\widetilde Z) S(\widetilde 
     Z)^{*} \right) (\cI^{*} \otimes I_{\cY}) \succeq 0
  \end{align*}
  and hence 
  \begin{equation}   \label{have3}
       P \otimes I_{\cY} - \left( \bigoplus_{1}^{k} 
     \Lambda_{0} \right) (P \otimes I_{\cU})  \left( \bigoplus_{1}^{k} 
     \Lambda_{0} \right)^{*}  \succeq 0,
  \end{equation}
  i.e., the inequality \eqref{strictStein2} holds with 
  $\bigoplus_{1}^{k} \Lambda_{0}$ in place of $\Lambda_{0}$.  This 
  completes the proof that the strict-Stein dominance condition
 $\bigoplus_{1}^{k}  \Lambda_{0} \succeq_{s\cS} \bigoplus_{1}^{k} Q(Z^{(0)})$ holds.

    Conversely, suppose that there is a nc function $S_{0}$ on 
    $\Omega'_{0}$ with $S(Z^{(0)}) = \Lambda_{0}$ and that 
    $\bigoplus_{1}^{k} \Lambda_{0} \succeq_{s\cS} \bigoplus_{1}^{k} Q(Z^{(0)})$. 
    By Remark \ref{R:weakhy} applied to the case 
    $a(Z^{(0)}) = I_{\cY^{n}}$, $b(Z^{(0)}) = \Lambda_{0}$, $\cE =\cY$, 
    to verify that there exists an $S \in \mathcal{SA}_{Q}(\cU, \cY)$ with 
    $S(Z^{(0)}) = \Lambda_{0}$, it suffices to show that the function 
    $S_{0}$ must in fact be contractive.  Let $\widetilde Z$ be a 
    point of $\Omega'_{0}$.  Then there is an injective  $\cI \in 
    {\mathbb C}^{k \cdot n \times m}$
    (for some $m$ with $1 \le m \le k \cdot n$) so that 
     \begin{equation}   \label{intertwine1}
      \cI \widetilde Z = \left(\bigoplus_{1}^{k} Z^{(0)} \right) (\cI 
      \otimes I_{{\mathbb C}^{d}}).
    \end{equation}
    If $\widetilde Z$ has $\| \widetilde Z \|^{2} \le \frac{1}{1- 
    \delta^{2}}$ for some $\delta < 1$, it follows that
    $ I_{{\mathbb C}^{m}}-  (1- \delta^{2}) \widetilde Z \widetilde 
    Z^{*} \succeq 0$. It then follows that
    \begin{align*}
  &  \cI \cI^{*}  -  (1 - \delta^{2}) \left(\bigoplus_{1}^{k} Z^{(0)} \right) 
  (\cI \cI^{*} \otimes I_{{\mathbb C}^{d}}) \left( \bigoplus_{1}^{k}  Z^{(0)} \right)^{*}  \\
    & \quad = \cI \left(  I_{{\mathbb C}^{m}} - (1 - \delta^{2}) \widetilde 
    Z \widetilde Z^{*} \right) \cI^{*} \succeq 0. 
    \end{align*}
    If we set $P = \cI \cI^{*}$, we then see that $P \succeq 0$ in ${\mathbb 
    C}^{k \cdot n \times k \cdot n}$ satisfies \eqref{ineq0}.
   For this $\widetilde Z$ to be in $\Omega'_{0}$, we also require 
    that $\| Q (\widetilde Z) \| < 1$.  Choosing $\delta < 1$ 
    sufficiently close to 1, we may assume that $\| Q(\widetilde Z) 
    \|^{2} < 1 - \delta^{2}$, or $(1 - \delta^{2}) I_{\cS^{m}} - Q(\widetilde Z) 
    Q(\widetilde Z)^{*} \succeq 0$.  Combining this with the 
    intertwining \eqref{intertwine1} then leads to the conclusion
    \begin{align*}
& (1 - \delta^{2})\cI \cI^{*} \otimes I_{\cS} - \left(\bigoplus_{1}^{k} 
	Q(Z^{(0)}) \right) ( \cI \cI^{*} \otimes I_{\cR}) 
\left(	\bigoplus_{1}^{k} Q(Z^{(0)}) \right)^{*}   \\
& \quad = ( \cI \otimes I_{\cS}) \left( (1 - \delta^{2}) I_{\cS^{m}} - Q(\widetilde Z) 
Q(\widetilde Z)^{*} \right) (\cI^{*} \otimes I_{\cS}) \succeq 0
\end{align*}
and consequently $P$ also satisfies \eqref{ineq1}.
As the inequalities \eqref{ineq0}--\eqref{ineq1} amount to the 
assumptions \eqref{strictStein1} in the strict-Stein dominance 
condition $\bigoplus_{1}^{k} \Lambda_{0} \succeq_{s\cS} 
\bigoplus_{1}^{k} Q(Z^{(0)})$, our standing hypothesis that this 
strict-Stein dominance condition holds implies that 
that \eqref{strictStein2} holds (with 
$\bigoplus_{1}^{k} Q(Z^{(0)})$ in place of $Z^{(0)}$ and
$\bigoplus_{1}^{k} \Lambda_{0}$ in place of $\Lambda_{0}$), i.e., that
$$
 \cI \cI^{*} \otimes I_{\cY} - \left( \bigoplus_{1}^{k} \Lambda_{0}\right)
 (\cI \cI^{*} \oplus I_{\cU}) 
\left( \bigoplus_{1}^{k} \Lambda_{0}\right)^{*} \succeq 0.
$$
An application of the Douglas lemma \cite{Douglas} then gives us 
the existence of an operator $Y \in \cL(\cU^{k}, \cY^{k})$ with $\| Y \| \le 1$
solving the factorization problem
\begin{equation} \label{defY}
   (\cI \otimes I_{\cY}) Y = \left(  \bigoplus_{1}^{k} 
   \Lambda_{0}\right)  (\cI \otimes I_{\cU}).
\end{equation}
As $S_{0}$ is a nc function on $\Omega'_{0}$ with $S_{0}(Z^{(0)}) = 
\Lambda_{0}$, from the relation \eqref{intertwine1} we deduce that
\begin{equation} \label{S0intertwine}
(\cI \otimes I_{\cY}) S_{0}(\widetilde Z) =  \left(\bigoplus_{1}^{k} 
   \Lambda_{0}\right)  (\cI \otimes I_{\cU}).
\end{equation}
As $\cI$ is injective, from \eqref{defY} and \eqref{S0intertwine} we 
deduce that $S_{0}(\widetilde Z) = Y$ and hence $\|S_{0}(\widetilde 
Z)\| = \| Y \| \le 1$.  As $\widetilde Z$ was an arbitrary point in 
$\Omega'_{0}$, we see that in fact $S_{0}$ is contractive on 
$\Omega'_{0}$.
It now follows from the content of Remark 
\ref{R:weakhy} that there exists a nc  Schur-Agler class function $S 
\in \mathcal{SA}_{Q}(\cU, \cY)$ satisfying the interpolation 
condition $S(Z^{(0)}) = \Lambda_{0}$.
 \end{proof}

  In the setting of Section \ref{S:mult} where $\cS = {\mathbb C}$ and 
 $Q(z) = Q_{0}(z)$ in the notation there, it is possible to use the 
 complete positivity condition in statement (2) of Theorem 
 \ref{T:multtanint} to get a more definitive version of the result 
 in Proposition \ref{P:singlepoint} for this case. 
 
 \begin{proposition}   \label{P:singlepoint0}  Suppose that $Q_{0}$ 
     is as in Section \ref{S:mult}, $Z^{(0)} \in {\mathbb 
     D}_{Q_{0},n}$ and $\Lambda_{0} \in \cL(\cU, \cY)^{n \times n}$.  
     Then the following conditions are equivalent:
     \begin{enumerate}
	 \item There exists $S$ in the Schur-Agler class 
	 $\mathcal{SA}_{Q_{0}}(\cU, \cY)$, or equivalently in the 
	 contractive multiplier class $\overline{\cB}\cM(k_{Q_{0}} 
	 \otimes I_{\cU}, k_{Q_{0}} \otimes I_{\cY})$, satisfying the 
	 interpolation condition 
\begin{equation}  \label{int-singlepoint'}
    S(Z^{(0)}) = \Lambda_{0}.
\end{equation}

\item $\bigoplus_{1}^{n} \Lambda_{0}$  dominates 
$\bigoplus_{1}^{n} Q_{0}(Z^{(0)})$ in the sense of Stein 
\eqref{Stein}.
\end{enumerate} 
 \end{proposition}
 
  \begin{remark}  \label{R:strictvsnonstrict}
  As explained in the proof of Theorem 
      \ref{T:multtanint} (``equivalence of conditions (1)''), the 
      contractive multiplier class $\overline{\cB} \cM(k_{Q_{0}} 
      \otimes I_{\cU}, k_{Q_{0}}  \otimes I_{\cY})$ coincides 
 with the Schur-Agler class $\mathcal{SA}_{Q_{0}}(\cU, \cY)$.  Thus 
 Proposition \ref{P:singlepoint} applied to this case tells us that 
 the interpolation problem \eqref{int-singlepoint'} has a contractive 
 multiplier solution $S \in \overline{\cB} \cM(k_{Q_{0}} \otimes 
 I_{\cU},  k_{Q_{0}} \otimes I_{\cY})$ if and only if the 
 strict-Stein dominance $\bigoplus_{1}^{k} \Lambda_{0} \succeq_{s\cS} 
 \bigoplus_{1}^{k} Q(Z^{(0)})$ with $k = n \cdot \dim \cY$
 (together with an extra hypothesis for the converse direction) holds.  The content of Proposition 
 \ref{P:singlepoint0} is that, for the special case where $Q = Q_{0}$ 
 has target space $\cS = {\mathbb C}$, the result of Proposition 
 \ref{P:singlepoint} holds with strict-Stein dominance 
 \eqref{strictStein1}--\eqref{strictStein2} replaced by simple 
 Stein dominance \eqref{Stein}, $n \le k = n \cdot \dim \cE$ 
 replacing $k$,   and with removal of the extra 
 hypothesis for the converse direction.
 \end{remark}

\begin{proof}[Proof of Proposition \ref{P:singlepoint0}]
  By the equivalence of (1) and (2) in Theorem \ref{T:multtanint}, we see that there is a 
 contractive multiplier solution $S \in \overline{\cB}\cM(k_{Q_{0}} 
 \otimes I_{\cU}, k_{Q_{0}} \otimes I_{\cY})$ of the interpolation 
 condition $S(Z^{(0)}) = \Lambda_{0}$ if and only if the map 
 $$
  P \mapsto k_{Q_{0}}(Z^{(0)}, Z^{(0)})(P) \otimes I_{\cY}  - 
  \Lambda_{0} \left( k_{Q_{0}}(Z^{(0)}, Z^{(0)})(P) \otimes I_{\cU} 
  \right) \Lambda_{0}^{*}
$$
is a completely positive map from ${\mathbb C}^{n \times n}$ into 
$\cL(\cY)^{n \times n}$.  By a result of M.-D.\ Choi (see 
\cite[Theorem 3.14]{Paulsen}, it suffices to check that this map in 
$n$-positive, i.e., that the map
\begin{align}
& P \mapsto k_{Q_{0}}\left(\bigoplus_{1}^{n} Z^{(0)}, \bigoplus_{1}^{n} 
Z^{(0)}\right) (P) \otimes I_{\cY}  \notag \\
& \quad - \left( \bigoplus_{1}^{n} \Lambda_{0} 
\right) \left( k_{Q_{0}}\left(\bigoplus_{1}^{n} Z^{(0)}, 
\bigoplus_{1}^{n} Z^{(0)}\right)(P) \otimes I_{\cU} \right) \left( 
\bigoplus_{1}^{n} \Lambda_{0} \right)^{*}
\label{n-posmap}
\end{align}
is a positive map from ${\mathbb C}^{n^{2} \times n^{2}}$ into 
$\cL(\cY)^{n^{2} \times n^{2}}$.  If we set 
$$
R = k_{Q_{0}}\left(\bigoplus_{1}^{n} Z^{(0)},\bigoplus_{1}^{n} 
Z^{(0)}\right)(P),
$$
then according to the identity \eqref{id1} we recover $P$ from $R$ via
$$
  P = R -\left(\bigoplus_{1}^{n} Q_{0}( Z^{(0)}) \right)\,  R\,
 \left(\bigoplus_{1}^{n} Q_{0}( Z^{(0)}) \right) ^{*}.
$$
Then the condition that the map \eqref{n-posmap} be positive can be 
reformulated as:
\begin{align*}
 & R \succeq 0 \text{ such that } R - \left(\bigoplus_{1}^{n} 
 Q(Z^{(0)}) \right) R 
   \left( \bigoplus_{1}^{n}Q_{0}( Z^{(0)}) \right)^{*} \succeq 0 \\
\Rightarrow
& \quad R -  \left(\bigoplus_{1}^{n} \Lambda_{0}\right) R\left ( \bigoplus_{1}^{n} 
\Lambda_{0}\right)^{*} \succeq 0
\end{align*}
This in turn amounts to the Stein dominance condition 
$\bigoplus_{1}^{n} \Lambda_{0} \succeq_{\cS} \bigoplus_{1}^{n} 
Q(Z^{(0)})$ (see \eqref{Stein}).
 \end{proof}
 
 \begin{remark}  \label{R:CL}  Propositions \ref{P:singlepoint} 
     and \ref{P:singlepoint0} were inspired by the work of 
     Cohen-Lewkowicz \cite{CL1, CL2} on the so-called Lyapunov order 
     on real symmetric matrices and the connection of this with the 
     Pick-matrix criterion for interpolation by positive real odd 
     functions (roughly, the real right-half-plane analogue of the 
     classical case of our topic here). 
  \end{remark}
 
 \subsection{Finite-Point Left-Tangential Pick Interpolation 
 Problem}  \label{S:LTTanPick}
 Let us consider the special case of Theorem \ref{T:multtanint} where 
 $\cV = \Xi = {\mathbb C}^{d}$.  We then write points $Z$ in $\cV_{{\rm 
 nc},n} = ({\mathbb C}^{d})^{n \times n}$ as $d$ tuples $Z = (Z_{1}, 
 \dots, Z_{d}) \in ({\mathbb C}^{n \times n})^{d} \cong ({\mathbb 
 C}^{d})^{n \times n}$.  As a nc function $Q$ on ${\mathbb 
 C}^{d}_{\rm nc}$ we choose $Q = Q_{\rm row}$ given by
 $$
  Q_{\rm row}(Z) = Q_{\rm}(Z_{1}, \dots, Z_{d}) = \begin{bmatrix} Z_{1} & 
  \cdots & Z_{d} \end{bmatrix}.
 $$
 The resulting nc domain ${\mathbb D}_{Q_{\rm row}}$ then amounts to the nc 
 operator ball
 $$
   {\mathbb B}^{d}_{\rm nc} = \amalg_{n=1}^{\infty} \{ Z = (Z_{1}, 
   \dots, Z_{d}) \in ({\mathbb C}^{n \times n})^{d} \colon
   Z_{1} Z_{1}^{*} + \cdots + Z_{d} Z_{d}^{*} \prec I_{n}\}.
 $$
 For this case the formula for the generalized Szeg\H{o} kernel can be 
 written out more concretely in coordinate form as
 $$
 k_{Q_{\rm row}}(Z,W)(P) = \sum_{\fa \in \free} Z^{\fa} P W^{* 
 \fa^{\top}}
 $$
 where we use nc functional calculus conventions as in Theorem 
 \ref{T:BGM2} in the Introduction: for $\fa = (i_{1}, \dots, 
 i_{N})$ in the unital free semigroup $\free$ and for
 $Z = (Z_{1}, \dots, Z_{d})$ and $W = (W_{1}, \dots, W_{d})$ in 
 ${\mathbb B}^{d}_{\rm nc}$, 
$$
 Z^{\fa}=  Z_{i_{N}} \cdots Z_{i_{1}},\quad
  W^{*} = (W_{1}^{*}, \dots, W_{d}^{*}), \quad
   W^{* \fa^{\top}} = W_{i_{1}}^{*} \cdots W_{i_{N}}^{*}. 
 $$
 Then nc functions $S \in \cT({\mathbb B}^{d}_{\rm nc}; \cL(\cU, 
 \cY)_{nc})$ 
 are given by a power-series representation as in 
 \eqref{powerseriesrep}
 \begin{equation}  \label{powerseriesrep'}
 S(Z) = \sum_{\fa \in \free} S_{\fa} \otimes Z^{\fa} 
 \text{ for } Z = (Z_{1}, \dots, Z_{d}) \in {\mathbb B}^{d}_{\rm nc}.
 \end{equation}
 We consider the \textbf{Finite-Point Left-Tangential Pick 
 Interpolation Problem} for the contractive multiplier class 
 associated with the kernel $k_{Q_{\rm row}}$:  {\sl Given $Z^{(1)}, 
 \dots, Z^{(N)} \in {\mathbb B}^{d}_{\rm nc}$ along with vectors 
 $A_{i} \in \cL(\cU^{n_{i}}, \cE_{i})$, $B_{i} \in \cL(\cY^{n_{i}}, \cE_{i})$
 ($n_{i}$ chosen so that $Z^{(i)} \in {\mathbb B}^{d}_{{\rm nc}, 
 n_{i}}$),  find 
 $S \in \overline{\cB} \cM(k_{Q_{\rm row}} \otimes I_{\cU}, k_{Q_{\rm 
 row}} \otimes I_{\cY})$ with}
 $$
    A_{i} S(Z^{(i)}) = B_{i} \text{ for } i = 1, \dots, N.
 $$
 Using the nc function structure, one can reduce any finite-point problem to a single-point 
 problem; specifically take the single-point data set $(Z^{(0)}, A_{0}, B_{0})$ to be
 $$
 Z^{(0)} = \sbm{ Z^{(1)} & & \\ & \ddots & \\ & & Z^{(N)} }, \quad
 A_{0} = \sbm{A_{1} & & \\ & \ddots & \\ & & A_{N}}, \quad
 B_{0} = \sbm{B_{1} & & \\ & \ddots & \\ & & B_{N}}.
 $$
 Thus we simplify the discussion here by considering only the 
 single-point version of the Left-Tangential Pick Interpolation 
 Problem with data set $\{ Z^{(0)}, A_{0}, B_{0}\}$.  The equivalence 
 of (1) $\Leftrightarrow$  (2) in Theorem \ref{T:multtanint} leads to 
 the following result.
 
 \begin{theorem}   \label{T:multtanintsingle}
     Suppose that we are given the data set $(Z^{(0)}, A_{0}, B_{0})$ 
     for a Single-Point Left-Tangential Pick Interpolation Problem. 
     Then the following are equivalent:
     \begin{enumerate}
	 \item The interpolation problem has a solution, i.e., there 
	 exists $S$ in $\overline{\cB}\cM(k_{Q_{\rm row}} \otimes I_{\cU},
	 k_{Q_{\rm row}} \otimes I_{\cU})$ with
\begin{equation}   \label{singletanint}
  A_{0} S(Z^{(0)}) = B_{0}.
  \end{equation}
  \item The map 
  \begin{equation}   \label{PickMS}
      P \mapsto \sum_{\fa \in \free} \left(A_{0} (Z^{(0) \fa} P 
      Z^{(0)* \fa^{\top}} \otimes  I_{\cY}) A_{0}^{*} -  B_{0}
      (Z^{(0) \fa} P  Z^{(0)* \fa^{\top}} \otimes I_{\cU}) 
      B_{0}^{*} \right)
  \end{equation}
  is completely positive.
  \end{enumerate}
  \end{theorem}
  
  This particular nc Pick interpolation problem has already been 
  considered by some other authors. Our result Theorem 
  \ref{T:multtanintsingle} agrees with the result of Muhly-Solel
  in \cite[Theorem 6.3]{MS2004} and is a slight variation of a result 
  of Popescu \cite[Corollary 2.3]{P98}.   We mention that Muhly-Solel 
  actually considered a much more general setting where the 
  multiplier algebra $\cM(k_{Q_{\rm row}} \otimes I_{\cU}, 
  k_{Q_{\rm row}} \otimes I_{\cY})$ is replaced by the generalized Hardy algebra 
  $H^{\infty}_{E}$ associated with a correspondence $E$ over a von 
  Neumann algebra $M$ (see also \cite{MS1998, MS2008}).  An 
  interesting topic for future research is to get a better 
  understanding about how this general von Neumann-algebra 
  correspondence setting fits with the free nc-function setting used 
  here; preliminary steps in such a program have already been made in 
  \cite{BBFtH, MS2013, Norton}.

 \subsection{NC-function versus Left-Tangential Operator-Argument point evaluation: 
 the nc ball setting}  \label{S:NCvsLTOA}
A popular formalism for handling intricate univariate matrix-valued 
interpolation problems over the years has been to make use of a 
\textbf{Left-Tangential Operator Argument} point-evaluation (see 
e.g.\ \cite{BGR90, FFGK}).  A nc version of the Left Tangential Operator 
Argument point evaluation (as well as right and two-sided versions 
which we need not go into here) was introduced in \cite{BB07}, along 
with a study of associated interpolation problems.  We now describe 
one such result for the nc ball setting in this Subsection.

Suppose that $Z^{(0)} = (Z^{(0)}_{1}, \dots, Z^{(0)}_{d})$ is a point 
in ${\mathbb B}^{d}_{{\rm nc}, n}$, $X$ is in operator in $\cL(\cY, {\mathbb 
C}^{n})$,  and $S$ is a nc operator-valued function in  
$\cT({\mathbb B}^{d}_{\rm nc}; \cL(\cU, \cY)_{\rm nc})$ with associated formal power series
$S(z) = \sum_{\fa \in \free} S_{\fa} z^{\fa}$.  We define
$(X S)^{\wedge L}(Z^{(0)})$ (the Left-Tangential Operator Argument 
evaluation of $S$ at $Z^{(0)}$ in direction $X$) by
\begin{equation} \label{LTOA}
    (X S)^{\wedge L}(Z^{(0)}) = \sum_{\fa \in \free} Z^{(0) \fa^{\top}} X 
    S_{\fa} \in \cL(\cU, {\mathbb C}^{n}).
\end{equation}
Note that in contrast to the nc-function point evaluation 
\eqref{powerseriesrep'}, the power on $Z^{(0)}$ involves $\fa^{\top}$ 
rather than $\fa$ and all multiplications in \eqref{LTOA} are 
operator compositions (no tensor products). Given an interpolation 
data set $(Z^{(0)}, X, Y)$, where $Z^{(0)}$, $X$ are as above along 
with an operator $Y$ in $\cL(\cY, {\mathbb C}^{n})$,  the 
\textbf{Left-Tangential Operator Argument}(\textbf{LTOA}) 
interpolation problem is:  {\em find $S \in \overline{\cB} 
\cM(k_{Q_{\rm row}} \otimes \cU, k_{Q_{\rm row}} \otimes \cY)$ (or 
equivalently in $\mathcal{SA}_{Q_{\rm row}}(\cU, \cY)$  so that the 
Left-Tangential Operator Argument interpolation condition 
\begin{equation}  \label{LTOAint}
    (X S)^{\wedge L}(Z^{(0)})   = Y
\end{equation}
holds.}  The solution is as follows (see \cite[Theorem 7.4]{BB07} as 
well as \cite[Theorem 3.4]{CJ03} for a different but equivalent 
formulation).

\begin{theorem}  \label{T:BB07}
    Suppose that we are given the data set $(Z^{(0)},X,Y)$ for a LTOA 
    interpolation problem for $S \in \overline{\cB} 
\cM(k_{Q_{\rm row}} \otimes \cU, k_{Q_{\rm row}} \otimes \cY)$ as 
above.  Then the LTOA interpolation problem has a solution if and 
only if
$$
  \sum_{\fa \in \free} Z^{(0) \fa} (X X^{*} - Y Y^{*}) Z^{(0)* 
  \fa^{\top}} \succeq 0.
$$
\end{theorem}

There is a curious connection 
between the nc-function Left-Tangential Pick Interpolation Problem 
versus the LTOA Interpolation Problem which we now discuss.  Let us again restrict to 
the scalar nc Schur class ${\mathcal SA}_{Q_{\rm row}}({\mathbb C}) = 
\overline{\cB}\cM(k_{Q_{\rm row}})$.  Choose a point $Z^{(0)} \in 
{\mathbb B}^{d}_{{\rm nc},n}$ along with an a matrix $\Lambda_{0} \in 
{\mathbb C}^{n}$ and consider the single-point interpolation problem:
{\em find $s \in {\mathcal SA}_{Q_{\rm row}}({\mathbb C})$ so that 
the interpolation condition 
\begin{equation}   \label{scalarint}
    s(Z^{(0)}) = \Lambda_{0}.
\end{equation}
holds.} According to Theorem 
\ref{T:multtanintsingle}, this  interpolation problem has a 
solution if and only if the map from ${\mathbb C}^{n 
\times n}$ into $\cL(\cE)$ given by 
\begin{equation}   \label{PickMS''}
     P \mapsto \sum_{\fa \in \free} \left( Z^{(0) \fa} P 
      Z^{(0)* \fa^{\top}}   -  \Lambda_{0}
      Z^{(0) \fa} P  Z^{(0)* \fa^{\top}}
      \Lambda_{0}^{*} \right)
\end{equation}
is completely positive.  As the domain for this map is ${\mathbb C}^{n \times n}$, 
the Choi criterion (see \cite[Theorem 3.14]{Paulsen}) gives a test for 
complete positivity in terms of positive definiteness of a single 
operator:  {\em the map \eqref{PickMS''} is completely positive if and 
only if the block matrix
\begin{equation}  \label{Choi}
  \left[ \sum_{\fa \in \free} 
\left( Z^{(0) \fa} e_{\kappa} e_{\kappa'}^{*} Z^{(0)* \fa^{\top}}  - 
\Lambda_{0} Z^{(0) \fa} e_{\kappa} e_{\kappa'}^{*} Z^{(0)*\fa^{\top}} \Lambda_{0}^{*} 
      \right)\right]_{\kappa, \kappa' \in {\mathfrak B}} 
\end{equation}
is positive semidefinite}, where $\{ e_{\kappa} \colon  \kappa  \in 
{\mathfrak B}\}$ is the 
standard basis for ${\mathbb C}^{n}$.

On the other hand, we can view the interpolation condition as a 
twisted version of a LTOA interpolation condition as follows.  Note 
that the interpolation condition \eqref{scalarint} can also be 
expressed as
$$
    s(Z^{(0)}) e_{\kappa} = \Lambda_{0} e_{\kappa} \text{ for each } 
    \kappa \in {\mathfrak B}.
$$
Writing this condition out in terms of a series and using that the 
coefficients $s_{\fa}$ are scalar, we get
\begin{equation}   \label{scalarint'}
  \sum_{\fa \in \free} s_{\fa} Z^{(0) \fa} e_{\kappa}
  = \sum_{\fa \in \free}  Z^{(0) \fa} e_{\kappa} s_{\fa} = 
  \Lambda_{0} e_{\kappa} \text{ for each } \kappa \in {\mathfrak B}.
\end{equation}
Let us introduce the twisted  LTOA point evaluation
\begin{equation}   \label{twistLTOA}
    (X S)^{\wedge \tau \circ L}(Z) = \sum_{\fa \in \free} 
    Z^{ \fa} X S_{\fa},
\end{equation}
i.e., the  formula \eqref{LTOA} but with the power of $Z^{(0)}$ equal 
to $\fa$ instead of $\fa^{\top}$.  
Then the nc-function interpolation condition \eqref{scalarint} can be 
reexpressed as the set of  nc twisted LTOA interpolation conditions
$$ 
  (e_{\kappa} s)^{\tau \circ L}(Z^{(0)}) = \Lambda_{0} e_{\kappa} 
  \text{ for } \kappa \in {\mathfrak B}.
$$
If we introduce the column vector $E = {\rm col}_{\kappa \in 
{\mathfrak B}} [e_{\kappa}]$, we can convert the problem to a single 
twisted LTOA interpolation condition
\begin{equation}   \label{twistedLTOAprob}
  (E s)^{\wedge \tau \circ L}(\oplus_{\kappa \in {\mathfrak B}} Z^{(0)}) = 
  \left( \bigoplus_{\kappa \in {\mathfrak B}}\Lambda_{0} \right) E.
\end{equation}
By our previous analysis we know that the positive semidefiniteness 
of the matrix \eqref{Choi} is necessary and sufficient for there to 
be a solution $s \in {\mathcal SA}_{Q_{\rm row}}({\mathbb C})$ of the 
interpolation condition \eqref{twistedLTOAprob}.

The same data set $(\oplus_{\kappa \in {\mathfrak B}}Z^{(0)}, E, 
(\bigoplus_{1}^{n} \Lambda_{0})E)$ is the data set 
for a (untwisted) LTOA interpolation problem:  {\em find $s \in 
{\mathcal SA}_{Q_{\rm row}}({\mathbb C})$ such that}
$$
    (E s)^{\wedge L}\left( \oplus_{\kappa \in {\mathfrak B}} 
    Z^{(0)}\right) = 
    \left( \bigoplus_{\kappa \in {\mathfrak B}}\Lambda_{0} \right) E.
$$
or equivalently, {\em such that}
\begin{equation}   \label{LTOAprob}
  \sum_{\fa \in \free} s_{\fa^{\top}} Z^{\fa} = \Lambda_{0}.
\end{equation}
The solution  criterion for this problem is positive semidefiniteness 
of the block matrix
\begin{equation}   \label{Choi'}
  \left[\sum_{\fa \in \free} Z^{(0) \fa} (e_{\kappa} 
  e_{\kappa'}^{*} - \Lambda_{0} e_{\kappa} e_{\kappa'}^{*} 
  \Lambda_{0}^{*} ) Z^{(0)* \fa^{\top}}   \right]_{\kappa, \kappa' 
  \in {\mathfrak B}}.
\end{equation}
or equivalently (by the Choi test), complete positivity of the map
\begin{equation}   \label{Choi''}
    P \mapsto \sum_{\fa \in \free} Z^{(0) \fa} (P - \Lambda_{0} 
    P \Lambda_{0}^{*}) Z^{(0) * \fa^{\top}}.
\end{equation}
Note that the problems \eqref{twistedLTOAprob} and \eqref{LTOAprob} 
are the same in case the components $Z^{(0)}_{1}, \dots, Z^{(0)}_{d}$ 
commute with each other (so $Z^{\fa} = Z^{\fa^{\top}}$).  A 
consequence of the ``respects intertwinings'' condition for nc 
functions is that $\Lambda_{0}$ must be in the double commutant of 
the collection $Z^{(0)}_{1}, \dots, Z^{(0)}_{d}$ if $\Lambda_{0} = 
S(Z^{0)})$ for a nc function $S$.  Thus, for the case of commutative 
$d$-tuple $Z^{(0)}$, positivity of the matrix 
\eqref{Choi} in fact implies that $\Lambda^{(0)}$ commutes with each 
$Z^{(0)}_{k}$ and the matrices \eqref{Choi} and \eqref{Choi'} are the 
same, consistent with positivity of either being the solution 
criterion for existence of a solution to the same problem 
\eqref{scalarint} or \eqref{LTOAprob}.  In case $Z^{(0)}$ is not a 
commutative tuple, we are led to the conclusion that the 
interpolation conditions \eqref{scalarint} and \eqref{LTOAprob} are 
different problems with each having its own independent solution 
criterion, positive semidefiniteness of \eqref{Choi} and of 
\eqref{Choi'} respectively.  For the case of commuting variables, 
nc-function point-evaluation (or Riesz-Dunford) interpolation conditions
can be reduced to the older theory of LTOA interpolation conditions 
and one can recover the solution criterion of one from the solution 
criterion for the other; this point is explored in more detail in 
\cite{BtH09}.  The recent paper of Norton \cite{Norton} explores 
similar connections between the interpolation theory of 
Constantinescu-Johnson \cite{CJ03} and that of Muhly-Solel 
\cite{MS2004}.

\smallskip

\noindent
\textbf{Acknowledgement:}  The research of the first and third authors was partially 
supported by the US-Israel Binational Science Foundation.  It is also  a pleasure to 
acknowledge the contribution of Orr Shalit for penetrating discussions leading to 
the observations in Section \ref{S:commutative}.

\end{document}